\documentclass[amssymb,amsfonts,reqno]{amsart}

 \usepackage{xcolor}
 \usepackage[citebordercolor={green}]{hyperref}

\usepackage{color}
\usepackage{latexsym}
\usepackage{amsmath}
\usepackage{amssymb}
\usepackage{amsfonts}
\usepackage{esint}
\usepackage{graphicx}
\usepackage{accents}
\newcommand{\Lie}{{\mathcal{L}}}

\newcommand{\der}{\nabla}

\newcommand{\les}{\lesssim}
\newcommand{\bea}{\begin{eqnarray}}
\newcommand{\eea}{\end{eqnarray}}

\newcommand{\derm}{ { \der^{(\bf{m})}} }
\newcommand{\rderm}{   {\mbox{$\nabla \mkern-13mu /$\,}^{(\bf{m})} }   }

\newcommand{\eps}{{\varepsilon}}\newcommand{\R}{{\mathbb R}}

\newcommand{\E}{{\cal E}}

\newcommand{\la}{\langle}\newcommand{\si}{\sigma}\renewcommand{\b}{\beta}

\newcommand{\cal}{\mathcal}

\def\a{\alpha}\def\ga{\gamma}\def\de{\delta}\def\Si{\Sigma}
\def\bm{\left( \begin{array}{cc}}
\def\endm{\end{array}\right)}\newcommand{\eq}{\end{equation}}
\def\a{\alpha}\def\b{\beta}
\def\ga{\gamma}\def\de{\delta}\def\Box{\square}\def\pa{\partial}

\def \rectangle#1#2{\hbox{\vrule\vbox to #2 {\hrule\hbox to #1{\hfil}\vfil\hrule}\vrule}}

\def\a{\alpha}\def\b{\beta}\def\ga{\gamma}
\def\de{\delta}\def\Box{\square}\def\pa{\partial}

\def\pa{\partial}

\def\beaa{\begin{eqnarray*}}
\def\eeaa{\end{eqnarray*}}
\def\pa{\partial}
\def\a{{\alpha}}
\def\b{{\beta}}
\def\ga{\gamma}
\def\Ga{\Gamma}
\def\de{\delta}

\def\eps{\epsilon}
\def\la{\lambda}
\def\La{\Lambda}
\def\si{\sigma}
\def\Si{\Sigma}

\def\Om{\Omega}

\def\g{{\bf g}}

\def\SSS{{\Bbb S}}

\def\nn{{\mathbb N}}

\def\R{{\mathbb R}}

\def \p{ \partial}

\def\12{\frac{1}{2}}

\def\N{{\mathcal N}}

\def\bep{\begin{proposition}}
\def\eep{\end{proposition}}

\def\4{\frac{1}{4}}

\def\i{{\rm i}}

\def \p{ \partial}

\def\12{\frac{1}{2}}

\def\N{\nn}

\def\bep{\begin{proposition}}
\def\eep{\end{proposition}}

\def\d{\text{d}}

\def\bm#1{\boldsymbol{#1}} 

\def\build#1_#2^#3{\mathrel{\mathop{\kern 0pt#1}\limits_{#2}^{#3}}}
\def\4{\frac{1}{4}}

\def\i{{\rm i}}

\def\<{\langle}
\def\>{\rangle}

\theoremstyle{plain}
\newtheorem{theorem}{Theorem}
\newtheorem{proposition}{Proposition}
\newtheorem{lemma}{Lemma}
\newtheorem{corollary}{Corollary}

\theoremstyle{remark}
\newtheorem{remark}{Remark}

\theoremstyle{definition}
\newtheorem{definition}{Definition}
\numberwithin{equation}{section}
\numberwithin{proposition}{section}
\numberwithin{definition}{section}
\numberwithin{lemma}{section}
\numberwithin{corollary}{section}
\numberwithin{remark}{section}
\begin{document}
\include{psfig}
\title[Einstein-Yang-Mills in Higher Dimensions]{The global stability of the Minkowski space-time solution to the Einstein-Yang-Mills equations in higher dimensions}
\author{Sari Ghanem}
\address{University of Lübeck}
\email{sari.ghanem@uni-luebeck.de}
\maketitle

\begin{abstract}
This is a first in a series of papers in which we study the stability of the $(1+n)$-Minkowski space-time, for $n \geq 3$\,, solution to the Einstein-Yang-Mills equations, in both the Lorenz and harmonic gauges, associated to any arbitrary compact Lie group $G$\,, and for arbitrary small perturbations. In this first, we prove global stability of the Minkowski space-time, $\R^{1+n}$\,, in higher dimensions $n \geq 5$ (both in the interior and in the exterior); in the paper that follows, we prove exterior stability for $n=4$\,; and its sequel, we prove exterior stability for $n=3$\,, and in all these cases, stability is studied as a solution to the fully coupled Einstein-Yang-Mills system in the Lorenz and harmonic gauges. We show here that for $n \geq 5$\,, the $\R^{1+n}$ Minkowski space-time in wave coordinates is stable as solution to the Einstein-Yang-Mills system in the Lorenz gauge on the Yang-Mills potential, for sufficiently small perturbations of the Einstein-Yang-Mills potential and metric, and leads to a global Cauchy development. We also obtain dispersive estimates in wave coordinates on the gauge invariant norm of the Yang-Mills curvature, on the Yang-Mills potential in the Lorenz gauge, and on the perturbations of the metric. In this manuscript, we detail all the material of our proof so as to provide lecture notes for Ph.D. students wanting to learn the Cauchy problem for the Einstein-Yang-Mills system.
\end{abstract}

\setcounter{page}{1}
\pagenumbering{arabic}


\section{Introduction}

This is a first paper, in a series of three papers where we study the non-linear stability of the Minkowski space-time solution to the Einstein-Yang-Mills equations in $(1+n)$-dimensions, where $n \geq 3$ is the number of space dimensions. In this first paper, we prove the global non-linear stability of Minkowski space-time for $n\geq 5$\,, however we carry out the computations with parameters that will be of use for the third paper concerning $n=3$\,, although these parameters will be chosen trivial both in this case of $n\geq 5$ and in the case of $n = 4$ in the paper that follows. We also define, in this paper, the Cauchy problem for the fully coupled Einstein-Yang-Mills system in generality for $n\geq 3$\,, so as to refer to it in the following papers. 

The problem that we look at, is that of the perturbation of the Minkowski space-time under the evolution problem in General Relativity with matter of which the governing equations are the Einstein-Yang-Mills equations
\bea\label{TheEinsteinYangMillsequationsasEinsteinwrotethem}
R_{ \mu \nu} - \frac{1}{2} g_{\mu\nu} \cdot R &=& 8\pi \cdot T_{\mu\nu} \; ,
\eea
where $T_{\mu\nu}$ is the Yang-Mills stress-energy-momentum tensor (see \eqref{definitionofstressenergymomenturmYangMillsmatter}), prescribed by the unknown Yang-Mills curvature $F$ given by (see \eqref{expF}),
\bea
F_{\a\b} = \der_{\a}A_{\b} - \der_{\b}A_{\a} + [A_{\a},A_{\b}]   \; ,
\eea
where $A$ is the unknown Yang-Mills potential valued in the Lie algebra $\cal G$ associated to the Lie group $G$\;, and where $\der_{\alpha}$ is the unknown space-time covariant derivative of Levi-Civita, prescribed by the unknown metric $\g$\;.

However, the Einstein-Yang-Mills equations \eqref{TheEinsteinYangMillsequationsasEinsteinwrotethem} imply the Yang-Mills equations (see \eqref{ThegaugecovariantdivergenceoftheYangMillscurvatureisequaltozero}), namely
\bea
 \der_{\alpha} F^{\a\b}  + [A_{\alpha}, F^{\a\b} ]  = 0 \; .
\eea
Thus, the Einstein-Yang-Mills system on $(\cal M, F, \g)$, is the following (see \eqref{simpEYM})
\bea
\begin{cases} \label{EYMsystemforintro}
R_{ \mu \nu}  =& 2 < F_{\mu\b}, F_{\nu}^{\;\;\b} > +   \frac{1}{(1-n)} \cdot g_{\mu\nu } \cdot < F_{\a\b },F^{\a\b } >\; ,   \\
0 \;\;\;\, \,=&  \der_{\alpha} F^{\a\b}  + [A_{\alpha}, F^{\a\b} ]    \;, \\
F_{\a\b} =& \der_{\a}A_{\b} - \der_{\b}A_{\a} + [A_{\a},A_{\b}] \; .  \end{cases} 
      \eea

The Einstein-Yang-Mills equations form an overdetermined system, not any initial data set leads to a Cauchy development. The initial data set (see Subsection \ref{Cauchyproblem}), namely $(\Sigma, \overline{A}, \overline{E}, \overline{g}, \overline{k})$\;, must satisfy the Einstein-Yang-Mills constraint equations which arise from the Gauss-Codazzi equations (see Lemma \ref{Gauss-Codazzi-equations}), as well as the Yang-Mills constraint equations (see Lemma \ref{TheconstraintequationsfortheEinstein-Yang-Mills system}). The Einstein-Yang-Mills constraints for the initial data are 
\bea
\begin{cases} 
 \mathcal{R}+ \overline{k}^i_{\,\, \, i} \overline{k}_{j}^{\,\,\,j}  -  \overline{k}^{ij} \overline{k}_{ij}    &=    \frac{4}{(n-1)}  < \overline{E}_{i}, \overline{E}^{ i}>    +  < \overline{F}_{ij },\overline{F}^{ij } > \, ,\\
\overline{D}_{i} \overline{k}^i_{\,\,\, j}    - \overline{D}_{j} \overline{k}^i_{\, \,\,i}   &=  2 < \overline{E}_{i}, {\overline{F}_{j}}^{\, i} > \, ,\\
\overline{D}^i \overline{E}_{ i} + [\overline{A}^i, \overline{E}_{ i} ]  &= 0 \, , \end{cases} 
\eea

where $\overline{F}$ is prescribed by $\overline{A}$ through 
\bea
\overline{F}_{\a\b} =  \overline{D}_{\a}\overline{A}_{\b} -  \overline{D}_{\b}\overline{A}_{\a} + [\overline{A}_{\a},\overline{A}_{\b}]   \; ,
\eea
and where $\mathcal{R}$ is given by contracting in \eqref{restrictedRiemanntensor}. Here $\overline{D}$ is defined as the Levi-Civita connection prescribed by Riemannian metric $\overline{g}$\;, and we raised indices with respect to the $\overline{g}$. Then, we are in fact looking for a Lorentzian metric $\g$\;, and therefore for $\der$\;, and for a Manifold $\cal M$\;, and therefore for $\hat{t}$ (see Definition \ref{definitionoftheunitorthogonaltimelikevectorefieldhatt}), such as on $\Sigma$\;, we have $\overline{D} = D$ (defined in \eqref{defrestrictedcovariantderivative}), and we have $\overline{k}  = k$ (see Definition \ref{definitionsosecondfundamentalform}), and we have $\overline{A} = A$ and $\overline{E}_{i} = F_{\hat{t} i}$\;.

Since the Einstein-Yang-Mills equations are invariant under gauge transformations (see Subsection \ref{The invariance under gauge transformation}) and under change of system of coordinates (see Subsection \ref{The diffeomorphism invariance}), we need to fix the system of coordinates and the gauge in order to make a precise statement on decay of the fields, which are the metric and the Yang-Mills potential. We choose to work in the Lorenz gauge (see \eqref{TheLorenzgaugeconditionthatwechoose}) and in wave coordinates (see \eqref{Thewavecoordinateconditionthatwechooseforoursystemofcoordinates}). The use of wave coordinates dates back to the celebrated work of Choquet-Bruhat, \cite{CB1}, where she proved existence of a maximal Cauchy development for the Einstein vacuum equations for sufficiently smooth initial data. Whereas to the Lorenz gauge, it is here being used to be able to make a statement on decay for the Yang-Mills potential.

However, if one chooses to work in the Lorenz gauge, namely $\der^{\a}   A_{\a}  =  0$\;, then the Yang-Mills equations implied by the Einstein-Yang-Mills equations (see \eqref{ThegaugecovariantdivergenceoftheYangMillscurvatureisequaltozero}), namely $ \textbf{D}^{(A)}_{\a}F^{\a\b}  = 0$\;, imply a system of non-linear wave equations in wave coordinates, on the Yang-Mills potential (see Lemma \ref{waveequationontheYangMillspotentialwithspurcesdependingonthemetricg}) with sources depending on both the Yang-Mills potential $A$ and the metric $g$\;. Furthermore, in wave coordinates, the Einstein-Yang-Mills equations (the original fields equations), imply a system of non-linear wave equations on the metric $g$ (see \eqref{waveequationonthemetricgwithsourcesonthemetricgdependingonRiccitensor}), with sources depending on the Ricci tensor $R$\;, that is here non-vanishing since we are treating the Einstein equations with matter, namely the Yang-Mills fields, which in its turn lead sources depending again on the Yang-Mils potential $A$ and the metric $g$\;. In fact, since we are interested in perturbations of the Minkowski space-time, the evolution problem that we are interested in, is on one hand that for the difference $h:= g - m$\;, where $m$ is defined to be the Minkowski metric $(-1, +1, \ldots, +1)$ in wave coordinates, and on the other hand, that for the Yang-Mills potential in the Lorenz gauge $A$\;.

The advantage of the use of both the wave coordinates (also referred to as the harmonic gauge) and of the Lorenz gauge, is that the field equations simplify to a system of coupled non-linear hyperbolic wave equations on both the unknown Yang-Mills potential $A$ and the unknown metric $h$ (see Lemma \ref{EYMsystemashyperbolicPDE}, or see Lemmas \ref{waveequationontheEinsteinYangMillspotentialderivedfromthegaugecovariantdivergenceoftheYangMillscurvatureisequaltozero} and \ref{waveequationontheEinsteinYangMillsmetricsmallhwithsourcesusingtheRiccitensorthatwascomoutedearlier}). 

Yet, we need to transform the initial data set $(\Sigma, \overline{A}, \overline{E}, \overline{g}, \overline{k})$ into an initial data of the type $(\Sigma, A_\Sigma, \pa_t A_\Sigma, g_\Sigma, \pa_t g_\Sigma)$, suitable for the considered coupled system of non-linear wave equations (given in Lemma \ref{EYMsystemashyperbolicPDE}), so as to give a hyperbolic formulation for the Cauchy problem.  

We are going to construct the initial data set $(\Sigma, A_\Sigma, \pa_t A_\Sigma, g_\Sigma, \pa_t g_\Sigma)$, considering that on one hand, the solution of the Einstein-Yang-Mills system that we are looking for is gauge invariant for both gauge transformations on $A$ and for diffeomorphisms on the system of coordinates (see Section \ref{The gauge conditions}), and on the other hand, in consistency with the fact that we are writing our equations in the Lorenz gauge and in wave coordinates conditions. Let us explain:

Given the gauge invariance of the Yang-Mills system (see Lemma \ref{gaugeinvarianceoftheYangMillssystem}), we can choose to look at our initial data for the Yang-Mills potential to be a section of the unknown solution $A$\;, such that $A_t = 0$ (only for the initial data $A_\Sigma$), which is a condition that will not necessarily be preserved for the evolution of $A_\Sigma$\;, namely $A$\;. Also, given the diffeomorphism invariance of the solution (see Subsection \ref{The diffeomorphism invariance}), we can choose which Cauchy hypersurface in the manifold $\cal M$\;, we would like our given initial data slice $\Sigma$ to ultimately be. We choose that we would like $\Sigma$ to be in $\cal M$ in a way such that $\pa_t$ is orthogonal to $\Sigma \subset \cal M$ (that is a condition that will not be preserved for the evolution of $\Sigma$, namely $\Sigma_t$). Differently speaking, one can always make a gauge transformation on $A$\;, and a diffeomorphism on $\cal M$\;, such that the initial data satisfies the conditions $A_t = 0$ (on $\Sigma$) and $g_{ti} = 0$ for spatial indices (on $\Sigma$). 

However, we wanted to look for a solution in both the Lorenz gauge and in wave coordinates. Thus, once we decided to look at our initial data set in a way that leads to $A_\Sigma$ and $g _\Sigma$ to be of the kind that we have just described (\eqref{initialdatadforzerothderivativeAsigma} and \eqref{constructionoggsigma}), namely with the properties that $(A_\Sigma)_t =0 $ and $({g_\Sigma})_{ti} = 0$\;, we can then proceed forward to construct $\pa_t A_\Sigma$ and $\pa_t g_\Sigma$ in consistency with the Lorenz gauge and the wave coordinates condition, which is possible for us to do, because in fact $\pa_t A_\Sigma$ and $\pa_t g_\Sigma$ are not part of the initial data set -- we just need to see what these gauges impose on $\pa A_\Sigma$ and $\pa g_\Sigma$ and deduce the expressions of their time partial derivatives in terms of $\overline{A}$\;, $\overline{E}$\;, $\overline{g}$\;, and $\overline{k}$, which are given from the initial data set (see \eqref{constructionopatialtimeAonsigma} and \eqref{constructionopatialtgonsigma}). 

It is not sufficient that the “new” initial data set that we constructed, namely $(\Sigma, A_\Sigma, \pa_t A_\Sigma, g_\Sigma, \pa_t g_\Sigma)$\;, is in the Lorenz gauge and in wave coordinates -- these gauges conditions will not necessarily be preserved for all time $t$\,, by the hyperbolic system of evolution that we deduced in \eqref{ThewaveequationontheYangMillspotentialwithhyperbolicwaveoperatorusingingpartialderivativesinwavecoordinates} and \eqref{Thewaveequationonthemetrichwithhyperbolicwaveoperatorusingingpartialderivativesinwavecoordinates} (in Lemma \ref{EYMsystemashyperbolicPDE}) by making implications on the original Einstein-Yang-Mills system \eqref{EYMsystemforintro}, assuming “sometimes” and in the first place that the solution will be in the Lorenz gauge and wave coordinates during the evolution. Let us explain:

The fact that we used “sometimes” the Lorenz gauge and the wave coordinates conditions, in order to simplify our original Einstein-Yang-Mills system \eqref{EYMsystemforintro}, only gives us an implication on the solution (implication given in \eqref{ThewaveequationontheYangMillspotentialwithhyperbolicwaveoperatorusingingpartialderivativesinwavecoordinates} and \eqref{Thewaveequationonthemetrichwithhyperbolicwaveoperatorusingingpartialderivativesinwavecoordinates} in Lemma \ref{EYMsystemashyperbolicPDE}), and this is if such a solution of \eqref{EYMsystemforintro} exists for all time $t$ while being in the Lorenz gauge and in wave coordinates simultaneously. Now, the question is: how do we know that solving the simplified system (that we derived by using “sometimes” the Lorenz gauge and wave coordinates condition to simplify \eqref{EYMsystemforintro}) gives rise to an actual solution of the original Einstein-Yang-Mills system \eqref{EYMsystemforintro} and that is indeed in the Lorenz gauge and in wave coordinates for all time $t$?

In fact, we are going to show that there is indeed a way to solve the Einstein-Yang-Mills system in the hyperbolic formulation (given in Lemma \ref{EYMsystemashyperbolicPDE}), where the evolution in time gives rise to a solution of the original Einstein-Yang-Mills system \eqref{EYMsystemforintro}, that will always be in the Lorenz gauge and in wave coordinates for all time $t$\;. Let us explain how we do that:

We shall in fact show that for a solution of the simplified coupled non-linear wave equations (namely, \eqref{ThewaveequationontheYangMillspotentialwithhyperbolicwaveoperatorusingingpartialderivativesinwavecoordinates} and \eqref{Thewaveequationonthemetrichwithhyperbolicwaveoperatorusingingpartialderivativesinwavecoordinates} in Lemma \ref{EYMsystemashyperbolicPDE}), the Einstein-Yang-Mills system \textit{implies} a system of non-linear wave equations for both the Lorenz gauge and the wave coordinate gauge conditions. We will also show that for our initial data $(\Sigma, A_\Sigma, \pa_t A_\Sigma, g_\Sigma, \pa_t g_\Sigma)$\;, constructed in consistency with the Lorenz and wave coordinates gauges, it is \textit{precisely} the Einstein-Yang-Mills constraint equations (given in Lemma \ref{TheconstraintequationsfortheEinstein-Yang-Mills system}) that will give us that the initial conditions for the propagation (through the Einstein-Yang-Mills system \eqref{EYMsystemforintro}) of the Lorenz and wave coordinates gauges are null. Thus, by starting with an initial data set $(\Sigma, \overline{A}, \overline{E}, \overline{g}, \overline{k})$ that satisfies the Einstein-Yang-Mills constraint equations (given in Lemma \ref{TheconstraintequationsfortheEinstein-Yang-Mills system}), we have constructed a “new” hyperbolic initial data set (in Subsections \ref{IntialdataforYangMills} and \ref{constructioninitialdataformetric}) in consistency with the Lorenz and wave coordinates conditions, for our non-linear coupled wave equations  \eqref{ThewaveequationontheYangMillspotentialwithhyperbolicwaveoperatorusingingpartialderivativesinwavecoordinates} and \eqref{Thewaveequationonthemetrichwithhyperbolicwaveoperatorusingingpartialderivativesinwavecoordinates}, and we show that this will give rise to a solution for which the original Einstein-Yang-Mills system will guarantee to us that the Lorenz and wave conditions are null: the Einstein-Yang-Mills system will imply that the gauges propagate by non-linear wave equations, with null initial data precisely because on one hand, we constructed our hyperbolic initial data in consistency with the Lorenz and wave coordinates gauges, and because on the other hand, we started with an initial data the solves the Einstein-Yang-Mills constraints.

Thus, we have showed that if a solution to the Einstein-Yang-Mills system \eqref{EYMsystemforintro} is in the Lorenz and wave coordinates gauges, then it must solve the non-linear hyperbolic system in \eqref{ThewaveequationontheYangMillspotentialwithhyperbolicwaveoperatorusingingpartialderivativesinwavecoordinates} and \eqref{Thewaveequationonthemetrichwithhyperbolicwaveoperatorusingingpartialderivativesinwavecoordinates} (given in Lemma \ref{EYMsystemashyperbolicPDE}), and we showed that for such a solution, the Einstein-Yang-Mills system implies that it is indeed in the Lorenz and wave coordinates gauges, and therefore, that it is also a solution for the original Einstein-Yang-Mills system (in the Lorenz and wave coordinates gauges). Consequently, solving the Einstein Yang-Mills system \eqref{EYMsystemforintro} in the Lorenz and wave coordinates gauges, with the initial data satisfying the Einstein-Yang-Mills constraints, is equivalent to solving the non-linear hyperbolic system (given in Lemma Lemma \ref{EYMsystemashyperbolicPDE}) with an initial data constructed in \eqref{initialdatadforzerothderivativeAsigma}, \eqref{constructionopatialtimeAonsigma}, \eqref{constructionoggsigma}, \eqref{constructionopatialtgonsigma}, in consistency with the Lorenz and wave coordinates gauges and with the Einstein-Yang-Mills constraints.

In fact, in order for us to construct a solution to the Yang-Mills equations \eqref{ThegaugecovariantdivergenceoftheYangMillscurvatureisequaltozero}, that is in the Lorenz gauge for the Yang-Mills potential, given the initial data set, all what we need to do is to solve the wave equation \eqref{ThewaveequationontheYangMillspotentialwithhyperbolicwaveoperatorusingingpartialderivativesinwavecoordinates}, that reads the equation that we show in Lemma \ref{equationtodefinepropgationofAtorespectLorenzgauge}. Then, Lemma \ref{LemmathatgaranteeshowtheLorenzgaugewillbepreserved} will tell to us that the original Yang-Mills equations \eqref{ThegaugecovariantdivergenceoftheYangMillscurvatureisequaltozero} implies that the Lorenz gauge will propagate in time $t$ through a non-linear wave equation (see \eqref{nonlinearwaveequationontheLorenzGaugeCondition}), and that the initial conditions for the propagation of the Lorenz gauge are null thanks to the fact that we started with a hyperbolic initial data that is both consistent with the Lorenz gauge and that satisfies the Yang-Mills constraints \eqref{theYangMillsconstraintforthepotentialandelectricchargeonsigma}. 

Also, in order for us to construct a solution to the Einstein-Yang-Mills equations \ref{TheEinsteinYangMillsequationsasEinsteinwrotethem}, that is in the wave coordinates gauge, given the initial data set, all what we need to do is to solve the wave equation on the metric \eqref{Thewaveequationonthemetrichwithhyperbolicwaveoperatorusingingpartialderivativesinwavecoordinates}. Now, thanks to Lemma \eqref{nonlinearwaveequationonGaforthepropagationofthewavecoordicondition}, we know that the Einstein-Yang-Mills equations \eqref{EYMsystemforintro} will imply, for such a solution, a system of non-linear wave equations on the wave coordinate condition (see \eqref{systemofnonlinearwaveeqonthewavecoordicondition}), a condition that we can write as a tensor (see Remark \ref{remarkondefinitionofGa} and \eqref{definitionofGainwavecoordinates}). Then, Lemma \ref{initialconditionsonthederivativesofGa} will tell us that the derivatives on the initial slice $\Sigma$ of the tensor that gives the wave coordinate gauge are null because the initial data satisfies the Einstein-Yang-Mills constraints \eqref{theEinsteinYangMillsconstriantsonSigmafirst} and \eqref{theEinsteinYangMillsconstriantsonSigmasecond}. Also, we constructed the hyperbolic initial data in consistency with the wave coordinates gauge and therefore the zeroth derivative of the tensor that gives the wave coordinate is null. Thus, the Einstein-Yang-Mills equations \eqref{EYMsystemforintro} will read exactly a non-linear wave equation on the propagation of the wave coordinates gauge, with null initial data.

Consequently, we have proved Corollary \ref{thefinalconstructiongivenconstriantsontheinitialdatatogaranteeLorenzandharmonicgauges}, that gives us a way to solve the Einstein-Yang-Mills system in both the Lorenz and wave coordinate gauges, by solving non-linear coupled wave equations on both the Yang-Mills potential and on perturbations of the metric, provided an initial data set that satisfies the Einstein-Yang-Mills constraints (given in Lemma \ref{TheconstraintequationsfortheEinstein-Yang-Mills system}).

However, instead of working in fixed wave coordinates, we can write the equations more geometrically, by viewing them as a system of tensorial wave equations, by defining covariant derivatives with respect to the metric $m$\;, and not $g$\;, which we write as $\derm$\; (definied in Definition \ref{definitionoftheMinkowskicovaruiantderivative}), and then look at the corresponding tensorial covariant wave operator we are interested in, which is $ g^{\a\b} \derm_\a \derm_\b $\;.

This leads in the Lorenz gauge, to a coupled system of tensorial covariant hyperbolic operators, with coupled non-linear sources, where this time, the fact that we privilege wave coordinates condition is hidden in the definition of the tensorial covariant derivative  $\derm$\;. It is precisely the study of the structure of these non-linear source terms, of both $ g^{\a\b} \derm_\a \derm_\b A_\si $ and $ g^{\a\b} \derm_\a \derm_\b h_{\mu\nu} $\;, that would allow us to make a statement about the dispersive estimates of the fields.

In the Lorenz gauge, we get the following system of coupled covariant tensorial wave equations on both $A$ and $h$ (see Lemma \ref{structureofthesourcetermsofthewaveequationonpoentialandmetricinLorenzandwavecoordinforEinsteinYangMillssystem}), where we lower and higher indices with respect to the metric $m$\;, where wave coordinates are hidden in the definition of $\derm$ (being the covariant derivative of the Minkowski space-time, that is defined to be Minkowski in wave coordinates),
  \bea
  \notag
 && g^{\la\mu} \derm_{\la}   \derm_{\mu}   A_{\si}     \\
   \notag
 &=&    ( \derm_{\si}  h^{\a\mu} )  \cdot (  \derm_{\a}A_{\mu} )  \\
    \notag
 &&   +   \frac{1}{2}    \big(   \derm^{\mu} h^{\nu}_{\, \, \, \si} +  \derm_\si h^{\nu\mu} -   \derm^{\nu} h^{\mu}_{\,\,\, \si}  \big)   \cdot  \big( \derm_{\mu}A_{\nu} -  \derm_{\nu}A_{\mu}  \big) \\
 \notag
&&     +   \frac{1}{2}    \big(   \derm^{\mu} h^{\nu}_{\, \, \, \si} +  \derm_\si h^{\nu\mu} -   \derm^{\nu} h^{\mu}_{\,\,\, \si}  \big)   \cdot   [A_{\mu},A_{\nu}] \\
 \notag
 && -  \big(  [ A_{\mu}, \derm^{\mu} A_{\si} ]  +    [A^{\mu},  \derm_{\mu}  A_{\si} - \derm_{\si} A_{\mu} ]    +    [A^{\mu}, [A_{\mu},A_{\si}] ]  \big)  \\
 \notag
  && + O( h \cdot  \derm  h \cdot  \derm A) + O( h \cdot  \derm h \cdot  A^2) + O( h \cdot  A \cdot \derm A) + O( h \cdot  A^3) \, ,\\
  \eea
 and
  \bea
\notag
 && g^{\alpha\beta}\derm_\alpha \derm_\beta h_{\mu\nu} \\
 \notag
  &=& P(\derm_\mu h,\derm_\nu h)  +  Q_{\mu\nu}(\derm h,\derm h)   + G_{\mu\nu}(h)(\derm h,\derm h)  \\
\notag
 &&   -4     <   \derm_{\mu}A_{\b} - \derm_{\b}A_{\mu}  ,  \derm_{\nu}A^{\b} - \derm^{\b}A_{\nu}  >    \\
 \notag
 &&   + m_{\mu\nu }       \cdot  <  \derm_{\a}A_{\b} - \derm_{\b}A_{\a} , \derm_{\a} A^{\b} - \derm^{\b}A^{\a} >   \\
 \notag
&&           -4  \cdot  \big( <   \derm_{\mu}A_{\b} - \derm_{\b}A_{\mu}  ,  [A_{\nu},A^{\b}] >   + <   [A_{\mu},A_{\b}] ,  \derm_{\nu}A^{\b} - \derm^{\b}A_{\nu}  > \big)  \\
\notag
&& + m_{\mu\nu }    \cdot \big(  <  \derm_{\a}A_{\b} - \derm_{\b}A_{\a} , [A^{\a},A^{\b}] >    +  <  [A_{\a},A_{\b}] , \derm^{\a}A^{\b} - \derm^{\b}A^{\a}  > \big) \\
\notag
 &&  -4     <   [A_{\mu},A_{\b}] ,  [A_{\nu},A^{\b}] >      + m_{\mu\nu }   \cdot   <  [A_{\a},A_{\b}] , [A^{\a},A^{\b}] >  \\
 \notag
     && + O \big(h \cdot  (\derm A)^2 \big)   + O \big(  h  \cdot  A^2 \cdot \derm A \big)     + O \big(  h   \cdot  A^4 \big)  \,  , \\
\eea

where $P$\;, $Q$ and $G$ are defined in \eqref{definitionofthetermbigPinsourcetermsforeinstein}, \eqref{definitionofthetermbigQinsourcetermsforeinstein} and \eqref{definitionofthetermbigGinsourcetermsforeinstein}, and where here, the notation $O$\;, for the zeroth Lie derivative of the given tensors, is defined in Definition \ref{definitionofbigOonlyforAandhadgradientofAandgardientofh}, which is a somewhat different notation than the one we use for the Lie derivatives of these tensors in Definition \ref{definitionofbigOforLiederivatives} (see Remark \ref{remarkaboutthedifferenceofdefinitionofbigOwhenwetalkaboutLiederivatives}). As far as this notation is concerned, let us point out in this whole paper, and in the ones that follow on these mathematical problems, when we write partial derivatives, namely $\pa$\;, this means that we fixed already the system of coordinates to be the wave coordinates, and when we write $\derm$, it is a a different way to see this geometrically, as tensors, which is sometimes useful for computation. Hence, the definition of the norms is also given along those lines, where there is either explicit, or implicit, choice of wave coordinates (see Definition \ref{definitionoftheMinkowskicovaruiantderivative}, and see \eqref{equationofthedefinitionofthenormofgradienteitheraspartialorcovariant}).

For certain systems of non-linear hyperbolic equations, such as this one at hand, one can prove that a local solution exists, under certain regularity assumptions on the initial data. One can also prove, as it is well-known, that such a local solution either exists for all time, or blows up in finite time if a higher order energy norm (such as the one defined in \eqref{definitionoftheenergynorm}) blows up. In other words, one has a global solution for all time if a higher order energy norm stays finite. Thus, proving the finiteness for all time of such a higher order energy norm, \eqref{definitionoftheenergynorm}, allows one to conclude that the local solution is in fact a global one.

An important feature of this higher order energy norm is that for non-linear hyperbolic differential equations which are locally well-posed, the time dependance of this higher order energy is continuous: it depends continuously on time. Thus, one looks at the maximal time such that the local solution’s higher energy norm is bounded by a certain constant, say $C$, and if one then proves that one has actually a better bound, say $\frac{C}{2}$, then this proves that the maximal time for which the higher order energy is bounded was in fact not maximal, or differently speaking, it is in fact infinity for the time, i.e. the higher order energy is bounded indeed for all time, and therefore does not blow up. Hence, this proves that the local solution of the locally well-posed non-linear hyperbolic equation, is in fact a global solution for all time (see Subsection \ref{Thebootstrapargumentandnotationonboundingtheenergy}). Such an argument is called a continuity argument or a bootstrap argument: one starts with an a priori estimate on the higher order energy (see Subsection \ref{Theassumptionforthebootstrapandthenotation}) and then one improves this a priori estimate and therefore one concludes, given the fact that the time dependance is continuous, that the a priori estimate is in fact a true estimate on this higher order energy norm. 
 
On the top of that, if one bounds this higher order energy using a bootstrap argument, as described above, or whatever argument that works, one can then use the Klainerman-Sobolev inequality (see \eqref{wksi}), that tells us that if we bound a certain weighted higher order norm (if this is the norm that was being used in the bootstrap argument), then one gets also pointwise decay in time of the solution, with a decay rate that depends on the space dimension $n$ of the space-time. Hence, this way, one can also get global dispersive estimates.

To effectively run such a bootstrap argument (see Subsection \ref{Thebootstrapargumentandnotationonboundingtheenergy}), it all depends on improving the a priori estimate on this weighted higher order energy norm. For this, one has to study the non-linear structure of the source terms in the non-linear hyperbolic equation (see Lemmas \ref{usingbootstraoptoshowstructorforsourcesonYangMillspotentialA} and \ref{boostraptostudystructureofsourcetermsonthemetricperturbationh}), as well as the structure of the wave operator itself: it is not the flat wave operator, but it is a wave operator that depends on the solution itself (see Lemma \ref{structureofthesourcetermsofthewaveequationonpoentialandmetricinLorenzandwavecoordinforEinsteinYangMillssystem}).

More precisely, the higher order energy norm in question is a certain norm of the gradient of the Lie derivatives in the direction of the Minkowski vector fields of the local solution (see \eqref{definitionoftheenergynorm}). Since we are talking about a wave operator that depends on the solution itself, commuting the wave operator with the Lie derivatives in the direction of the Minkowski vector fields, gives a structure that depends on the wave operator itself and on the solution (see Lemma \ref{thecommutatortermusingthebootstrap}). Also, such a commutation gives a quantity that depends on the Lie derivatives in the direction of the Minkowski vector fields of the source terms of the wave operator. Thus, one has to study these Lie derivatives of the source terms of the non-linear hyperbolic wave equation in order to bound the higher order energy norm. 

Speaking of bounding the higher order energy norm, one applies a conservation law (see Lemma \ref{Conservationlawwithweightforwaveequations}), that is nothing else but the divergence theorem applied to suitably chosen tensors so as the boundary terms would look like the energy norm that we would like to bound (see Lemma \ref{howtogetthedesirednormintheexpressionofenergyestimate}), however the divergence theorem generates a space-time integral that one would then need to control. This space-time integral involves the source terms of the non-linear wave equation (see Lemma \ref{Theenerhyestimatefornequalfive}). In other words, in order to control the higher order energy norm in question (as in Lemma \ref{the main energy estimate forAandhforn4}), one needs to control the source terms of non-linear hyperbolic equation satisfied by the Lie derivatives of the solutions, which are nothing else but the Lie derivatives of the source terms of the original equation and the Lie derivatives of the structure of the wave equation itself which we would call the commutator term (see Lemma \ref{estimatingtheequareofthesourcetermsusingbootsrapassumption}).

However, in order to close the argument, one needs to improve the bound on the higher order energy without using even more higher order energy for which a bound would also be assumed -- such an argument obviously does not close, as the bound on this even higher order energy could then not be improved. For this, one controls the Lie derivatives of the source terms using the fact that it is a product of Lie derivatives: one does not need to control them all, but one needs to control one factor in the product, as long as the control on that factor is good enough (see Lemmas \ref{squareofthesourcetermsforAforngeq5withdeltaequalzeroforenergyestimate} and \ref{squareofthesourcetermsforthemetrichforngeq5withdeltaequalzeroforenergyestimate}). With that control, one can then look forward to establishing a Grönwall inequality on the higher order energy norm (see Lemmas \ref{estimateonthesoacetimeintegralforLiederivativesZofthesourcetermsAusefulforagronwallinequalityinenergyestimate} and \ref{estimateonthesoacetimeintegralforLiederivativesZofthesourcetermsformetrichusefulforagronwallinequalityinenergyestimate}).

The celebrated Grönwall lemma tells us that if the factor in the integrand is decaying fast enough, in such a way that it is integrable, then the quantity in question, which is here the higher order energy norm, will be bounded. If the initial conditions are small enough, then the bound on the energy will then be improved from what was initially assumed and used in the argument (see Lemma \ref{TheGronwalltypeinequalityonenergyforngeq5}). Thus, using the continuity of the growth of the energy, this would close the bootstrap argument (see Proposition \ref{Thetheoremofglobalstabilityanddecayforngeq 5}). 

In the case of higher dimensions $n\geq 4$, the Klainerman-Sobolev inequality gives a pointwise decay that is fast enough to be integrable in time, and hence one could close a bootstrap argument that concludes that the higher order energy will remain bounded for all time. In the case of $n\geq 5$, using an energy estimate associated to wave equations, combined with the Klainerman-Sobolev inequality, one could get a suitable Grönwall inequality everywhere: for an energy norm defined as an integral on the whole space slice. This allows one to conclude global stability of the Minkowski space-time. In the case of $n=4$, there is a lack of integrability for a term in the interior region: inside an outgoing light cone, where one could get concentration of energy. Thus, in the case of $n=4$, one defines the energy to be only an integral on the exterior: exterior to the outgoing light cone. With this exterior notion, that we will reat in the paper that follows, one could then get an integrable factor in the Grönwall inequality and thereby conclude exterior stability of the Minkowski space-time under perturbations governed by the coupled Einstein-Yang-Mills equations.

In this paper, we will prove the following theorem.

\subsection{The statement of the theorem}\

\begin{theorem}
Let $n \geq 5$\;. Assume that we are given an initial data set $(\Sigma, \overline{A}, \overline{E}, \overline{g}, \overline{k})$ for \eqref{EYMsystemforintro}
\;. We assume that $\Sigma$ is diffeomorphic to $\R^n$\;. Then, there exists a global system of coordinates $(x^1, ..., x^n) \in \R^n$ for $\Sigma$\;. We define
\bea
r := \sqrt{ (x^1)^2 + ...+(x^n)^2  }\;.
\eea
Furthermore, we assume that the data $(\overline{A}, \overline{E}, \overline{g}, \overline{k}) $ is smooth and asymptotically flat. 

Let $\de_{ij}$ be the Kronecker symbol and let $\overline{h}_{ij} $ be defined in this system of coordinates $x^i$\;, by
 \bea
\overline{h}_{ij} := \overline{g}_{ij} - \de_{ij} \; .
\eea
We then define the weighted $L^2$ norm on $\Sigma$\;, namely $\overline{\E}_N$\;, for $\ga > 0$\;, by
 \bea\label{definitionoftheenergynormforinitialdata}
 \notag
&&\overline{\E}_N \\
\notag
&:=&  \sum_{|I|\leq N} \big(   \| (1+r)^{1/2 + \ga + |I|}   \overline{D} (  \overline{D}^I  \overline{A}    )  \|_{L^2 (\Sigma)} +  \|(1+r)^{1/2 + \ga + |I|}    \overline{D}  ( \overline{D}^I \overline{h}   )  \|_{L^2 (\Sigma)} \big) \\
\notag
&:=&  \sum_{|I|\leq N}      \big(   \sum_{i=1}^{n}  \|(1+r)^{1/2 + \ga + |I|}     \overline{D} (  \overline{D}^I  \overline{A_i}    )  \|_{L^2 (\Sigma)} +  \sum_{i, j =1}^{n}  \|(1+r)^{1/2 + \ga + |I|}     \overline{D}  ( \overline{D}^I \overline{h}_{ij}   )  \|_{L^2 (\Sigma)} \big) \; , \\
\eea
where the integration is taken on $\Sigma$ with respect to the Lebesgue measure $dx_1 \ldots dx_n$\;, and where $\overline{D} $ is the Levi-Civita covariant derivative associated to the given Riemannian metric $\overline{g}$\;.

We also assume that the initial data set $(\Sigma, \overline{A}, \overline{E}, \overline{g}, \overline{k})$ satisfies the Einstein-Yang-Mills constraint equations, namely
\bea
\notag
  \mathcal{R}+ \overline{k}^i_{\,\, \, i} \overline{k}_{j}^{\,\,\,j}  -  \overline{k}^{ij} \overline{k}_{ij}   &=&    \frac{4}{(n-1)}   < \overline{E}_{i}, \overline{E}^{ i}>   \\
 \notag
 && +  < \overline{D}_{i}  \overline{A}_{j} - \overline{D}_{j} \overline{A}_{i} + [ \overline{A}_{i},  \overline{A}_{j}] ,\overline{D}^{i}  \overline{A}^{j} - \overline{D}^{j} \overline{A}^{i} + [ \overline{A}^{i},  \overline{A}^{j}] >  \;  ,\\
 \notag
\overline{D}_{i} \overline{k}^i_{\,\,\, j}    - \overline{D}_{j} \overline{k}^i_{\, \,\,i}  &=&  2 < \overline{E}_{i}, \overline{D}_{j}  \overline{A}^{i} - \overline{D}^{i} \overline{A}_{j} + [ \overline{A}_{j},  \overline{A}^{i}]  >  \;  ,\\
\notag
\overline{D}^i \overline{E}_{ i} + [\overline{A}^i, \overline{E}_{ i} ]  &=& 0  \;  . \\
\eea
For any $n \geq 5$\;, and for any $N \geq 2 \lfloor  \frac{n}{2} \rfloor  + 2$\;, there exists a constant $ \overline{c} (N, \ga)$ depending on $N$ and on $\ga$\;, such that if 
\bea\label{Assumptiononinitialdataforglobalexistenceanddecay}
\overline{\E}_N \leq   \overline{c} ( N, \ga) \; ,
\eea
then there exists a solution $(\cal{M}, A, g)$ to the Cauchy problem for the fully coupled Einstein-Yang-Mills system \eqref{EYMsystemforintro} in the future of $\Sigma$ converging to the null Yang-Mills potential and to the Minkowski space-time in the following sense: if we define the metric $ m_{\mu\nu}$ to be the Minkowski metric in wave coordinates $(x^0, x^1, \ldots, x^n)$\, and define $t = x^0$\;, and if we define in this system of wave coordinates 
\bea
h_{\mu\nu} := g_{\mu\nu} - m_{\mu\nu}  \;  ,
\eea
then, for $\overline{h}^1_{ij} $ and $\overline{A}_i$ decaying sufficiently fast as exhibited in Proposition \ref{Thetheoremofglobalstabilityanddecayforngeq 5}, we have the following estimates on $h$\;, and on $A$ in the Lorenz gauge, for the norm constructed using wave coordinates (see Subsection \ref{subsectiondefinitionofthenorms}), by taking the sum over all indices in wave coordinates. That there exists a constant $E(N)$\;, that depends on $\overline{c} (N, \ga)$\;, such that for all $|I| \leq N -  \lfloor  \frac{n}{2} \rfloor  - 1$\;, we have

\bea
 \notag
&&  \sum_{\mu= 0}^{n} |\derm  ( \Lie_{Z^I}  A_{\mu} ) (t,x)  |     +   \sum_{\mu, \nu = 0}^{n}  |\derm  ( \Lie_{Z^I}  h_{\mu\nu} ) (t,x)  |   \\
 \notag
    &\les& \begin{cases}   \frac{E ( N ) }{(1+t+|r-t|)^{\frac{(n-1)}{2}} (1+|r-t|)^{1+\gamma}}  \; ,\quad\text{when }\quad r-t  >0 \;  ,\\
           \notag
        \frac{E ( N )  }{(1+t+|r-t|)^{\frac{(n-1)}{2}}(1+|r-t|)^{\frac{1}{2} }}   \; ,\quad\text{when }\quad r-t  <0 \;  , \end{cases} \\
      \eea
and
 \bea
 \notag
  \sum_{\mu= 0}^{n}  |\Lie_{Z^I} A_{\mu} (t,x)  | +   \sum_{\mu, \nu = 0}^{n}  |\Lie_{Z^I}  h_{\mu\nu} (t,x)  |  &\les& \begin{cases}   \frac{ c (\gamma) \cdot E ( N )  }{(1+t+|r-t|)^{\frac{(n-1)}{2}} (1+|r-t|)^{\gamma}}  \; ,\quad\text{when }\quad r-t >0  \;  ,\\
   \notag
        \frac{E ( N )  \cdot (1+|r-t|)^{\frac{1}{2} }  }{(1+t+|r-t|)^{\frac{(n-1)}{2}}}   \; ,\quad\text{when }\quad r-t <0  \; , \end{cases} \\
      \eea
            where $Z^I$ are the Minkowski vector fields (see Subsection \ref{TheMinkwoskivectorfieldsdefinition}).
            
     In particular, the gauge invariant norm on the Yang-Mills curvature decays as follows, for all $|I| \leq N -  \lfloor  \frac{n}{2} \rfloor  - 1$\;,
 \bea
 \notag
&& \sum_{\mu, \nu = 0}^{n}  |\Lie_{Z^I} F_{\mu\nu}  (t,x) |  \\
 \notag
&\les& \begin{cases}   \frac{E ( N ) }{(1+t+|r-t|)^{\frac{(n-1)}{2}} (1+|r-t|)^{1+\gamma}} + \frac{ c (\gamma) \cdot E ( N )  }{(1+t+|r-t|)^{(n-1)} (1+|r-t|)^{2\gamma}} \;,\quad\text{when }\quad r-t >0  \;  ,\\
           \notag
        \frac{E ( N )  }{(1+t+|r-t|)^{\frac{(n-1)}{2}}(1+|r-t|)^{\frac{1}{2} }} +   \frac{E ( N )  \cdot (1+|r-t|)  }{(1+t+|r-t|)^{(n-1)}}  \; ,\quad\text{when }\quad r-t <0   \;  . \end{cases}  \\
      \eea
      
      Furthermore, if one defines $w$ as follows (see Definition \ref{defw}), 
\bea
w(q):=\begin{cases} (1+|r-t|)^{1+2\gamma} \quad\text{when }\quad r-t>0 \;, \\
         1 \,\quad\text{when }\quad r-t<0 \; , \end{cases}
\eea
and if we define $\Sigma_t$ as being the time evolution in wave coordinates of $\Sigma$\;, then for all time $t$\;, we have
\bea\label{theboundinthetheoremonEnbyconstantEN}
\notag
\E_{N} (t) &:=&  \sum_{|J|\leq N} \big( \|w^{1/2}   \derm ( \Lie_{Z^J} h   (t,\cdot) )  \|_{L^2(\Sigma_t)} +  \|w^{1/2}   \derm ( \Lie_{Z^J}  A   (t,\cdot) )  \|_{L^2 \Sigma_t) } \big) \\
&\leq& E(N) \; .
\eea
     
More precisely, for any constant $E(N)$\;, there exist two constants, a constant $c_1$ that depends on $\ga > 0$ and on $n \geq 5$\;, and a constant $c_2$ (to bound $ \overline{\E}_N (0)$ defined in \eqref{definitionoftheenergynormforinitialdata}), that depends on $E(N)$\;, on $N \geq 2 \lfloor  \frac{n}{2} \rfloor  + 2$ and on $w$ (i.e. depends on $\gamma$), such that if
\bea\label{AssumptiononinitialdataforcertainfirstminimalLiederivativesglobalexistenceanddecay}
 \overline{\E}_{ ( \lfloor  \frac{n}{2} \rfloor  +1)} (0)  \leq c_1(\ga, n ) \; ,
\eea
and if
\bea\label{Assumptiononinitialdataforglobalexistenceanddecay}
 \overline{\E}_N (0) \leq c_2 (E(N), N, \ga) \; ,
\eea
then, we have for all time $t$\,, 
\bea\label{Theboundontheglobaleenergybyaconstantthatwechoose}
 \E_{N} (t) \leq E(N) \; .
\eea

\end{theorem}

\section{The Einstein-Yang-Mills equations}

\subsection{The set-up} \

We consider that we are given an arbitrary compact Lie group $G$\;, and a positive definite Ad-invariant scalar product, $<$ $,$ $>$\;, on the Lie algebra $\cal G$\;, associated to the Lie group $G$\;.

The unknowns that we are looking for are $(\cal{M}, A, \g)$\;, where $\cal M$ is an unknown manifold, where $A$ is an unknown Yang-Mills potential, which in a given system of coordinates $x^{\alpha}$, is a one-form $A$ on the manifold $\cal M$\:, valued in the Lie algebra $\cal G$, and can be written as
$$A  = A_{\alpha}dx^{\alpha} \, ,$$
and where $\g$ is an unknown Lorentzian metric.

Let $\der_{\alpha}$ be the Levi-Civita covariant derivative, that is a torsion free connection and compatible with the metric $\g$ that is part of the unknowns $(\cal{M}, A, \g)$ that we are looking for. We define the gauge covariant derivative of any arbitrary tensor $\Psi$ valued in the Lie algebra $\cal G$\;, as
\bea
\textbf{D}^{(A)}_{\alpha}\Psi := \der_{\alpha}\Psi + [A_{\alpha},\Psi] \, .
\eea
The Yang-Mills curvature, $F$, is a two-form defined by
\bea\label{expF}
F_{\a\b} = \der_{\a}A_{\b} - \der_{\b}A_{\a} + [A_{\a},A_{\b}]   \, .
\eea

\subsection{The field equations} \

In a given system of coordinates, we define
\beaa
e_\mu = \frac{\pa}{\pa x_\mu} \, .
\eeaa

Let $R_{\a \beta \gamma \delta}$ be the Riemann tensor that is 
\bea
R_{\a \beta \gamma \delta}&:=& g( e_{\alpha}, \der_{e_{\gamma}} \der_{e_{\delta}} e_{\beta} - \der_{e_{\delta}} \der_{e_{\gamma}} e_{\beta} - \der_{[e_{\gamma}, e_{\delta}]} e_{\beta} ) \, .
\eea

We are in fact looking for a (1+$n$)-dimensional globally hyperbolic Lorentzian manifold $(\cal M, \g)$\;, and a one-form $A$ defined on this manifold, which satisfy the Einstein-Yang-Mills equations, which are 
\bea
R_{ \mu \nu} - \frac{1}{2} g_{\mu\nu} R &=& 8\pi \cdot T_{\mu\nu} \, , \label{EYM}
\eea
where
\bea\label{definitionofstressenergymomenturmYangMillsmatter}
 T_{\mu\nu} =  \frac{1}{4 \pi} \cdot  ( < F_{\mu\b}, F_{\nu}^{\; \; \b}> - \frac{1}{4} g_{\mu\nu } < F_{\a\b },F^{\a\b } > ) \, ,
 \eea
where $R_{ \mu \nu}$ is the Ricci tensor, that is
\bea
R_{ \mu \nu} := R^{\alpha}_{\,\,\, \mu\alpha \nu} :=  g^{\a\si} R_{\si \mu\alpha \nu}  \, ,
\eea
and where the scalar of Ricci $R$ is given by
\bea
R := R_{ \mu}^{\,\;\; \mu} := g^{\a\si} R_{ \mu\si} \, .
\eea
Here, we have used the Einstein summation convention of lowering and highering indices with respect to the unknown background metric $\g$\;.

However, the expression of $F$ in terms of $A$, \eqref{expF}, leads to the Bianchi identities for the Yang-Mills curvature (see \cite{G1}), 
\bea\label{BianchiforYangMillscurvature}
\textbf{D}^{(A)}_{\a}F_{\mu\nu} + \textbf{D}^{(A)}_{\mu}F_{\nu\a} + \textbf{D}^{(A)}_{\nu}F_{\a\mu} = 0 \, . \label{eq:Bianchi}
\eea
Since $\der$ is the Levi-Civita covariant derivative, we have the Bianchi identities for the Riemann tensor
\bea
\der_{\a} R^{\ga}_{\,\,\, \b\mu\nu} + \der_{\mu} R^{\ga}_{\,\,\, \b\nu\a} + \der_{\nu} R^{\ga}_{\,\,\, \b\a\mu} = 0 \, . \label{eq:BianchiEinstein}
\eea
Contracting, we get
\beaa
\der_{\a} R^{\a}_{\,\,\, \b\mu\nu} + \der_{\mu} R^{\a}_{\,\,\, \b\nu\a} + \der_{\nu} R^{\a}_{\,\,\, \b\a\mu}  = \der_{\a} R^{\a}_{\,\,\, \b\mu\nu} - \der_{\mu} R_{\b\nu}  +  \der_{\nu} R_{\b\mu} = 0 \, .
\eeaa
Contracting again, we obtain
\beaa
 \der^{\a} R_{\a \b\mu}^{\,\,\, \,\,\;  \;\;\; \b}   - \der_{\mu} R_{\b}^{\,\;\;\b}  +  \der^{\b} R_{\b\mu} = 0 \, ,
\eeaa
which leads to
\beaa
\der^{\a} R_{\a\mu} - \der_{\mu} R  +  \der^{\b} R_{\b\mu} = 0 \, ,
\eeaa
and hence
\bea\label{fromthebianchiidentityfortheRiemanntensor}
2 ( \der^{\a} R_{\a\mu} - \frac{1}{2} \der_{\mu} R )  = 0 \, .
\eea
Therefore, 
\beaa
\der^{\mu} (R_{\mu\nu} - \frac{1}{2} \g_{\mu\nu}R ) = 0 \,  ,
\eeaa
which in its turn implies that
\bea\label{energymomentumtensordivergencefree}
\der^{\mu}  T_{\mu\nu} = 0 \, .
\eea
Using the Bianchi identities for the Yang-Mills curvature \eqref{BianchiforYangMillscurvature}, the fact that  $<$ $,$ $>$ is Ad-invariant, that the connection $\der$ is compatible with the metric $\g$\,, then the fact that the energy-momentum tensor is divergence free, \eqref{energymomentumtensordivergencefree}, leads to the following Yang-Mills equation (see \cite{G1}),
\bea\label{ThegaugecovariantdivergenceoftheYangMillscurvatureisequaltozero}
\textbf{D}^{(A)}_{\a}F^{\a\b} := \der_{\alpha} F^{\a\b}  + [A_{\alpha}, F^{\a\b} ]  = 0 \; . 
\eea

Now, contracting the left hand side of the Einstein-Yang-Mills equations, \eqref{EYM}, gives
\beaa
{R_{ \mu}}^{\mu} - \frac{1}{2} {g_{\mu}}^{\mu} R &=& {R_{ \mu}}^{\mu} - \frac{1}{2} g^{\mu\a} g_{\mu\a} R  \\
&=& R - \frac{(n+1)}{2} R  = \frac{(1-n)}{2}  R \, .
\eeaa
Thus, the full contraction of the Einstein-Yang-Mills equations leads to
\beaa
 \frac{(1-n)}{2} \cdot R &=& 8\pi \cdot {T_{\mu}}^{\mu}  \\
&=& 8\pi  \cdot \frac{1}{4 \pi} ( < F_{\mu\b}, F^{\mu\b} > - \frac{1}{4} {g_{\mu}}^{\mu }   < F_{\a\b },F^{\a\b } > )\\
&=& 2( < F_{\a\b}, F^{\a\b}> - \frac{(n+1)}{4} \cdot < F_{\a\b },F^{\a\b } > ) \\
&=&  2 \cdot  \frac{(3-n)}{4} \cdot < F_{\a\b },F^{\a\b } > \\
&=&  \frac{(3-n)}{2} \cdot <F_{\a\b },F^{\a\b } >  .
\eeaa
Therefore,
\bea\label{expressionofRicciscalarR}
R &=&  \frac{(3-n)}{(1-n)} \cdot  < F_{\a\b },F^{\a\b }> .
\eea
Consequently, the Einstein-Yang-Mills equations in (1+$n$)-dimensions, \eqref{EYM}, can be written as
\beaa
R_{ \mu \nu} - \frac{1}{2} g_{\mu\nu}   \frac{(3-n)}{(1-n)} \cdot  < F_{\a\b },F^{\a\b } > &=&  2 < F_{\mu\b}, F_{\nu}^{\; \; \b}> - \frac{1}{2} g_{\mu\nu } < F_{\a\b },F^{\a\b } >  \, ,
 \eeaa
which yields to
\bea\label{simpEYM}  
R_{ \mu \nu}  &=& 2 < F_{\mu\b}, F_{\nu}^{\;\;\b} > + g_{\mu\nu }  \frac{1}{(1-n)} \cdot < F_{\a\b },F^{\a\b } >\; .  
\eea

Finally, the Einstein Yang-Mills equations are given by the following system
\bea \label{EYMsystemofequationsongandA}  
\begin{cases} 
R_{ \mu \nu}  =& 2 < F_{\mu\b}, F_{\nu}^{\;\;\b} > +   \frac{1}{(1-n)} \cdot g_{\mu\nu } \cdot < F_{\a\b },F^{\a\b } >\; ,   \\
0 \;\;\;\, \,=&  \der_{\alpha} F^{\a\b}  + [A_{\alpha}, F^{\a\b} ]    \;, \\
F_{\a\b} =& \der_{\a}A_{\b} - \der_{\b}A_{\a} + [A_{\a},A_{\b}] \; .  \end{cases} 
      \eea

\section{The Cauchy problem and the constraints for the Einstein-Yang-Mills system }

\subsection{The Cauchy problem}\label{Cauchyproblem}

\begin{definition}\label{definitionoftheunitorthogonaltimelikevectorefieldhatt}
Since the unknown space-time $(\cal{M},\g)$ is  globally hyperbolic, we know by then that there exists a smooth vector field $\frac{\pa}{\pa t}$ such that $\cal{M}$ is foliated by Cauchy hypersurfaces $\Sigma_{t}$\;. The one-form $(dt)_{\mu}$ defines a vector field $ g^{\mu\nu} (dt)_{\nu}$ orthogonal to the hypersurfaces $\Sigma_{t}$\;. This vector field could then in turn be normalised to define a unit timelike vector $\hat{t}$ orthogonal to $\Sigma_{t}$\;.

In fact, let 
\bea\label{definitionoflapseN}
N = \big( - (dt)^{\mu} (dt)_{\mu} \big)^{\frac{1}{2}}= \big(- g_{\mu\nu} (dt)^{\mu} (dt)^{\nu} \big)^{\frac{1}{2}} \;.
\eea
 Then, at each point $p$ on $\Sigma_{t}$\;, we define
\bea
\hat{t}^{\nu} = \frac{1}{N} (dt)^{\nu} \; .
\eea
\end{definition}

\begin{definition}\label{definitionsosecondfundamentalform}
For $U, V$ vector fields tangent to $\Sigma_{t}$\,, let second fundamental form $k$ be defined by 
\bea
\label{secondfundamentalfrmdef}
k (U, V) &:=& g(\der_U \hat{t}, V) \; .
\eea
\end{definition}

We are looking for an unknown (1+$n$)-dimensional globally hyperbolic manifold $(\cal M,\g)$, therefore foliated by space-like hypersurfaces $\Sigma_t$\;, where $t$ is a smooth time function, and $\hat{t}$ is a timelike vector orthogonal to $\Sigma_{t}$ (defined in Definition \ref{definitionoftheunitorthogonaltimelikevectorefieldhatt}), and we are looking for an unknown Yang-Mills curvature $F$ on $\cal M$\;, which solve the Einstein-Yang-Mills equations \eqref{EYMsystemofequationsongandA} on $(\cal M, F, \g)$ . The Cauchy problem for the Einstein-Yang-Mills equations can be formulated as follows:

We consider that we are given an initial data set for the space-time, that is $(\Sigma, \overline{A}, \overline{E}, \overline{g}, \overline{k})$\;, which consists of an $n$-dimensional manifold $\Sigma$ with a Riemannian metric ${\overline{g}}$\;, and a symmetric two-tensor $k_{0}$\;, and consists of an initial data for the Yang-Mills fields which are two one-tensors ${\overline{A}} = {\overline{A}}_{i} dx^i$ and $E = E_{i}dx^i$ on $\Sigma$ valued in the Lie algebra $\cal G$\;. We are then looking for a (1+$n$)-dimensional Lorentzian manifold $\cal M$\;, with Yang-Mills curvature $F$\;, which solve the Einstein-Yang-Mills equations \eqref{EYMsystemofequationsongandA}, such that $\Sigma = \Sigma_{t_0}\subset \cal M$\;, and such that ${\overline{g}}$ is the restriction of $\g$ on $\Sigma_{t_0} \subset \cal M$\;, and $\overline{k}$ is the restriction of the second fundamental form $k$ on $\Sigma_{t_0} \subset \cal M$ (defined in Definition \ref{definitionsosecondfundamentalform} in \eqref{secondfundamentalfrmdef}), and such that $E_{i} = F_{\hat{t} i}$\;.

\begin{remark}
In this series of papers, we are going to impose that $ \{t, x_1, ..., x_n \}$ satisfy the wave coordinate condition (see Section \ref{The gauge conditions}), we are going to construct the Minkowski metric using this wave coordinates system (see Section \ref{LookingmetricperturbationMinkowskispacetime}), and we are going to impose that $A$ satisfies the Lorenz condition (see Section \ref{The gauge conditions}). We are going to construct $t$ such that $t = 0$ on $\Sigma \subset \cal M$\,. We also use the notation $x^0 = t$\;.

\end{remark}

\subsection{The constraint equations}\label{derivationofEinsteinYangMillsconstraintequationsforthecoupledsystem}\

The Einstein-Yang-Mills equations are overdetermined -- not any initial data set, $(\Sigma, \overline{A}, \overline{E}, \overline{g}, \overline{k})$\;, for the Einstein-Yang-Mills equations \eqref{EYMsystemofequationsongandA}, leads to a Cauchy development. In fact, the initial data must satisfy itself the Einstein-Yang-Mills equations. Let us explain in what follows.

\begin{definition}\label{definitionsforconstraints}
For $U, V$ vector fields tangent to $\Sigma_{t}$\,, let
\bea
\label{defrestrictedcovariantderivative}
D_{ U} V  &:=&   \der_U V - k (U,V) \, \hat{t} \; ,\\
\label{restrictedRiemanntensor}
\mathcal{R}^a_{\ bcd}&:=& g( e^a, D_{e_c} D_{e_d} e_b - D_{e_d} D_{e_c} e_b- D_{[e_c, e_d]} e_b ) \; .
\eea
\end{definition}

We will show the following well-known lemmas.

\begin{lemma}
The tensor $k$, the second fundamental form of the hypersurface $\Sigma_t$\,, is symmetric, that is for all $U, V \in T\Sigma_{t}$\,, we have $k (U, V) = k (V, U)$\,.
\end{lemma}
\begin{proof}
We know that for all $U \in T\Sigma_{t}$, we have
\beaa
g(\hat{t}, U) = 0 \, .
\eeaa
This along with the fact that $\der_{V} g(\hat{t}, U) = 0 $, we obtain 
\beaa
0= V (g(\hat{t}, U) ) = g(\der_{V}\hat{t}, U)  +g(\hat{t}, \der_{V} U) \, .
\eeaa
Now, for all $U, V \in T\Sigma_{t}$, we have $[U, V] := UV-VU \in T\Sigma_{t}$. Thus,
\beaa
0 = g(\hat{t}, [U, V]) =  g(\hat{t}, \der_{U} V -\der_{V} U  ) 
\eeaa
(since the metric is torsion free). Consequently,
\beaa
g(\der_{V}\hat{t}, U)  = - g(\hat{t}, \der_{V} U) = - g(\hat{t}, \der_{U} V  ) \, .
\eeaa
However, since $0= U (g(\hat{t}, V) )$, we get
\beaa
g(\der_{V}\hat{t}, U)  =  g(\der_{U} \hat{t},  V  ) \, .
\eeaa
\end{proof}

\begin{lemma}
The connection $D$ is compatible with the metric $\g$, that is $D g = 0$\,.
\end{lemma}
\begin{proof}

Since the metric $g$ is compatible with the metric, we have
\beaa
\pa_{c}  g( e_a,   e_b)  &=&    g( \der_{c} e_a,   e_b)  + g(  e_a,   \der_{c} e_b)  \\
 &=&    g( D_{c} e_a + k_{ca} \hat{t}   ,   e_b)  + g(  e_a,   D_{c} e_b + k_{cb} \hat{t} ) \\
  &=&    g( D_{c} e_a  ,   e_b)  + g(  e_a,   D_{c} e_b  ) \\
  && \text{(since $ \hat{t} $ is orthogonal to $\Sigma_{t})$. } \\
\eeaa
\end{proof}

\begin{lemma}\label{Gauss-Codazzi-equations}
We have the Gauss-Codazzi equations which say that for a spatial frame $ \{ e_a, e_b, e_c \} $ tangent to the hypersurface $\Sigma_{t}$\;, we have
\bea
\label{Gauss}
R^a_{\,\,\, bcd} &=&  \mathcal{R}^a_{\,\,\, bcd} - k^a_{\,\,\, d} k_{bc} + k^a_{\,\, \, c} k_{bd}  \; ,\\
R^a_{\,\,\,  \hat{t} cd} &=& D_{e_c}k^a_{\,\,\, d} - D_{e_d}k^a_{\,\,\, c}  \; .
\eea
\end{lemma}

\begin{proof}

\textbf{The Gauss equations:}\\

We will show the well-known proof of the following Gauss equations
\beaa
R^a_{\,\,\, bcd} &=&  \mathcal{R}^a_{\,\,\, bcd} - k^a_{\,\,\, d} k_{bc} + k^a_{\,\, \, c} k_{bd}  \; .
\eeaa

In fact, we have,
\beaa
R^a_{\,\,\, bcd} &=& g( e^a, \der_{e_c}  \der_{e_d} e_b -  \der_{e_d}  \der_{e_c} e_b -  \der_{[e_c, e_d]} e_b ) \\
&=& g( e^a,  \der_{e_c}  \der_{e_d} e_b) - g( e^a,  \der_{e_d}  \der_{e_c} e_b ) - g( e^a,  \der_{( \der_{c} e_{d} -  \der_{d} e_{c})} e_b ) \\
&& \text{(since $\der$ is a Levi-Civita connection and therefore torsion free) } \\
&=&   \pa_{e_c} g( e^a,   \der_{e_d} e_b) -  g( \der_{e_c} e^a,   \der_{e_d} e_b)  - \pa_{e_d} g( e^a,  \der_{e_c} e_b )\\
&&  +   g(  \der_{e_d} e^a,  \der_{e_c} e_b )   - g( e^a,  \der_{( \der_{c} e_{d} -  \der_{d} e_{c})} e_b ) \\
&=&   \pa_{e_c} g( e^a,   D_{e_d} e_b + k_{db} \hat{t} ) -  g( D_{e_c} e^a + {k_{c}}^{a} \hat{t} ,   D_{e_d} e_b + k_{db} \hat{t} )  -  \pa_{e_d} g( e^a,  D_{e_c} e_b + k_{cb} \hat{t} ) \\
&& +   g(  D_{e_d} e^a + {k_{d}}^{a} \hat{t} ,  D_{e_c} e_b + k_{cb} \hat{t} )    - g( e^a,  \der_{( D_{c} e_{d} + k_{cd}  \hat{t}  -  D_{d} e_{c} -  k_{dc}  \hat{t}  )} e_b ) \\
&=&   \pa_{e_c} g( e^a,   D_{e_d} e_b  ) -  g( D_{e_c} e^a + {k_{c}}^{a} \hat{t} ,   D_{e_d} e_b + k_{db} \hat{t} )  -  \pa_{e_d} g( e^a,  D_{e_c} e_b  ) \\
&& +   g(  D_{e_d} e^a +  {k_{d}}^{a}  \hat{t} ,  D_{e_c} e_b + k_{cb} \hat{t} )    - g( e^a,  \der_{( D_{c} e_{d} + k_{cd}  \hat{t}  -  D_{d} e_{c} -  k_{dc}  \hat{t}  )} e_b ) \\
&& \text{(where we used the fact that $\hat{t}$ is orthogonal to $\Sigma_{t}$) } 
\eeaa
\beaa
&=&   g(  D_{e_c} e^a,   D_{e_d} e_b  ) +  g( e^a,   D_{e_c}  D_{e_d} e_b  )  \\
&& -  g( D_{e_c} e^a ,   D_{e_d} e_b  )  -    g( D_{e_c} e^a  ,   k_{db} \hat{t} )  -  g( {k_{c}}^{a} \hat{t} ,   D_{e_d} e_b ) -    g(  {k_{c}}^{a} \hat{t} ,  k_{db} \hat{t} ) \\
&&  -  \pa_{e_d} g( e^a,  D_{e_c} e_b  ) \\
&& +   g(  D_{e_d} e^a  ,  D_{e_c} e_b )+   g(  D_{e_d} e^a  ,  k_{cb} \hat{t} )+   g(    {k_{d}}^{a}  \hat{t} ,  D_{e_c} e_b  )+   g(   {k_{d}}^{a}  \hat{t} ,   k_{cb} \hat{t} )  \\
&&  - g( e^a,  \der_{( D_{c} e_{d}   -  D_{d} e_{c}  )} e_b )  - g( e^a,  \der_{(  k_{cd}  \hat{t}  -  k_{dc}  \hat{t}  )} e_b ) \\
&& \text{(where we used the fact that also $D$ is compatible with the metric $g$) } \\
&=&   g(  D_{e_c} e^a,   D_{e_d} e_b  ) +  g( e^a,   D_{e_c}  D_{e_d} e_b  )    - g( e^a,  D_{e_d} D_{e_c} e_b  ) - g( e^a,  \der_{( D_{c} e_{d}   -  D_{d} e_{c}  )} e_b )  \\
&& -  g( D_{e_c} e^a ,   D_{e_d} e_b  )  -    g( D_{e_c} e^a  ,   k_{db} \hat{t} )  -  g( {k_{c}}^{a} \hat{t} ,   D_{e_d} e_b ) -    g(  {k_{c}}^{a} \hat{t} ,  k_{db} \hat{t} ) \\
&&  -   g(D_{e_d} e^a,  D_{e_c} e_b  ) \\
&& +   g(  D_{e_d} e^a  ,  D_{e_c} e_b )+   g(  D_{e_d} e^a  ,  k_{cb} \hat{t} )+   g(    {k_{d}}^{a} \hat{t} ,  D_{e_c} e_b  )+   g(   {k_{d}}^{a} \hat{t} ,   k_{cb} \hat{t} )  \\
&& \text{(where we used the fact that $k$ is symmetric) } \\
&=&  \mathcal{R}^a_{\ bcd}    -    g( D_{e_c} e^a  ,   k_{db} \hat{t} )  -  g( {k_{c}}^{a} \hat{t} ,   D_{e_d} e_b ) +     {k_{c}}^{a}   k_{db}  \\
&&+   g(  D_{e_d} e^a  ,  k_{cb} \hat{t} )+   g(    {k_{d}}^{a} \hat{t} ,  D_{e_c} e_b  ) -   {k_{d}}^{a}  k_{cb}  \, . \\
\eeaa

However, we have for $U, V$ tangent to $\Sigma_{t}$
\beaa
D_{ U} V  &=&   \der_U V - k (U,V) \hat{t} \, . \\
\eeaa
Thus,
\bea
\notag
g(  D_{ U} V,  \hat{t} )  &=&  g( \der_U V  , \hat{t} )- k (U,V) g( \hat{t},  \hat{t} ) \\
\notag
&=& - k (U,V)  - k (U,V) g( \hat{t},  \hat{t} ) \\
\notag
&=& - k (U,V)  + k (U,V)  \\
&=& 0 \, .
\eea
Hence,
\beaa
R^a_{\,\,\, bcd} &=&  \mathcal{R}^a_{\,\,\, bcd}  + k^{\,\, \, a}_c k_{bd} - k^{\,\,\, a}_d k_{bc}   \\
&=&  \mathcal{R}^a_{\ bcd}   + k^a_{\,\, \, c} k_{bd} - k^a_{\,\,\, d} k_{bc}  \\
&& \text{(using the symmetry of the second fundamental form $k$). } 
\eeaa

\textbf{The Codazzi equations}\\

Now, we prove the Codazzi equations
\beaa
\label{Gauss}
R^a_{\,\,\,  \hat{t} cd}  &=& D_{e_c}k^a_{\,\,\, d} - D_{e_d}k^a_{\,\,\, c} \; .
\eeaa

We have,
\beaa
 R^a_{\,\,\,  \hat{t} cd} &=& g( e^{a}, \der_{e_{c}} \der_{e_{d}} \hat{t} - \der_{e_{d}} \der_{e_{c}} \hat{t} - \der_{[e_{c}, e_{d}]} \hat{t} ) \\
&=& g( e^{a}, \der_{e_{c}} \der_{e_{d}} \hat{t} )  - g( e^{a},   \der_{e_{d}} \der_{e_{c}} \hat{t} ) -  g( e^{a}, \der_{[e_{c}, e_{d}]} \hat{t} ) \\
&=&\pa_{e_{c}}  g( e^{a}, \der_{e_{d}} \hat{t} )  -g( \der_{e_{c}} e^{a},  \der_{e_{d}} \hat{t} )  -  \pa_{e_{d}}  g( e^{a},  \der_{e_{c}} \hat{t} ) \\
&&  + g(  \der_{e_{d}}  e^{a},    \der_{e_{c}} \hat{t} )  -  g( e^{a}, \der_{(\der_{e_{c}} e_{d} - \der_{e_{d}} e_{c} ) } \hat{t} ) \\
&=&\pa_{e_{c}}  g( e^{a}, \der_{e_{d}} \hat{t} )  -g( \der_{e_{c}} e^{a},  \der_{e_{d}} \hat{t} )  - \pa_{e_{d}}  g( e^{a},  \der_{e_{c}} \hat{t} )  \\
&&+ g(  \der_{e_{d}}  e^{a},    \der_{e_{c}} \hat{t} )  -  g( e^{a}, \der_{(D_{e_{c}} e_{d} + k_{cd}\hat{t}   - D_{e_{d}} e_{c} - k_{dc}\hat{t} ) } \hat{t} )\, . \\
\eeaa
However, we have
\bea
\notag
  \der_{e_{\a}} \hat{t} &=&    g( \der_{e_{\a}} \hat{t}, e_{\mu} ) e^{\mu} -  g( \der_{e_{\a}} \hat{t}, \hat{t}) \hat{t}  \\
  &=&  k_{\a\mu} e^{\mu}  \\
  \notag
  && \text{(where we used the fact that $\hat{t}$ is a unit vector field). } 
 \eea
 Thus,
 \beaa
R^a_{\,\,\,  \hat{t} cd} &=& \pa_{e_{c}}  g( e^{a}, \der_{e_{d}} \hat{t} )  -g( \der_{e_{c}} e^{a},  \der_{e_{d}} \hat{t} )  -  \pa_{e_{d}}  g( e^{a},  \der_{e_{c}} \hat{t} )  + g(  \der_{e_{d}}  e^{a},    \der_{e_{c}} \hat{t} ) \\
&& -  g( e^{a}, \der_{(D_{e_{c}} e_{d} + k_{cd}\hat{t}   - D_{e_{d}} e_{c} - k_{dc}\hat{t} ) } \hat{t} ) \\
&=& \pa_{e_c}{k^a}_d - \pa_{e_d}{k^a}_c   -g( \der_{e_{c}} e^{a}, k_{d\mu} e^{\mu} )   + g(  \der_{e_{d}}  e^{a},    k_{c\mu} e^{\mu}) \\
&& - g( e^{a}, \der_{(D_{e_{c}} e_{d}   - D_{e_{d}} e_{c}  ) } \hat{t} ) \\
&& \text{(using the symmetry of $k$).} 
\eeaa
Yet,
\beaa
  g( e^{a}, \der_{(D_{e_{c}} e_{d}   - D_{e_{d}} e_{c}  ) } \hat{t} ) &=&   g( e^{a}, k_{\nu\mu} (D_{e_{c}} e_{d}   - D_{e_{d}} e_{c}  )^{\nu} e^{\mu} ) =k_{\nu a} (D_{e_{c}} e_{d}   - D_{e_{d}} e_{c}  )^{\nu} \, , \\
 \der_{e_{c}} e^{a} &=&  D_{e_{c}} e^{a}  + k_{c}^{\,\,\, a} \hat{t} \, .
\eeaa
Hence,
 \beaa
R^a_{\,\,\,  \hat{t} cd}   &=& \pa_{e_c}{k^a}_d - \pa_{e_d}{k^a}_c   -g(  D_{e_{c}} e^{a}  + {k_{c}}^{a} \hat{t}, k_{d\mu} e^{\mu} )   + g(   D_{e_{d}} e^{a}  + {k_{d}}^{a} \hat{t} ,    k_{c\mu} e^{\mu}) \\
&& - k_{\nu a} (D_{e_{c}} e_{d}   - D_{e_{d}} e_{c}  )^{\nu} \\
 &=& \pa_{e_c}{k^a}_d - \pa_{e_d}{k^a}_c   -g(  D_{e_{c}} e^{a} , k_{d\mu} e^{\mu} )   + g(   D_{e_{d}} e^{a} ,    k_{c\mu} e^{\mu}) \\
&& - {k^{a}}_{\nu } (D_{e_{c}} e_{d}   - D_{e_{d}} e_{c}  )^{\nu} \\
 &=& \pa_{e_c}{k^a}_d   -k_{d\mu}  g(  D_{e_{c}} e^{a} , e^{\mu} ) -{k^{a}}_{\nu } (D_{e_{c}} e_{d}  )^{\nu}  - \pa_{e_d}{k^a}_c  + {k^{a}}_{\nu } (  D_{e_{d}} e_{c}  )^{\nu} +  k_{c\mu} g(   D_{e_{d}} e^{a} ,    e^{\mu})   \\
 &=& \pa_{e_c}{k^a}_d   -k_{\mu d}  (  D_{e_{c}} e^{a} )^{\mu}  -{k^{a}}_{\nu } (D_{e_{c}} e_{d}  )^{\nu}  - \pa_{e_d}{k^a}_c  + {k^{a}}_{\nu }(  D_{e_{d}} e_{c}  )^{\nu} +  k_{\mu c} (   D_{e_{d}} e^{a} )^{\mu}   \\
 && \text{(where we used the symmetry of $k$) } \\
  &=& D_{e_c} k^a_{\,\,\, d}    - D_{e_d} k^a_{\, \,\,c}     \, .
 \eeaa
 
 \end{proof}

\begin{lemma}\label{TheconstraintequationsfortheEinstein-Yang-Mills system}

The constraint equations for the Einstein-Yang-Mills system are 
\bea
\label{theEinsteinYangMillsconstriantsonSigmafirst}
 \mathcal{R}+ k^a_{\,\, \, a} k_{c}^{\,\,\,c}  -  k^{ac} k_{ac}    &=&    \frac{4}{(n-1)}   < E_{b}, E^{ b}>    +  < F_{ab },F^{ab } > \, ,\\
 \label{theEinsteinYangMillsconstriantsonSigmasecond}
 D_{e_a} k^a_{\,\,\, d}    - D_{e_d} k^a_{\, \,\,a}  &=&  2 < E_{b}, {F_{d}}^{\, b}> \, ,\\ 
\label{theYangMillsconstraintforthepotentialandelectricchargeonsigma} D^i E_{ i} + [A^i, E_{ i} ]  &=& 0 \, , 
\eea
where $E_{i} = F_{\hat{t} i}$\;, and where the summation is carried only over spatial indices. Here $\hat{t}$\;, $k$\;, $D$ and $\mathcal{R}$ are given in Definitions \ref{definitionoftheunitorthogonaltimelikevectorefieldhatt}, \ref{definitionsosecondfundamentalform}, and \ref{definitionsforconstraints}.

\end{lemma}

\begin{proof}
Based on Lemma \ref{Gauss-Codazzi-equations}, and by raising indices, we obtain
\beaa
R_{\,\, \, ba}^{a \, \,  \;\;\; d}   &=& g^{d\mu} R^a_{\,\,\, ba\mu} = g^{dc} R^a_{\,\,\, bac} =  {\mathcal{R}^a_{\,\,\, ba}}^{d} + k^a_{\,\, \, a} k_{b}^{\,\,\,d} - k^{ad} k_{ba}   \\
&& \text{(where we used the fact that $\hat{t}$ is orthogonal to the Cauchy hypersurfaces $\Sigma_{t}$). } 
\eeaa

Thus, summing over spatial indices, we obtain
\beaa
R_{\,\, \, ba}^{a \, \,  \;\;\; b}   &=&   \mathcal{R}_{\,\,\, ba}^{a \, \,  \;\;\; b} + k^a_{\,\, \, a} k_{b}^{\,\,\,b} - k^{ab} k_{ba} \, .
\eeaa
But,
\beaa
R &=& R_{\mu}^{\,\,\; \mu} = g^{\nu\mu} R_{\mu\nu} = g^{\nu\mu} R_{\mu \a \nu}^{\,\,\, \,\,\,  \;\;\; \a} = g^{\hat{t} \hat{t} } R_{\hat{t}  \a \hat{t} }^{\,\,\, \,\,\,  \,\,\, \a} + g^{a b } R_{a  \a b}^{\,\,\, \,\,\,  \,\,\, \a} = g^{\hat{t} \hat{t} } {R_{\hat{t}  a \hat{t} }}^{a} + g^{a b }  g^{\a\b}  {R_{a  \a b \b}}^{}  \\
&& \text{(where we used the symmetries of the Riemann tensor) } \\
&=&  g^{\hat{t} \hat{t} } R_{\hat{t}  a \hat{t} }^{\,\,\, \,\,\,  \,\,\, a} + g^{a b }  g^{\hat{t} \hat{t} }  {R_{a  \hat{t}  b \hat{t} }}^{}  + g^{a b }  g^{ c d }  {R_{a  c b d }}^{}\\
&=&  g^{\hat{t} \hat{t} }  R_{\hat{t}  a \hat{t} }^{\,\,\, \,\,\,  \,\,\, a}+   g^{\hat{t} \hat{t} }  R_{a  \hat{t}  \,\,\, \hat{t} }^{\,\,\,  \,\,\, a}   + g^{a b }  g^{ c d }  {R_{a  c b d }}^{} =   g^{\hat{t} \hat{t} } R_{\hat{t}  a \hat{t} }^{\,\,\, \,\,\,  \,\,\, a} +   g^{\hat{t} \hat{t} }  R_{a  \hat{t}  \,\,\, \hat{t} }^{\,\,\,  \,\,\, a}  +   R_{a  c  }^{\,\,\, \,\,\; ac}\\
&& \text{(where we used again the fact that $\hat{t}$ is orthogonal to the Cauchy hypersurfaces $\Sigma_{t}$) } \\
&=& 2 g^{\hat{t} \hat{t} } R_{\hat{t}  a \hat{t} }^{\,\,\, \,\,\,  \,\,\, a} +   R_{a  c  }^{\,\,\, \,\,\; ac} = - 2  R_{\hat{t}  a \hat{t} }^{\,\,\, \,\,\,  \,\,\, a}  +  R_{a  c  }^{\,\,\, \,\,\; ac} =  - 2  R_{\hat{t}  \a \hat{t} }^{\,\,\, \,\,\,  \,\,\, \a}  +   R_{a  c  }^{\,\,\, \,\,\; ac}=  - 2  R_{\hat{t}  \hat{t} }  +   R_{a  c  }^{\,\,\, \,\,\; ac}\\
&& \text{(where we used the symmetries of the Riemann tensor) } \\
 &=& - 2  R_{\hat{t}  \hat{t} }  +    R_{a  c  }^{\,\,\, \,\,\; ac} \, .
 \eeaa
 However, since $R =  \frac{(3-n)}{(1-n)} < F_{\a\b },F^{\a\b } > $, we get
 \bea
   \frac{(3-n)}{(1-n)} < F_{\a\b },F^{\a\b }> + 2  R_{\hat{t}  \hat{t} }  &=& R_{a  c  }^{\,\,\, \,\,\; ac} \, ,
 \eea
 and hence,
 \beaa
 \notag
R_{a  c  }^{\,\,\, \,\,\; ac}&=& 2  R_{\hat{t}  \hat{t} } +   \frac{(3-n)}{(1-n)} < F_{\a\b },F^{\a\b } >  = 2 ( 8\pi T_{\hat{t} \hat{t} } ) +   \frac{(3-n)}{(1-n)} < F_{\a\b },F^{\a\b } >  \\
 &=& 2 [ 2 \big( < F_{\hat{t}\b}, {F_{\hat{t}}}^{\b} >  - \frac{1}{4} g_{\hat{t}\hat{t} } < F_{\a\b },F^{\a\b } > \big) ] +   \frac{(3-n)}{(1-n)} < F_{\a\b },F^{\a\b } > \\
&=& 4 < F_{\hat{t}\b}, {F_{\hat{t}}}^{\b} >  + < F_{\a\b },F^{\a\b } >  +   \frac{(3-n)}{(1-n)} < F_{\a\b },F^{\a\b } > \\
&=& 4 < F_{\hat{t}\b}, {F_{\hat{t}}}^{\b} > +  \frac{2(2-n)}{(1-n)} < F_{\a\b },F^{\a\b } > \, .
 \eeaa

However,
\beaa
&& < F_{\a\b },F^{\a\b } > \\
&=& < F_{\hat{t}\b },F^{\hat{t}\b } >  +  < F_{a\b },F^{a\b } > \\
&=&  < F_{\hat{t}b },F^{\hat{t}b } >  +  < F_{a\b },F^{a\b } >  =  < F_{\hat{t}b },F^{\hat{t}b } >  +  < F_{a\hat{t} },F^{a\hat{t} } > +  < F_{ab },F^{ab } > \\
&& \text{(using the anti-symmetry of the Yang-Mills curvature) } \\
&=& 2 < F_{\hat{t}b },F^{\hat{t}b } >  +  < F_{ab },F^{ab } > = 2 g^{\hat{t}\mu} < F_{\hat{t}b },{F_{\mu}}^{\, b } >  +  < F_{ab },F^{ab } > \\
&=&  2 g^{\hat{t}\hat{t}} < F_{\hat{t}b },{F_{\hat{t}}}^{\, b } >  +  < F_{ab },F^{ab } > =  - 2 < F_{\hat{t}b },{F_{\hat{t}}}^{\, b } >  +  < F_{ab },F^{ab } > \\
&& \text{(where we used the fact that $\hat{t}$ is a unit-orthogonal vector to the} \\
&& \text{  Cauchy hypersurfaces foliation). } \\
\eeaa
Thus,
 \beaa
 \notag
 R_{a  c  }^{\,\,\, \,\,\; ac} &=&  4 < F_{\hat{t}\b}, {F_{\hat{t}}}^{\, \b}>  +  \frac{2(2-n)}{(1-n)}  < F_{\a\b },F^{\a\b } > \\
  \notag
 &=&   4 < F_{\hat{t}b}, {F_{\hat{t}}}^{\, b}>   -  \frac{4(2-n)}{(1-n)}  < F_{\hat{t}b },{F_{\hat{t}}}^{\, b } >  +  \frac{2(2-n)}{(1-n)}  < F_{ab },F^{ab } > \\
 &=&   \frac{4}{(n-1)} < F_{\hat{t}b}, {F_{\hat{t}}}^{\, b}>    +  < F_{ab },F^{ab } > \, .
\eeaa
However, Gauss equations give
\beaa
R_{a  c  }^{\,\,\, \,\,\; ac} &=&  \mathcal{R}_{a  c  }^{\,\,\, \,\,\; ac} - k^{ac} k_{ca} +  k^a_{\,\, \, a} k_{c}^{\,\,\,c} \, .
\eeaa
Thus,
 \bea
 \notag
 \mathcal{R}_{a  c  }^{\,\,\, \,\,\; ac} + k^a_{\,\, \, a} k_{c}^{\,\,\,c}  -  k^{ac} k_{ac}    &=&    \frac{4}{(n-1)}  < F_{\hat{t}b}, {F_{\hat{t}}}^{\, b}>    +  < F_{ab },F^{ab } > \, . \\
\eea
Now, we look at
\bea
 \notag
R^a_{\,\,\,  \hat{t} ad} &=& R^\mu_{\,\,\,  \hat{t} \mu d} = R_{\hat{t} d}   \\ 
 \notag
&& \text{(using the symmetries of the Riemann curvature) } \\
 \notag
&=& 8\pi T_{\hat{t} d } + \frac{1}{2} g_{\hat{t} d }   \frac{(3-n)}{(1-n)} < F_{\a\b },F^{\a\b } > \\
 \notag
&=&   2 ( < F_{\hat{t}\b}, {F_{d}}^{\, \b}> - \frac{1}{4} g_{\hat{t} d } < F_{\a\b },F^{\a\b } > ) + \frac{1}{2} g_{\hat{t} d }   \frac{(3-n)}{(1-n)} < F_{\a\b },F^{\a\b }> \\
 \notag
&=&     2 < F_{\hat{t}b}, {F_{d}}^{\, b}> \\
&=& D_{e_a} k^a_{\,\,\, d}    - D_{e_d} k^a_{\, \,\,a}  \, .
\eea

Hence, the constraint equations for the Einstein-Yang-Mills system are 
\bea
 \mathcal{R}+ k^a_{\,\, \, a} k_{c}^{\,\,\,c}  -  k^{ac} k_{ac}    &=&   \frac{4}{(n-1)}   < F_{\hat{t}b}, {F_{\hat{t}}}^{\, b}>    +  < F_{ab },F^{ab } > \, ,\\
 D_{e_a} k^a_{\,\,\, d}    - D_{e_d} k^a_{\, \,\,a}  &=&  2 < F_{\hat{t}b}, {F_{d}}^{\, b}> \, .
\eea

Also, since we want $E_{i} = F_{\hat{t} i}$ on $\Sigma$\;, then in view of the fact that $\textbf{D}^{(A) \,  i}F_{\hat{t} i}   = 0$ (which is implied from the Einstein-Yang-Mills system), we get
\bea
\textbf{D}^{(A)\,  i}  F_{\hat{t} i} = \textbf{D}^{(A)\,  i}  E_{ i} = \der^i E_{ i} + [A^i, E_{ i} ]  = 0  \, .
\eea

\end{proof}

\section{The gauges invariance of the equations and fixing the gauges}\label{The gauge conditions}

The Einstein-Yang-Mills equations are invariant under both gauge transformations and diffeomorphisms. We explain in what follows.

\subsection{The invariance under gauge transformation}\label{The invariance under gauge transformation}\

For any Yang-Mills potential $A_\a $ solution to the Einstein-Yang-Mills system and for any element $ \cal{O} \in G$, since $G$ is a Lie group and therefore a group, there exists an inverse $ \cal{O}^{-1} \in G$, and therefore we can define
\bea
\tilde{A_\a} =  \cal{O} A_\a  \cal{O}^{-1} - ( \der_\a  \cal{O} )  \cal{O}^{-1} \, ,
\eea
and
\bea
\tilde{F_{\a\b}} = \der_{\a}\tilde{A_{\b}} - \der_{\b}\tilde{A_{\a}} + [\tilde{A_{\a}},\tilde{A_{\b}}]   \, .
\eea
We have the following well-known lemma:
\begin{lemma}
We have
\beaa
\tilde{F_{\a\b}} = \cal{O} F_{\a\b}  \cal{O}^{-1} \, ,
\eeaa
and for any tensor $\Psi$ valued in the Lie algebra $\cal G$ associated to the Lie group $G$, if 
\beaa
\tilde{\Psi} =  \cal{O} \Psi   \cal{O}^{-1} \, ,
\eeaa
then,
\beaa
 \textbf{D}^{(\tilde{A})}_{\alpha}\tilde{\Psi} =  \cal{O} ( \textbf{D}^{(A)}_{\alpha}\Psi )  \cal{O}^{-1} \, .
\eeaa

\end{lemma}

\begin{proof}
Computing
\beaa
\tilde{F_{\a\b}} &=& \der_{\a}\tilde{A_{\b}} - \der_{\b}\tilde{A_{\a}} + [\tilde{A_{\a}},\tilde{A_{\b}}]  \\
&=&  \der_{\a} (  \cal{O} A_{\b}  \cal{O}^{-1}   ) - \der_{\a}  \big(  ( \der_\b  \cal{O} )  \cal{O}^{-1}   \big) - \der_{\b} (  \cal{O} A_{\a}  \cal{O}^{-1}   ) + \der_{\b}  \big(  ( \der_\a  \cal{O} )  \cal{O}^{-1}   \big)  \\
&& + [\cal{O} A_\a \cal{O}^{-1} - ( \der_\a  \cal{O} )  \cal{O}^{-1},  \cal{O} A_\b  \cal{O}^{-1} - ( \der_\b  \cal{O} )  \cal{O}^{-1} ]   \\
&=& ( \der_{\a} \cal{O} ) A_{\b} \cal{O}^{-1}   +   \cal{O} (\der_{\a} A_{\b} ) \cal{O}^{-1}    + \cal{O} A_{\b} ( \der_{\a}  \cal{O}^{-1}  ) \\
&& - ( \der_{\a}  \der_\b \cal{O} ) \cal{O}^{-1}   -    ( \der_\b \cal{O} ) ( \der_{\a} \cal{O}^{-1} )  \\
&& -  ( \der_{\b} \cal{O} ) A_{\a} \cal{O}^{-1}  -  \cal{O}  (\der_{\b} A_{\a} ) \cal{O}^{-1}    -  \cal{O} A_{\a} ( \der_{\b} \cal{O}^{-1}  ) \\
&& + ( \der_{\b}  \der_\a \cal{O} ) \cal{O}^{-1}    -    ( \der_\a \cal{O} ) ( \der_{\b} \cal{O}^{-1} )  \\
&& + [\cal{O} A_\a \cal{O}^{-1} , \cal{O} A_\b \cal{O}^{-1}  ]    - [\cal{O} A_\a \cal{O}^{-1} ,  ( \der_\b \cal{O} ) \cal{O}^{-1} ]    - [ ( \der_\a \cal{O} ) \cal{O}^{-1}, \cal{O} A_\b \cal{O}^{-1}  ] \\
&&  + [ ( \der_\a \cal{O} ) \cal{O}^{-1},  ( \der_\b \cal{O} ) \cal{O}^{-1} ]  \, , \\
\eeaa
and developing the terms in the commutators, we get
\beaa
\tilde{F_{\a\b}} &=&  \cal{O} (\der_{\a} A_{\b} - \der_{\b} A_{\a} ) \cal{O}^{-1}  + \cal{O} A_\a \cal{O}^{-1}  \cal{O} A_\b \cal{O}^{-1}   - \cal{O} A_\b \cal{O}^{-1}  \cal{O} A_\a \cal{O}^{-1} \\
&& + ( \der_{\a} \cal{O} ) A_{\b} \cal{O}^{-1}     + \cal{O} A_{\b} ( \der_{\a}  \cal{O}^{-1}  ) \\
&& - ( \der_{\a}  \der_\b \cal{O} ) \cal{O}^{-1}   -    ( \der_\b \cal{O} ) ( \der_{\a} \cal{O}^{-1} )  \\
&& -  ( \der_{\b} \cal{O} ) A_{\a} \cal{O}^{-1}   -  \cal{O} A_{\a} ( \der_{\b} \cal{O}^{-1}  ) \\
&& + ( \der_{\b}  \der_\a \cal{O} ) \cal{O}^{-1}    -    ( \der_\a \cal{O} ) ( \der_{\b} \cal{O}^{-1} )  \\
&&  - \big( \cal{O} A_\a \cal{O}^{-1}   ( \der_\b \cal{O} ) \cal{O}^{-1}  -  ( \der_\b \cal{O} ) \cal{O}^{-1}   \cal{O} A_\a \cal{O}^{-1}  \big) \\
&&   - \big(  ( \der_\a \cal{O} ) \cal{O}^{-1} \cal{O} A_\b \cal{O}^{-1}   - \cal{O} A_\b \cal{O}^{-1}  ( \der_\a \cal{O} ) \cal{O}^{-1} \big) \\
&&  +  \big( ( \der_\a \cal{O} ) \cal{O}^{-1}  ( \der_\b \cal{O} ) \cal{O}^{-1}  -  ( \der_\a \cal{O} ) \cal{O}^{-1}  ( \der_\b \cal{O} ) \cal{O}^{-1} \big)  \, . \\
\eeaa
Since $\cal{O} \cal{O}^{-1} = I$ and in a system of coordinates $\der_{\a}  \der_\b \cal{O} - \der_{\b}  \der_\a \cal{O} = 0$, we get
\beaa
\tilde{F_{\a\b}}  &=&  \cal{O} \big( \der_{\a} A_{\b} - \der_{\b} A_{\a} + [A_{\a}, A_{\b}] \big) \cal{O}^{-1}  \\
&& + ( \der_{\a} \cal{O} ) A_{\b} \cal{O}^{-1}     + \cal{O} A_{\b} ( \der_{\a}  \cal{O}^{-1}  ) \\
&&  -    ( \der_\b \cal{O} ) ( \der_{\a} \cal{O}^{-1} )  \\
&& -  ( \der_{\b} \cal{O} ) A_{\a} \cal{O}^{-1}   -  \cal{O} A_{\a} ( \der_{\b} \cal{O}^{-1}  ) \\
&&    -    ( \der_\a \cal{O} ) ( \der_{\b} \cal{O}^{-1} )  \\
&&  - \big( \cal{O} A_\a \cal{O}^{-1}   ( \der_\b \cal{O} ) \cal{O}^{-1}  -  ( \der_\b \cal{O} )  A_\a \cal{O}^{-1}  \big) \\
&&   - \big(  ( \der_\a \cal{O} ) A_\b \cal{O}^{-1}   - \cal{O} A_\b \cal{O}^{-1}  ( \der_\a \cal{O} ) \cal{O}^{-1} \big) \\
&&  +  \big( ( \der_\a \cal{O} ) \cal{O}^{-1}  ( \der_\b \cal{O} ) \cal{O}^{-1}  -  ( \der_\a \cal{O} ) \cal{O}^{-1}  ( \der_\b \cal{O} ) \cal{O}^{-1} \big) \, . \\
\eeaa
On the other hand, since
\beaa
\cal{O} \cal{O}^{-1} = I \, ,
\eeaa 
we have
\beaa
( \der_\a \cal{O} ) \cal{O}^{-1} + \cal{O} ( \der_\a \cal{O}^{-1} ) = 0 \, ,
\eeaa
and thus,
\beaa
\der_\a \cal{O}^{-1} = - \cal{O}^{-1} ( \der_\a \cal{O} ) \cal{O}^{-1}  \, .
\eeaa
Therefore 
\beaa
\tilde{F_{\a\b}}  &=&  \cal{O}  F_{\a\b} \cal{O}^{-1}  \\
&& + ( \der_{\a} \cal{O} ) A_{\b} \cal{O}^{-1}     - \cal{O} A_{\b} \cal{O}^{-1} ( \der_\a \cal{O} ) \cal{O}^{-1}   \\
&&  +  ( \der_\b \cal{O} ) \cal{O}^{-1} ( \der_\a \cal{O} ) \cal{O}^{-1}   \\
&& -  ( \der_{\b} \cal{O} ) A_{\a} \cal{O}^{-1}   +  \cal{O} A_{\a} \cal{O}^{-1} ( \der_\b \cal{O} ) \cal{O}^{-1}  \\
&&    -   ( \der_\a \cal{O} ) ( \cal{O}^{-1} ( \der_\b \cal{O} ) \cal{O}^{-1} )  \\
&&  - \cal{O} A_\a \cal{O}^{-1}   ( \der_\b \cal{O} ) \cal{O}^{-1}  + ( \der_\b \cal{O} )  A_\a \cal{O}^{-1}  \\
&&   -   ( \der_\a \cal{O} ) A_\b \cal{O}^{-1}   + \cal{O} A_\b \cal{O}^{-1}  ( \der_\a \cal{O} ) \cal{O}^{-1}  \\
&&  +  ( \der_\a \cal{O} ) \cal{O}^{-1}  ( \der_\b \cal{O} ) \cal{O}^{-1}  -  ( \der_\a \cal{O} ) \cal{O}^{-1}  ( \der_\b \cal{O} ) \cal{O}^{-1} )  \\
&=&  \cal{O}  F_{\a\b} \cal{O}^{-1}  \, .
\eeaa

Now, let
\beaa
\tilde{\Psi} = \cal{O} \Psi  \cal{O}^{-1} \, ,
\eeaa
we compute
\beaa
 \textbf{D}^{(\tilde{A})}_{\alpha}\tilde{\Psi} &=& \der_{\alpha}\tilde{\Psi}  + [\tilde{A_\a}, \tilde{\Psi}] \\
 &=& \der_{\alpha}\tilde{\Psi}  + [\cal{O} A_\a \cal{O}^{-1} - ( \der_\a \cal{O} ) \cal{O}^{-1}, \tilde{\Psi}] \\
 &=& \der_{\alpha}\tilde{\Psi}  + \cal{O} A_\a \cal{O}^{-1}  \tilde{\Psi}  + \tilde{\Psi} \cal{O} A_\a \cal{O}^{-1}  - ( \der_\a \cal{O} ) \cal{O}^{-1} \tilde{\Psi}  +  \tilde{\Psi} ( \der_\a \cal{O} ) \cal{O}^{-1}  \\
 &=& \der_{\alpha} (\cal{O} \Psi  \cal{O}^{-1} )   + \cal{O} A_\a \cal{O}^{-1}  \cal{O} \Psi  \cal{O}^{-1}  + \cal{O} \Psi  \cal{O}^{-1} \cal{O} A_\a \cal{O}^{-1} \\
 && - ( \der_\a \cal{O} ) \cal{O}^{-1} \cal{O} \Psi  \cal{O}^{-1} +  \cal{O} \Psi  \cal{O}^{-1}( \der_\a \cal{O} ) \cal{O}^{-1}  \\
  &=& ( \der_{\alpha} \cal{O} ) \Psi  \cal{O}^{-1}  + \cal{O} ( \der_{\alpha} \Psi ) \cal{O}^{-1}  + \cal{O} \Psi (  \der_{\alpha} \cal{O}^{-1} ) \\
 &&   + \cal{O} A_\a \Psi  \cal{O}^{-1}  + \cal{O} \Psi  A_\a \cal{O}^{-1} \\
 && - ( \der_\a \cal{O} )  \Psi  \cal{O}^{-1} +  \cal{O} \Psi  \cal{O}^{-1}( \der_\a \cal{O} ) \cal{O}^{-1}  \\
  &=&  \cal{O} ( \der_{\alpha} \Psi ) \cal{O}^{-1}  + \cal{O} \Psi (  \der_{\alpha} \cal{O}^{-1} ) + \cal{O}  [A_\a, \Psi]  \cal{O}^{-1}  \\
   && - \cal{O} \Psi  \cal{O} ^{-1} ( \cal{O} \der_\a \cal{O}^{-1} ) \\
   &=&  \cal{O} ( \der_{\alpha} \Psi ) \cal{O}^{-1}  + \cal{O}  [A_\a, \Psi]  \cal{O}^{-1}  \\
    &=&  \cal{O} ( \textbf{D}^{(A)}_{\alpha}\Psi ) \cal{O}^{-1} \, .
 \eeaa

\end{proof}
Now, we state the following well-known lemma:
\begin{lemma}\label{gaugeinvarianceoftheYangMillssystem}
 If $(M, A, \g)$ is a solution to the Einstein-Yang-Mills equations, then $(M, \tilde{A}, \g)$ is also a solution which is what we call the gauge invariance of the equations.
\end{lemma}
\begin{proof}
Let $F_{\a\b}$ be a solution to the Einstein-Yang-Mills system
\beaa
R_{ \mu \nu} - \frac{1}{2} \cdot g_{\mu\nu} R &=&   2 \big( < F_{\mu\b}, F_{\nu}^{\;\;\b}> - \frac{1}{4} g_{\mu\nu } < F_{\a\b },F^{\a\b } > \big) \;.
\eeaa
Since
\beaa
 &&   < \tilde{F}_{\mu\b},  \tilde{F}_{\nu}^{\;\;\b}> - \frac{1}{4} \cdot g_{\mu\nu } <  \tilde{F}_{\a\b }, \tilde{F}^{\a\b } > \\
 &=&  < \cal{O} F_{\mu\b} \cal{O}^{-1}, \cal{O} F_{\nu}^{\;\;\b} \cal{O}^{-1}> - \frac{1}{4} \cdot g_{\mu\nu } < \cal{O} F_{\a\b } \cal{O}^{-1} , \cal{O} F^{\a\b } \cal{O}^{-1}>  \\
  &=&  <  F_{\mu\b} , F_{\nu}^{\;\;\b}> - \frac{1}{4} \cdot g_{\mu\nu } <  F_{\a\b }  ,  F^{\a\b } > \;,
  \eeaa
then, $\tilde{F}_{\a\b}$ is also a solution to the Einstein-Yang-Mills system
\bea
R_{ \mu \nu} - \frac{1}{2} \cdot g_{\mu\nu} R &=&   2 \big(  < \tilde{F}_{\mu\b},  \tilde{F}_{\nu}^{\;\;\b}> - \frac{1}{4} \cdot g_{\mu\nu } <  \tilde{F}_{\a\b }, \tilde{F}^{\a\b } > \big) \;,
\eea
which will in turn enforce, by the symmetries of the Riemann tensor, that
\beaa
\der^{\mu}  T_{\mu\nu} (\tilde{F})= 0 \;,
\eeaa
and since we also have the Bianchi identities for $\tilde{F}$ (given the expression of $\tilde{F}$ in terms of the potential $\tilde{A}$), we also have
\bea
 \textbf{D}^{(\tilde{A})}_{\a} \tilde{F}_{\mu\nu} + \textbf{D}^{(\tilde{A})}_{\mu} \tilde{F}_{\nu\a} + \textbf{D}^{(\tilde{A})}_{\nu} \tilde{F}_{\a\mu}\; ,
 \eea
which all together leads to
\bea
\textbf{D}^{(\tilde{A})}_{\a} \tilde{F}^{\a\b} = \der_{\alpha} \tilde{F}^{\a\b}  + [\tilde{A}_{\alpha}, \tilde{F}^{\a\b} ]  = 0 \;.
\eea
This is consistent with the fact that
\beaa
\textbf{D}^{(\tilde{A})}_{\a} \tilde{F}^{\a\b} = \cal{O} ( \textbf{D}^{(A)}_{\a} F^{\a\b} ) \cal{O}^{-1} = 0\;,
\eeaa
and that 
\beaa
 \textbf{D}^{(\tilde{A})}_{\a} \tilde{F}_{\mu\nu} + \textbf{D}^{(\tilde{A})}_{\mu} \tilde{F}_{\nu\a} + \textbf{D}^{(\tilde{A})}_{\nu} \tilde{F}_{\a\mu}  = \cal{O} ( \textbf{D}^{(A)}_{\a}F_{\mu\nu} + \textbf{D}^{(A)}_{\mu}F_{\nu\a} + \textbf{D}^{(A)}_{\nu}F_{\a\mu}  )  \cal{O}^{-1} = 0 \;.
\eeaa

\end{proof}
Consequently, for each solution $F$ of the Einstein-Yang-Mills equation, we can make a gauge transformation
\beaa
\tilde{A_\a} = \cal{O} A_\a \cal{O}^{-1} - ( \der_\a \cal{O} ) \cal{O}^{-1} \, ,
\eeaa
and define a new solution $\tilde{A}$, which is what we call the gauge invariance of the equations. Hence, a solution $A$ to the Einstein-Yang-Mills system is only defined up to a class.

We know that a global existence result for the Yang-Mills fields for any arbitrary gauge on the Yang-Mills potential fails. In fact, given any global solution for the Einstein-Yang-Mills equations, one can always perform a gauge transformation on the Yang-Mills potential so that the gauge transformed solution remains a solution to the Einstein-Yang-Mills equations but blows-up in finite time. Thus, fixing a gauge condition on the Yang-Mills fields is essential in order to obtain a global solution. We choose here to work in the Lorenz gauge, which impose on the solution to satisfy the following condition
\bea\label{TheLorenzgaugeconditionthatwechoose}
\der^{\a}   A_{\a}  &=&  0 \; .
\eea

\subsection{The diffeomorphism invariance}\label{The diffeomorphism invariance}\

Let $(\cal{M}, A, \g)$ be a solution to the Einstein-Yang-Mills system. Now, consider a diffeomorphism $\phi : \cal{M} \to \cal{M}^{\prime}$ and define a metric $\g^\prime$ on $\cal{M}^{\prime}$ by $\phi^* \g^\prime = \g$ where $\phi^*$ is the pull-back of $\phi$. In other words, at each point $p \in \cal{M}$, we define $\g^\prime$ by
\beaa
g^\prime (\phi_* X, \phi_*Y)  (\phi (p)) = g (X, Y) (p) \, ,
\eeaa
where $\phi_* : T_p\cal{M} \to T_p\cal{M}^\prime $ is the push-forward of $\phi$ defined through
\beaa
\phi_* X(f) (\phi (p)) = X (\phi^* f) (p) \, ,
\eeaa
for all smooth functions $f$ on $\cal{M}^{\prime}$ and for all $X \in T_p\cal{M}$, and where $\phi^* f(p) := f(\phi(p))$.

Then, $(\cal{M}^\prime, A, \g^\prime)$ is also a solution to the Einstein-Yang-Mills equations, which is what we call the diffeomorphism invariance of the equations. Hence, a solution to the Einstein-Yang-Mills system is only defined up to a class, where two solutions are the same if they are isometric. However, this gives us the freedom to choose a representative of this class. 

Another way to look at the solution, is that one can eliminate the diffeomorphism invariance by fixing a system of coordinates. We choose to look at the manifold in harmonic coordinates, which 
means that we are fixing our system of coordinates $\{ x^\mu \} $ such that:
\bea\label{Thewavecoordinateconditionthatwechooseforoursystemofcoordinates}
\Box_{\g} x^\mu := \der^{\a}  \der_{\a}  x^\mu =  0 \, ,
\eea
and such that $x^0 = 0 $ on $ \Sigma $.

Since $x^{\mu}$ are scalar functions on $\cal{M}$, and not tensors, we have,
\beaa
\der_{\a}  x^{\mu}  &=&  \pa_{\a} ( x^\mu  )     .
\eeaa

We evaluate
\beaa
\der^{\a}  \der_{\a}  x^{\mu} &=& \pa^{\a}  \der_{\a}  x^{\mu}  -    \der_{\der^{\a } e_\a}  x^{\mu}   \\
&=&  \pa^{\a}  \pa_{\a} x^\mu     - { {\Ga^{\a}}_{\a}}^{\b}  \der_{\b }  x^{\mu}  \\
&=&  \pa^{\a}  \pa_{\a} x^\mu   - { {\Ga^{\a}}_{\a}}^{\b}  \pa_{\b }  x^{\mu}  \, ,
\eeaa
where $\Ga_{\, \, \, \a}^{\a  \, \, \, \mu}$ are the Christoffel symbols. We have
\beaa
\Ga_{\, \, \, \a}^{\a  \, \, \, \mu} = g^{\a\b}\Ga_{\b\a}^{\, \, \, \, \, \, \, \mu}  \, .
\eeaa

Computing the contraction, we get
 \beaa
 \pa_{\a}  x^\mu  = \begin{cases} 1,\quad\text{for}\quad \a= \mu \; ,\\
    0,\quad\text{for }\quad \a \neq \mu \; .\end{cases} 
\eeaa

Hence,
\beaa
 \pa^{\a}  \pa_{\a} x^\mu  = \pa^{\mu}  \pa_{\mu}   x^\mu  = \pa^{\mu} 1 = 0 \, .
  \eeaa
Thus,
\beaa
\der^{\a}  \der_{\a}  x^{\mu} &=&   -\Ga_{\, \, \, \a}^{\a  \, \, \, \b} \pa_{\b }  x^{\mu}  \\
 &=&   -\Ga_{\, \, \, \a}^{\a  \, \, \, \mu}  \pa_{\mu }  x^{\mu}  \\
  &=&   -\Ga_{\, \, \, \a}^{\a  \, \, \, \mu} \, .
\eeaa
Consequently,
\beaa
\der^{\a}  \der_{\a}  x^{\mu} &=& 0   
\eeaa
is equivalent to
\bea\label{wavecoordinatecondition}
 \Ga_{\, \, \, \a}^{\a  \, \, \, \mu} &=&g^{\a\nu} \Ga_{\nu \a}^{\, \,  \,  \,   \, \, \, \mu} =  0  \, .
\eea

Now, for any arbitrary tensor $\Psi$, we have
\beaa
\der^{\a}  \der_{\a}  \Psi  &=&   \der^{\a} ( \der_{\a}  \Psi ) -  \der_{ \der^{\a}  e_{\a}  }  \Psi \\
&=&   \der^{\a} ( \der_{\a}  \Psi ) -  \der_{ \Ga_{\, \, \, \a}^{\a  \, \, \, \mu} e_{\mu}  }  \Psi  \\
&=&   \der^{\a} ( \der_{\a}  \Psi ) -  \Ga_{\, \, \, \a}^{\a  \, \, \, \mu} \der_{  \mu  }  \Psi  \, .
\eeaa

Consequently, either in harmonic coordinates or in a geodesic frame (i.e. a frame where the Christoffel symbols vanish), we can write
\bea
\der^{\a}  \der_{\a}  \Psi &=&   \der^{\a} ( \der_{\a}  \Psi ) \, .
\eea

\begin{lemma}
In either wave coordinates or in a geodesic frame, the Lorenz gauge can be written as
\bea
\pa^{\a}  A_{\a}  & =& 0 .
 \eea
\end{lemma}
\begin{proof}
We have
\beaa
  \der^{\a}  A_{\a} &=&  \pa^{\a}  A_{\a}  - A (\der^{\a} e_{\a} )    \\
&=& \pa^{\a}  A_{\a}  -\Ga_{\, \, \, \a}^{\a  \, \, \, \mu}  A (e_{\mu} ) \, .
 \eeaa
Thus, the result follows.
\end{proof}

\section{Looking at the metric as a perturbation of the Minkowski space-time}\label{LookingmetricperturbationMinkowskispacetime}

Now that we have fixed the coordinates to be the wave coordinates, let $m$ be Minkowski metric in these wave coordinates $\{x^0, ..., x^n\}$, i.e. $m$ is the metric prescribed by: 
\beaa
m_{00}&:=&-1\;,\qquad  m_{ii}:=1,\quad \text{if}\quad i=1, ...,n\,, \\
\quad\text{and}\quad m_{\mu\nu}&:=&0\;,\quad \text{if}\quad \mu\neq \nu \quad \text{for} \quad  \mu, \nu \in \{0, 1, ..., n\} \,.
\eeaa
\begin{definition}\label{definitionofsmallhandbigH}
We define $h$ as the 2-tensor given by:
\bea
h_{\mu\nu} := g_{\mu\nu} - m_{\mu\nu} \; .
\eea

Let $m^{\mu\nu}$ be the inverse of $m_{\mu\nu}$\,. We define
\bea
h^{\mu\nu} &:=& m^{\mu\mu^\prime}m^{\nu\nu^\prime}h_{\mu^\prime\nu^\prime} \;,\\
H^{\mu\nu} &:=& g^{\mu\nu}-m^{\mu\nu} \;.
\eea
\end{definition}

\begin{definition}\label{definitionofbigOonlyforAandhadgradientofAandgardientofh}
Let $K$ be a tensor that is either $A$ or $h$ or $H$\;, or $\derm A$\;, $\derm h$ or $\derm H$\;.

Let $ P_n (K )$ be tensors that are Polynomials of degree $n$\;, and $Q_1 (K)$ a tensor that is a Polynomial of degree $1$ such that $Q_1 (0) = 0$ and $Q_1 \neq 0$\;, of which the coefficients are components in wave coordinates of the metric $\textbf m$ and of the inverse metric $\textbf m^{-1}$, and of which the variables are components in wave coordinates of the covariant tensor $K$, leaving some indices free, so that the following product gives a tensor that we define as,
\bea
\notag
O_{\mu_{1} \ldots \mu_{k} } (K  ) &:=& Q_1 ( K) \cdot \Big( \sum_{n=0}^{\infty} P_n ( K ) \Big) \; .\\
\eea
For a family of tensors $K^{(1)}, \ldots,  K^{(m)}$, where each tensor $K^{(l)}$ is again either $A$ or $h$ or $H$\;, or $\derm A$\;, $\derm h$ or $\derm H$\;, we define
\bea
\notag
O_{\mu_{1} \ldots \mu_{k} } (K^{(1)} \cdot \ldots \cdot K^{(m)} ) &:=& \prod_{l=1}^{m} Q_{1}^{l} ( K^{(l)} ) \cdot \Big( \sum_{n=0}^{\infty} P_{n}^{l}  ( K^{(l)} ) \Big) \; .\\
\eea
where again $P_n^{l} (K^{l} )$ and $Q_1^l (K)$, are tensors that are Polynomials of degree $n$ and $1$, respectively, with $Q_1 (0) = 0$ and $Q_1 \neq 0$\;, of which the coefficients are components in wave coordinates of the metric $\textbf m$ and of the inverse metric $\textbf m^{-1}$, and of which the variables are components in wave coordinates of the covariant tensor $K^{l}$, leaving some indices free, so that at the end the whole product $\prod_{l=1}^{m} Q_{1}^{l} ( K^{(l)} ) \cdot \Big( \sum_{n=0}^{\infty} P_{n}^{l}  ( K^{(l)} ) \Big)$ gives a tensor which we define as $O_{\mu_{1} \ldots \mu_{k} } (K^{(1)} \cdot \ldots \cdot K^{(m)} )$. To lighten the notation, we shall sometimes drop the indices and just write $O (K^{(1)} \cdot \ldots \cdot K^{(m)} )$\;.

\begin{remark}\label{remarkaboutthedifferenceofdefinitionofbigOwhenwetalkaboutLiederivatives}
Note that in this Definition \ref{definitionofbigOonlyforAandhadgradientofAandgardientofh}, we did not include Lie derivatives of $A$ or $h$ or $H$ neither Lie derivatives of $\derm A$\;, $\derm h$ or $\derm H$. We will however generalize this definition, in a separate definition to include the Lie derivatives (see Definition \ref{definitionofbigOforLiederivatives}).
\end{remark}

\begin{remark}\label{thedefinitionofBigOforthesespecifictensorsalsoappliesforpartialderivativeinwavecoordinates}
The same definition for $O$ as in Definition \ref{definitionofbigOonlyforAandhadgradientofAandgardientofh}, is considered when we use the notation $\pa A$\;, $\pa h$ or $\pa H$ (instead of the Minkowski covariant derivatives), where naturally, the tensors are simply replaced by their partial derivatives in wave coordinates.
\end{remark}

\end{definition}

\begin{lemma}\label{BigHintermsofsmallh}
We have,
\bea
H^{\mu\nu}=-h^{\mu\nu}+O^{\mu\nu}(h^2) \, ,
\eea
or differently written
\beaa
g^{\mu\nu} = m^{\mu\nu}-h^{\mu\nu}  +O^{\mu\nu}(h^2) \, .
\eeaa

\end{lemma}

\begin{proof}
We compute,
\beaa
g_{\mu\a} \big(m^{\a\nu}  - h^{\a\nu} \big) &=& ( h_{\mu\a} + m_{\mu\a} )  \big(m^{\a\nu}  - h^{\a\nu} \big) =  ( h_{\mu\a} + m_{\mu\a} )  \big(m^{\a\nu}  -m^{\a\mu^\prime}m^{\nu\nu^\prime}h_{\mu^\prime\nu^\prime}   \big) \\
&=&   h_{\mu\a}  \big(m^{\a\nu}  -m^{\a\mu^\prime}m^{\nu\nu^\prime}h_{\mu^\prime\nu^\prime}   \big)  +   m_{\mu\a}   \big(m^{\a\nu}  -m^{\a\mu^\prime}m^{\nu\nu^\prime}h_{\mu^\prime\nu^\prime}   \big) \\
&=& h_{\mu\a}  m^{\a\nu}  -  h_{\mu\a} m^{\a\mu^\prime}m^{\nu\nu^\prime}h_{\mu^\prime\nu^\prime}     +   m_{\mu\a}   m^{\a\nu}  -m_{\mu\a}  m^{\a\mu^\prime}m^{\nu\nu^\prime}h_{\mu^\prime\nu^\prime}    .\\
\eeaa

Thus,
\beaa
g_{\mu\a} \big(m^{\a\nu}  - h^{\a\nu} \big) &=& h_{\mu\a}  m^{\a\nu}  -  h_{\mu\a} m^{\a\mu^\prime}m^{\nu\nu^\prime}h_{\mu^\prime\nu^\prime}     +   m_{\mu\a}   m^{\a\nu}  -m_{\mu\a}  m^{\a\mu^\prime}m^{\nu\nu^\prime}h_{\mu^\prime\nu^\prime}    \\
&=& h_{\mu\a}  m^{\a\nu}      + O_{\mu}^{\ \; \nu} (h^2) +   {I_{\mu}}^{\nu}     -{I_{\mu}}^{\mu^\prime}m^{\nu\nu^\prime}h_{\mu^\prime\nu^\prime}    \\
&=&          I_{\mu}^{\ \; \nu} + O_{\mu}^{\ \; \nu}(h^2)  +h_{\mu\a}  m^{\a\nu}     -   m^{\nu\nu^\prime} {I_{\mu}}^{\mu^\prime} h_{\mu^\prime\nu^\prime}    \\
&=&          I_{\mu}^{\ \; \nu} + O_{\mu}^{\ \; \nu}(h^2)  +h_{\mu\a}  m^{\a\nu}     -   m^{\nu\nu^\prime} h_{\mu\nu^\prime}    \\
&=&          I_{\mu}^{\ \; \nu} + O_{\mu}^{\ \; \nu}(h^2)  +h_{\mu\a}  m^{\a\nu}     -   m^{\nu^\prime\nu} h_{\mu\nu^\prime}    \\
&=&          I_{\mu}^{\ \; \nu}  + O_{\mu}^{\ \; \nu}(h^2)  +h_{\mu\a}  m^{\a\nu}     -    h_{\mu\nu^\prime}m^{\nu^\prime\nu}     \\
&=&          I_{\mu}^{\ \; \nu} + O_{\mu}^{\ \; \nu}(h^2)     .
\eeaa

Hence,
\beaa
g_{\mu\a} \big(m^{\a\nu}  - h^{\a\nu} \big) &=&          I_{\mu}^{\ \; \nu}+ O_{\mu}^{\ \; \nu}(h^2)  \, ,  \\
\eeaa
and multiplying on both sides, we get
\beaa
g^{\la \mu} g_{\mu\a} \big(m^{\a\nu}  - h^{\a\nu} \big) &=&  g^{\la \mu}        I_{\mu}^{\ \; \nu} + g^{\la \mu} O_{\mu}^{\ \; \nu} (h^2)     \\
 I^{\la}_{\ \; \a}  \big(m^{\a\nu}  - h^{\a\nu} \big)    &=&  g^{\la \nu}   + g^{\la \mu} O_{\mu}^{\ \; \nu} (h^2)     \\
m^{\la\nu}  - h^{\la\nu}     &=&  g^{\la \nu}   + g^{\la \mu} O_{\mu}^{\ \; \nu} (h^2)     \, , 
\eeaa
which gives,
\beaa
 g^{\la \nu}    &=&  m^{\la\nu}  - h^{\la\nu}  + g^{\la \mu} O_{\mu}^{\ \; \nu}  (h^2)   . \\
\eeaa

Consequently,
\beaa
  g^{\la \nu}  \big(1 +  O (h^2) \big) &=&  m^{\la\nu}  - h^{\la\nu}  \, ,
  \eeaa
and therefore,
\beaa
  g^{\la \nu}   &=&   \big( m^{\la\nu}  - h^{\la\nu}  \big)  \big(1 + O (h^2) \big)^{-1} \\
  &=&   \big( m^{\la\nu}  - h^{\la\nu}  \big)  \big(1 +  O (h^2) \big) \\
    &=&   m^{\la\nu}  - h^{\la\nu}  + m^{\la\nu}  O (h^2) + h^{\la\nu}  O (h^2) \\
     &=&   m^{\la\nu}  - h^{\la\nu}  +  O^{\la\nu} (h^2) + m^{\la\mu^\prime}m^{\nu\nu^\prime}h_{\mu^\prime\nu^\prime} O (h^2) \\
          &=&   m^{\la\nu}  - h^{\la\nu}  +  O^{\la\nu} (h^2) +  O^{\la\nu} (h^3) \\
          &=&   m^{\la\nu}  - h^{\la\nu}  +  O^{\la\nu} (h^2)  \, .
  \eeaa
  Thus, using Definition \ref{definitionofsmallhandbigH}, we get
\beaa
H^{\mu\nu}=g^{\mu\nu}-m^{\mu\nu}=-h^{\mu\nu}+O^{\mu\nu}(h^2) \, .
\eeaa

  \end{proof}
  
  \section{The Einstein-Yang-Mills equations in a given system of coordinates}

\begin{lemma}\label{EinsteinYangMillssystemusingpotentialandusingmetricg}

The Einstein-Yang-Mills equations read in a given system of coordinates, i.e where $\a\;, \b\;, \si\;, \la\;,$ run over a given system of coordinates, as follows,
\beaa
 && R_{ \mu \nu}  \\
 &=&  2  g^{\si\b} <   \pa_{\mu}A_{\b} - \pa_{\b}A_{\mu}  ,  \pa_{\nu}A_{\si} - \pa_{\nu}A_{\si}  >    - \frac{1}{(n-1)} g_{\mu\nu } \cdot g^{\si\b} g^{\a\la}  <  \pa_{\a}A_{\b} - \pa_{\b}A_{\a} , \pa_{\la}A_{\si} - \pa_{\si}A_{\la} >  \\
&& + 2  g^{\si\b} <   \pa_{\mu}A_{\b} - \pa_{\b}A_{\mu}  ,  [A_{\nu},A_{\si}] >  + 2  g^{\si\b} <   [A_{\mu},A_{\b}] ,  \pa_{\nu}A_{\si} - \pa_{\si}A_{\nu}  >  \\
&&  - \frac{1}{(n-1)} g_{\mu\nu } \cdot g^{\si\b} g^{\a\la}   <  \pa_{\a}A_{\b} - \pa_{\b}A_{\a} , [A_{\la},A_{\si}] >     - \frac{1}{(n-1)} g_{\mu\nu } \cdot g^{\si\b} g^{\a\la}  <  [A_{\a},A_{\b}] , \pa_{\la}A_{\si} - \pa_{\si}A_{\la}  >  \\
&& + 2  g^{\si\b} <   [A_{\mu},A_{\b}] ,  [A_{\nu},A_{\si}] >   - \frac{1}{(n-1)} g_{\mu\nu } \cdot g^{\si\b} g^{\a\la}  <  [A_{\a},A_{\b}] , [A_{\la},A_{\si}] >  .
\eeaa

\end{lemma}

\begin{proof}
The definition of the Yang-Mills curvature in \eqref{expF} gives that
\beaa
 F_{\mu\b}  &=& \der_{\mu}A_{\b} - \der_{\b}A_{\mu} + [A_{\mu},A_{\b}] \\
 &=& \pa_{\mu}A_{\b} - A (\der_\mu e_b)   - \pa_{\b}A_{\mu}  + A (\der_b e_\mu)+ [A_{\mu},A_{\b}]\\
 &=& \pa_{\mu}A_{\b}    - \pa_{\b}A_{\mu} + A (\der_b e_\mu - \der_\mu e_b)+ [A_{\mu},A_{\b}]\\
  &=& \pa_{\mu}A_{\b}    - \pa_{\b}A_{\mu} + A ([ e_b,  e_\mu ]) + [A_{\mu},A_{\b}] \; .
 \eeaa
 In a given system of coordinates, we have $[ e_b,  e_\mu ] = 0$. Therefore, in a system of coordinates, we have 
 \beaa
 F_{\mu\b} &=& \pa_{\mu}A_{\b}    - \pa_{\b}A_{\mu}  + [A_{\mu},A_{\b}] \; .
 \eeaa
 
We know from \eqref{simpEYM}, that the Einstein-Yang-Mills equations are

\beaa
R_{ \mu \nu}  &=& 2  \big( < F_{\mu\b}, F_{\nu}^{\;\;\b}> - g_{\mu\nu } \frac{1}{2(n-1)} < F_{\a\b },F^{\a\b }  > \big) \\
&=&  2 \big( g^{\si\b} <  F_{\mu\b}, {F_{\nu\si}}> - \frac{1}{2(n-1)} g_{\mu\nu } g^{\si\b} g^{\a\la} < F_{\a\b },F_{\la\si } > \big) .
\eeaa
Computing the right hand side, we get
\beaa
&&R_{ \mu \nu} \\
&=&  2 \big( g^{\si\b} <  F_{\mu\b}, {F_{\nu\si}}> - \frac{1}{2(n-1)} g_{\mu\nu } g^{\si\b} g^{\a\la}  < F_{\a\b },F_{\la\si } > \big) \\
&=&  2  g^{\si\b} <   \pa_{\mu}A_{\b} - \pa_{\b}A_{\mu} + [A_{\mu},A_{\b}] ,  \pa_{\nu}A_{\si} - \pa_{\si}A_{\nu} + [A_{\nu},A_{\si}] >  \\
&&  - \frac{1}{(n-1)} g_{\mu\nu } g^{\si\b} g^{\a\la}  <  \pa_{\a}A_{\b} - \pa_{\b}A_{\a} + [A_{\a},A_{\b}] , \pa_{\la}A_{\si} - \pa_{\si}A_{\la} + [A_{\la},A_{\si}] >  \\
&=&  2  g^{\si\b} <   \pa_{\mu}A_{\b} - \pa_{\b}A_{\mu}  ,  \pa_{\nu}A_{\si} - \pa_{\si}A_{\nu} + [A_{\nu},A_{\si}] >  \\
&& + 2  g^{\si\b} <   [A_{\mu},A_{\b}] ,  \pa_{\nu}A_{\si} - \pa_{\si}A_{\nu} + [A_{\nu},A_{\si}] >  \\
&&  - \frac{1}{(n-1)} g_{\mu\nu } g^{\si\b} g^{\a\la}  <  \pa_{\a}A_{\b} - \pa_{\b}A_{\a} , \pa_{\la}A_{\si} - \pa_{\si}A_{\la} + [A_{\la},A_{\si}] >  \\
&&  - \frac{1}{(n-1)} g_{\mu\nu } g^{\si\b} g^{\a\la}  <  [A_{\a},A_{\b}] , \pa_{\la}A_{\si} - \pa_{\si}A_{\la} + [A_{\la},A_{\si}] >  \\
&=&  2  g^{\si\b} <   \pa_{\mu}A_{\b} - \pa_{\b}A_{\mu}  ,  \pa_{\nu}A_{\si} - \pa_{\si}A_{\nu}  >  \\
&& + 2  g^{\si\b} <   \pa_{\mu}A_{\b} - \pa_{\b}A_{\mu}  ,  [A_{\nu},A_{\si}] >  \\
&& + 2  g^{\si\b} <   [A_{\mu},A_{\b}] ,  \pa_{\nu}A_{\si} - \pa_{\si}A_{\nu}  >  \\
&& + 2  g^{\si\b} <   [A_{\mu},A_{\b}] ,  [A_{\nu},A_{\si}] >  \\
&&  - \frac{1}{(n-1)} g_{\mu\nu } g^{\si\b} g^{\a\la}  <  \pa_{\a}A_{\b} - \pa_{\b}A_{\a} , \pa_{\la}A_{\si} - \pa_{\si}A_{\la} >  \\
&&  - \frac{1}{(n-1)} g_{\mu\nu } g^{\si\b} g^{\a\la} <  \pa_{\a}A_{\b} - \pa_{\b}A_{\a} , [A_{\la},A_{\si}] >  \\
&&  - \frac{1}{(n-1)} g_{\mu\nu } g^{\si\b} g^{\a\la}  <  [A_{\a},A_{\b}] , \pa_{\la}A_{\si} - \pa_{\si}A_{\la}  >  \\
&&  - \frac{1}{(n-1)} g_{\mu\nu } g^{\si\b} g^{\a\la} <  [A_{\a},A_{\b}] , [A_{\la},A_{\si}] >  .
\eeaa
Hence, we get the stated result.
\end{proof}

  \begin{lemma}\label{EinsteinYangMillssystemusingpotentialandusingMinkowskimetricmandperturbationh},
  The Einstein-Yang-Mills equations in a given system of coordinates, i.e. where $\a\;, \b\;, \si\;, \la\;,$ run over a given system of coordinates,  can be written as
\beaa
\notag
&& R_{ \mu \nu}   \\
\notag
 &=&   2  m^{\si\b} \cdot  <   \pa_{\mu}A_{\b} - \pa_{\b}A_{\mu}  ,  \pa_{\nu}A_{\si} -\pa_{\si}A_{\nu}  >           - \frac{1}{(n-1)} m_{\mu\nu } \cdot m^{\si\b}  m^{\a\la}    \cdot   <  \pa_{\a}A_{\b} - \pa_{\b}A_{\a} , \pa_{\la}A_{\si} - \pa_{\si}A_{\la} >   \\
 \notag
&&  +           2  m^{\si\b}  \cdot  \big( <   \pa_{\mu}A_{\b} - \pa_{\b}A_{\mu}  ,  [A_{\nu},A_{\si}] >   + <   [A_{\mu},A_{\b}] ,  \pa_{\nu}A_{\si} - \pa_{\si}A_{\nu} > \big)  \\
\notag
&&    - \frac{1}{(n-1)} m_{\mu\nu }  \cdot m^{\si\b}  m^{\a\la}   \cdot  \big(  <  \pa_{\a}A_{\b} - \pa_{\b}A_{\a} , [A_{\la},A_{\si}] >    +  <  [A_{\a},A_{\b}] , \pa_{\la}A_{\si} - \pa_{\si}A_{\la}  > \big) \\
\notag
 && +      2 m^{\si\b}  \cdot  <   [A_{\mu},A_{\b}] ,  [A_{\nu},A_{\si}] >       - \frac{1}{(n-1)} m_{\mu\nu } \cdot   m^{\si\b}  m^{\a\la}   .  <  [A_{\a},A_{\b}] , [A_{\la},A_{\si}] >  \\
     && + O \big( h \cdot  (\pa A)^2 \big)   + O \big(  h \cdot  A^2 \cdot  \pa A \big)     + O \big(  h  \cdot  A^4 \big)    \, ,
\eeaa

where here the notation $O$ is defined as in Remark \ref{thedefinitionofBigOforthesespecifictensorsalsoappliesforpartialderivativeinwavecoordinates}.

\end{lemma}

\begin{proof}

In Lemma \ref{EinsteinYangMillssystemusingpotentialandusingmetricg}, we compute the terms on the right hand side of the equality, one by one in order,
   \beaa
&& R_{ \mu \nu}  \\
&=&  2  g^{\si\b} <   \pa_{\mu}A_{\b} - \pa_{\b}A_{\mu}  ,  \pa_{\nu}A_{\si} - \pa_{\si}A_{\nu}  >    - \frac{1}{(n-1)} g_{\mu\nu } g^{\si\b} g^{\a\la}  <  \pa_{\a}A_{\b} - \pa_{\b}A_{\a} , \pa_{\la}A_{\si} - \pa_{\si}A_{\la} >  \\
&& + 2  g^{\si\b} \big( <   \pa_{\mu}A_{\b} - \pa_{\b}A_{\mu}  ,  [A_{\nu},A_{\si}] >  +  <   [A_{\mu},A_{\b}] ,  \pa_{\nu}A_{\si} - \pa_{\si}A_{\nu} >   \big) \\
&&  - \frac{1}{(n-1)} g_{\mu\nu } g^{\si\b} g^{\a\la}  \big(  <  \pa_{\a}A_{\b} - \pa_{\b}A_{\a} , [A_{\la},A_{\si}] >    +  <  [A_{\a},A_{\b}] , \pa_{\la}A_{\si} - \pa_{\si}A_{\la}  > \big)  \\
&& + 2  g^{\si\b} <   [A_{\mu},A_{\b}] ,  [A_{\nu},A_{\si}] >   - \frac{1}{(n-1)} g_{\mu\nu } g^{\si\b} g^{\a\la}  <  [A_{\a},A_{\b}] , [A_{\la},A_{\si}] >  \, . 
\eeaa

\textbf{First term}

We have
\beaa
 &&  2  g^{\si\b} <   \pa_{\mu}A_{\b} - \pa_{\b}A_{\mu}  ,  \pa_{\nu}A_{\si} - \pa_{\si}A_{\nu} >  \\
 &=&   2  m^{\si\b} <   \pa_{\mu}A_{\b} - \pa_{\b}A_{\mu}  ,  \pa_{\nu}A_{\si} - \pa_{\si}A_{\nu}  >   -  2  h^{\si\b} <   \pa_{\mu}A_{\b} - \pa_{\b}A_{\mu}  ,  \pa_{\nu}A_{\si} - \pa_{\si}A_{\nu}>    \\
 && + O (h^2 \cdot  (\pa A)^2 ) \\
  &=&   2  m^{\si\b} <   \pa_{\mu}A_{\b} - \pa_{\b}A_{\mu}  ,  \pa_{\nu}A_{\si} - \pa_{\si}A_{\nu} > \\
   && + O ((h + h^2 )\cdot  (\pa A)^2 ) \\
     &=&   2  m^{\si\b} <   \pa_{\mu}A_{\b} - \pa_{\b}A_{\mu}  ,  \pa_{\nu}A_{\si} - \pa_{\si}A_{\nu} > \\
   && + O (h  \cdot (\pa A)^2 ) \, .
\eeaa

\textbf{Second term}

\beaa
 && - \frac{1}{(n-1)} g_{\mu\nu } g^{\si\b} g^{\a\la}  <  \pa_{\a}A_{\b} - \pa_{\b}A_{\a} , \pa_{\la}A_{\si} - \pa_{\si}A_{\la} >   \\
 &=& - \frac{1}{(n-1)} ( m_{\mu\nu } + h_{\mu\nu}  ) (m^{\si\b} -h^{\si\b}   +O^{\si\b} (h^2) ) (m^{\a\la} -h^{\a\la}   +O^{\a\la} (h^2) )  \\
 && .<  \pa_{\a}A_{\b} - \pa_{\b}A_{\a} , \pa_{\la}A_{\si} - \pa_{\si}A_{\la} >   \\
 &=& \big[ - \frac{1}{(n-1)} m_{\mu\nu }  \big( m^{\si\b} -h^{\si\b}   +O^{\si\b} (h^2) \big)  \big( m^{\a\la} -h^{\a\la}   +O^{\a\la} (h^2) \big)  \\
 && - \frac{1}{(n-1)} h_{\mu\nu}  \big( m^{\si\b} -h^{\si\b}   +O^{\si\b} (h^2) \big) \big( m^{\a\la} -h^{\a\la}   +O^{\a\la} (h^2) \big)  \big] \\
 && .  <  \pa_{\a}A_{\b} - \pa_{\b}A_{\a} , \pa_{\la}A_{\si} - \pa_{\si}A_{\la} >   \\
 &=& I_1 + I_2 \, .
\eeaa

We have
\beaa
&& I_1 \\
 &=&  - \frac{1}{(n-1)} m_{\mu\nu } (m^{\si\b} -h^{\si\b}   +O^{\si\b} (h^2) ) (m^{\a\la} -h^{\a\la}   +O^{\a\la} (h^2) ) \\
   && .  <  \pa_{\a}A_{\b} - \pa_{\b}A_{\a} , \pa_{\la}A_{\si} - \pa_{\si}A_{\la} >   \\
  &=&   - \frac{1}{(n-1)} m_{\mu\nu } ( m^{\si\b}  m^{\a\la} -m^{\si\b}  h^{\a\la}   +O^{\a\la} (h^2) )   \\
   && .  <  \pa_{\a}A_{\b} - \pa_{\b}A_{\a} , \pa_{\la}A_{\si} - \pa_{\si}A_{\la} >   \\
  &&  + \frac{1}{(n-1)} m_{\mu\nu } ( h^{\si\b}  m^{\a\la} - h^{\si\b}  h^{\a\la}   +O^{\a\la} (h^3) )  \\
   && .  <  \pa_{\a}A_{\b} - \pa_{\b}A_{\a} , \pa_{\la}A_{\si} - \pa_{\si}A_{\la} >   \\
&& +  O\big( ( h^2 + h^3 + h^4 )  \cdot  (\pa A)^2 \big) \, . \\
\eeaa
 On one hand,
 \beaa
   &&   - \frac{1}{(n-1)} m_{\mu\nu } ( m^{\si\b}  m^{\a\la} -m^{\si\b}  h^{\a\la}   +O^{\a\la} (h^2) )  .  <  \pa_{\a}A_{\b} - \pa_{\b}A_{\a} , \pa_{\la}A_{\si} - \pa_{\si}A_{\la} >   \\
   &=&     - \frac{1}{(n-1)} m_{\mu\nu }  m^{\si\b}  m^{\a\la}     .  <  \pa_{\a}A_{\b} - \pa_{\b}A_{\a} , \pa_{\la}A_{\si} - \pa_{\si}A_{\la} >   \\
         &&   + \frac{1}{(n-1)} m_{\mu\nu }  m^{\si\b}  h^{\a\la}   .  <  \pa_{\a}A_{\b} - \pa_{\b}A_{\a} , \pa_{\la}A_{\si} - \pa_{\si}A_{\la} >   \\
         && + O \big( ( h^2  ) \cdot (\pa A)^2 \big) \, .
 \eeaa
 
 On the other hand,
 \beaa
   &&   \frac{1}{(n-1)} m_{\mu\nu } ( h^{\si\b}  m^{\a\la} - h^{\si\b}  h^{\a\la}   +O^{\a\la} (h^3) )   .  <  \pa_{\a}A_{\b} - \pa_{\b}A_{\a} , \pa_{\la}A_{\si} - \pa_{\si}A_{\la} >   \\
&=&    \frac{1}{(n-1)} m_{\mu\nu }  h^{\si\b}  m^{\a\la}   .  <  \pa_{\a}A_{\b} - \pa_{\b}A_{\a} , \pa_{\la}A_{\si} - \pa_{\si}A_{\la} >   \\
   &&  - \frac{1}{(n-1)} m_{\mu\nu } h^{\si\b}  h^{\a\la}    .  <  \pa_{\a}A_{\b} - \pa_{\b}A_{\a} , \pa_{\la}A_{\si} - \pa_{\si}A_{\la} >   \\
         && + O \big( ( h^3  )  \cdot  (\pa A)^2 \big) \, .
\eeaa
In conclusion,
\beaa
 && - \frac{1}{(n-1)} g_{\mu\nu } g^{\si\b} g^{\a\la}  <  \pa_{\a}A_{\b} - \pa_{\b}A_{\a} , \pa_{\la}A_{\si} - \pa_{\si}A_{\la} >   \\
   &=&     - \frac{1}{(n-1)} m_{\mu\nu }  m^{\si\b}  m^{\a\la}     .  <  \pa_{\a}A_{\b} - \pa_{\b}A_{\a} , \pa_{\la}A_{\si} - \pa_{\si}A_{\la} >   \\
         &&   + \frac{1}{(n-1)} m_{\mu\nu }  m^{\si\b}  h^{\a\la}   .  <  \pa_{\a}A_{\b} - \pa_{\b}A_{\a} , \pa_{\la}A_{\si} - \pa_{\si}A_{\la} >   \\
         && + O \big( ( h^2  )  \cdot  (\pa A)^2 \big). \\
         &&+    \frac{1}{(n-1)} m_{\mu\nu }  h^{\si\b}  m^{\a\la}   .  <  \pa_{\a}A_{\b} - \pa_{\b}A_{\a} , \pa_{\la}A_{\si} - \pa_{\si}A_{\la} >   \\
   &&  - \frac{1}{(n-1)} m_{\mu\nu } h^{\si\b}  h^{\a\la}    .  <  \pa_{\a}A_{\b} - \pa_{\b}A_{\a} , \pa_{\la}A_{\si} - \pa_{\si}A_{\la} >   \\
         && + O \big( ( h^3  )  \cdot  (\pa A)^2 \big) \\
            &=&     - \frac{1}{(n-1)} m_{\mu\nu }  m^{\si\b}  m^{\a\la}     .  <  \pa_{\a}A_{\b} - \pa_{\b}A_{\a} , \pa_{\la}A_{\si} - \pa_{\si}A_{\la} >   \\
             && + O \big( ( h+ h^2 + h^3  )  \cdot (\pa A)^2 \big) \, . \\
 \eeaa
 Finally,
 \beaa
 && - \frac{1}{(n-1)} g_{\mu\nu } g^{\si\b} g^{\a\la}  <  \pa_{\a}A_{\b} - \pa_{\b}A_{\a} , \pa_{\la}A_{\si} - \pa_{\si}A_{\la} >   \\
     &=&     - \frac{1}{(n-1)} m_{\mu\nu }  m^{\si\b}  m^{\a\la}     .  <  \pa_{\a}A_{\b} - \pa_{\b}A_{\a} , \pa_{\la}A_{\si} - \pa_{\si}A_{\la} >   \\
             && + O \big(  h   \cdot  (\pa A)^2 \big)  \, . \\
 \eeaa
 
\textbf{Third term}

\beaa
&&  2  g^{\si\b}   \big(  <   \pa_{\mu}A_{\b} - \pa_{\b}A_{\mu}  ,  [A_{\nu},A_{\si}] >  +  <   [A_{\mu},A_{\b}] ,  \pa_{\nu}A_{\si} - \pa_{\si }A_{\nu}  > \big)  \\
&=& 2 ( m^{\si\b}-h^{\si\b}  +O^{\si\b}(h^2) ) \big( <   \pa_{\mu}A_{\b} - \pa_{\b}A_{\mu}  ,  [A_{\nu},A_{\si}] >   + <   [A_{\mu},A_{\b}] ,  \pa_{\nu}A_{\si} - \pa_{\si }A_{\nu}  > \big)  \\
&=& 2  m^{\si\b}  \big( <   \pa_{\mu}A_{\b} - \pa_{\b}A_{\mu}  ,  [A_{\nu},A_{\si}] >   + <   [A_{\mu},A_{\b}] ,  \pa_{\nu}A_{\si} - \pa_{\nu}A_{\si}  > \big)  \\
&& - 2 h^{\si\b} \big( <   \pa_{\mu}A_{\b} - \pa_{\b}A_{\mu}  ,  [A_{\nu},A_{\si}] >   + <   [A_{\mu},A_{\b}] ,  \pa_{\nu}A_{\si} - \pa_{\nu}A_{\si}  > \big)  \\
  && + O \big(  h^2   \cdot  A^2  \cdot  \pa A \big) \\
&=& 2  m^{\si\b}  \big( <   \pa_{\mu}A_{\b} - \pa_{\b}A_{\mu}  ,  [A_{\nu},A_{\si}] >   + <   [A_{\mu},A_{\b}] ,  \pa_{\nu}A_{\si} - \pa_{\nu}A_{\si}  > \big)  \\
  && + O \big( (h + h^2 )  \cdot  A^2  \cdot  \pa A \big)  \, . \\
\eeaa
Thus,

\beaa
&&  2  g^{\si\b}   \big(  <   \pa_{\mu}A_{\b} - \pa_{\b}A_{\mu}  ,  [A_{\nu},A_{\si}] >  +  <   [A_{\mu},A_{\b}] ,  \pa_{\nu}A_{\si} - \pa_{\nu}A_{\si}  > \big)  \\
&=& 2  m^{\si\b}  \big( <   \pa_{\mu}A_{\b} - \pa_{\b}A_{\mu}  ,  [A_{\nu},A_{\si}] >   + <   [A_{\mu},A_{\b}] ,  \pa_{\nu}A_{\si} - \pa_{\nu}A_{\si}  > \big)  \\
  && + O \big(  h  \cdot  A^2  \cdot  \pa A \big) \, . \\
\eeaa

\textbf{Fourth term}

\beaa
&&  - \frac{1}{(n-1)} g_{\mu\nu } g^{\si\b} g^{\a\la}  \big(  <  \pa_{\a}A_{\b} - \pa_{\b}A_{\a} , [A_{\la},A_{\si}] >    +  <  [A_{\a},A_{\b}] , \pa_{\la}A_{\si} - \pa_{\si}A_{\la}  > \big)  \\
&=& - \frac{1}{(n-1)} ( m_{\mu\nu } + h_{\mu\nu}  ) (m^{\si\b} -h^{\si\b}   +O^{\si\b} (h^2) ) (m^{\a\la} -h^{\a\la}   +O^{\a\la} (h^2) )  \\
&& . \big(  <  \pa_{\a}A_{\b} - \pa_{\b}A_{\a} , [A_{\la},A_{\si}] >    +  <  [A_{\a},A_{\b}] , \pa_{\la}A_{\si} - \pa_{\si}A_{\la}  > \big) \\
&=& - \frac{1}{(n-1)}  m_{\mu\nu }  (m^{\si\b} -h^{\si\b}   +O^{\si\b} (h^2) ) (m^{\a\la} -h^{\a\la}   +O^{\a\la} (h^2) )  \\
&& . \big(  <  \pa_{\a}A_{\b} - \pa_{\b}A_{\a} , [A_{\la},A_{\si}] >    +  <  [A_{\a},A_{\b}] , \pa_{\la}A_{\si} - \pa_{\si}A_{\la}  > \big) \\
&& - \frac{1}{(n-1)}   h_{\mu\nu}  (m^{\si\b} -h^{\si\b}   +O^{\si\b} (h^2) ) (m^{\a\la} -h^{\a\la}   +O^{\a\la} (h^2) )  \\
&& . \big(  <  \pa_{\a}A_{\b} - \pa_{\b}A_{\a} , [A_{\la},A_{\si}] >    +  <  [A_{\a},A_{\b}] , \pa_{\la}A_{\si} - \pa_{\si}A_{\la}  > \big) \\
&=& J_1 + J_2 + O \big( (h +  h^2 + h^3 + h^4 )  \cdot  A^2  \cdot  \pa A \big) \, ,
\eeaa
where

\beaa
 && J_1 \\
 &=&   (  - \frac{1}{(n-1)} m_{\mu\nu }  m^{\si\b}  m^{\a\la}      + \frac{1}{(n-1)} m_{\mu\nu }  m^{\si\b}  h^{\a\la} ) . \big(  <  \pa_{\a}A_{\b} - \pa_{\b}A_{\a} , [A_{\la},A_{\si}] >  \\
 && +  <  [A_{\a},A_{\b}] , \pa_{\la}A_{\si} - \pa_{\si}A_{\la}  > \big) + O \big( ( h^2  )  \cdot  A^2  \cdot  \pa A \big) \\
     &=&    - \frac{1}{(n-1)} m_{\mu\nu }  m^{\si\b}  m^{\a\la}    . \big(  <  \pa_{\a}A_{\b} - \pa_{\b}A_{\a} , [A_{\la},A_{\si}] >    +  <  [A_{\a},A_{\b}] , \pa_{\la}A_{\si} - \pa_{\si}A_{\la}  > \big) \\
       && + O \big( ( h + h^2  )  \cdot  A^2  \cdot  \pa A \big)    \, ,
 \eeaa
 
  \beaa
  && J_2 \\
   &=&   \frac{1}{(n-1)} m_{\mu\nu } ( h^{\si\b}  m^{\a\la} - h^{\si\b}  h^{\a\la}   +O^{\a\la} (h^3) )   . \big(  <  \pa_{\a}A_{\b} - \pa_{\b}A_{\a} , [A_{\la},A_{\si}] >   \\
   && +  <  [A_{\a},A_{\b}] , \pa_{\la}A_{\si} - \pa_{\si}A_{\la}  > \big)  \\
&=& (   \frac{1}{(n-1)} m_{\mu\nu }  h^{\si\b}  m^{\a\la}     - \frac{1}{(n-1)} m_{\mu\nu } h^{\si\b}  h^{\a\la}   ) .   \big(  <  \pa_{\a}A_{\b} - \pa_{\b}A_{\a} , [A_{\la},A_{\si}] >   \\
&& +  <  [A_{\a},A_{\b}] , \pa_{\la}A_{\si} - \pa_{\si}A_{\la}  > \big)   + O \big( ( h^3  )   \cdot  A^2  \cdot  \pa A \big)  \\
         &=&     O \big( ( h +h^2+ h^3  )  \cdot A^2  \cdot \pa A \big) \, .
\eeaa

Hence,

\beaa
&&  - \frac{1}{(n-1)} g_{\mu\nu } g^{\si\b} g^{\a\la}  \big(  <  \pa_{\a}A_{\b} - \pa_{\b}A_{\a} , [A_{\la},A_{\si}] >    +  <  [A_{\a},A_{\b}] , \pa_{\la}A_{\si} - \pa_{\si}A_{\la}  > \big)  \\
&=&    - \frac{1}{(n-1)} m_{\mu\nu }  m^{\si\b}  m^{\a\la}    . \big(  <  \pa_{\a}A_{\b} - \pa_{\b}A_{\a} , [A_{\la},A_{\si}] >    +  <  [A_{\a},A_{\b}] , \pa_{\la}A_{\si} - \pa_{\si}A_{\la}  > \big) \\
       && + O \big(  h  \cdot  A^2  \cdot  \pa A \big)    \, .
\eeaa

\textbf{Fifth term}

\beaa
&&  2  g^{\si\b} <   [A_{\mu},A_{\b}] ,  [A_{\nu},A_{\si}] >  \\
&=& 2 (m^{\si\b} -h^{\si\b}   +O^{\si\b} (h^2) ) . <   [A_{\mu},A_{\b}] ,  [A_{\nu},A_{\si}] >  \\
       &=& 2 m^{\si\b}  . <   [A_{\mu},A_{\b}] ,  [A_{\nu},A_{\si}] >  - 2 h^{\si\b}  . <   [A_{\mu},A_{\b}] ,  [A_{\nu},A_{\si}] >  \\
       && + O \big(  h^2  \cdot  A^4 \big)    \\
              &=& 2 m^{\si\b}  . <   [A_{\mu},A_{\b}] ,  [A_{\nu},A_{\si}] >   \\
       && + O \big(  h  \cdot  A^4 \big)    \, .
\eeaa

\textbf{Sixth term}

\beaa
&& - \frac{1}{(n-1)} g_{\mu\nu } g^{\si\b} g^{\a\la}  <  [A_{\a},A_{\b}] , [A_{\la},A_{\si}] >  \\
&=& - \frac{1}{(n-1)} ( m_{\mu\nu } + h_{\mu\nu}  ) (m^{\si\b} -h^{\si\b}   +O^{\si\b} (h^2) ) (m^{\a\la} -h^{\a\la}   +O^{\a\la} (h^2) )  .  <  [A_{\a},A_{\b}] , [A_{\la},A_{\si}] >  \\
&=& - \frac{1}{(n-1)}  m_{\mu\nu } (m^{\si\b} -h^{\si\b}   +O^{\si\b} (h^2) ) (m^{\a\la} -h^{\a\la}   +O^{\a\la} (h^2) )  .  <  [A_{\a},A_{\b}] , [A_{\la},A_{\si}] >  \\
&& - \frac{1}{(n-1)}  h_{\mu\nu}   (m^{\si\b} -h^{\si\b}   +O^{\si\b} (h^2) ) (m^{\a\la} -h^{\a\la}   +O^{\a\la} (h^2) )  .  <  [A_{\a},A_{\b}] , [A_{\la},A_{\si}] >  \\
&=&   \big(  - \frac{1}{(n-1)} m_{\mu\nu }  m^{\si\b}  m^{\a\la}      + \frac{1}{(n-1)} m_{\mu\nu }  m^{\si\b}  h^{\a\la}  +   \frac{1}{(n-1)} m_{\mu\nu }  h^{\si\b}  m^{\a\la} \\
&&    - \frac{1}{(n-1)} m_{\mu\nu } h^{\si\b}  h^{\a\la}   \big)  .  <  [A_{\a},A_{\b}] , [A_{\la},A_{\si}] >   + O \big( ( h + h^2 + h^3 + h^4 ) \cdot A^4 \big)    \\
    &=&     - \frac{1}{(n-1)} m_{\mu\nu }  m^{\si\b}  m^{\a\la}   .  <  [A_{\a},A_{\b}] , [A_{\la},A_{\si}] >  \\
    && + O \big( ( h +  h^2 + h^3 + h^4 ) \cdot A^4 \big)    \\
        &=&     - \frac{1}{(n-1)} m_{\mu\nu }  m^{\si\b}  m^{\a\la}   .  <  [A_{\a},A_{\b}] , [A_{\la},A_{\si}] >  \\
    && + O \big(  h   \cdot  A^4 \big)    \, .
 \eeaa
 
\textbf{Final result}

   \beaa
&& R_{ \mu \nu}  \\
 &=&   2  m^{\si\b} <   \pa_{\mu}A_{\b} - \pa_{\b}A_{\mu}  ,  \pa_{\nu}A_{\si} - \pa_{\si}A_{\nu} >       \\
 && + O (h  \cdot  (\pa A)^2 ). \\
    &&     - \frac{1}{(n-1)} m_{\mu\nu }  m^{\si\b}  m^{\a\la}     .  <  \pa_{\a}A_{\b} - \pa_{\b}A_{\a} , \pa_{\la}A_{\si} - \pa_{\si}A_{\la} >   \\
             && + O \big(  h  \cdot (\pa A)^2 \big) \\
&&  +           2  m^{\si\b}  \big( <   \pa_{\mu}A_{\b} - \pa_{\b}A_{\mu}  ,  [A_{\nu},A_{\si}] >   + <   [A_{\mu},A_{\b}] ,  \pa_{\nu}A_{\si} -\pa_{\si}A_{\nu} > \big)  \\
  && + O \big(  h  \cdot  A^2  \cdot \pa A \big) \\
&&    - \frac{1}{(n-1)} m_{\mu\nu }  m^{\si\b}  m^{\a\la}    . \big(  <  \pa_{\a}A_{\b} - \pa_{\b}A_{\a} , [A_{\la},A_{\si}] >    +  <  [A_{\a},A_{\b}] , \pa_{\la}A_{\si} - \pa_{\si}A_{\la}  > \big) \\
       && + O \big(  h  \cdot A^2  \cdot  \pa A \big)    \\
 && +      2 m^{\si\b}  . <   [A_{\mu},A_{\b}] ,  [A_{\nu},A_{\si}] >   \\
       && + O \big(  h  \cdot  A^4 \big)   \\
    &&     - \frac{1}{(n-1)} m_{\mu\nu }  m^{\si\b}  m^{\a\la}   .  <  [A_{\a},A_{\b}] , [A_{\la},A_{\si}] >  \\
    && + O \big(  h   \cdot  A^4 \big)  \, .   \\
\eeaa
Thus, we get the desired result.
\end{proof}

\section{The Einstein-Yang-Mills system in the harmonic and Lorenz gauges as a non-linear hyperbolic system}

\begin{lemma}\label{waveequationontheYangMillspotentialwithspurcesdependingonthemetricg}

The equation $ \textbf{D}^{(A)}_{\a}F^{\a\b}  = 0 $, implies that in the Lorenz gauge, and in wave coordinates $\mu\;, \nu\;,\a\;,\b\;,\si \in \{0, 1, ..., n\} $\,,

     \beaa
        \notag
g^{\mu\nu}  \pa_{\mu}   \pa_{\nu}    A_{\si}    &=&  -  ( \pa_{\si}  g^{\a\mu} ) \cdot   \pa_{\a}A_{\mu}         +    \frac{1}{2}   g^{\a\mu} g^{\b\nu}  \cdot \big(   \pa_\a g_{\b\si} + \pa_\si g_{\b\a}- \pa_\b g_{\a\si}  \big) . \big( \pa_{\mu}A_{\nu} - \pa_{\nu}A_{\mu} \big) \\
 \notag
 && -   g^{\a\mu} \cdot [ A_{\mu}, \pa_{\a} A_{\si} ]  - g^{\a\mu}  \cdot [A_{\alpha},  \pa_{\mu}  A_{\si} - \pa_{\si} A_{\mu} ]   \\
 \notag
 && +  \frac{1}{2}   g^{\a\mu} g^{\b\nu}  \cdot  \big(   \pa_\a g_{\b\si} + \pa_\si g_{\b\a}- \pa_\b g_{\a\si}  \big) \cdot  [A_{\mu},A_{\nu}]   - g^{\a\mu} \cdot  [A_{\alpha}, [A_{\mu},A_{\si}] ] \, . \\
   \eeaa

\end{lemma}

\begin{proof}

We know from \eqref{ThegaugecovariantdivergenceoftheYangMillscurvatureisequaltozero} that the Yang-Mills fields satisfy
\beaa
\textbf{D}^{(A)}_{\a}F^{\a\b}  &=& 0 \\
&=& \der_{\alpha} F^{\a\b}  + [A_{\alpha}, F^{\a\b} ]  \\
&=&  \der_{\alpha}  ( g^{\a\mu} g^{\b\nu} F_{\mu\nu} ) + g^{\a\mu} g^{\b\nu}  [A_{\alpha}, F_{\mu\nu} ]  \, . \\
\eeaa
Since $\der g = 0$, we get
\beaa
 \textbf{D}^{(A)}_{\a}F^{\a\b}   = g^{\a\mu} g^{\b\nu}   \der_{\alpha} F_{\mu\nu}  + g^{\a\mu} g^{\b\nu}  [A_{\alpha}, F_{\mu\nu} ]  = 0 \, .
\eeaa
However,
\beaa
 g^{\a\mu} g^{\b\nu}    \der_{\alpha} F_{\mu\nu}    &=&     g^{\a\mu} g^{\b\nu} \big(    \pa_{\alpha}  F_{\mu\nu}  - F(  \der_{\alpha} e_\mu, e_\nu  ) - F( e_\mu , \der_{\alpha} e_\nu  ) \big) \\
 &=&     g^{\a\mu} g^{\b\nu}     \pa_{\alpha}  F_{\mu\nu}  - g^{\b\nu}   F\big(  \der_{\alpha} ( g^{\a\mu}  e_\mu ) , e_\nu \big) -   g^{\a\mu} g^{\b\nu}   F( e_\mu , \der_{\alpha} e_\nu  ) \\
  &=&     g^{\a\mu} g^{\b\nu}     \pa_{\alpha}  F_{\mu\nu}  - g^{\b\nu}   F (  \der_{\alpha}  e^\a ) , e_\nu ) -   g^{\a\mu} g^{\b\nu}   F( e_\mu , \der_{\alpha} e_\nu  ) \\
    &=&     g^{\a\mu} g^{\b\nu}     \pa_{\alpha}  F_{\mu\nu}  - g^{\b\nu}  \Ga_{\, \, \, \a}^{\a  \, \, \, \la} F_{\la\nu}  -   g^{\a\mu} g^{\b\nu}   F( e_\mu , \der_{\alpha} e_\nu  ) \, .
\eeaa

Since it is a trace, it does not depend on the system of coordinates used to compute it, in particular one could compute the trace over $\mu, \nu, \a$ indices using wave coordinates. In wave coordinates, we get

\bea
 g^{\a\mu} g^{\b\nu}    \der_{\alpha} F_{\mu\nu}       &=&     g^{\a\mu} g^{\b\nu}     \pa_{\alpha}  F_{\mu\nu}    -   g^{\a\mu} g^{\b\nu}  \Ga_{\a \nu}^{\, \, \,  \, \,  \, \, \, \la}  F_{\mu\la}   \, .
\eea

Now, since 
\beaa
\notag
F_{\mu\nu} &=&  \der_{\mu}A_{\nu} - \der_{\nu}A_{\mu} + [A_{\mu},A_{\nu}]  \\
&=&  \pa_{\mu}A_{\nu} - \pa_{\nu}A_{\mu} + [A_{\mu},A_{\nu}]  \, ,
\eeaa
we have

\beaa
 \pa_{\alpha} F_{\mu\nu} &=& \pa_{\alpha}  \pa_{\mu}A_{\nu} -    \pa_{\alpha} \pa_{\nu}A_{\mu} + [\pa_{\alpha}  A_{\mu},A_{\nu}] + [A_{\mu},\pa_{\alpha}  A_{\nu}]  \, . \\
\eeaa 
Consequently,

\beaa
\textbf{D}^{(A)}_{\a}  F^{\a\b}  &=&     g^{\a\mu}   g^{\b\nu}  \pa_{\alpha}  F_{\mu\nu}   -   g^{\a\mu} g^{\b\nu}  \Ga_{\a \nu}^{\, \, \,  \, \,  \, \, \, \la}  F_{\mu\la}  + g^{\a\mu} g^{\b\nu}  [A_{\alpha}, F_{\mu\nu} ]  \\
&=&   g^{\a\mu}   g^{\b\nu}  ( \pa_{\alpha}  \pa_{\mu}A_{\nu} -    \pa_{\alpha} \pa_{\nu}A_{\mu} + [\pa_{\alpha}  A_{\mu},A_{\nu}] + [A_{\mu},\pa_{\alpha}  A_{\nu}]  ) \\
&& + g^{\a\mu} g^{\b\nu}  [A_{\alpha},  \pa_{\mu}A_{\nu} - \pa_{\nu}A_{\mu} + [A_{\mu},A_{\nu}] ] \\
&& -   g^{\a\mu} g^{\b\nu}  \Ga_{\a \nu}^{\, \, \,  \, \,  \, \, \, \la}  F_{\mu\la} \, .
  \eeaa

On one hand,
\beaa
&& g^{\a\mu} g^{\b\nu}   \big(  \pa_{\alpha}   \pa_{\mu}A_{\nu}   - \pa_{\alpha}   \pa_{\nu}A_{\mu}   + [\pa_{\a} A_{\mu}  ,A_{\nu}] + [ A_{\mu}, \pa_{\a} A_{\nu} ]    \big)\\
  &=&  g^{\b\nu}   \big(  \pa^{\mu}   \pa_{\mu}A_{\nu}    - g^{\a\mu}   \pa_{\a}   \pa_{\nu}A_{\mu}  +  [\pa^{\mu} A_{\mu}  ,A_{\nu}]    + g^{\a\mu} [ A_{\mu}, \pa_{\a} A_{\nu} ]      \big) \, .
 \eeaa

Now, in wave coordinates, the derivates commute as it is a system of coordinates, and therefore

\beaa
 - g^{\a\mu}     \pa_{\a}   \pa_{\nu}A_{\mu}   &=&  - g^{\a\mu} \pa_{\nu}    \pa_{\a}  A_{\mu} \\
     &=& -  \pa_{\nu} ( g^{\a\mu}     \pa_{\a}A_{\mu} )  + ( \pa_{\nu}  g^{\a\mu} )    \pa_{\a}A_{\mu} \\
 &=& -  \pa_{\nu} (  \pa^{\mu}A_{\mu} )  + ( \pa_{\nu}  g^{\a\mu} )    \pa_{\a}A_{\mu} \, .
\eeaa

Thus,
\beaa
&& g^{\a\mu} g^{\b\nu}   \big(  \pa_{\alpha}   \pa_{\mu}A_{\nu}   - \pa_{\alpha}   \pa_{\nu}A_{\mu}   + [\pa_{\a} A_{\mu}  ,A_{\nu}] + [ A_{\mu}, \pa_{\a} A_{\nu} ]    \big)\\
       &=&  g^{\b\nu}   \big(  \pa^{\mu}   \pa_{\mu}A_{\nu}   +  ( \pa_{\nu}  g^{\a\mu} )    \pa_{\a}A_{\mu}      + g^{\a\mu} [ A_{\mu}, \pa_{\a} A_{\nu} ]     +  [\pa^{\mu} A_{\mu}  ,A_{\nu}]  -  \pa_{\nu} (  \pa^{\mu}A_{\mu} )   \big) \, . \\
 \eeaa

Hence,
\beaa
\textbf{D}^{(A)}_{\a}  F^{\a\b}  &=&  g^{\b\nu}   \big(  \pa^{\mu}   \pa_{\mu}A_{\nu}   +  ( \pa_{\nu}  g^{\a\mu} )    \pa_{\a}A_{\mu}      + g^{\a\mu} [ A_{\mu}, \pa_{\a} A_{\nu} ]       +  [\pa^{\mu} A_{\mu}  ,A_{\nu}]  -  \pa_{\nu} (  \pa^{\mu}A_{\mu} )  \big)\\
&& + g^{\a\mu} g^{\b\nu}  [A_{\alpha},  \pa_{\mu}A_{\nu} - \pa_{\nu}A_{\mu} + [A_{\mu},A_{\nu}] ] \\
 &=& g^{\b\nu}    \pa^{\mu}   \pa_{\mu}A_{\nu}   + g^{\b\nu} ( \pa_{\nu}  g^{\a\mu} )    \pa_{\a}A_{\mu}      + g^{\b\nu}  g^{\a\mu} [ A_{\mu}, \pa_{\a} A_{\nu} ]    \\\
&& + g^{\a\mu} g^{\b\nu}  [A_{\alpha},  \pa_{\mu}A_{\nu} - \pa_{\nu}A_{\mu} ]  + g^{\a\mu} g^{\b\nu}  [A_{\alpha}, [A_{\mu},A_{\nu}] ] \\
 && + g^{\b\nu} [\pa^{\mu} A_{\mu}  ,A_{\nu}] -  g^{\b\nu}  \pa_{\nu} (  \pa^{\mu}A_{\mu} ) -   g^{\a\mu} g^{\b\nu}  \Ga_{\a \nu}^{\, \, \,  \, \,  \, \, \, \la}  F_{\mu\la} \, .
   \eeaa

Computing now  $ [\pa^{\mu} A_{\mu}  ,A_{\nu}]  $ and $ -  \pa_{\nu} (  \pa^{\mu}A_{\mu} ) $. Since the Lorenz gauge does not depend on the system of coordinates, but is a geometric condition on the Yang-Mills potential $A$, we can compute it in wave coordinates, and therefore
\beaa
 [\pa^{\mu} A_{\mu}  ,A_{\nu}] &=& 0 ,\\
 -  \pa_{\nu} (  \pa^{\mu}A_{\mu} )  &=& 0 .
\eeaa
Finally, in wave coordinates and in the Lorenz gauge,
\beaa
\textbf{D}^{(A)}_{\a}  F^{\a\b}   &=& g^{\b\nu}    \pa^{\mu}   \pa_{\mu}A_{\nu}   + g^{\b\nu} (  \pa_{\nu}  g^{\a\mu} )    \pa_{\a}A_{\mu}      + g^{\b\nu}  g^{\a\mu} [ A_{\mu}, \pa_{\a} A_{\nu} ]    \\
&& + g^{\a\mu} g^{\b\nu}  [A_{\alpha},  \pa_{\mu}A_{\nu} - \pa_{\nu}A_{\mu} ]  + g^{\a\mu} g^{\b\nu}  [A_{\alpha}, [A_{\mu},A_{\nu}] ] \\
&& -   g^{\a\mu} g^{\b\nu}  \Ga_{\a \nu}^{\, \, \,  \, \,  \, \, \, \la}  F_{\mu\la} \\
&=& 0 \;.
  \eeaa
  
Multiplying the equation above by $g_{\si\b}$, and using the fact that $g_{\si\b} g^{\nu\b} =  I_{\si}^{\ \; \nu}$    (where  $I_{\si}^{\ \; \nu}$ is the identity matrix), we get
\beaa
 0 &=& I_{\si}^{\ \; \nu}   \pa^{\mu}   \pa_{\mu}A_{\nu}    +  I_{\si}^{\ \; \nu} ( \pa_{\nu}  g^{\a\mu} )    \pa_{\a}A_{\mu}    +I_{\si}^{\ \; \nu}  g^{\a\mu} [ A_{\mu}, \pa_{\a} A_{\nu} ]    \\
&& + g^{\a\mu} I_{\si}^{\ \; \nu}   [A_{\alpha},  \pa_{\mu}A_{\nu} - \pa_{\nu}A_{\mu} ]  + g^{\a\mu} I_{\si}^{\ \; \nu}   [A_{\alpha}, [A_{\mu},A_{\nu}] ] \\
&& -   g^{\a\mu}  I_{\si}^{\ \; \nu}  \Ga_{\a \nu}^{\, \, \,  \, \,  \, \, \, \la}  F_{\mu\la} \\
&=&   \pa^{\mu}   \pa_{\mu} ( I_{\si}^{\ \; \nu}  A_{\nu} )  +   ( \pa_{\si}  g^{\a\mu} )    \pa_{\a}A_{\mu}      +   g^{\a\mu} [ A_{\mu}, \pa_{\a} A_{\si} ]    \\
&& + g^{\a\mu}   [A_{\alpha},  \pa_{\mu} ( I_{\si}^{\ \; \nu} A_{\nu} )- \pa_{\si} A_{\mu} ]  + g^{\a\mu}   [A_{\alpha}, [A_{\mu},A_{\si}] ] \\
&& -   g^{\a\mu}  \Ga_{\a \si}^{\, \, \,  \, \,  \, \, \, \la}  F_{\mu\la} \, .
  \eeaa
    
We obtain
    \beaa
 \pa^{\mu}   \pa_{\mu}   A_{\si}    &=&  -   ( \pa_{\si}  g^{\a\mu} )    \pa_{\a}A_{\mu}      -   g^{\a\mu} [ A_{\mu}, \pa_{\a} A_{\si} ]    \\
&& - g^{\a\mu}   [A_{\alpha},  \pa_{\mu}  A_{\si} - \pa_{\si} A_{\mu} ]  - g^{\a\mu}   [A_{\alpha}, [A_{\mu},A_{\si}] ] \\
&& +  g^{\a\mu}  \Ga_{\a \si}^{\, \, \,  \, \,  \, \, \, \nu}  \big( \pa_{\mu}A_{\nu} - \pa_{\nu}A_{\mu} + [A_{\mu},A_{\nu}]   \big) \, .
 \eeaa

 However, the Christoffel symbols are
 \beaa
\Gamma^{\,\,\, \, \, \, \, \, \nu}_{\a\si} & =& \frac{1}{2} g^{\nu\b} \big( \pa_\a g_{\b\si} + \pa_\si g_{\b\a}- \pa_\b g_{\a\si} \big) \, .
\eeaa

 At the end, we obtain
   \bea
   \notag
 \pa^{\mu}   \pa_{\mu}   A_{\si}    &=&  -  ( \pa_{\si}  g^{\a\mu} )    \pa_{\a}A_{\mu}         +    \frac{1}{2}   g^{\a\mu} g^{\b\nu}  \big(   \pa_\a g_{\b\si} + \pa_\si g_{\b\a}- \pa_\b g_{\a\si}  \big) . \big( \pa_{\mu}A_{\nu} - \pa_{\nu}A_{\mu} \big) \\
 \notag
 && -   g^{\a\mu} [ A_{\mu}, \pa_{\a} A_{\si} ]  - g^{\a\mu}   [A_{\alpha},  \pa_{\mu}  A_{\si} - \pa_{\si} A_{\mu} ]   \\
 \notag
 && +  \frac{1}{2}   g^{\a\mu} g^{\b\nu}  \big(   \pa_\a g_{\b\si} + \pa_\si g_{\b\a}- \pa_\b g_{\a\si}  \big)  [A_{\mu},A_{\nu}]   - g^{\a\mu}   [A_{\alpha}, [A_{\mu},A_{\si}] ] \, .
   \eea
We have,
\beaa
 && -   g^{\a\mu} [ A_{\mu}, \pa_{\a} A_{\si} ]  - g^{\a\mu}   [A_{\alpha},  \pa_{\mu}  A_{\si} - \pa_{\si} A_{\mu} ] \\
   &=&  -   g^{\a\mu} [ A_{\mu}, \pa_{\a} A_{\si} ]  - g^{\a\mu}   [A_{\alpha},  \pa_{\mu}  A_{\si}  ]  +   g^{\a\mu} [ A_{\a}, \pa_{\a} A_{\si} ]  \\
 &=&  - 2  g^{\a\mu} [ A_{\mu}, \pa_{\a} A_{\si} ]  +   g^{\a\mu} [ A_{\a}, \pa_{\a} A_{\si} ] \, .
\eeaa
Thus, we obtain the stated result.

   \end{proof}
   
   \begin{lemma}\label{waveequationontheEinsteinYangMillspotentialderivedfromthegaugecovariantdivergenceoftheYangMillscurvatureisequaltozero}
The equation $ \textbf{D}^{(A)}_{\a}F^{\a\b}  = 0 $, implies in the Lorenz gauge, and in wave coordinates $\mu\;, \nu\;,\la\;,\a\;, \b\;, \ga\;,\si \in \{0, 1, ..., n\} $\,,
      \bea
   \notag
g^{\la\mu}  \pa_{\la}   \pa_{\mu}   A_{\si}      &=&  m^{\a\ga} m ^{\mu\la}  (  \pa_{\si}  h_{\ga\la} ) \cdot  \pa_{\a}A_{\mu}       +   \frac{1}{2}  m^{\a\mu}m^{\b\nu} \cdot   \big(   \pa_\a h_{\b\si} + \pa_\si h_{\b\a}- \pa_\b h_{\a\si}  \big)   \cdot  \big( \pa_{\mu}A_{\nu} - \pa_{\nu}A_{\mu}  \big) \\
 \notag
&& +      \frac{1}{2}  m^{\a\mu}m^{\b\nu} \cdot    \big(   \pa_\a h_{\b\si} + \pa_\si h_{\b\a}- \pa_\b h_{\a\si}  \big)   \cdot   [A_{\mu},A_{\nu}] \\
 \notag
 && -  m^{\a\mu} \cdot  \big(  [ A_{\mu}, \pa_{\a} A_{\si} ]  +    [A_{\alpha},  \pa_{\mu}  A_{\si} - \pa_{\si} A_{\mu} ]    +    [A_{\alpha}, [A_{\mu},A_{\si}] ]  \big)  \\
 \notag
  && + O( h \cdot \pa h \cdot  \pa A) + O( h \cdot \pa h \cdot  A^2) + O( h \cdot  A \cdot  \pa A) + O( h \cdot  A^3) \; .
  \eea

\end{lemma}

\begin{proof}
Using from Lemma \ref{BigHintermsofsmallh}, the fact that
\beaa
g^{\mu\nu} = m^{\mu\nu}-h^{\mu\nu}  +O^{\mu\nu}(h^2)\, ,
\eeaa
 we have by differentiation, that
 \beaa
   \notag
  \pa_{\si}  g^{\a\mu}  &=&  \pa_{\si}  m^{\a\mu} -  \pa_{\si}  h^{\a\mu} + \pa_{\si}   \big( O^{\a\nu}(h^2) \big) \\
    \notag
  &=& -  \pa_{\si}  h^{\a\mu} +  O_{\si}^{\, \, \, \a\nu}(h \cdot  \pa h) \, ,
    \eeaa
  and
   \bea
    \notag
  \pa_{\si}  g^{\a\mu}  &=& -  \pa_{\si}  h^{\a\mu} + \pa_{\si}  \big( O^{\a\nu}(h^2) \big) \\
      \notag
  &=& -  \pa_{\si}  h^{\a\mu} +  O_{\si}^{\, \, \, \a\nu} (h \cdot  \pa h) \\
      \notag
   &=& -  \pa_{\si}  \big( m^{\a\ga} m ^{\mu\la}  h_{\ga\la}  \big) +  O_{\si}^{\, \, \, \a\nu} (h \cdot  \pa h) \\
     \notag
                 &=& - m^{\a\ga} m ^{\mu\la}   \pa_{\si} ( h_{\ga\la} ) +  O_{\si}^{\, \, \, \a\nu} (h \cdot  \pa h) \, .
  \eea
  Also,
  \beaa
g_{\mu\nu} =  m_{\mu\nu} + h_{\mu\nu} \; ,
\eeaa
yields to
 \bea
 \notag
 \pa_{\alpha}   g_{\b\si}  &=&  \pa_{\alpha} m_{\b\si} +  \pa_{\alpha} h_{\b\si} \\
 &=&  \pa_{\alpha} h_{\b\si} \, .
 \eea
  
 We get then, from Lemma \ref{waveequationontheYangMillspotentialwithspurcesdependingonthemetricg}, that
    \beaa
   \notag
 \pa^{\mu}   \pa_{\mu}   A_{\si}     &=&  -   ( \pa_{\si}  g^{\a\mu} )    \pa_{\a}A_{\mu}     \\
  &&     +   \frac{1}{2}   g^{\a\mu} g^{\b\nu}  \big(   \pa_\a g_{\b\si} + \pa_\si g_{\b\a}- \pa_\b g_{\a\si}  \big) \cdot \big( \pa_{\mu}A_{\nu} - \pa_{\nu}A_{\mu} + [A_{\mu},A_{\nu}] \big) \\
 \notag
 && -   g^{\a\mu} \big(  [ A_{\mu}, \pa_{\a} A_{\si} ]  +   [A_{\alpha},  \pa_{\mu}  A_{\si} - \pa_{\si} A_{\mu} ]    +  [A_{\alpha}, [A_{\mu},A_{\si}] ]  \big)  \\
 &=&  m^{\a\ga} m ^{\mu\la}   ( \pa_{\si}  h_{\ga\la} )   \pa_{\a}A_{\mu}  + O_{\si}^{\, \, \, \a\nu} (h \cdot  \pa h)  \cdot \pa_{\a}A_{\mu} \\
 &&     +    \frac{1}{2}  \big(  m^{\a\mu}-h^{\a\mu}  +O^{\a\mu}(h^2) \big)  \big(  m^{\b\nu} -h^{\b\nu}  +O^{\b\nu}  (h^2) \big) \cdot \big(   \pa_\a h_{\b\si} + \pa_\si h_{\b\a}- \pa_\b h_{\a\si}  \big)   \cdot  \\
 && \big( \pa_{\mu}A_{\nu} - \pa_{\nu}A_{\mu} + [A_{\mu},A_{\nu}] \big) \\
 \notag
 && -  \big( m^{\a\mu} -h^{\a\mu}  +O^{\a\mu} (h^2) \big) \cdot \big(  [ A_{\mu}, \pa_{\a} A_{\si} ]  +    [A_{\alpha},  \pa_{\mu}  A_{\si} - \pa_{\si} A_{\mu} ]    +    [A_{\alpha}, [A_{\mu},A_{\si}] ]  \big) \, . \\
  \eeaa
  
  We have
  \beaa
   &&   \big( m^{\a\mu}-h^{\a\mu}  +O^{\a\mu}(h^2) \big)  \cdot \big( m^{\b\nu} -h^{\b\nu}  +O^{\b\nu}  (h^2) \big) \\
   &=& m^{\a\mu}m^{\b\nu} -m^{\a\mu} h^{\b\nu}  +O^{\a\mu\b\nu}  (h^2) \\
   &&  -h^{\a\mu}  m^{\b\nu}  + h^{\a\mu}h^{\b\nu}  +O^{\a\mu\b\nu}  (h^3)  \\
     && +O^{\a\mu\b\nu}(h^2 + h^3 + h^4) \\
      &=& m^{\a\mu}m^{\b\nu}  +O^{\a\mu\b\nu}(h ) \, . \\
  \eeaa
  
  We get
    \beaa
   \notag
 \pa^{\mu}   \pa_{\mu}   A_{\si}      &=&  m^{\a\ga} m ^{\mu\la}  ( \pa_{\si}  h_{\ga\la} )   \pa_{\a}A_{\mu}  + O_{\si}^{\, \, \, \a\nu} (h \cdot  \pa h)  \cdot \pa_{\a}A_{\mu}   \\
 &&     +  \frac{1}{2}  \big( m^{\a\mu}m^{\b\nu}  +O^{\a\mu\b\nu}(h ) \big)  \cdot \big(   \pa_\a h_{\b\si} + \pa_\si h_{\b\a}- \pa_\b h_{\a\si}  \big)   \cdot \big( \pa_{\mu}A_{\nu} - \pa_{\nu}A_{\mu} + [A_{\mu},A_{\nu}] \big) \\
 \notag
 && -  m^{\a\mu}\cdot  \big(  [ A_{\mu}, \pa_{\a} A_{\si} ]  +    [A_{\alpha},  \pa_{\mu}  A_{\si} - \pa_{\si} A_{\mu} ]    +    [A_{\alpha}, [A_{\mu},A_{\si}] ]  \big)  \\
 && + \big( h^{\a\mu}  +O^{\a\mu} (h^2) \big) \cdot \big(  [ A_{\mu}, \pa_{\a} A_{\si} ]  +    [A_{\alpha},  \pa_{\mu}  A_{\si} - \pa_{\si} A_{\mu} ]    +    [A_{\alpha}, [A_{\mu},A_{\si}] ]  \big) \, .
  \eeaa
  
  Therefore,
      \beaa
   \notag
 \pa^{\mu}   \pa_{\mu}   A_{\si}      &=&  m^{\a\ga} m ^{\mu\la}  (  \pa_{\si}  h_{\ga\la} )   \pa_{\a}A_{\mu}    \\
 &&     +  \frac{1}{2}   m^{\a\mu}m^{\b\nu}  \cdot  \big(   \pa_\a h_{\b\si} + \pa_\si h_{\b\a}- \pa_\b h_{\a\si}  \big)   \cdot   \big( \pa_{\mu}A_{\nu} - \pa_{\nu}A_{\mu} + [A_{\mu},A_{\nu}] \big) \\
 \notag
 && -  m^{\a\mu} \cdot  \big(  [ A_{\mu}, \pa_{\a} A_{\si} ]  +    [A_{\alpha},  \pa_{\mu}  A_{\si} - \pa_{\si} A_{\mu} ]    +    [A_{\alpha}, [A_{\mu},A_{\si}] ]  \big)  \\
  && + O( h \cdot \pa h \cdot  \pa A) + O( h \cdot \pa h\cdot  A^2) + O( h \cdot  A \cdot  \pa A) + O( h\cdot   A^3) \, .
  \eeaa
  
Hence, we get the result.
\end{proof}

We obtained in Lemma \ref{EinsteinYangMillssystemusingpotentialandusingMinkowskimetricmandperturbationh}, that the Einstein-Yang-Mills equations read
\beaa
\notag
&& R_{ \mu \nu}   \\
\notag
 &=&   2  m^{\si\b} \cdot  <   \pa_{\mu}A_{\b} - \pa_{\b}A_{\mu}  ,  \pa_{\nu}A_{\si} -\pa_{\si}A_{\nu}  >           - \frac{1}{(n-1)} m_{\mu\nu }  m^{\si\b}  m^{\a\la}    \cdot   <  \pa_{\a}A_{\b} - \pa_{\b}A_{\a} , \pa_{\la}A_{\si} - \pa_{\si}A_{\la} >   \\
 \notag
&&  +           2  m^{\si\b}  \cdot  \big( <   \pa_{\mu}A_{\b} - \pa_{\b}A_{\mu}  ,  [A_{\nu},A_{\si}] >   + <   [A_{\mu},A_{\b}] ,  \pa_{\nu}A_{\si} - \pa_{\si}A_{\nu} > \big)  \\
\notag
&&    - \frac{1}{(n-1)} m_{\mu\nu }  m^{\si\b}  m^{\a\la}   \cdot  \big(  <  \pa_{\a}A_{\b} - \pa_{\b}A_{\a} , [A_{\la},A_{\si}] >    +  <  [A_{\a},A_{\b}] , \pa_{\la}A_{\si} - \pa_{\si}A_{\la}  > \big) \\
\notag
 && +      2 m^{\si\b}  \cdot  <   [A_{\mu},A_{\b}] ,  [A_{\nu},A_{\si}] >       - \frac{1}{(n-1)} m_{\mu\nu }  m^{\si\b}  m^{\a\la}   .  <  [A_{\a},A_{\b}] , [A_{\la},A_{\si}] >  \\
     && + O \big( h \cdot  (\pa A)^2 \big)   + O \big(  h \cdot  A^2 \cdot  \pa A \big)     + O \big(  h  \cdot  A^4 \big)    \, .
\eeaa

Now, we would like to write differently the left hand side of the equality.

As shown by Lindblad-Rodnianski in Lemma 3.1 in \cite{LR2} (in particular, in equation (3.17)), the Ricci tensor in wave coordinates can be expressed as
 \beaa
R_{\mu\nu} &=&-\frac{1}{2} g^{\alpha\beta}\pa_\alpha\pa_\beta g_{\mu\nu}+g^{\alpha\alpha^\prime}g^{\beta\beta^\prime}\,\Big( -\frac{1}{4} \pa_\nu g_{\alpha\beta}\,\, \pa_\mu g_{\alpha^\prime\beta^\prime} +\frac{1}{8} \pa_{\mu} g_{\beta\beta^\prime}\, \, \pa_{\nu} g_{\alpha\alpha^\prime} \Big)\\
&&  +\frac{1}{2}g^{\alpha\alpha^\prime}g^{\beta\beta^\prime}  \pa_{\alpha} g_{\beta\mu}\, \,  \pa_{\alpha^\prime} g_{\beta^\prime\nu} -\frac{1}{2}g^{\alpha\alpha^\prime}g^{\beta\beta^\prime}\, \Big(\pa_{\alpha} g_{\beta\mu}\, \, \pa_{\beta^\prime} g_{\alpha^\prime\nu}-\pa_{\beta^\prime} g_{\beta\mu}\, \, \pa_{\alpha} g_{\alpha^\prime\nu}\Big)\\
&&+\frac{1}{2}g^{\alpha\alpha^\prime}g^{\beta\beta^\prime}\,\Big( \big(\pa_\mu g_{\a^\prime\b^\prime}\, \pa_\a g_{\b\nu}- \pa_\a g_{\a^\prime\b^\prime}\, \pa_\mu g_{\b\nu}\big) +\big(\pa_\nu g_{\a^\prime\b^\prime}\, \pa_\a g_{\b\mu}- \pa_\a g_{\a^\prime\b^\prime}\, \pa_\nu g_{\b\mu}\big) \Big)\\ 
&& +\frac{1}{4}g^{\alpha\alpha^\prime}g^{\beta\beta^\prime} \Big(\big( \pa_{\beta^\prime} g_{\a^\prime\alpha} \,\pa_\mu g_{\b\nu} -\pa_{\mu} g_{\a^\prime\alpha} \,\pa_{\beta^\prime}g_{\b\nu}\big) +\big( \pa_{\beta^\prime} g_{\a^\prime\alpha} \,\pa_\nu g_{\b\mu} -\pa_{\nu} g_{\a^\prime\alpha} \,\pa_{\beta^\prime}g_{\b\mu}\big)\Big) \, .
\eeaa
By defining 
\bea
\widetilde{P}(\pa_\mu g,\pa_\nu g) := \frac{1}{4}  g^{\alpha\alpha^\prime}\pa_\mu g_{\alpha\alpha^\prime} \,  g^{\beta\beta^\prime}\pa_\nu g_{\beta\beta^\prime}- \frac{1}{2} g^{\alpha\alpha^\prime}g^{\beta\beta^\prime} \pa_\mu g_{\alpha\beta}\, \pa_\nu g_{\alpha^\prime\beta^\prime}  \, ,
\eea
and
\bea
\notag
&& \widetilde{Q}_{\mu\nu}(\pa g,\pa g) \\
\notag
&:=& \pa_{\alpha} g_{\beta\mu}\,  \,g^{\alpha\alpha^\prime}g^{\beta\beta^\prime}
 \pa_{\alpha^\prime} g_{\beta^\prime\nu} -g^{\alpha\alpha^\prime}g^{\beta\beta^\prime} \big(\pa_{\alpha}
g_{\beta\mu}\,\,\pa_{\beta^\prime} g_{\alpha^\prime \nu} -\pa_{\beta^\prime} g_{\beta\mu}\,\,\pa_{\alpha} g_{\alpha^\prime\nu}\big)
\label{eq:tildenullform}\\
\notag
&& +g^{\a\a'}g^{\b\b'}\big (\pa_\mu g_{\a'\b'} \pa_\a g_{\b\nu}- \pa_\a g_{\a'\b'} \pa_\mu g_{\b\nu}\big )  + g^{\a\a'}g^{\b\b'}\big (\pa_\nu g_{\a'\b'} \pa_\a g_{\b\mu} - \pa_\a g_{\a'\b'} \pa_\nu g_{\b\mu}\big )\\
\notag
&& +\frac 12 g^{\a\a'}g^{\b\b'}\big (\pa_{\b'} g_{\a\a'} \pa_\mu g_{\b\nu} - \pa_{\mu} g_{\a\a'} \pa_{\b'} g_{\b\nu} \big ) +\frac 12 g^{\a\a'}g^{\b\b'} \big (\pa_{\b'} g_{\a\a'} \pa_\nu g_{\b\mu} - \pa_{\nu} g_{\a\a'} \pa_{\b'} g_{\b\mu} \big ) \, \\
\eea
we get
\bea\label{waveequationonthemetricgwithsourcesonthemetricgdependingonRiccitensor}
 g^{\alpha\beta}\pa_\alpha\pa_\beta
g_{\mu\nu} =  \widetilde{P}(\pa_\mu g,\pa_\nu g) + \widetilde{Q}_{\mu\nu}(\pa g,\pa g)  -2 R_{\mu\nu} \, .
\eea

Now, we want to prove the following lemma:
\begin{lemma}\label{waveequationontheEinsteinYangMillsmetricsmallhwithsourcesusingtheRiccitensorthatwascomoutedearlier}
Let,  
\bea\label{definitionofthetermSinsourcetermsforeinstein}
S_{\mu\nu} (h) (\pa h, \pa h)  &:=& P(\pa_\mu h,\pa_\nu h)+Q_{\mu\nu}(\pa h,\pa h) +G_{\mu\nu}(h)(\pa h,\pa h)
\eea

where
\bea\label{definitionofthetermbigPinsourcetermsforeinstein}
 P(\pa_\mu h,\pa_\nu h) :=\frac{1}{4} m^{\alpha\alpha^\prime}\pa_\mu h_{\alpha\alpha^\prime} \, m^{\beta\beta^\prime}\pa_\nu h_{\beta\beta^\prime}  -\frac{1}{2} m^{\alpha\alpha^\prime}m^{\beta\beta^\prime} \pa_\mu h_{\alpha\beta}\, \pa_\nu h_{\alpha^\prime\beta^\prime} \, ,
\eea
 
\bea\label{definitionofthetermbigQinsourcetermsforeinstein}
\notag
&& Q_{\mu\nu}(\pa h,\pa h) \\
\notag
&:=& \pa_{\alpha} h_{\beta\mu}\, \, m^{\alpha\alpha^\prime}m^{\beta\beta^\prime} \pa_{\alpha^\prime} h_{\beta^\prime\nu} -m^{\alpha\alpha^\prime}m^{\beta\beta^\prime} \big(\pa_{\alpha} h_{\beta\mu}\,\,\pa_{\beta^\prime} h_{\alpha^\prime \nu} -\pa_{\beta^\prime} h_{\beta\mu}\,\,\pa_{\alpha} h_{\alpha^\prime\nu}\big)\\
\notag
&& +m^{\a\a'}m^{\b\b'}\big (\pa_\mu h_{\a'\b'}\, \pa_\a h_{\b\nu}- \pa_\a h_{\a'\b'} \,\pa_\mu h_{\b\nu}\big )  \\
\notag
&& + m^{\a\a'}m^{\b\b'}\big (\pa_\nu h_{\a'\b'} \,\pa_\a h_{\b\mu} - \pa_\a h_{\a'\b'}\, \pa_\nu h_{\b\mu}\big )\\
\notag
&& +\frac 12 m^{\a\a'}m^{\b\b'}\big (\pa_{\b'} h_{\a\a'}\, \pa_\mu h_{\b\nu} - \pa_{\mu} h_{\a\a'}\, \pa_{\b'} h_{\b\nu} \big ) \\
\notag
&& +\frac 12 m^{\a\a'}m^{\b\b'} \big (\pa_{\b'} h_{\a\a'}\, \pa_\nu h_{\b\mu} - \pa_{\nu} h_{\a\a'} \,\pa_{\b'} h_{\b\mu} \big ) \, ,\\
\eea
and
\bea\label{definitionofthetermbigGinsourcetermsforeinstein}
G_{\mu\nu}(h)(\pa h,\pa h) :=  O  (h \cdot (\pa h)^2) \, ,
\eea
i.e. $G_{\mu\nu}(h)(\pa h,\pa h) $ is a quadratic form in $\pa h$ with
coefficients smoothly dependent on $h$ and vanishing when $h$
vanishes: $G_{\mu\nu}(0)(\pa h,\pa h)=0$\,. Then, we have in wave coordinates $\mu\;, \nu\;, \si\;, \a \in \{0, 1, ..., n\} $\,,
\bea
g^{\si\a}\pa_{\si} \pa_{\a} h_{\mu\nu}  -S_{\mu\nu} (h) (\pa h, \pa h)  = -2 R_{\mu\nu} \, .
\eea
\end{lemma}

\begin{proof}

In view of the fact that
\beaa
g^{\mu\nu} &=& m^{\mu\nu}-h^{\mu\nu}  +O^{\mu\nu}(h^2)   \, , \\
g_{\mu\nu} &=& m_{\mu\nu} + h_{\mu\nu}  \, ,
\eeaa
we have
\beaa
\pa_\mu g &=&\pa_\mu  h  \, ,
\eeaa
and
\bea
\notag
g^{\a\a'}g^{\b\b'}  &=& \big(  m^{\a\a'}-h^{\a\a'} +O^{\a\a'}(h^2)  \big) \cdot \big( m^{\b\b'} -h^{\b\b'}   +O^{\b\b'}  (h^2) \big) \\
\notag
&=&  m^{\a\a'} m^{\b\b'} - m^{\a\a'} h^{\b\b'}   +O^{\a\a'\b\b'}  (h^2)  \\
\notag
&& -h^{\a\a'} m^{\b\b'} + h^{\a\a'}h^{\b\b'}  +O^{\a\a'\b\b'}  (h^3) \\
\notag
&& +O^{\a\a'\b\b'}  (h^2 + h^3 + h^4) \\
&=&  m^{\a\a'} m^{\b\b'}  +O^{\a\a'\b\b'}  (h)  \, .
\eea
We know
\beaa
 g^{\alpha\beta}\pa_\alpha\pa_\beta g_{\mu\nu} =  \widetilde{P}(\pa_\mu g,\pa_\nu g) + \widetilde{Q}_{\mu\nu}(\pa g,\pa g)  -2 R_{\mu\nu}  \, .
\eeaa

We have
\beaa
 \widetilde{P}(\pa_\mu g,\pa_\nu g) = P(\pa_\mu h,\pa_\nu h) + O  (h \cdot (\pa h)^2)  \, ,
 \eeaa
 
 \beaa
   \widetilde{Q}_{\mu\nu}(\pa g,\pa g) &=& Q_{\mu\nu}(\pa h,\pa h) + O  (h \cdot (\pa h)^2)  \, .
\eeaa

Thus,
\bea
\notag
 g^{\alpha\beta}\pa_\alpha\pa_\beta
g_{\mu\nu} &=&  P(\pa_\mu h,\pa_\nu h) + Q_{\mu\nu}(\pa h,\pa h)  + O  (h \cdot (\pa h)^2)   -2 R_{\mu\nu} \\
\notag
&=& P(\pa_\mu h,\pa_\nu h) + Q_{\mu\nu}(\pa h,\pa h)  +G_{\mu\nu}(h)(\pa h,\pa h)  -2 R_{\mu\nu} \, ,
\eea
where $G_{\mu\nu}(h)(\pa h,\pa h)$ is a quadratic form in $\pa h$ with
coefficients smoothly dependent on $h$ and vanishing when $h$
vanishes, i.e. $G_{\mu\nu}(0)(\pa h,\pa h)=0$.\\

\end{proof}

\begin{lemma}\label{EYMsystemashyperbolicPDE}
 The Einstein-Yang-Mills equations in Lorenz gauge and in wave coordinates implies that 
      \bea\label{ThewaveequationontheYangMillspotentialwithhyperbolicwaveoperatorusingingpartialderivativesinwavecoordinates}
   \notag
g^{\la\mu} \pa_{\la}   \pa_{\mu}   A_{\si}      &=&  m^{\a\ga} m ^{\mu\la}  (  \pa_{\si}  h_{\ga\la} )   \pa_{\a}A_{\mu}       +   \frac{1}{2}  m^{\a\mu}m^{\b\nu}   \big(   \pa_\a h_{\b\si} + \pa_\si h_{\b\a}- \pa_\b h_{\a\si}  \big)   \cdot  \big( \pa_{\mu}A_{\nu} - \pa_{\nu}A_{\mu}  \big) \\
 \notag
&& +      \frac{1}{2}  m^{\a\mu}m^{\b\nu}   \big(   \pa_\a h_{\b\si} + \pa_\si h_{\b\a}- \pa_\b h_{\a\si}  \big)   \cdot   [A_{\mu},A_{\nu}] \\
 \notag
 && -  m^{\a\mu} \big(  [ A_{\mu}, \pa_{\a} A_{\si} ]  +    [A_{\alpha},  \pa_{\mu}  A_{\si} - \pa_{\si} A_{\mu} ]    +    [A_{\alpha}, [A_{\mu},A_{\si}] ]  \big)  \\
 \notag
  && + O( h \cdot  \pa h \cdot  \pa A) + O( h \cdot  \pa h \cdot  A^2) + O( h \cdot  A \cdot \pa A) + O( h \cdot  A^3) \, ,\\
  \eea
      
  and 
  \bea\label{Thewaveequationonthemetrichwithhyperbolicwaveoperatorusingingpartialderivativesinwavecoordinates}
\notag
 g^{\alpha\beta}\pa_\alpha\pa_\beta
h_{\mu\nu} &=& P(\pa_\mu h,\pa_\nu h) + Q_{\mu\nu}(\pa h,\pa h)  +G_{\mu\nu}(h)(\pa h,\pa h)  \\
\notag
 &&   -4   m^{\si\b} \cdot  <   \pa_{\mu}A_{\b} - \pa_{\b}A_{\mu}  ,  \pa_{\nu}A_{\si} - \pa_{\si}A_{\nu}  >    \\
 \notag
 &&   + \frac{2}{(n-1)} m_{\mu\nu }  m^{\si\b}  m^{\a\la}    \cdot  <  \pa_{\a}A_{\b} - \pa_{\b}A_{\a} , \pa_{\la}A_{\si} - \pa_{\si}A_{\la} >   \\
 \notag
&&           -4 m^{\si\b}  \cdot  \big( <   \pa_{\mu}A_{\b} - \pa_{\b}A_{\mu}  ,  [A_{\nu},A_{\si}] >   + <   [A_{\mu},A_{\b}] ,  \pa_{\nu}A_{\si} - \pa_{\si}A_{\nu}  > \big)  \\
\notag
&& + \frac{2}{(n-1)} m_{\mu\nu }  m^{\si\b}  m^{\a\la}    \cdot \big(  <  \pa_{\a}A_{\b} - \pa_{\b}A_{\a} , [A_{\la},A_{\si}] >    +  <  [A_{\a},A_{\b}] , \pa_{\la}A_{\si} - \pa_{\si}A_{\la}  > \big) \\
\notag
 &&  -4 m^{\si\b}  \cdot   <   [A_{\mu},A_{\b}] ,  [A_{\nu},A_{\si}] >      +  \frac{2}{(n-1)} m_{\mu\nu }  m^{\si\b}  m^{\a\la}   \cdot   <  [A_{\a},A_{\b}] , [A_{\la},A_{\si}] >  \\
     && + O \big(h \cdot  (\pa A)^2 \big)   + O \big(  h  \cdot  A^2 \cdot \pa A \big)     + O \big(  h   \cdot  A^4 \big)    \, ,
\eea
where $P$\;, $Q$ and $G$ are defined in \eqref{definitionofthetermbigPinsourcetermsforeinstein}, \eqref{definitionofthetermbigQinsourcetermsforeinstein} and \eqref{definitionofthetermbigGinsourcetermsforeinstein}.

\end{lemma}

\begin{proof}
As a result of Lemmas \ref{waveequationontheEinsteinYangMillspotentialderivedfromthegaugecovariantdivergenceoftheYangMillscurvatureisequaltozero}, \ref{waveequationontheEinsteinYangMillsmetricsmallhwithsourcesusingtheRiccitensorthatwascomoutedearlier}, and \ref{EinsteinYangMillssystemusingpotentialandusingMinkowskimetricmandperturbationh}. we have proved the stated lemma.
\end{proof}

  \section{Construction of the initial data and the gauges conditions constraints}\label{Construction of the initial data and the gauges conditions constraints}
  
 We assume that we are given already two one-tensors $\overline{A} = \overline{A}_{i}dx^i$ and $E = E_{i}dx^i$ valued in the Lie algebra $\cal G$\,, associated to the group $G$\,, prescribed on a given $n$-dimensional manifold $\Sigma$ diffeomorphic to $\R^n$, with a Riemannian metric $\overline{g}$ on the initial slice $\Sigma$\,, and with a symmetric two-tensor $\overline{k}$ on $\Sigma$\,. The given initial data set is then $(\Sigma, \overline{A}, \overline{E}, \overline{g}, \overline{k})$\,. Let $A_\Sigma$ and  $g_\Sigma$  be the restrictions of $A$ and $g$ on $\Sigma$\,. We want to translate this initial data set to the form of $(\Sigma, A_\Sigma, \pa_t A_\Sigma, g_\Sigma, \pa_t g_\Sigma)$ taking into consideration that we already chose to look at our solution in both the Lorenz gauge and in wave coordinates -- this will be useful for a hyperbolic formulation of the Cauchy problem.
 
\begin{remark}
We note that to impose that the solution remains in wave coordinates, imposes an additional wave coordinate constraint on the initial data set that is not present in the case of the Einstein vacuum equations. In other words, we will show that any arbitrary initial data satisfying the Einstein-Yang-Mills constraints and that is in the Lorenz gauge, will remain in the Lorenz gauge, but if the initial data is in wave coordinates, it will not remain in wave coordinates unless a wave coordinates constraint is imposed on the initial data set which we will exhibit.
\end{remark}

\subsection{The initial data for the Yang-Mills potential}\label{IntialdataforYangMills}\

For any solution of the Einstein-Yang-Mills equations, one can perform a gauge transformation on the Yang-Mills potential such that on $\Sigma$, we have $A_{t} = 0$ at $t=0$\;. However, this is not necessarily preserved for $t > 0$, since we want in fact to satisfy the Lorenz gauge condition as well. In any case, we can always define on the initial slice $\Sigma$\;, the following
 \bea\label{initialdatadforzerothderivativeAsigma}
A_{\Sigma} = \begin{cases} (A_{\Sigma})_{t} = 0 \; ,  \\
(A_{\Sigma})_{i}= \overline{A}_{i} \quad\text{prescribed arbitrarily for }\quad i \neq t \; , \end{cases} 
\eea
and we then look forward to choosing $\p_{t} A_{\Sigma}$ such that the Lorenz condition and the Einstein-Yang-Mills constraints are satisfied.

In wave coordinates $x^\mu$\;, we compute
\beaa
\der_{\mu} A^{\mu} &=& \pa_{\mu} A^{\mu} =  \pa_{t} A^{t} + \pa_{i} A^{i} \\
&=&  \pa_{t} (g^{t\mu} A_{\mu} )+ \pa_{i} ( g^{i\mu} A_{\mu}) \, .
\eeaa
We are going to construct, in Subsection \ref{constructioninitialdataformetric} (see \ref{constructionoggsigma}), the initial data for the metric such that at $t=0$, $g_{ti} = 0$ for all $i\neq t$. Then, we also have $g^{t i} = 0$ for all $i\neq t$. Thus, at $t=0$, 
\bea
\notag
\der_{\mu} A^{\mu} &=& \pa_{t} (g^{tt} A_{t} )+ \pa_{i} ( g^{ij} A_{j}) \\
\notag
&=& \pa_{t} (g^{tt} ) A_{t} + g^{tt} \pa_{t}  A_{t}  + \pa_{i} ( g^{ij} A_{j}) \\
\notag
&=&   g^{tt} \pa_{t}  A_{t}  + \pa_{i} ( g^{ij} A_{j}) \\
\notag
&& \text{(since $A_t = 0$ at $t=0$).} 
\eea
Hence, the Lorenz gauge reads in wave coordinates
\beaa
\notag
0 &=& \der_{\mu} A^{\mu} =  g^{tt} \pa_{t}  A_{t} + \pa_{i} ( g^{ij} A_{j}) \, ,
\eeaa
from which we deduce that 
\bea
\pa_{t}  A_{t}  = - \frac{1}{g^{tt}}   \pa_{i} ( g^{ij} A_{j}) = N^2  \pa_{i} ( g^{ij} A_{j}) \, .
\eea

However, on one hand, the initial data $A_i$ must satisfy the following Yang-Mills constraint equation at $t=0$,
\bea
\notag
0 &=& \textbf{D}^{(A)}_{i}F^{i t}\\
\notag
&=& g^{t\mu}\textbf{D}^{(A)}_{i}{F^{i}}_{\mu} = g^{tt}\textbf{D}^{(A)}_{i}{F^{i}}_{t}\\
\notag
&=& g^{tt}  \der^{i}  (   \der_{i}A_{t} - \der_{t}A_{i} + [A_{i},A_{t}]  ) + g^{tt}   [A^{i},  \der_{i}A_{t} - \der_{t}A_{i} + [A_{i},A_{t}]  ]  \\
\notag
&=&   \der^{i}  (  g^{tt} \der_{i}A_{t} - g^{tt} \der_{t}A_{i} + g^{tt} [A_{i},A_{t}]  ) + g^{tt}   [A^{i},  \der_{i}A_{t} - \der_{t}A_{i} + [A_{i},A_{t}]  ]  \\
\notag
&=&  \der^{i}  (   g^{tt} \pa_{i}A_{t} -  g^{tt} \pa_{t}A_{i}  ) + g^{tt}   [A^{i},  \pa_{i}A_{t} - \pa_{t}A_{i}   ]  \\
\notag
&& \text{(using the fact that $A_t = 0$ at $t=0$)} \\
\notag
&=& \der^{i}  (  -  g^{tt} \pa_{t}A_{i}  ) + g^{tt}   [A^{i},   - \pa_{t}A_{i}   ]  \, .
\eea
On the other hand,
\beaa
E_i &=& F_{\hat{t} i} = \frac{1}{N} F_{t i } =   \frac{1}{N} \big( \der_{i}A_{t} -  \der_{t}A_{i} + [A_{i},A_{t}]  \big) =   \frac{1}{N} \big( \pa_{i}A_{t} -  \pa_{t}A_{i} + [A_{i},A_{t}]  \big)  \\
&=&  - \frac{1}{N}  \pa_{t}A_{i}    \\
&& \text{(using the fact that $A_t = 0$ at $t=0$)} \\
&=& g^{tt}   \pa_{t}A_{i} \, .
\eeaa
Consequently, the Yang-Mills constraint equation reads
\beaa
0 &=& \textbf{D}^{(A)}_{i}F^{i t}\\
\notag
 &=& g^{tt}  \der^{i}  (   -E_{i} ) + g^{tt}   [A_{i},   -E_{i}   ]  \\
\notag
&=& g^{tt}  D^{i}  (   -E_{i} ) + g^{tt}   [A_{i},   -E_{i}   ]  \, ,
\eeaa
with $\pa_{t} A_{i} = - N E_{i}$, where $M$ is defined in \eqref{definitionoflapseN}.

Thus, we choose
 \bea\label{constructionopatialtimeAonsigma}
\p_{t} A_{\Sigma} = \begin{cases} (\pa_{t } A_{\Sigma})_{t} =  N^2  \pa_{i} ( \overline{g}^{ij} \overline{A}_{j}) \, ,\\
(\pa_{t } A_{\Sigma})_{i} = - N \overline{E}_{i} \quad\text{ where $\overline{E}_i$ is prescribed arbitrarily for } i \neq t \; ,\\
\quad\quad\quad\quad\quad\quad\;\;\quad\text{such that } D^i \overline{E}_{ i} + [\overline{A}^i, \overline{E}_{ i} ] = 0 \,. \\ \end{cases} 
\eea

 \subsection{The initial data for the metric}\label{constructioninitialdataformetric}\
 
Since the Einstein-Yang-Mills equations are invariant under diffeomorphisms, for a given solution $(M, F, \g)$ of the Einstein-Yang-Mills system, one can always perform a diffeomorphism such that $\Sigma$ corresponds to $t=0$ and such that $\frac{\pa }{\pa t}$ is orthogonal to $\Sigma$ at $t=0$\;, i.e. such that the metric $g_\Sigma$ has the following form
\bea\label{constructionoggsigma}
g_{\Sigma} = \begin{cases} (g_{\Sigma})_{tt} = - N^2 \, , \\
(g_{\Sigma})_{ij} =   \overline{g}_{ij}  \quad\text{given by the initial data}, \\
(g_{\Sigma})_{tj} =  (g_{\Sigma})_{jt} = 0\, .
 \end{cases} 
\eea
In fact,
\bea
\frac{\pa }{\pa t} = N \, \hat{t} + X^i e_i \;,
\eea
where $N$ is the lapse function and $X^i$ is the shift vector, and choose that on on the given Cauchy hypersurface $\Sigma$\;,
\bea
X^i = 0 \;.
\eea
In other words, we make a diffeomorphism of the solution such that $\hat{t}$\;, the unitary time-like vector orthogonal to $\Sigma_t$\;, agrees on $\Sigma$ with $ \frac{1}{N} \frac{\pa}{\pa t}$\;, i.e. that we have on $\Sigma$\;,  $${\hat{t}}_\Sigma = \frac{1}{N} \frac{\pa}{\pa t} \;.$$

However, we want to perform the diffeomorphism in a way such that the metric $g_\Sigma$ is not only in the form prescribed above but also in wave coordinates simultaneously: we will show that this is indeed possible provided that on $\Sigma$, we choose $ \frac{\pa}{\pa t} {g_\Sigma}_{\mu\nu} $ adequately in terms of the initial data $\overline{g}$ and $\overline{k}$. In fact, we could do so without adding any constraint on the initial data since $(\frac{\pa}{\pa t}  g_{\Sigma})_{\mu\nu} $ is not part of the initial data set.

 In other words, we are going to construct $(\frac{\pa}{\pa t}  g_{\Sigma})_{\mu\nu} $ on $\Sigma$\;, using the wave coordinates condition. To start with, for a solution $(M, \g)$\;, let us compute on $\Sigma$\;,
\beaa
\der_{\frac{1}{N} \frac{\pa}{\pa t}} {g}_{\mu\nu} = 0  = \frac{1}{N}  \pa_{t} {g}_{\mu\nu} - {g} (\der_{\frac{1}{N}  \frac{\pa}{\pa t}}  e_\mu , e_\nu ) - {g} (  e_\mu , \der_{\frac{1}{N}  \frac{\pa}{\pa t}}  e_\nu ) \, ,
\eeaa
which gives
\beaa
\frac{1}{N}  \pa_{t} {g}_{\mu\nu} &=&  {g} (\der_{\frac{1}{N} \frac{\pa}{\pa t}} e_\nu, e_\mu) + {g} ( e_\mu, \der_{\frac{1}{N}  \frac{\pa}{\pa t}} e_\nu) \, .\\
\eeaa
In particular, for spatial indices $i, j$, we get
\beaa
\frac{1}{N}  \pa_{t} {g}_{ij} &=&  {g} (\der_{\frac{1}{N} \frac{\pa}{\pa t}} e_{i}, e_{j}) + {g} ( e_{i}, \der_{\frac{1}{N}  \frac{\pa}{\pa t}} e_{j}) \\
&=&  {g} (\frac{1}{N} \der_{ \frac{\pa}{\pa t}} e_{i}, e_{j}) + {g} ( e_{i}, \frac{1}{N} \der_{ \frac{\pa}{\pa t}} e_{j}) \, .\\
\eeaa

Since $ \frac{\pa}{\pa t}$ and $e_i$ are coordinate vector fields, we have

\beaa
[ \frac{\pa}{\pa t}, e_{i}] = 0 = \der_{\frac{\pa}{\pa t}}  e_{i} - \der_{e_{i}} \frac{\pa}{\pa t} \, .
\eeaa
Hence,
\beaa
\frac{1}{N}  \pa_{t} {g}_{ij} &=&  {g} (\frac{1}{N} \der_{  e_{i}} \frac{\pa}{\pa t} , e_{j}) + {g} ( e_{i}, \frac{1}{N} \der_{  e_{j}} \frac{\pa}{\pa t}) \\
&=&  {g} ( \der_{  e_{i}} ( \frac{1}{N} \frac{\pa}{\pa t} ) - \pa_{e_i} ( \frac{1}{N} ) \frac{\pa}{\pa t}  , e_{j}) + {g} ( e_{i}, \der_{  e_{j}} (  \frac{1}{N}\frac{\pa}{\pa t} ) - \pa_{e_j} ( \frac{1}{N} ) \frac{\pa}{\pa t}  ) \\
&=&  {g} ( \der_{  e_{i}} ( \frac{1}{N} \frac{\pa}{\pa t} )  , e_{j}) + {g} ( e_{i}, \der_{  e_{j}} (  \frac{1}{N}\frac{\pa}{\pa t} ) ) \\
&& \text{(since $\frac{\pa}{\pa t}$ is orthogonal to the hypersurface $\Sigma$).}
\eeaa

Since 
\beaa
k ( e_{i} ,  e_{j}) &=& g(\der_{ e_{i}} \hat{t},  e_{j}) \, , \\
\eeaa
this leads to 
\beaa
\frac{1}{N}  \pa_{t} {g}_{ij} &=&  {g} (\der_{e_{i}} \hat{t}, e_{j}) + {g} ( e_{i}, \der_{ e_{j}} \hat{t}) \\
&=& 2 k_{ij}  \, . \\
\eeaa
Hence, we impose
\bea
 \pa_{t} {g}_{ij} &=&  2 N k_{ij}  \, .
\eea

\begin{lemma}\label{Lemmaonwritingthewaveconditiononginwavecoordinates}
In wave coordinates, we have
\bea
  g^{\mu\nu}  \pa_\mu  g_{\si\nu}   -\frac{1}{2}  g^{\mu\nu}  \pa_\si g_{\mu\nu} = 0 \, .
\eea
\end{lemma}

\begin{proof}

On one hand, the formula for the Christoffel symbols are given by
\beaa
\Gamma^{\,\,\, \, \, \, \, 
 \la}_{\mu\nu} =\frac{1}{2} g^{\lambda\delta} \big( \pa_\mu
g_{\delta\nu} + \pa_\nu g_{\delta\mu}-\pa_\delta g_{\mu\nu}\big)  \, .
\eeaa
Thus,
\beaa
g_{\la \si} \Gamma^{\,\,\, \, \, \, \, 
 \la}_{\mu\nu}&=& \frac{1}{2} g_{\la \si} g^{\lambda\delta} \big( \pa_\mu
g_{\delta\nu} + \pa_\nu g_{\delta\mu}-\pa_\delta g_{\mu\nu}\big) \\
&=& \frac{1}{2} I_{\si}^{\, \, \, \delta} \big( \pa_\mu
g_{\delta\nu} + \pa_\nu g_{\delta\mu}-\pa_\delta g_{\mu\nu}\big) \\
&=& \frac{1}{2} \big( \pa_\mu
 ( I_{\si}^{\, \, \, \delta}  g_{\delta\nu} )+ \pa_\nu ( I_{\si}^{\, \, \, \delta}  g_{\delta\mu} ) -\pa_\si g_{\mu\nu}\big) \\
 &=& \frac{1}{2} \big( \pa_\mu  g_{\si\nu} + \pa_\nu  g_{\si\mu}  -\pa_\si g_{\mu\nu}\big)  \, . \\
\eeaa

On the other hand, the wave coordinate condition, \eqref{wavecoordinatecondition}, reads

\beaa
g^{\mu\nu} \Gamma^{\,\,\, \, \, \, \, 
 \la}_{\mu\nu} =0 \, .
\eeaa
Thus, injecting, we obtain
\beaa
0 &=&  g^{\mu\nu}g_{\la \si} \Gamma^{\,\,\, \, \, \, \, 
 \la}_{\mu\nu} \\
  &=& g^{\mu\nu} \frac{1}{2} \big( \pa_\mu  g_{\si\nu} + \pa_\nu  g_{\si\mu}  -\pa_\si g_{\mu\nu}\big) \\
    &=&  \frac{1}{2} \big( \pa^\nu  g_{\si\nu} + \pa_\mu  g_{\si\mu}  -\pa_\si g_{\mu\nu}\big) \\
     &=&  \pa^\nu  g_{\si\nu}   -\frac{1}{2}  g^{\mu\nu}  \pa_\si g_{\mu\nu} \\
          &=& g^{\mu\nu}  \pa_\mu  g_{\si\nu}   -\frac{1}{2}  g^{\mu\nu}  \pa_\si g_{\mu\nu}  \, .
\eeaa

\end{proof}

Using Lemma \ref{Lemmaonwritingthewaveconditiononginwavecoordinates}, for $\si=t$, we obtain that in wave coordinates,
\beaa
0 &=&  g^{\mu\nu}  \pa_\mu  g_{t\nu}   -\frac{1}{2}  g^{\mu\nu}  \pa_t g_{\mu\nu} \\
&=& g^{\mu t}  \pa_\mu  g_{tt} +  g^{\mu i}  \pa_\mu  g_{ti}    -\frac{1}{2}  g^{tt}  \pa_t g_{tt}  -\frac{1}{2}  g^{ij}  \pa_t g_{ij}  \\
&=& g^{t t}  \pa_t  g_{tt}   + g^{j i}  \pa_j  g_{t i} -\frac{1}{2}  g^{tt}  \pa_t g_{tt}  -\frac{1}{2}  g^{ij}  \pa_t g_{ij}  \\
&=&\frac{1}{2}  g^{tt}  \pa_t g_{tt}  -\frac{1}{2}  g^{ij}  \pa_t g_{ij}  \\
&& \text{(where we used the fact that $g_{ti} = 0$ on $\Sigma$).}
\eeaa

Hence,
\beaa
g^{tt}  \pa_t g_{tt}  &=& g^{ij}  \pa_t g_{ij}  \, .  \\
\eeaa
Consequently,
\bea
\notag
 \pa_t g_{tt}  &=& \frac{1}{g^{tt} } g^{ij}  \pa_t g_{ij}  \\
 &=& -N^2   g^{ij}  \pa_t g_{ij}   \, .
\eea

For $\si=j$, the wave coordinates condition reads
\beaa
  g^{\mu\nu}  \pa_\mu  g_{j \nu}   -\frac{1}{2}  g^{\mu\nu}  \pa_j g_{\mu\nu} = 0  \, .
\eeaa
We get
\beaa
 0 &=&  g^{\mu t}  \pa_\mu  g_{j t}  + g^{\mu i}  \pa_\mu  g_{j i}   -\frac{1}{2}  g^{tt}  \pa_j g_{tt}  -\frac{1}{2}  g^{ik}  \pa_j g_{ik} \\
 0 &=&   g^{t t}  \pa_t  g_{j t} +  g^{k i}  \pa_k g_{j i}   -\frac{1}{2}  g^{tt}  \pa_j g_{tt}  -\frac{1}{2}  g^{ik}  \pa_j g_{ik}  \, .\\
 \eeaa
Thus,
\bea
\notag
 \pa_t  g_{j t}   &=&  \frac{1}{ g^{t t} } \big( -  g^{k i}  \pa_k g_{j i}   + \frac{1}{2}  g^{tt}  \pa_j g_{tt}  + \frac{1}{2}  g^{ik}  \pa_j g_{ik}  \big) \\
 \notag
   &=& - N^2 \big( -  g^{k i}  \pa_k g_{j i}   -  \frac{1}{2 N^2 } \pa_j (-N^2)  + \frac{1}{2}  g^{ik}  \pa_j g_{ik}  \big) \\
      &=& - N^2 \big( \frac{1}{ N }   \pa_j N   -  g^{k i}  \pa_k g_{j i} + \frac{1}{2}  g^{ik}  \pa_j g_{ik}  \big)  \, .
 \eea

Finally, in consistency with the wave coordinate condition, we take the initial data
\bea\label{constructionopatialtgonsigma}
\pa_{t} {g_\Sigma} = \begin{cases}  ( \pa_{t} {g_\Sigma})_{ij}  =  2 N \overline{k}_{ij}  \, ,\\
 ( \pa_{t} {g_\Sigma})_{tt} =  - N^2  g^{ij}  \pa_t g_{ij}   =   - 2 N^3  \overline{g}^{ij} \overline{k}_{ij} \, ,\\
  ( \pa_{t} {g_\Sigma})_{tj} =   ( \pa_{t} {g_\Sigma})_{jt} =  - N \pa_j N   + N^2  \overline{g}^{k i}  \pa_k  \overline{g}_{j i} -  \frac{N^2}{2}   \overline{g}^{ik}  \pa_j  \overline{g}_{ik}    \, .
 \end{cases} 
\eea

Furthermore, to ensure that the solution remains in wave coordinates, the initial data must satisfy an additional wave coordinate constraint to ensure that the wave coordinates condition propagates -- this will be discussed in the next section.

\subsection{The propagation of the Lorenz gauge condition}\

We will show that there is indeed a way to solve the Einstein-Yang-Mills system in a manner that guarantees that the Lorenz gauge propagates in time. In other words, given the initial data set  $(\Sigma, A_\Sigma, \pa_t A_\Sigma, g_\Sigma, \pa_t g_\Sigma)$\;, that we have just constructed in Subsections \ref{IntialdataforYangMills} and \ref{constructioninitialdataformetric}, we will show a way to construct a solution to the Einstein-Yang-Mills equations, \eqref{EYMsystemofequationsongandA}, that is in the  Lorenz gauge for all time. 

\begin{lemma}\label{waveoperatorontheYangMillspoentialwithsourcesdependingonA}
The Yang-Mills equations, \eqref{ThegaugecovariantdivergenceoftheYangMillscurvatureisequaltozero}, read
\bea
\notag
\Box_{\g} A^{\b} =     \der^{\b}  \der_{\alpha} A^{\a} + R_{ \mu}^{\, \, \, \, \b}   A^{\mu} - [  \der_{\alpha} A^{\a},A^{\b}] - 2 [  A_{\a},\der^{\alpha}  A^{\b}]     + [A_{\alpha},  \der^{\b}A^{\a} ] -  [A_{\alpha}, [A^{\a},A^{\b}]  ]  \, . \\
\eea

\end{lemma}

\begin{proof}
We write the Yang-Mills equations, \eqref{ThegaugecovariantdivergenceoftheYangMillscurvatureisequaltozero}, in terms of the Yang-Mills potential $A$\, and we get
\beaa
\notag
0 &=& \textbf{D}^{(A)}_{\a}F^{\a\b}\\
\notag
&=& \der_{\alpha} F^{\a\b}  + [A_{\alpha}, F^{\a\b} ]  \\
\notag
&=&  \der_{\alpha}  (   \der^{\a}A^{\b} - \der^{\b}A^{\a} + [A^{\a},A^{\b}]  ) +   [A_{\alpha},  \der^{\a}A^{\b} - \der^{\b}A^{\a} + [A^{\a},A^{\b}]  ]  \\
\notag
&=&  \der_{\alpha}    \der^{\a}A^{\b} -   \der_{\alpha} \der^{\b}A^{\a} + [  \der_{\alpha} A^{\a},A^{\b}] +  [  A^{\a},\der_{\alpha}  A^{\b}]  +   [A_{\alpha},  \der^{\a}A^{\b} - \der^{\b}A^{\a} + [A^{\a},A^{\b}]  ]  \\
\notag
&=&  \Box_{\g} A^{\b} -   \der_{\alpha} \der^{\b}A^{\a} + [  \der_{\alpha} A^{\a},A^{\b}] +  [  A_{\a},\der^{\alpha}  A^{\b}]   +   [A_{\alpha},  \der^{\a}A^{\b}]  - [A_{\alpha},  \der^{\b}A^{\a} ] +  [A_{\alpha}, [A^{\a},A^{\b}]  ] \, . \\
\eeaa
Using the fact that
\beaa
  \der_{\alpha} \der^{\b}A^{\a} &=&   \der^{\b}  \der_{\alpha} A^{\a} +R_{\,\,\, \mu \a }^{\a \,\, \,\,\, \,\, \,  \b}    A^{\mu} \\
  &=&   \der^{\b}  \der_{\alpha} A^{\a} + R_{ \mu}^{\, \, \, \, \b}    A^{\mu} \, ,
  \eeaa
we get
\beaa
\notag
0 &=& \textbf{D}^{(A)}_{\a}F^{\a\b}\\
\notag
&=&  \Box_{\g} A^{\b} -    \der^{\b}  \der_{\alpha} A^{\a} - R_{ \mu}^{\, \, \, \, \b}   A^{\mu} + [  \der_{\alpha} A^{\a},A^{\b}] +  2 [  A_{\a},\der^{\alpha}  A^{\b}]     - [A_{\alpha},  \der^{\b}A^{\a} ] +  [A_{\alpha}, [A^{\a},A^{\b}]  ]  \, .
\eeaa
\end{proof}

\begin{lemma}\label{equationtodefinepropgationofAtorespectLorenzgauge}
The non-linear wave equation on $A$ given in \eqref{ThewaveequationontheYangMillspotentialwithhyperbolicwaveoperatorusingingpartialderivativesinwavecoordinates}, reads
\bea
  \Box_{\g} A^{\b} &=&   R_{ \mu}^{\, \, \, \, \b}   A^{\mu}  -  2 [  A_{\a},\der^{\alpha}  A^{\b}]     + [A_{\alpha},  \der^{\b}A^{\a} ] -  [A_{\alpha}, [A^{\a},A^{\b}]  ]  \, ,
\eea

\end{lemma}

\begin{proof}

Based on the proofs of Lemmas \ref{waveequationontheYangMillspotentialwithspurcesdependingonthemetricg} and \ref{waveequationontheEinsteinYangMillspotentialderivedfromthegaugecovariantdivergenceoftheYangMillscurvatureisequaltozero}, we see that we got the non-linear wave equation on $A$ in \eqref{ThewaveequationontheYangMillspotentialwithhyperbolicwaveoperatorusingingpartialderivativesinwavecoordinates}, by using the Lorenz gauge \eqref{TheLorenzgaugeconditionthatwechoose}, for the terms $\der^{\b}  \der_{\alpha} A^{\a}$ and $ [  \der_{\alpha} A^{\a},A^{\b}]$ that appear on the right hand side of the equation in Lemma \ref{waveoperatorontheYangMillspoentialwithsourcesdependingonA}. Thus, we get the result.

\end{proof}

Now, we want to show that solving the equation on $A$ given in Lemma \ref{equationtodefinepropgationofAtorespectLorenzgauge}, with the initial data $(A_{\Sigma}, \pa_{t} A_{\Sigma})$ constructed as in Subsection \ref{IntialdataforYangMills}\;, would guarantee that the Lorenz gauge is satisfied for all time $t$\;, provided that the Yang-Mills constraint \eqref{theYangMillsconstraintforthepotentialandelectricchargeonsigma} is satisfied.

\begin{lemma}\label{waveoperatorontheLorenzgaugeconditionwithsourcesdependingonA}
Let
\bea\label{definitionofquantityforLorenzgaugecondi}
\La := \der_{\mu} A^{\mu} \, .
\eea
For a solution of the Yang-Mills equations \eqref{ThegaugecovariantdivergenceoftheYangMillscurvatureisequaltozero}, we have
\bea
\notag
  \Box_{\g} \La =       [  \der_{\b}  \La ,A^{\b}]  +   \der_{\b}  \big( \Box_{\g} A^{\b}  -    R_{ \mu}^{\, \, \, \, \b}   A^{\mu}   +  2 [  A_{\a},\der^{\alpha}  A^{\b}]     - [A_{\alpha},  \der^{\b}A^{\a} ] +  [A_{\alpha}, [A^{\a},A^{\b}]  ]   \big) \, .\\
\eea

\end{lemma}

\begin{proof}
We compute the covariant divergence of \eqref{ThegaugecovariantdivergenceoftheYangMillscurvatureisequaltozero}, 
\beaa
0 &=&  \der_{\b} \textbf{D}^{(A)}_{\a}F^{\a\b}\\
&=&  \der_{\b} \big(   \Box_{\g} A^{\b} -    \der^{\b}  \der_{\alpha} A^{\a} - R_{ \mu}^{\, \, \, \, \b}   A^{\mu} + [  \der_{\alpha} A^{\a},A^{\b}] +  2 [  A_{\a},\der^{\alpha}  A^{\b}]     - [A_{\alpha},  \der^{\b}A^{\a} ] +  [A_{\alpha}, [A^{\a},A^{\b}]  ]   \big) \\
&=&  -     \der_{\b} \der^{\b}  \der_{\alpha} A^{\a} +  \der_{\b}  [  \der_{\alpha} A^{\a},A^{\b}]  +  \der_{\b}  \big( \Box_{\g} A^{\b}  -  R_{ \mu}^{\, \, \, \, \b}  A^{\mu}   +  2 [  A_{\a},\der^{\alpha}  A^{\b}]     - [A_{\alpha},  \der^{\b}A^{\a} ] +  [A_{\alpha}, [A^{\a},A^{\b}]  ]   \big) \\
&=&  -      \Box_{\g}\der_{\alpha} A^{\a} +    [  \der_{\b} \der_{\alpha} A^{\a},A^{\b}]  +[  \der_{\alpha} A^{\a},  \der_{\b}A^{\b}]  \\
&& +   \der_{\b}  \big( \Box_{\g} A^{\b}  -   R_{ \mu}^{\, \, \, \, \b}   A^{\mu}   +  2 [  A_{\a},\der^{\alpha}  A^{\b}]     - [A_{\alpha},  \der^{\b}A^{\a} ] +  [A_{\alpha}, [A^{\a},A^{\b}]  ]   \big) \\
&=&  -      \Box_{\g}  \La  +    [  \der_{\b}  \La ,A^{\b}]    \\
&& +   \der_{\b}  \big( \Box_{\g} A^{\b}  -    R_{ \mu}^{\, \, \, \, \b}   A^{\mu}   +  2 [  A_{\a},\der^{\alpha}  A^{\b}]     - [A_{\alpha},  \der^{\b}A^{\a} ] +  [A_{\alpha}, [A^{\a},A^{\b}]  ]   \big) \, .
 \eeaa
  
\end{proof}

\begin{lemma}\label{LemmathatgaranteeshowtheLorenzgaugewillbepreserved}
If $A$ is a solution to \eqref{ThewaveequationontheYangMillspotentialwithhyperbolicwaveoperatorusingingpartialderivativesinwavecoordinates}, with the initial data $(A_{\Sigma}, \pa_{t} A_{\Sigma})$ as constructed in Subsection \ref{IntialdataforYangMills} and such that it solves the Yang-Mills constraint \eqref{theYangMillsconstraintforthepotentialandelectricchargeonsigma}, then for $\La$ defined in \eqref{definitionofquantityforLorenzgaugecondi}, then the Yang-Mills equations \eqref{ThegaugecovariantdivergenceoftheYangMillscurvatureisequaltozero} imply
\bea\label{nonlinearwaveequationontheLorenzGaugeCondition}
  \Box_{\g} \La =       [  \der_{\b}  \La ,A^{\b}]  \; ,
\eea
and we have $\La_{\Sigma} = 0$ and $\pa_{t} \La_{\Sigma} = 0$\;, and thus $\La = 0$ for all time $t$\;.
\end{lemma}

\begin{proof}

First, based on Lemmas \ref{equationtodefinepropgationofAtorespectLorenzgauge} and \ref{waveoperatorontheLorenzgaugeconditionwithsourcesdependingonA}, we get that for a solution to \eqref{ThewaveequationontheYangMillspotentialwithhyperbolicwaveoperatorusingingpartialderivativesinwavecoordinates}, the Yang-Mills equations read
\bea
  \Box_{\g} \La =       [  \der_{\b}  \La ,A^{\b}]  \; .
\eea

Now, based on Lemma \ref{waveoperatorontheYangMillspoentialwithsourcesdependingonA}, we showed that the Yang-Mills equations read
 \beaa
\notag
0 &=& \textbf{D}^{(A)}_{\a}F^{\a\b}\\
\notag
&=&  \Box_{\g} A^{\b} -    \der^{\b}  \La - R_{ \mu}^{\, \, \, \, \b}  A^{\mu} + [  \La,A^{\b}] +  2 [  A_{\a},\der^{\alpha}  A^{\b}]     - [A_{\alpha},  \der^{\b}A^{\a} ] +  [A_{\alpha}, [A^{\a},A^{\b}]  ]  \, , \\
\eeaa
which in its turn implies 
  \beaa
\notag
0 &=& \textbf{D}^{(A)}_{\a}F^{\a t}\\
\notag
&=&  \Box_{\g} A^{t} -    \der^{t}  \La - R_{ \mu}^{\, \, \, \, t}  A^{\mu} + [  \La,A^{t}] +  2 [  A_{\a},\der^{\alpha}  A^{t}]     - [A_{\alpha},  \der^{t}A^{\a} ] +  [A_{\alpha}, [A^{\a},A^{t}]  ] \, , \\
\eeaa
which yields to
  \beaa
\notag
 \der^{t}  \La&=&  \Box_{\g} A^{t} - R_{ \mu}^{\, \, \, \, t}   A^{\mu} + [  \La,A^{t}] +  2 [  A_{\a},\der^{\alpha}  A^{t}]     - [A_{\alpha},  \der^{t}A^{\a} ] +  [A_{\alpha}, [A^{\a},A^{t}]  ] \, . \\
\eeaa
Since $(g_{\Sigma})_{ti} = 0$ for all $i\neq t$\;, we have $(g_{\Sigma})^{t i} = 0$ for all $i\neq t$\;. Thus,
  \beaa
\notag
g^{t\mu} \der_{t}  \La&=&g^{t\mu}  \big(  \Box_{\g} A_{t} - R_{ \mu t}   A^{\mu} + [  \La,A_{t}] +  2 [  A_{\a},\der^{\alpha}  A_{t}]     - [A_{\alpha},  \der_{t}A^{\a} ] +  [A_{\alpha}, [A^{\a},A_{t}]  ]   \big) \, , \\
\eeaa
implies on $\Sigma$\;, that
\beaa
\notag
g^{tt} \pa_{t}  \La&=&g^{tt}  \big(  \Box_{\g} A_{t} - R_{ \mu t}   A^{\mu} + [  \La,A_{t}] +  2 [  A_{\a},\der^{\alpha}  A_{t}]     - [A_{\alpha},  \der_{t}A^{\a} ] +  [A_{\alpha}, [A^{\a},A_{t}]  ]   \big) \, . \\
\eeaa
Since by construction of the initial data for the Yang-Mills potential $A_\a$, we have $\La_\Sigma =0$ , and since $(g_{\Sigma})^{tt} = - \frac{1}{N^2} \neq 0$, we obtain that on $\Sigma$\;, 
\beaa
\notag
 \pa_{t}  \La_{\Sigma} &=&  \Box_{\g} A_{t} - R_{ \mu t}   A^{\mu} + [  \La,A_{t}] +  2 [  A_{\a},\der^{\alpha}  A_{t}]     - [A_{\alpha},  \der_{t}A^{\a} ] +  [A_{\alpha}, [A^{\a},A_{t}]  ]    \\
&=&  \Box_{\g} A_{t} - R_{ \mu t}   A^{\mu} +  2 [  A_{\a},\der^{\alpha}  A_{t}]     - [A_{\alpha},  \der_{t}A^{\a} ] +  [A_{\alpha}, [A^{\a},A_{t}]  ]    \\
  &=& 0   \\
 &&\text{(since the initial data for $A$ was chosen such that it solves the Yang-Mills constraints \eqref{theYangMillsconstraintforthepotentialandelectricchargeonsigma},}\\
 && \text{which here reads as the equation \eqref{equationtodefinepropgationofAtorespectLorenzgauge} on $\Sigma$).}
\eeaa
In summary, from \eqref{equationtodefinepropgationofAtorespectLorenzgauge} and Lemma  \ref{waveoperatorontheYangMillspoentialwithsourcesdependingonA}, we have 
\bea
  \Box_{\g} \La =       [  \der_{\b}  \La ,A^{\b}]  \, ,
\eea
with $\La_{\Sigma} = 0$ and $\pa_{t} \La_{\Sigma} = 0$. This is a wave equation in $\La$ with identically zero initial conditions, thus the unique global solution is $\La = 0$ for all $t$.
\end{proof}

\subsection{The propagation of the wave coordinates condition}\

We want to show that the Einstein-Yang-Mills system in the wave coordinate condition is consistent with the wave coordinates condition being imposed for all time, i.e. that the Einstein-Yang-Mills system in wave coordinates preserve the wave coordinates for all time $t$, if the initial data for the Einstein-Yang-Mills system satisfies some wave coordinate constraint.

\begin{definition}\label{definitionofquantityforthewavecoordcondition}
Let
\bea
\Ga^{\la}   := g^{\a\b} \Gamma_{\a\b}^{\, \, \, \, \, \, \, \, \,  \la}  \, .
\eea

\begin{remark}\label{remarkondefinitionofGa}
Note that $\Ga^{\la}$ is not a tensor. However, the difference between two Christoffel symbols coming out from two different metrics, is a tensor. Thus, we view the wave coordinate condition as the difference of the two traces
\bea
0 = g^{\a\b} \Gamma_{\a\b}^{\, \, \, \, \, \, \, \, \,  \la}- g^{\a\b} \hat{\Gamma}_{\a\b}^{\, \, \, \, \, \, \, \, \,  \la} (m) \, ,
\eea
where $\hat{\Gamma}_{\a\b}^{\, \, \, \, \, \, \, \, \,  \la} (m) $ are the Christoffel symbols of the metric $m$ defined to be the Minkowski metric in wave coordinates. Hence, the wave coordinate condition can be viewed as a tensorial geometric quantity equal to zero, which in wave coordinates can be written as
\bea
\Ga^{\la}  = 0 \, ,
\eea
since $\hat{\Gamma}_{\a\b}^{\, \, \, \, \, \, \, \, \,  \la} (m) = 0 $ in wave coordinates.

Hence, when we write $\Ga^{\la} = 0$\;, this means that we already fixed the system of coordinates to be the wave coordinates, however we view it as the tensor given in wave coordinates by
\bea\label{definitionofGainwavecoordinates}
\Ga^{\la} = g^{\a\b} \Gamma_{\a\b}^{\, \, \, \, \, \, \, \, \,  \la}- g^{\a\b} \hat{\Gamma}_{\a\b}^{\, \, \, \, \, \, \, \, \,  \la} (m)
\eea
and we differentiate it using the Levi-Civita covariant derivative associated to the metric $g$.
\end{remark}

\end{definition}

In turns out, based on the proof of Lemma 3.1 in \cite{LR2}, that for a solution of the system of non-linear wave equation on the metric, namely \eqref{Thewaveequationonthemetrichwithhyperbolicwaveoperatorusingingpartialderivativesinwavecoordinates}, the Einstein-Yang-Mills equations, \ref{simpEYM}, read
\bea\label{fromtheproofofLindbladRodnianski}
\notag
R_{\a\b} - 2 < F_{\a\si}, F_{\b}^{\;\;\si} > -   \frac{1}{(1-n)} \cdot g_{\a\b } \cdot < F_{\si\la },F^{\si\la } >  =  \frac 12 (\der_{\a} \Ga_{\b} + \der_{\b}\Ga_{\a}) + \Ga_\si N^\si_{\, \, \, \a\b}(g,\pa g) \; , \\
\eea
where $N^\si_{\, \, \, \a\b}(g,\pa g)$ is a non-linearity depending on $g$ and $\pa g$\;. In other words, in \eqref{Thewaveequationonthemetrichwithhyperbolicwaveoperatorusingingpartialderivativesinwavecoordinates}, the terms on the right hand side of \eqref{fromtheproofofLindbladRodnianski} are the ones we added using the wave coordinates condition, where $\Ga_{\b} = 0$ is identically zero (and therefore, all derivatives up to any order of $\Ga_{\b} $ are null).

Now, we shall use this equation, \eqref{fromtheproofofLindbladRodnianski}, to show that $\Ga$ satisfies a non-linear wave equation. 

\begin{lemma}\label{nonlinearwaveequationonGaforthepropagationofthewavecoordicondition}
For a solution $g = m + h$ of the non-linear wave equation given in \eqref{Thewaveequationonthemetrichwithhyperbolicwaveoperatorusingingpartialderivativesinwavecoordinates}, we have that the Einstein-Yang-Mills equations \ref{simpEYM} imply in wave coordinates the following equation on $\Ga_{\b}$ (defined in Definition \ref{definitionofquantityforthewavecoordcondition}, and considering Remark \ref{remarkondefinitionofGa} and in particular \ref{definitionofGainwavecoordinates}),
\bea\label{systemofnonlinearwaveeqonthewavecoordicondition}
\notag
 \Box_\g  \Ga_{\b}  &=&   - 2 \big( \der^\a \Ga_\si \big) N^\si_{\a\b}(g,\pa g)  +  \big( \der_{\b}  \Ga_\si  \big) N^{\si \, \, \, \,  \a }_{\, \, \, \a}(g,\pa g)  \\
 \notag
 &&   - R_{\mu  \b }     \Ga^{\mu} - 2 \Ga_\si \big( \der^\a N^\si_{\a\b}(g,\pa g)  \big)  + \Ga_\si \big(  \der_{\b} N^{\si \, \, \, \,  \a }_{\, \, \, \a}(g,\pa g) \big) \\
 &&  - 4    \der^\a< F_{\a\si}, F_{\b}^{\;\;\si} >   + \der_{\b}  < F_{\a\b },F^{\a\b }> \; ,
\eea
where $\der$ is the Levi-Civita covariant derivative associate to the metric $\g$\;. 
\end{lemma}

\begin{proof}
On one hand, since $\der$ is the Levi-Civita covariant derivative, we have the Bianchi identities for the Riemann tensor (see \eqref{fromthebianchiidentityfortheRiemanntensor}), that
\bea
 \der^{\a} R_{\a\b} - \frac{1}{2} \der_{\b} R   = 0 \; .
\eea

On the other hand, based on \eqref{fromtheproofofLindbladRodnianski} and by differentiating, we obtain
\beaa
 \der^\a  R_{\a\b} &=&     \frac 12  \der^\a (\der_{\a} \Ga_{\b} + \der_{\b}\Ga_{\a}) +\der^\a \big( \Ga_\si N^\si_{\a\b}(g,\pa g)  \big)  \\
 && +  \der^\a \big( 2 < F_{\a\si}, F_{\b}^{\;\;\si} >  +   \frac{1}{(1-n)} \cdot g_{\a\b } \cdot < F_{\si\la },F^{\si\la } > \big) \\
 &=&     \frac 12  \der^\a (\der_{\a} \Ga_{\b} + \der_{\b}\Ga_{\a}) +\der^\a \big( \Ga_\si N^\si_{\a\b}(g,\pa g)  \big)  \\
 && +    2 \der^\a< F_{\a\si}, F_{\b}^{\;\;\si} >  +   \frac{1}{(1-n)} \cdot  \der_\b \cdot < F_{\si\la },F^{\si\la } > \; .
\eeaa
We compute $R$ from \eqref{fromtheproofofLindbladRodnianski}, and we get
\beaa
R   = R_{\a}^{\, \,\, \, \a} &=& \frac 12 (\der_{\a} \Ga^{\a} + \der_{\a}\Ga^{\a}) + \Ga_\si N^{\si \, \, \, \,  \a }_{\, \, \, \a}(g,\pa g)  +  2 < F_{\a\b}, F_{\a\b}> +   \frac{(1+n)}{(1-n)}  \cdot < F_{\si\la },F^{\si\la } >\\
&=&  \der_{\a} \Ga^{\a} + \Ga_\si N^{\si \, \, \, \,  \a }_{\, \, \, \a}(g,\pa g)  +  \frac{(3-n)}{(1-n)} \cdot  < F_{\a\b },F^{\a\b }> \;  ,
\eeaa

and therefore 
\bea
\notag
\der_{\b}  R    &=& \der_{\b}  \der_{\a} \Ga^{\a} + \der_{\b}  \big( \Ga_\si N^{\si \, \, \, \,  \a }_{\, \, \, \a}(g,\pa g) \big)  +   \frac{(3-n)}{(1-n)} \cdot  \der_{\b}  < F_{\a\b },F^{\a\b }> \; .
\eea
Thus, injecting in the Bianchi identity, we obtain
\beaa
0 &=& \der^{\a} R_{\a\b} - \frac{1}{2} \der_{\b} R   \\
&=&    \frac 12  \der^\a (\der_{\a} \Ga_{\b} + \der_{\b}\Ga_{\a}) +\der^\a \big( \Ga_\si N^\si_{\a\b}(g,\pa g)  \big)  - \frac12 \big( \der_{\a} \Ga^{\a} + \Ga_\si N^{\si \, \, \, \,  \a }_{\, \, \, \a}(g,\pa g) \big) \\
&&  +    2 \der^\a< F_{\a\si}, F_{\b}^{\;\;\si} >  +   \frac{1}{(1-n)} \cdot  \der_\b \cdot < F_{\si\la },F^{\si\la } >   - \frac{(3-n)}{2\cdot (1-n)} \cdot  \der_{\b}  < F_{\a\b },F^{\a\b }> \\
&=&    \frac 12  \der^\a (\der_{\a} \Ga_{\b} + \der_{\b}\Ga_{\a}) +\der^\a \big( \Ga_\si N^\si_{\a\b}(g,\pa g)  \big)  - \frac12 \big( \der_{\a} \Ga^{\a} + \Ga_\si N^{\si \, \, \, \,  \a }_{\, \, \, \a}(g,\pa g) \big) \\
&&  +    2 \der^\a< F_{\a\si}, F_{\b}^{\;\;\si} >    - \frac{1}{2} \cdot  \der_{\b}  < F_{\a\b },F^{\a\b }> \; .
\eeaa
Consequently, we have that $\Ga_\la$ satisfies the following equation
\beaa
0   &=&    \frac 12  \der^\a \der_{\a} \Ga_{\b} +  \frac 12 \der^\a  \der_{\b}\Ga_{\a}  +\der^\a \big( \Ga_\si N^\si_{\a\b}(g,\pa g)  \big)  - \frac12 \big( \der_{\b}  \der_{\a} \Ga^{\a} + \der_{\b}  \big( \Ga_\si N^{\si \, \, \, \,  \a }_{\, \, \, \a}(g,\pa g) \big)  \big) \\
&&  +    2 \der^\a< F_{\a\si}, F_{\b}^{\;\;\si} >    - \frac{1}{2} \cdot  \der_{\b}  < F_{\a\b },F^{\a\b }> \\
&=&   \frac 12  \Box_\g  \Ga_{\b} +  \frac 12 \big( \der_\a  \der_{\b}\Ga^{\a} - \frac12  \der_{\b}  \der_{\a} \Ga^{\a} \big) +  \big( \der^\a \Ga_\si \big) N^\si_{\a\b}(g,\pa g) +\Ga_\si \big( \der^\a N^\si_{\a\b}(g,\pa g)  \big) \\
&& -\frac12 \big( \der_{\b}  \Ga_\si  \big) N^{\si \, \, \, \,  \a }_{\, \, \, \a}(g,\pa g)  - \frac12 \Ga_\si \big(  \der_{\b} N^{\si \, \, \, \,  \a }_{\, \, \, \a}(g,\pa g) \big) \\
&&  +    2 \der^\a< F_{\a\si}, F_{\b}^{\;\;\si} >    - \frac{1}{2} \cdot  \der_{\b}  < F_{\a\b },F^{\a\b }> \; .
\eeaa
Hence, we get
\bea
\notag
 \Box_\g  \Ga_{\b}   &=&   - R_{\,\,\, \mu \a \b }^{\a  }     \Ga^{\mu} - 2 \big( \der^\a \Ga_\si \big) N^\si_{\a\b}(g,\pa g) - 2 \Ga_\si \big( \der^\a N^\si_{\a\b}(g,\pa g)  \big) \\
 \notag
&& +  \big( \der_{\b}  \Ga_\si  \big) N^{\si \, \, \, \,  \a }_{\, \, \, \a}(g,\pa g)  + \Ga_\si \big(  \der_{\b} N^{\si \, \, \, \,  \a }_{\, \, \, \a}(g,\pa g) \big) \\
&&  - 4    \der^\a< F_{\a\si}, F_{\b}^{\;\;\si} >   + \der_{\b}  < F_{\a\b },F^{\a\b }> \;.
\eea
which gives the stated result.

\end{proof}

Thus, the source terms for $\Box_\g  \Ga_{\b} $ depend only on $\g$ and $ \Ga_{\b} $ and their first derivatives only, as well as on $F$ and the first derivatives of $F$\;. Consequently, $\Ga_{\b} $ satisfies a non-linear wave equation. The initial data set for both $A$ and $\g$ on $\Sigma$ was constructed in Subsections \ref{IntialdataforYangMills} and \ref{constructioninitialdataformetric}, so that $\Ga_{\b} = 0$ on $\Sigma$\;. Now, we would like to see if the wave coordinates condition propagates in time, so as to have $\Ga_{\b} = 0 $ on $\Sigma_t$ for all time $t$\;.

In fact, since we already have $\Ga_{\b} = 0 $ on $\Sigma$\;, if in addition, we would have $\der_t \Ga_{\b} = 0 $ on $\Sigma$\;, then, we would have a non-linear wave equation for $\Ga_{\b} $ (that we exhibited  in Lemma \ref{nonlinearwaveequationonGaforthepropagationofthewavecoordicondition}) with identically null initial data, and thus this would prove that $\Ga_{\b} = 0$ is identically zero in the time evolution (and therefore, all derivatives up to any order of $\Ga_{\b} $ are null) and, consequently, $\Ga_{\b} = 0$ for $t \geq 0$\;. Hence, the condition $\der_t \Ga_{\b} = 0 $ if true on the initial slice $\Sigma$\;, it would guarantee that the wave coordinate gauge propagates in time $t$\;.

However, we did already construct the initial data $(g_\Sigma, \pa_t g_\Sigma)$ in wave coordinates, i.e. such that $\Ga^{\la} = 0$ on $\Sigma$\;. Yet, we want to see if the whole initial data set $(\Sigma, A_\Sigma, \pa_t A_\Sigma, g_\Sigma, \pa_t g_\Sigma)$ that we already constructed from our original initial data set solution to the Einstein-Yang-Mills constraint equations, namely $(\Sigma, \overline{A}, \overline{E}, \overline{g}, \overline{k})$\;, would give this additional wave coordinate condition constraint, namely $\der_t \Ga_{\b} = 0 $ on $\Sigma$\;.

\begin{lemma}\label{initialconditionsonthederivativesofGa}
For a solution of the non-linear wave equations on the metric given in \eqref{Thewaveequationonthemetrichwithhyperbolicwaveoperatorusingingpartialderivativesinwavecoordinates}, with initial data solving the Einstein-Yang-Mills constraint equations \eqref{theEinsteinYangMillsconstriantsonSigmafirst} and \eqref{theEinsteinYangMillsconstriantsonSigmasecond}, and constructed as in Subsections \ref{IntialdataforYangMills} and \ref{constructioninitialdataformetric}, we have in wave coordinates on $\Sigma$\,,
\bea
\notag
  \der_{t} \Ga_{t} &=& 0   \; ,\\
\eea
and
\bea
\der_{t} \Ga_{\i} &=&  0  \; .
\eea

\end{lemma}

\begin{proof}

We first recall that the Einstein-Yang-Mills system, \eqref{EYMsystemofequationsongandA}, implies the following equation
\beaa
R_{ \mu \nu}  - 2  < F_{\mu\b}, F_{\nu}^{\;\;\b}> - g_{\mu\nu } \cdot \frac{1}{(1-n)} < F_{\a\b },F^{\a\b }  > &=& 0  \, .
\eeaa
Thanks to \eqref{fromtheproofofLindbladRodnianski}, we have for a solution of \eqref{Thewaveequationonthemetrichwithhyperbolicwaveoperatorusingingpartialderivativesinwavecoordinates}, that in wave coordinates,
\beaa
\notag
 \frac 12 (\der_{\mu} \Ga_{\nu} + \der_{\nu}\Ga_{\mu}) + \Ga_\si N^\si_{\mu\nu}(g,\pa g) &=& R_{ \mu \nu}  - 2  < F_{\mu\b}, F_{\nu}^{\;\;\b}> - g_{\mu\nu } \cdot \frac{1}{(1-n)} < F_{\a\b },F^{\a\b }  >  , \\
 \eeaa
 Since on $\Sigma$\;, we constructed our initial data $(\Sigma, A_\Sigma, \pa_t A_\Sigma, g_\Sigma, \pa_t g_\Sigma)$\, in way such that $\Ga_\si =0$\;, then the initial data constructed in Subsection \ref{constructioninitialdataformetric}, gives that we have on $\Sigma$\;, 
 \bea
\notag
 \frac{1}{2} (\der_{\mu} \Ga_{\nu} + \der_{\nu}\Ga_{\mu}) &=& R_{ \mu \nu}  - 2  < F_{\mu\b}, F_{\nu}^{\;\;\b}> - g_{\mu\nu } \cdot \frac{1}{(1-n)} < F_{\a\b },F^{\a\b }  >  \;. \\
\eea

Since we choose the initial data for our system in \eqref{Thewaveequationonthemetrichwithhyperbolicwaveoperatorusingingpartialderivativesinwavecoordinates} such that it solves the Einstein-Yang-Mills equations, namely by choosing the initial data as a solution to the Einstein-Yang-Mills constraint equations given in Lemma \ref{TheconstraintequationsfortheEinstein-Yang-Mills system} -- in other words, we have the initial data set that we constructed in Subsections \ref{IntialdataforYangMills} and \ref{constructioninitialdataformetric}, taken solutions to the Einstein-Yang-Mills constraints derived in Subsection \ref{derivationofEinsteinYangMillsconstraintequationsforthecoupledsystem}
--, therefore, we have on $\Sigma$\;, 
\beaa
R_{ \mu \nu}  - 2  < F_{\mu\b}, F_{\nu}^{\;\;\b}> - g_{\mu\nu } \cdot \frac{1}{(1-n)} < F_{\a\b },F^{\a\b }  > = 0
\eeaa
and thus, by injecting, we have on $\Sigma$
 \bea
\notag
 \frac{1}{2} (\der_{\mu} \Ga_{\nu} + \der_{\nu}\Ga_{\mu}) &=& 0   \;. \\
\eea

Hence, we get

\beaa
\notag
 \frac{1}{2}   (\der_{\hat{t}} \Ga_{\hat{t}} + \der_{\hat{t}}\Ga_{\hat{t}})  &=&  0 \; .
 \eeaa
Thus, solutions to the non-linear wave equation on the metric, given in \eqref{Thewaveequationonthemetrichwithhyperbolicwaveoperatorusingingpartialderivativesinwavecoordinates}, with initial data solving the Einstein-Yang-Mills constraint equations and the wave coordinates condition, give that
\beaa
\notag
 \der_{\hat{t}} \Ga_{\hat{t}}  &=& 0 \; .\\
\eeaa
Yet, on $\Sigma$\;, we have $\hat{t}=  \frac{1}{N} \frac{\pa}{ \pa t}$\;, thus,
\bea
 \der_{t} \Ga_{t}  &=& 0 \; .
\eea

Also, we have for spatial indices $i$ in wave coordinates, the following on $\Sigma$\;,
\beaa
\notag
 \frac 12 (\der_{\hat{t}} \Ga_{\i} + \der_{i}\Ga_{\hat{t}})  &=& 0 \; .
\eeaa

However, since by construction $\Ga_{\la} = 0$ on $\Sigma$\;, we have $$\der_{i}\Ga_{\hat{t}} = 0 \;, $$ and hence, we get on $\Sigma$\;,
\bea
\notag
\der_{t} \Ga_{\i} &=& 0 \; .
\eea

\end{proof}

\subsection{Construction of the initial data for the hyperbolic system given an initial data set that solves the Einstein-Yang-Mills constraints}\

Finally, we have proven, in this Section \ref{Construction of the initial data and the gauges conditions constraints}, the following corollary.

\begin{corollary}\label{thefinalconstructiongivenconstriantsontheinitialdatatogaranteeLorenzandharmonicgauges}
Assume that we are given an initial data set $(\Sigma, \overline{A}, \overline{E}, \overline{g}, \overline{k})$ that satisfies the Einstein-Yang-Mills constraint equations given in Lemma \ref{TheconstraintequationsfortheEinstein-Yang-Mills system}, which are 
\bea
\notag
  \mathcal{R}+ \overline{k}^i_{\,\, \, i} \overline{k}_{j}^{\,\,\,j}  -  \overline{k}^{ij} \overline{k}_{ij}   &=&    \frac{4}{(n-1)}   < \overline{E}_{i}, \overline{E}^{ i}>   \\
 \notag
 && +  < \overline{D}_{i}  \overline{A}_{j} - \overline{D}_{j} \overline{A}_{i} + [ \overline{A}_{i},  \overline{A}_{j}] ,\overline{D}^{i}  \overline{A}^{j} - \overline{D}^{j} \overline{A}^{i} + [ \overline{A}^{i},  \overline{A}^{j}] > \, ,\\
 \notag
\overline{D}_{i} \overline{k}^i_{\,\,\, j}    - \overline{D}_{j} \overline{k}^i_{\, \,\,i}  &=&  2 < \overline{E}_{i}, {\overline{F}_{j}}^{\, i}> \, ,\\
\notag
\overline{D}^i \overline{E}_{ i} + [\overline{A}^i, \overline{E}_{ i} ]  &=& 0 \, ,
\eea
where $\overline{D} $ is the Levi-Civita covariant derivative associated to the given Riemannian metric $\overline{g}$, and where the summation is carried only over spatial indices, and we raise indices with respect to $\overline{g}$\;.

Then, we can construct an initial data set $(\Sigma, A_\Sigma, \pa_t A_\Sigma, g_\Sigma, \pa_t g_\Sigma)$, as prescribed in 
\ref{initialdatadforzerothderivativeAsigma}, \ref{constructionopatialtimeAonsigma}, \ref{constructionoggsigma} and \ref{constructionopatialtgonsigma}, for the coupled system of non-linear hyperbolic equations given in Lemma \ref{EYMsystemashyperbolicPDE}, such that solving that system gives rise to a solution of the Einstein-Yang-Mills system that is in the Lorenz gauge and in wave coordinates for all time $t$\;.
 
\end{corollary}

\section{Set-up for the proof}

\subsection{The rotations and the Lorentz boosts}\ \label{TheMinkwoskivectorfieldsdefinition}

At a point $p$ in the space-time, let $x^{\mu}$  be the wave coordinates, with $ x^{0} = t$, and let
\beaa
x_{\b} = m_{\mu\b} x^{\mu} \, ,
\eeaa
where we raised and lowered indices with respect to the Minkowski metric $m$, defined in wave coordinates to be the Minkowski metric, i.e. in wave coordinates, we have $m_{tt} = -1$, $m_{ii} = 1$, $m_{ti} = 0$ and $m_{ij} = \delta_{ij} $. Here $i$ and $j$ denote always spatial indices.
The Lorentz boosts and rotations are
\beaa
Z_{\a\b} = x_{\b} \pa_{\a} - x_{\a} \pa_{\b}  \, ,
\eeaa
and they form a representation of the Lie algebra of the Lorentz group. Here, what we call Lorentz boosts are $L_{ti}$ and the rotations are $L_{ij}$.
We also define the well-known space-time dilation vector field, or the scaling vector field, as
\beaa
S = t \pa_t + \sum_{i=1}^{n} x^i \pa_{i} \, .
\eeaa
The Lorentz boosts and rotations along with the scaling vector field $S$ and the time and space translations $\pa_t$ and $\pa_{x_i}$, form a representation of the Lie algebra of the Poincare group, which is the group of isometries of the Minkowski space-time, which we will call the Minkowski vector fields and will be denoted by ${\cal Z}$. Vector fields belonging to Minkowski vector fields will be denoted by $Z$, i.e.
\bea
Z \in {\cal Z}  := \big\{ Z_{\a\b}\;,\; S\;,\; \pa_{\a} \, \,  | \, \,   \a\;,\; \b \in \{ 0, \ldots, n \} \big\} \quad \text{with} \quad x^0=t = - x_0 .
\eea

Note that the family ${\cal Z}$ has $\frac{(n^2 + 3n + 4)}{2}$ vector fields: $\frac{(n+1) \cdot n}{2}$ vectors for the Lorentz boosts and rotations, $(n+1)$ space-time translations and one scaling vector field. One can order them and assign to each vector an $\frac{(n^2 + 3n + 4)}{2}$-dimensional integer index $(0, \ldots, 1, \ldots, 0)$. Hence, a collection of $k$ vector fields from the family ${\cal Z}$, can be described by the set $I=(\iota_1, \ldots,\iota_k)$, where each $\iota_i$ is an $\frac{(n^2 + 3n + 4)}{2}$-dimensional integer, where $|I|=k = \sum_{i=1}^{k} | \iota_i |$, with $|\iota_i|=1$. Thus, we make the following definition:

\begin{definition} \label{DefinitionofZI}
We define
\bea
Z^I :=Z^{\iota_1} \ldots Z^{\iota_k} \quad \text{for} \quad I=(\iota_1, \ldots,\iota_k),  
\eea
where $\iota_i$ is an $\frac{(n^2 + 3n + 4)}{2}$-dimensional integer index, with $|\iota_i|=1$, and $Z^{\iota_i}$ representing each a vector field from the family ${\cal Z}$.

For a tensor $T$, of arbitrary order, either a scalar or valued in the Lie algebra, we define the Lie derivative as
\bea
\Lie_{Z^I} T :=\Lie_{Z^{\iota_1}} \ldots \Lie_{Z^{\iota_k}} T \quad \text{for} \quad I=(\iota_1, \ldots,\iota_k) .
\eea

In addition, when we write $I = I_1+I_2$, it means that we divided the set $I$ into two sets $I_1$ and in $I_2$, while preserving the order of $I$ in $I_1$ and in $I_2$, i.e., if $I=(\iota_1, \ldots, \iota_k)$, then $I_1=(\iota_{i_1}, \ldots,\iota_{i_n})$ and $I_2=(\iota_{i_{n+1}}, \ldots,\iota_{i_k})$, where $i_1< \ldots<i_n$ and $i_{n+1}< \ldots<i_k$. By a sum $\sum_{I_1+I_2=I}$, we mean that we make the sum over all such partitions for a given $I$. With this convention, the Leibniz rule holds and reads for sufficiently smooth functions $f$ and $g$, 
\bea
Z^I(fg)=\sum_{I_1+I_2=I} (Z^{I_1} f)(Z^{I_2} g) \;.
\eea

\end{definition}

\subsection{Weighted Klainerman-Sobolev inequality}\

We now state a weighted Klainerman-Sobolev inequality, see for example \cite{BFJST1} or \cite{Hu1} for a proof. It is a weighted version of the standard Klainerman-Sobolev inequality. We define
\bea
q := r - t \; ,
\eea
which is a null coordinate for the Minkowski metric in wave coordinates. The weight is defined by the following

\begin{definition}\label{defw}
We define $w(q)$ by
\bea
w(q):=\begin{cases} (1+|q|)^{1+2\gamma} \quad\text{when }\quad q>0 \;, \\
         1 \,\quad\text{when }\quad q<0 \; , \end{cases}
\eea
for some $ \gamma > 0$\;. Note that we put the notational factor of $2$ infront of $\gamma$ since we are going to compute $\big[ (1+|q|)w(q)\big]^{1/2}$\,, and this way, for that expression, this notational factor disappears (see \eqref{whynotationalfactorof2infrontofgammaindefofw}).

\end{definition}

Then, we have globally the following pointwise estimate for any smooth scalar function $\phi$ vanishing at spatial infinity, i.e. $\lim_{r \to \infty} \phi (t, x^1, \ldots, x^n) = 0$,
\bea\label{wksi}
\notag
|\phi(t,x)| \cdot (1+t+|q|)^{\frac{(n-1)}{2}} \cdot \big[ (1+|q|) \cdot w(q)\big]^{1/2} \leq
C\sum_{|I|\leq  \lfloor  \frac{n}{2} \rfloor  +1 } \|\big(w(q)\big)^{1/2} Z^I \phi(t,\cdot)\|_{L^2} \; , \\
\eea
where here the $L^2$ norm is taken on $t= \text{constant}$ slice.

\subsection{Definition of the norms}\label{subsectiondefinitionofthenorms} \

We recall that we defined $m$ to be the Minkowski metric in wave coordinates $\{t, x^1,  \ldots, x^n \}$, such that
\beaa
m_{00} = m(\frac{\pa}{\pa t}, \frac{\pa}{\pa t}) = - 1 \; ,
\eeaa
and for the spatial coordinates  $ \{ \frac{\pa}{\pa x^1},  \ldots,  \frac{\pa}{\pa x^n } \} $ tangent to $\Sigma_{t}$ prescribed by $t = constant$ hypersurfaces, we have
\beaa
m_{ij} = m(\frac{\pa}{\pa x^i}, \frac{\pa}{\pa x^j}) = \delta_{ij} \, ,
\eeaa
where $\delta_{ij}$ is the Kronecker symbol, and
\beaa
m_{i0}= m(\frac{\pa}{\pa x^i}, \frac{\pa}{\pa t}) = 0 \, .
\eeaa
Denoting $t = x^0 = - x_0$\;, we define for all $\mu, \nu \in \{0, 1, \ldots, n \}$, the following euclidian metric in wave coordinates
\bea
\notag
E_{\mu\nu} = E (\frac{\pa}{\pa x^\mu}, \frac{\pa}{\pa x^\nu}) = m(\frac{\pa}{\pa x^\mu}, \frac{\pa}{\pa x^\nu})+2 m(\frac{\pa}{\pa x^\mu}, \frac{\pa}{\pa t}) \cdot m(\frac{\pa}{\pa x^\nu}, \frac{\pa}{\pa t}) \, .
\eea
Then, we define for a scalar tensor $Q_{\a}$,
\bea
\notag
 | Q  |_{\text{scal}}^2 &:=& E^{\mu\nu}       Q_{\mu} \cdot    Q_{\nu}    \\
&=&E_{\mu\nu}   Q^{\mu}  \cdot    Q^{\nu}   \, ,
\eea
where we took here the scalar product, and where one lowers and highers indices with respect to the metric $E$ and where $E^{\a\b}$ is the inverse matrix of $E_{\a\b}$.

Similarly, for a tensor $K_\a$ valued in the Lie algebra associated to the Lie group $G$\;, we define
\bea
 | K  |^2_{\cal G} := E^{\mu\nu}  < K_{\mu} ,   K_{\nu} >  \, ,
\eea
where here $< \, , \, >$ is the Ad-invariant norm on the Lie algebra. 

Similarly, we define the norms for tensors of arbitrarily order by taking a full contraction with respect to the euclidian metric $E$ of the scalar product of a scalar tensor, or of the scalar product on the Lie algebra of a ${\cal G}$-valued tensor.

To lighten the notation, we will use the same notation for both the scalar product for scalar components or for ${\cal G}$-valued components. Also, we will drop the indices $\mid . \mid_{\text{scal}}$ and $\mid . \mid_{\cal G}$ and use $\mid . \mid$ for norms on tensors. 

Using this notation, and viewing the gradient of a sufficiently smooth scalar function $f$ as the tensor $\pa_a f$\;, we have
\bea
 | \pa f |^2 :=  E^{\mu\nu}    \pa_{\mu} f\ \cdot    \pa_{\nu} f  \; , 
 \eea
which is a definition that generalises for taking instead of $f$, a tensor of arbitrarily order, either a scalar tensor or valued in the Lie algebra $\cal G$\;, by replacing the partial derivatives with a covariant derivative with respect to the Minkowski metric. We will be more precise in the definitions in what follows.
 
  \begin{definition}\label{definitionoftheMinkowskicovaruiantderivative}
 Now, defining the connection $\der^{(\bf{m})}$ to be the flat connection in the wave coordinates such that its Christoffel symbols are vanishing in wave coordinates, i.e. such that for all $\mu, \nu \in \{0, 1, \ldots, n \}$\;,
 \bea
{ \der^{(\bf{m})}}_{e_{\mu}} e_\nu := 0 \; ,
 \eea
 where $e_\mu = \frac{\pa}{\pa x_\mu} $ and where $\{x^0, x^1, \ldots, x^n \}$ are the wave coordinates.
\end{definition}

We then define for a scalar tensor $Q_{\a}$
\bea
 | \pa Q |^2 :=  E^{\a\b} E^{\mu\nu}    { \der^{(\bf{m})}}_{\mu} Q_\a \ \cdot    { \der^{(\bf{m})}}_{\nu} Q_\b , 
 \eea
 and similarly for scalar tensors of arbitrarily order. We define for a tensor $K_\a$ valued in the Lie algebra
 \bea
|  \pa K | ^2  := E^{\a\b}  E^{\mu\nu}  <  {\der^{(\bf{m})}}_{\mu} K_{\a} ,    {\der^{(\bf{m})}}_{\nu} K_{\b} > \, ,
\eea
and similarly for scalar tensors of arbitrary order. Note that by contracting in wave coordinates, we get
 \bea
 \notag
| \pa K |^2  &=&   |  {\der^{(\bf{m})}}_{t} K |^2  +  |  {\der^{(\bf{m})}}_{x^1} K |^2 +\ldots  + |  {\der^{(\bf{m})}}_{x^n} K |^2 \\
&=&  \sum_{\a,\; \b \in  \{t, x^1, \ldots, x^n \}} |  \pa_{\a} K_\b |^2   \, ,
\eea
since in wave coordinates, the Minkowski covariant derivative of the contraction of a tensor expressed in wave coordinates, is in fact a partial derivative. We shall also write
\bea\label{equationofthedefinitionofthenormofgradienteitheraspartialorcovariant}
| \derm K | :=  | \pa K | \; .
\eea

\begin{lemma}
At a point $p$ of the space-time, let $x^{\mu}$  be the wave coordinate system. For a sufficiently smooth function $f$ and for a norm $\mid .\mid $\,, we define for all $I$ and $Z^I$ as previously defined, the following norm in the wave coordinates system $\{t, x^1, \ldots, x^n \}$\,,
\bea
| Z^I \pa f | := \sqrt{ | Z^I \pa_{t} f  |^2 + \sum_{i=1}^{n} | Z^I  \pa_{i} f  |^2 } \;  .
\eea
Then, we have,
\bea
|  Z^I \pa  f |  \leq C (|I| ) \cdot    \sum_{|J| \leq |I| }  | \pa  (  Z^J  f ) | \; ,
\eea
where $C (|I| )$ is a constant that depends only on $|I|$\,.
\end{lemma}

\begin{proof}
Recall that
\beaa
x_{\b} = m_{\mu\b} x^{\mu} \; ,
\eeaa
and
\beaa
Z_{\a\b} = x_{\b} \pa_{\a} - x_{\a} \pa_{\b} \; .
\eeaa

Computing for a sufficiently smooth function $f$\; , 
\beaa
[\pa_{\mu}, Z_{\a\b} ]  f &=& \pa_{\mu} (Z_{\a\b}f )  - Z_{\a\b} (\pa_{\mu} f ) = \pa_{\mu}(  x_{\b} \pa_{\a} f - x_{\a} \pa_{\b} f ) - (  x_{\b} \pa_{\a}  - x_{\a} \pa_{\b} ) \pa_{\mu} f  \\
&=& \pa_{\mu}(  x_{\b}  ) \pa_{\a} f +   x_{\b}   \pa_{\mu}  \pa_{\a} f    -  \pa_{\mu}( x_{\a} ) \pa_{\b} f  -   x_{\a} \pa_{\mu}  \pa_{\b} f -  x_{\b} \pa_{\a}\pa_{\mu} f + x_{\a} \pa_{\b} \pa_{\mu}  f \\
&=& \pa_{\mu}(  x_{\b}  ) \pa_{\a} f    -  \pa_{\mu}( x_{\a} ) \pa_{\b}  f \;  . 
\eeaa
However, we have,
\beaa
\pa_{\mu}(  x_{\b}  )&=& \pa_{\mu}(  m_{\si\b} x^{\si} ) =  \pa_{\mu}(  m_{\si\b} ) x^{\si} +   m_{\si\b} \pa_{\mu}( x^{\si} ) =  m_{\si\b} \de_{\mu}^{\si} \\
&=& m_{\mu\b} \; .
\eeaa
Hence,
\beaa
[\pa_{\mu}, Z_{\a\b} ] f &=&  \pa_{\mu} (Z_{\a\b}f )  - Z_{\a\b} (\pa_{\mu} f )   = m_{\mu\b}   \pa_{\a} f   - m_{\mu\a}   \pa_{\b} f \;   . 
\eeaa

Thus,

\bea
[\pa_{\mu}, Z_{\a\b} ] = \begin{cases}   m_{\mu\b}   \pa_{\a} f   - m_{\mu\a}   \pa_{\b} f   = 0   \quad\text{ if}\quad (\a = \b) , \quad\text{ or if} \quad  ( \mu \neq \a \quad\text{and}\quad \mu\neq \b ),\\
 m_{\mu\b}   \pa_{\a} f   - m_{\mu\a}   \pa_{\b} f  =  m_{\b\b}   \pa_{\a}   \quad\text{if}\quad \mu = \b   \quad\text{and}\quad \a \neq \b , \\
 m_{\mu\b}   \pa_{\a} f   - m_{\mu\a}   \pa_{\b} f    =   - m_{\a\a}   \pa_{\b}    \quad\text{if}\quad \mu = \a   \quad\text{and}\quad \a \neq \b\;  .\\
 \end{cases} 
\eea

Thus, we conclude that for a sufficiently smooth function $f$, we have

\bea\label{commuting relation Z and f}
Z_{\a\b} \pa_{\mu} f  = \begin{cases}    \pa_{\mu} ( Z_{\a\b}  f )  \quad\text{ if}\quad (\a = \b ) \quad\text{ or}\quad (\mu \neq \a \quad\text{ and} \quad \mu \neq \b ) ,\\
 \pa_{\mu} (  Z_{\a\b}  f ) - m_{\b\b}   \pa_{\a} f   \quad\text{if}\quad \mu = \b   \quad\text{and}\quad \a \neq \b , \\
  \pa_{\mu} ( Z_{\a\b}  f )   + m_{\a\a}   \pa_{\b}  f  \quad\text{if}\quad \mu = \a   \quad\text{and}\quad \a \neq \b \; . \\
 \end{cases} 
 \eea

Thus,
\bea
\mid Z_{\a\b} \pa_{\mu} f  \mid  &\leq& \mid \pa_{\mu} ( Z_{\a\b}   f ) \mid+ \mid  \pa_{\a} f \mid  + \mid  \pa_{\b} f \mid  \, .
\eea

Moreover, considering a commutation with the scaling vector field
\beaa
S = t \pa_t + \sum_{i=1}^{3} x^i \pa_{i} = \sum_{\a=0}^{3} x^\a \pa_{\a} \,  ,
\eeaa
we get
\beaa
[\pa_{\mu}, S]  f &=& \pa_{\mu} (S f )  - S (\pa_{\mu} f ) = \pa_{\mu}(  x^{\a} \pa_{\a} f  ) -   x^{\a} \pa_{\a}  ( \pa_{\mu} f ) \\
&=& \pa_{\mu}(  x^{\a}  ) \pa_{\a} f +   x^{\a}   \pa_{\mu}  \pa_{\a} f     -  x^{\a} \pa_{\a}\pa_{\mu} f \\
&=& \pa_{\mu}(  x^{\a}  ) \pa_{\a} f    \\
&=& \de_{\mu}^{\a}   \pa_{\a} f  \\
&=&   \pa_{\mu} f  \, .
\eeaa
Consequently,
\bea
| S \pa_{\mu} f |  &\leq&  |  \pa_{\mu} ( S   f )  |  +  |  \pa_{\mu} f  |   .
\eea
Now, in wave coordinates, $| \pa f |$ is defined to be 
\bea
| \pa f | := \sqrt{ | \pa_{t} f  |^2 + \sum_{i=1}^{n} | \pa_{i} f  |^2  }  \, .
\eea
Thus, in wave coordinates $x^\mu$, we have
\beaa
| \pa f |  \geq | \pa_{\mu} f | \, ,
\eeaa
and as a result, we have shown that for all $Z \in {\cal Z}$, we have in wave coordinates, the following estimate
 \bea
 |  Z \pa_{\mu} f   |   &\leq&  |  \pa ( Z   f )  |  +  |   \pa  f  |  .
\eea

 Let $Z^{\iota_j}$, $j \in \{1, 2, ..., k\}$ be a family of vector fields from the family ${\cal Z}$. By induction, we get that there exists a constant $C_1 (k)$, depending on $k$, such that
\beaa
 |   \prod_{j=1}^{k} Z^{\iota_j} \pa_{\mu}  f |   \leq C_1 (k) \cdot   \sum_{i=0}^{k}   |  \pa  ( \prod_{j=0}^{i} Z^{\iota_j}  f ) |  ,
\eeaa
which leads to
\bea
 |   Z^I \pa_{\mu}  f |   \leq C_2 (|I| ) \cdot    \sum_{|J| \leq |I| }   |  \pa  (  Z^J  f )   |  .
\eea

\end{proof}

\subsection{The energy norm}\

We recall that we are given an initial data set which we write as $(\Sigma, \overline{A}, \overline{E}, \overline{g}, \overline{k})$\;, and that $\Sigma$ is diffeomorphic to $\R^n$\;, and therefore there exists a global system of coordinates $(x^1, ..., x^n) \in \R^n$ for $\Sigma$\;, and we define
\bea
r := \sqrt{ (x^1)^2 + ...+(x^n)^2  }\;.
\eea
We assume that the initial data set is smooth and asymptotically flat. Now, this initial data set looks differently depending on the space-dimension $n$\;. Let us explain: if we define $M(n)$\;, to be the mass, defined by
 \bea\label{defM}
M(n)  := \begin{cases} M > 0  \quad\text{for }\quad n = 3 \; ,\\
0 \quad\text{for }\quad n \geq 4 \; ,\end{cases} 
\eea
and if we define a smooth function $\chi$\;, given by
 \bea\label{defXicutofffunction}
\chi (r)  := \begin{cases} 1  \quad\text{for }\quad r \geq \frac{3}{4} \;  ,\\
0 \quad\text{for }\quad r \leq \frac{1}{2} \;, \end{cases} 
\eea
and if we let $\de_{ij}$ be the Kronecker symbol, and if we define $\overline{h}^1_{ij} $ in this system of coordinates $x^i$\;, by
 \bea
\overline{h}^1_{ij} := \overline{g}_{ij} - (1 + \chi (r)\cdot  \frac{ M(n)}{r}  ) \de_{ij} \; ,
\eea
then, the initial data can be written as
\bea
\overline{g}_{ij} = \overline{h}^1_{ij} +\de_{ij}  +  \chi (r)\cdot  \frac{ M(n)}{r} \de_{ij}   \; .
\eea
We can then define
\bea
\overline{h}_{ij} = \overline{h}^1_{ij}  + +  \chi (r)\cdot  \frac{ M(n)}{r} \de_{ij} \;. 
\eea
and this way, we can write
\bea
\overline{g}_{ij} = \overline{h}_{ij} +   +\de_{ij} \;.
\eea
However, in the case $n \geq 4$\,, the mass $M(n)=0$\;, and thus, on the initial slice $\Sigma$\;, we have  $\overline{h} =  \overline{h}^1$\;. Hence, in higher dimensions, we look for a solution in the following form in wave coordinates,

 \bea
g_{\mu\nu}= h_{\mu\nu} +m_{\mu\nu}   \; ,
\eea
where $h$ is the propagation of $\overline{h}^1 = \overline{h}$ (n higher dimension).
We want to define the energy as a quantity in a form that could dominate the right-hand side of the weighted Klainerman-Sobolev inequality for the functions $\pa \Lie_{Z^I}    h^1_{\mu\nu} $ and $ \pa  \Lie_{Z^I}  A_{\mu}$\;,\, $\mu, \nu \in (t, x^1, \ldots, x^n)$\;, instead of $\phi$\;, and for higher order $N$\,. We note that such a definition for the energy would not give a finite energy for $\chi (r)\cdot  \frac{2\cdot M(n)}{r}$\;, that is the part that carries the mass $M(n)$\;. Thus, in $n=3$\;, we need to write instead $h= h^1 + h^0$\;, where $h^0$ represents the propagation of the mass $M$  (we shall explain this more in the third paper that follows).

Yet, keeping the discussion above in mind and the fact that we aim to study the case of $n=3$ in a paper that follows, we shall often write the equations on $h^1$ (instead of $h$), with
\bea
h^1 = h - h^0 \;, 
\eea
where $h^0$ is vanishing in higher dimensions, as in this paper.

In fact, we define the higher order energy norm as the following $L^2$ norms on $A$ and $h^1$, using the scalar products either on the Lie algebra $\cal G$ or the usual scalar product, and we set
 \bea\label{definitionoftheenergynorm}
\E_{N} (t) :=  \sum_{|I|\leq N} \big(   \|w^{1/2}   \pa ( \Lie_{Z^I}  A   (t,\cdot) )  \|_{L^2} +  \|w^{1/2}   \pa ( \Lie_{Z^I} h^1   (t,\cdot) )  \|_{L^2} \big) \, ,
\eea
where the integration is taken with respect to the Lebesgue measure $dx_1 \ldots dx_n$.

Here, we have in wave coordinates $(t, x^1, \ldots, x^n)$\;,
 \bea
 \notag
\mid \pa (  \Lie_{Z^I}  A ) \mid^2 :=  \mid \pa   \Lie_{Z^I}   A_t \mid^2 +   \mid \pa    \Lie_{Z^I}  A_{1} \mid^2  +   \ldots +   \mid  \pa  \Lie_{Z^I}  A_n \mid^2 \, , \\
\eea
where for $\mu \in (t, x^1, \ldots, x^n)$,
 \bea
  \notag
\mid \pa  \Lie_{Z^I}  A_\mu \mid^2 :=    \mid \pa_t   \Lie_{Z^I}   A_\mu \mid^2 + \mid \pa_{x^1} \Lie_{Z^I}   A_\mu \mid^2 + \ldots  +\mid \pa_{x^n}  \Lie_{Z^I}  A_\mu \mid^2 \, , \\
\eea
 and similarly for the metric $h_{\mu\nu}^1$ using the absolute value and a summation over all indices $\mu, \nu $ in wave coordinates.
 
 However, since for $n\geq 4$\;, we have $h = h^1$, and in particular for the case $n=5$ that we consider here, we therefore write
 \bea\label{definitionoftheenergynormwithh}
\E_{N} (t) :=  \sum_{|I|\leq N} \big(   \|w^{1/2}   \pa ( \Lie_{Z^I}  A   (t,\cdot) )  \|_{L^2} +  \|w^{1/2}   \pa ( \Lie_{Z^I} h   (t,\cdot) )  \|_{L^2} \big) \, .
\eea

To sum up: we shall nevertheless often use in many equations, in this paper, the tensor $h^1$ which coincides with $h$ in the case of higher dimensions, since our goal is to continue the work in the third paper that follows where the part that carries the mass $M$ is non-vanishing. Thus, we shall often refer to the energy as defined in \eqref{definitionoftheenergynorm}.

\subsection{The bootstrap argument} \label{Thebootstrapargumentandnotationonboundingtheenergy}\

It is a continuity argument. We start with a local solution defined on a maximum time interval $[0, T_{\text{loc}})$ and that is well-posed in the energy norm $\E_{N} (t)$ for some $N \in \N$\;. This means that the time dependance of the energy $\E_{N} (t)$ is continuous: in other words, the map $[0, T_{\text{loc}}) \to \R$, which assigns $t \to \E_{N} (t)$ is continuous in the standard sense. Furthermore, by maximality of $T_{\text{loc}}$ and the well-posedness of the solution, the time interval for the local solution must be excluding $T_{\text{loc}}$\,, otherwise the energy will be finite at $t=T_{\text{loc}}$ and this means that we could extend the local solution again beyond the time $t=T_{\text{loc}} $ by repeating the argument for establishing a local solution starting at time $t = T_{\text{loc}} $. In other words, the maximal $T_{\text{loc}}$ is characterised by
\beaa
\lim_{t \to T_{\text{loc}} } \E_{ N } (t)  =  \infty \;.
\eeaa
We look at any time $T \in [0, T_{\text{loc}})$, such that for all $t$ in the interval of time $[0, T]$, we have
\bea\label{aprioriestimate}
\E_{ N } (t)  \leq E (N )  \cdot \eps \cdot (1 +t)^\delta \;,
\eea
where $E ( N )$\, is a constant that depends on $ N $\,, where $\eps \geq 0$ is a constant to be chosen later small enough, and where $\delta \geq 0$ is to be chosen later. In addition, we start with an initial data such that this estimate holds true for $t=0$\,, i.e.
\bea \label{theboostrapimposestheconditiononinitialdatasothatnonemptyset}
\E_{ N } (0)  \leq E (N ) \cdot  \eps \;,
\eea
and thus we know that such a $T$ exists, since at least $T=0$ satisfies the estimate.

We will then show that for $t \in [0, T]$\,, the same estimate holds true but with $\eps$ replaced with $\frac{\eps}{2}$\,, i.e. we then prove that for all $t$ in the time interval $[0, T]$\,,
\bea\label{improvedapriori}
\E_{ N } (t)  \leq E ( N )  \cdot \frac{\eps}{2}   \cdot (1 +t)^\delta \; .
\eea

As a result, we would have shown that the set
$$\{ T  \; \;  | \; \text{for all} \quad  t \in [0, T]\,,  \quad  \E_{ N } (t)  \leq E (N )  \cdot \eps  \cdot (1 +t)^\delta  \}$$ is relatively open in $[0, T_{\text{loc}})$\,, non-empty since $0$ belongs to the set, and we know that it is relatively closed in $[0, T_{\text{loc}})$ since the map $t \to \E_{N} (t)$  is continuous, and thus, the set is the whole interval $[0, T_{\text{loc}})$\,.

Consequently, we would have shown that for all $t \in [0, T_{\text{loc}})$\,, we have
\beaa
\E_{ N } (t)  \leq E ( N )  \cdot \frac{\eps}{2}  \cdot (1 +t)^\delta \;.
\eeaa
As a result, we have
\beaa
\lim_{t \to T_{\text{loc}} } \E_{ N } (t)  \leq E ( N )  \cdot \frac{\eps}{2}  \cdot (1 +T_{\text{loc}})^\delta < \infty \;.
\eeaa
By continuity of the energy, this means that  $\E_{ N } (T_{\text{loc}})$ is finite and we can then repeat the argument for establishing a local solution starting at time $t = T_{\text{loc}} $ which would lead to a local solution defined beyond the time $t=T_{\text{loc}} $\,, which contradicts the maximality of $T_{\text{loc}} $\,.

To sum this up, we started by an a priori estimate \eqref{aprioriestimate}, we improved the a priori estimate in \eqref{improvedapriori}, and we therefore showed using the local well-posedness of the solution that it is an actual estimate. Since the estimate \eqref{aprioriestimate} is therefore true, this provides the finiteness of the higher order energy $\E_{ N } (t)$ for all time $t$ and therefore that the local solution for \eqref{EYMsystemashyperbolicPDE} is a in fact a global solution and furthermore, the improved estimate on the energy is true for all time $t$\,.

\subsection{The bootstrap assumption}\label{Theassumptionforthebootstrapandthenotation}\

As explained above, to run our bootstrap argument, we start by assuming that for all $k \leq c \in \N$\,, where $c$ is to be determined later, we have
  \bea\label{bootstrap}
\E_{ k } (t)  \leq E (k )  \cdot \eps \cdot (1 +t)^\delta \;.
\eea
In the case here, where $n \geq 5$\,, and also for the case that follows in the next paper for $n=4$\,, we choose in fact
\bea
\delta &=& 0 \;, \\ \label{delataqualtozero}
\eps &=& 1 \; . \label{epsequaltoone}
\eea
However, we carry out the calculations sometimes with $\de \geq 0$\,, and always with $0 < \eps \leq 1$\,, since in the case of $n = 3$\,, in a paper that follows, we will use indeed $\de > 0$ and we shall indeed fix $\eps $ small, and we would like therefore to use some of the calculations carried out here without repeating them. Thus, our bootstrap assumption here is
  \bea\label{actuallyusedincalculationsbootstrapassumption}
\E_{ k } (t)  \leq E (k )  \cdot \eps  \; .
\eea
The choice, for next papers, of 
\bea\label{epsissmallerthan1}
0 < \eps \leq 1\;,
\eea
is so that any powers of $\eps$ are in fact bounded by $\eps$\,. To lighten the notation, we also choose here
\bea\label{assumptionontheconstantboundsforenergy}
E ( k ) \leq 1\;,
\eea
so that any sum of powers of $E ( k )$ is in fact bounded by a constant multiplied by $E ( k )$\,. In addition, we choose
\bea\label{assumptionontheorderingoftheconstantboundsforenergy}
E ( k_1 )  \leq E(k_2)\;,
\eea
for all $k_1 \leq k_2$\,, with $k_1\, , k_2 \in \N$\,, given the fact that $\E ( k_1 )  \leq \E(k_2)$. The reason we choose to put the constants $E ( k )$\,, rather than want an $\eps$ to be fixed, is to show in the estimates the dependance on the energy and mainly, on the number of Lie derivatives involved. In other words, these constants $E ( k)$ are not needed but are there to make clearer in the argument the number of Lie derivatives of fields for which we use the bootstrap assumption. Speaking of this, in turns out in fact, that we could close the argument for $\E_{N}$\,, as in \eqref{aprioriestimate} with $\de=0$ and $\eps = 1$\,, by assuming \eqref{actuallyusedincalculationsbootstrapassumption} for all
\beaa
k \leq   \lfloor \frac{N}{2} \rfloor + \lfloor  \frac{n}{2} \rfloor  + 1 \;,
\eeaa
provided that $ N\geq 2 \lfloor  \frac{n}{2} \rfloor  + 2$\, (see Proposition \ref{Thetheoremofglobalstabilityanddecayforngeq 5}).

To sum up, our actual bootstrap assumption for this paper is that for $k \leq   \lfloor \frac{N}{2} \rfloor + \lfloor  \frac{n}{2} \rfloor  + 1$\,, with $N \geq 2 \lfloor  \frac{n}{2} \rfloor  + 2$\,, we have
  \bea\label{bootstrapassumption}
\E_{ k } (t)  \leq E (k )   \;,
\eea
where $E ( k ) \leq 1$ and $E ( k_1 )  \leq E(k_2)$ for all $k_1\, , k_2 \in \N$\,. And we are looking forward to upgrading the estimate \eqref{bootstrapassumption}. For this, we will have to exploit the special structure of the equations. 

In Section \ref{Aprioridecayestimatessectionreference}, we will show how an a priori estimate on the energy, \eqref{bootstrap}, translates into decay estimates on the pointwise norm of the solution that is the metric and the Yang-Mills potential -- this is derived using the weighted Klainerman-Sobolev inequality.

\subsection{The $O$ notation}\

\begin{definition}\label{definitionofbigOforLiederivatives}

For a family of tensors Let $\Lie_{Z^{I_1}}K^{(1)}, \ldots, \Lie_{Z^{I_m}} K^{(m)}$, where each tensor $ K^{(l)}$ is again either $A$ or $h$ or $H$\;, or $\derm A$\;, $\derm h$ or $\derm H$\;, we define
\bea
\notag
&& O_{\mu_{1} \ldots \mu_{k} } (\Lie_{Z^{I_1}}K^{(1)} \cdot \ldots \cdot \Lie_{Z^{I_m}} K^{(m)}  ) \\
\notag
&:=& \prod_{l=1}^{m} \Big[  \prod_{|J_l| \leq |I_l|} Q_{1}^{J_l} ( \Lie_{Z^{J_l}} K^{(l)} ) \cdot \Big( \sum_{n=0}^{\infty} P_{n}^{J_l}  ( \Lie_{Z^{J_l}} K^{(l)} ) \Big) \Big] \; . \\
\eea
where again $P_n^{J_l} (K^{l} )$ and $Q_1^{J_l} (K)$, are tensors that are Polynomials of degree $n$ and $1$, respectively, with $Q^{J_l}_1 (0) = 0$ and $Q^{J_l}_1 \neq 0$\;, of which the coefficients are components in wave coordinates of the metric $\textbf m$ and of the inverse metric $\textbf m^{-1}$, and of which the variables are components in wave coordinates of the covariant tensor $\Lie_{Z^{J_l}} K^{l}$, leaving some indices free, so that at the end the whole product $$  \prod_{l=1}^{m} \Big[  \prod_{|J_l| \leq |I_l|}  Q_{1}^{J_l} ( \Lie_{Z^{J_l}} K^{(l)} ) \cdot \Big( \sum_{n=0}^{\infty} P_{n}^{J_l}  ( \Lie_{Z^{J_l}} K^{(l)} ) \Big) \Big]$$ gives a tensor with free indices $\mu_{1} \ldots \mu_{k}$\;. To lighten the notation, we shall drop the indices and just write $O (\Lie_{Z^{I_1}}K^{(1)} \cdot \ldots \cdot \Lie_{Z^{I_m}} K^{(m)}  )$.

\end{definition}

\begin{remark}
Note that if we use a bootstrap assumption, \eqref{bootstrap}, to bound $$  Q_{1}^{|I_l|} ( \Lie_{Z^{|I_l|}} K^{(l)} ) \cdot \Big( \sum_{n=0}^{\infty} P_{n}^{|I_l|}  ( \Lie_{Z^{|I_l|}} K^{(l)} ) \Big) \; ,$$ the bound will then hold true for $$\prod_{|J_l| \leq |I_l|}^{m}  Q_{1}^{J_l} ( \Lie_{Z^{J_l}} K^{(l)} ) \cdot \Big( \sum_{n=0}^{\infty} P_{n}^{J_l}  ( \Lie_{Z^{J_l}} K^{(l)} ) \Big) \; .$$

\end{remark}

\section{A priori decay estimates}\label{Aprioridecayestimatessectionreference}

The a priori estimates are decay estimates that are generated from the weighted Klainerman-Sobolev inequality combined with the bootstrap assumption \eqref{bootstrap} which is the fact that we look at a time $T \in [0, T_{\text{loc}})$\,, such that for all $t \in [0, T]$\,, we have
\beaa
\E_{ N } (t)  \leq E (N )  \cdot \eps \cdot (1 +t)^\delta \;.
\eeaa

This will generate decay estimates which have nothing to do with the Einstein-Yang-Mills equations, but they come from the fact that we chose the energy to be in the form of what dominates the right hand side of the Klainerman-Sobolev inequality when applied to  $\pa \Lie_{Z^I} h^1$ and $ \pa Z^I A$. In other words, this bootstrap assumption \ref{bootstrap}, is an assumption on the bound of such an energy (an assumption that needs yet to be improved in order to turn it into a true estimate) translates into pointwise decay estimates through Klainerman-Sobolev inequality. The fact that these estimates are generated from the bootstrap assumption, and are not proven yet to be true estimates, is the reason why we call them “a priori decay estimates”.

\begin{lemma} \label{aprioriestimatesongradientoftheLiederivativesofthefields}
Under the bootstrap assumption \eqref{bootstrap}, taken for $N =  |I| +  \lfloor  \frac{n}{2} \rfloor  +1$\;, if for all $\mu, \nu \in (t, x^1, \ldots, x^n)$\;, and for any functions $ \pa_\mu \Lie_{Z^I} h^1_\nu \; , \; \pa_\mu \Lie_{Z^I} A_\nu  \in C^\infty_0(\R^n)$\;, then we have
 \bea
 \notag
|\pa  ( \Lie_{Z^I}  A ) (t,x)  |           &\leq& \begin{cases} C ( |I| ) \cdot E ( |I| + \lfloor  \frac{n}{2} \rfloor  +1 )  \cdot \frac{\eps }{(1+t+|q|)^{\frac{(n-1)}{2}-\delta} (1+|q|)^{1+\gamma}},\quad\text{when }\quad q>0,\\
           \notag
       C ( |I| ) \cdot E ( |I| + \lfloor  \frac{n}{2} \rfloor  +1)  \cdot \frac{\eps  }{(1+t+|q|)^{\frac{(n-1)}{2}-\delta}(1+|q|)^{\frac{1}{2} }}  \,\quad\text{when }\quad q<0 , \end{cases} \\
      \eea
and 
 \bea
 \notag
|\pa ( \Lie_{Z^I}  h^1 ) (t,x)  |   &\leq& \begin{cases} C ( |I| ) \cdot E ( |I| + \lfloor  \frac{n}{2} \rfloor  +1)  \cdot \frac{\eps }{(1+t+|q|)^{\frac{(n-1)}{2}-\delta} (1+|q|)^{1+\gamma}},\quad\text{when }\quad q>0,\\
       C ( |I| ) \cdot E ( |I| + \lfloor  \frac{n}{2} \rfloor  +1)  \cdot \frac{\eps  }{(1+t+|q|)^{\frac{(n-1)}{2}-\delta}(1+|q|)^{\frac{1}{2} }}  \,\quad\text{when }\quad q<0 . \end{cases} \\
      \eea

\end{lemma}

\begin{proof}

In fact, the weighted Sobolev estimate gives that for all $\mu, \nu \in (t, x^1, \ldots, x^n)$, and for any functions $ \pa_\mu \Lie_{Z^I} h^1_\nu \, ,  \pa_\mu \Lie_{Z^I} A_\nu  \in C^\infty_0(\R^n)$, i.e. if they are smooth functions vanishing at spatial infinity, $$\lim_{|x| \to \infty} \pa_\mu \Lie_{Z^I} h^1_\nu (t, x) = \lim_{|x| \to \infty} \pa_\mu \Lie_{Z^I} A_\nu (t, x) = 0 ,$$ and for any arbitrary $(t,x)$,

\beaa
|\pa_\mu \Lie_{Z^I} A_\nu (t,x)| \cdot (1+t+|q|) \cdot \big[ (1+|q|)w(q)\big]^{1/2} \leq
C\sum_{|J|\leq  \lfloor  \frac{n}{2} \rfloor  +1} \|\big(w(q)\big)^{1/2} Z^J  \pa_\mu \Lie_{Z^J} A_\nu (t,\cdot)\|_{L^2} ,
\eeaa
and
\beaa
|\pa_\mu \Lie_{Z^I} h^1_\nu (t,x)| \cdot (1+t+|q|) \cdot \big[ (1+|q|)w(q)\big]^{1/2} \leq
C\sum_{|J|\leq \lfloor  \frac{n}{2} \rfloor  +1 } \|\big(w(q)\big)^{1/2} Z^J  \pa_\mu \Lie_{Z^I} h^1_\nu (t,\cdot)\|_{L^2} .
\eeaa
However, we have established that for a sufficiently smooth function $f$,

\beaa
| w^{1/2}  Z^J  \pa Z^I f (t,\cdot)| \leq C_1( |J|) \sum_{|K|\leq |J|} | w^{1/2} \pa Z^K  Z^I f (t,\cdot) | ,
\eeaa
which leads to
\beaa
| w^{1/2}  Z^J  \pa Z^I f (t,\cdot)|^2 &\leq& C_2( |J|)  \sum_{|K|\leq  |J|} | w^{1/2} \pa Z^K  Z^I f (t,\cdot) |^2 \\
&& \text{(using $ab \les a^2 + b^2$  )} .
\eeaa

Thus,
\bea
\notag
\|\big(w(q)\big)^{1/2} Z^J  \pa Z^I f (t,\cdot)\|_{L^2} &\leq& C_3( |J|) \sum_{|K|\leq  |J|} \|\big(w(q)\big)^{1/2}  \pa Z^K  Z^I f (t,\cdot)\|_{L^2}  \\
\notag
&& \text{(using $\sqrt{a + b} \leq \sqrt{a} + \sqrt{ b} $  )} \\
\notag
&\leq & C(| I|,  |J|) \cdot \sum_{|K|\leq | I| + |J|} \|\big(w(q)\big)^{1/2}  \pa Z^K  f (t,\cdot)\|_{L^2}  . \\
\eea

Hence, for all $\mu, \nu \in (t, x^1, \ldots, x^n)$, we have
\beaa
\|\big(w(q)\big)^{1/2} Z^J  \pa_\mu \Lie_{Z^I} A_\nu (t,\cdot)\|_{L^2}  &\leq & C(| I|,  |J|) \sum_{|K|\leq  |J|} \|\big(w(q)\big)^{1/2}  \pa Z^K  \Lie_{Z^I} A_\nu (t,\cdot)\|_{L^2}  \; .
\eeaa
Using the fact that a commutation of two vector fields in $\cal Z$, i.e. $[Z^{\iota_i}, Z^{\iota_k}]$, gives a combination of vector fields in $\cal Z$, and using the fact that we have already showed, that a commutation of a vector field in $\cal Z$ and $\pa_\mu$ gives a linear combination of vectors of the form $\pa_\mu$, we get that for all $\nu \in (t, x^1, \ldots, x^n)$, $Z^K  \Lie_{Z^I} A_\nu$ is a linear combination of elements of the form  $\Lie_{Z^L} A_\mu$ with $|L|\leq |K| + |I|$ and $\mu \in (t, x^1, \ldots, x^n)$. Hence, for any $\nu \in (t, x^1, \ldots, x^n)$,
\beaa
\|\big(w(q)\big)^{1/2}  \pa Z^K  \Lie_{Z^I} A_\nu (t,\cdot)\|_{L^2}  \les \sum_{|L|\leq  |K| + |I|} \|\big(w(q)\big)^{1/2}  \pa   \Lie_{Z^K} A (t,\cdot)\|_{L^2}  \, ,
\eeaa
and therefore,
\beaa
 \sum_{|K|\leq  |J|} \|\big(w(q)\big)^{1/2}  \pa Z^K  \Lie_{Z^I} A_\nu (t,\cdot)\|_{L^2}  \les  \sum_{|K|\leq | I| +  |J|}  \|\big(w(q)\big)^{1/2}  \pa   \Lie_{Z^K} A (t,\cdot)\|_{L^2}  \, .
\eeaa

Consequently, for any $\nu \in (t, x^1, \ldots, x^n)$,

\beaa
\|\big(w(q)\big)^{1/2} Z^J  \pa_\mu \Lie_{Z^I} A_\nu (t,\cdot)\|_{L^2}  &\leq & C(| I|,  |J|) \sum_{|K|\leq | I| +  |J|} \|\big(w(q)\big)^{1/2}  \pa \Lie_{Z^K}  A (t,\cdot)\|_{L^2}  \\
&\les &C(| I|,  |J|) \cdot \E_{| I| +  |J|} (t) ,
\eeaa
and 
\beaa
\|\big(w(q)\big)^{1/2} Z^J  \pa_\mu \Lie_{Z^I} h^1_\nu (t,\cdot)\|_{L^2}   &\leq & C(| I|,  |J|)\cdot \E_{| I| +  |J|} (t) .
\eeaa
Plugging these estimates to the right hand side of the weighted Sobolev inequalities, \eqref{wksi}, gives that for all $\mu, \nu \in (t, x^1, \ldots, x^n)$,
\beaa
&& | \pa_\mu \Lie_{Z^I} A_\nu (t,x)| \cdot (1+t+|q|)^\frac{(n-1)}{2} \cdot \big[ (1+|q|)w(q)\big]^{1/2} \\
&\leq& C\sum_{|J|\leq \lfloor  \frac{n}{2} \rfloor  +1 } \|\big(w(q)\big)^{1/2} Z^J  \pa Z^I A (t,\cdot)\|_{L^2} \\
&\leq& \sum_{|J|\leq \lfloor  \frac{n}{2} \rfloor  +1} C(| I|,  |J|) \cdot \E_{| I| +  |J|} (t)  \\
&\leq&  C(| I|) \cdot \E_{| I| +  \lfloor  \frac{n}{2} \rfloor  +1 } (t)  ,
\eeaa
and hence,
\bea
\notag
| \pa ( \Lie_{Z^I} A ) (t,x)| \cdot (1+t+|q|)^\frac{(n-1)}{2} \cdot \big[ (1+|q|)w(q)\big]^{1/2} &\leq&  C(| I|) \cdot \E_{| I| +  \lfloor  \frac{n}{2} \rfloor  +1} (t)  , \\
\eea
and similarly,
\bea
\notag
|\pa ( \Lie_{Z^I} h^1 ) (t,x)| \cdot (1+t+|q|)^\frac{(n-1)}{2} \cdot \big[ (1+|q|)w(q)\big]^{1/2} &\leq& C(| I|) \cdot \E_{| I| +  \lfloor  \frac{n}{2} \rfloor  +1} (t) . \\
\eea

Thus,
 \bea
|\pa ( \Lie_{Z^I}  A  ) (t,x) ) |  &\leq&   C(| I|)  \cdot \frac{   \E_{| I| +  \lfloor  \frac{n}{2} \rfloor  +1} (t) }{(1+t+|q|)^\frac{(n-1)}{2} \big[ (1+|q|)w(q)\big]^{1/2}} ,
\eea
and 
 \bea
|\pa  ( \Lie_{Z^I}  h^1 ) (t,x) ) |  &\leq&   C(| I|)  \cdot \frac{   \E_{| I| +  \lfloor  \frac{n}{2} \rfloor  +1} (t) }{(1+t+|q|)^\frac{(n-1)}{2} \big[ (1+|q|)w(q)\big]^{1/2}} .
\eea

By definition of the weight $w$ (see Definition \ref{defw}), for some $0<\gamma<1$, we have
\beaa
(w(q))^{1/2} &=& \begin{cases}\big[ (1+|q|)^{1+2\gamma} \big]^{1/2},\quad\text{when }\quad q>0 ,\\
     1    \,\quad\text{when }\quad q<0 .\end{cases} \\
        &=& \begin{cases} (1+|q|)^{\frac{1}{2}+\gamma} ,\quad\text{when }\quad q>0 ,\\
       1 \,\quad\text{when }\quad q<0 .\end{cases}
\eeaa
Hence,
\bea\label{whynotationalfactorof2infrontofgammaindefofw}
\big[ (1+|q|)w(q)\big]^{1/2}   &=& \begin{cases} (1+|q|)^{1+\gamma} ,\quad\text{when }\quad q>0 \; ,\\
       (1+|q|)^{\frac{1}{2}} \,\quad\text{when }\quad q<0 \; .\end{cases}
\eea
Consequently,
 \beaa
 \notag
|\pa ( \Lie_{Z^I}  A ) (t,x)  | &\leq& \begin{cases} C ( |I| )   \cdot \frac{\E_{| I| +  \lfloor  \frac{n}{2} \rfloor  +1} (t) }{(1+t+|q|)^\frac{(n-1)}{2}(1+|q|)^{1+\gamma}},\quad\text{when }\quad q>0 \; ,\\
       C ( |I| )   \cdot \frac{\E_{| I| +  \lfloor  \frac{n}{2} \rfloor  +1} (t) }{(1+t+|q|)^\frac{(n-1)}{2} (1+|q|)^{\frac{1}{2} }}  \,\quad\text{when }\quad q<0 \; , \end{cases} \\
       \eeaa
       
       and 
       \beaa
|\pa ( \Lie_{Z^I}  h^1 ) (t,x)  |       &\leq& \begin{cases} C ( |I| ) \cdot   \frac{\E_{| I| +  \lfloor  \frac{n}{2} \rfloor  +1} (t) }{(1+t+|q|)^\frac{(n-1)}{2}(1+|q|)^{1+\gamma}},\quad\text{when }\quad q>0 \; ,\\
       C ( |I| ) \cdot    \frac{\E_{| I| +  \lfloor  \frac{n}{2} \rfloor  +1} (t) }{(1+t+|q|)^\frac{(n-1)}{2} (1+|q|)^{\frac{1}{2} }}  \,\quad\text{when }\quad q<0 \; .\end{cases} \\
      \eeaa

Using the bootstrap assumption on the growth of the higher order energy, we get
 \bea
 \notag
|\pa  ( \Lie_{Z^I}  A ) (t,x)  | &\leq& \begin{cases} C ( |I| ) \cdot E ( |I| + \lfloor  \frac{n}{2} \rfloor  +1)  \cdot \frac{\eps (1 +t)^\delta }{(1+t+|q|)^\frac{(n-1)}{2} (1+|q|)^{1+\gamma}},\quad\text{when }\quad q>0,\\
 \notag
       C ( |I| ) \cdot E ( |I| + \lfloor  \frac{n}{2} \rfloor  +1)  \cdot \frac{\eps (1 +t)^\delta }{(1+t+|q|)^\frac{(n-1)}{2} (1+|q|)^{\frac{1}{2} }}  ,\,\quad\text{when }\quad q<0 .\end{cases} \\
        \notag
       &\leq& \begin{cases} C ( |I| ) \cdot E ( |I| + \lfloor  \frac{n}{2} \rfloor  +1)  \cdot \frac{\eps (1+t+|q|)^\delta }{(1+t+|q|)^\frac{(n-1)}{2} (1+|q|)^{1+\gamma}},\quad\text{when }\quad q>0,\\
        \notag
       C ( |I| ) \cdot E ( |I| + \lfloor  \frac{n}{2} \rfloor  +1)  \cdot \frac{\eps (1+t+|q|)^\delta }{(1+t+|q|)^\frac{(n-1)}{2} (1+|q|)^{\frac{1}{2} }} , \,\quad\text{when }\quad q<0 . \end{cases} \\
        \notag
          &\leq& \begin{cases} C ( |I| ) \cdot E ( |I| + \lfloor  \frac{n}{2} \rfloor  +1)  \cdot \frac{\eps }{(1+t+|q|)^{\frac{(n-1)}{2}-\delta} (1+|q|)^{1+\gamma}},\quad\text{when }\quad q>0,\\
           \notag
       C ( |I| ) \cdot E ( |I| + \lfloor  \frac{n}{2} \rfloor  +1)  \cdot \frac{\eps  }{(1+t+|q|)^{\frac{(n-1)}{2}-\delta}(1+|q|)^{\frac{1}{2} }} , \,\quad\text{when }\quad q<0 , \end{cases} \\
      \eea
and similarly,
 \bea
 \notag
|\pa ( \Lie_{Z^I}  h^1 ) (t,x)  |   &\leq& \begin{cases} C ( |I| ) \cdot E ( |I| + \lfloor  \frac{n}{2} \rfloor  +1)  \cdot \frac{\eps }{(1+t+|q|)^{\frac{(n-1)}{2}-\delta} (1+|q|)^{1+\gamma}},\quad\text{when }\quad q>0,\\
       C ( |I| ) \cdot E ( |I| + \lfloor  \frac{n}{2} \rfloor  +1)  \cdot \frac{\eps  }{(1+t+|q|)^{\frac{(n-1)}{2}-\delta}(1+|q|)^{\frac{1}{2} }}  , \,\quad\text{when }\quad q<0 . \end{cases} \\
      \eea

\end{proof}

\subsection{The spatial asymptotic behaviour of $ \Lie_{Z^I} A (t, x) $ at $t=0$}\

\begin{lemma}\label{estimateonZderivativeofafuncionbypartialderivativeoff}
We have for all vector $Z \in  {\cal Z}$, and for all sufficiently smooth function $f$, the following estimate for $t \geq 0$,
\beaa
 | Z f  |&\les& ( 1 + t +  |x  | ) \cdot | \pa f  | .
\eeaa
\end{lemma}
\begin{proof}

As a reminder, in wave coordinates $x^\mu$, we have
\beaa
Z_{\a\b} &=& x_{\b} \pa_{\a} - x_{\a} \pa_{\b} ,\\
S &=&  x^{\a} \pa_{\a} ,
\eeaa
where 
\beaa
x_{\b} = m_{\mu\b} x^{\mu} .
\eeaa
Thus, we have,
\beaa
 | Z_{\a\b} f  | &=&  | x_{\b} \pa_{\a} f - x_{\a} \pa_{\b} f  | \leq   | x_{\b} \pa_{\a} f  | +  | x_{\a} \pa_{\b} f  | \\
&\leq&   | x_{\b} \pa_{\a} f | +  | x_{\a} \pa_{\b} f  | \\
&\leq&   | x_{\b} \pa_{\a} f  | +  | x_{\a} \pa_{\b} f  | \; .
\eeaa
Hence, in wave coordinates, i.e. for $\a, \b \in \{0, 1, \ldots, n  \}$, we have
\beaa
 | Z_{\a\b} f  | &\leq&   |x_{\b} \pa_{\a} f  |+ | x_{\a} \pa_{\b} f  | \\
&\les& (  | t  | +  | x  |  \cdot  | \pa f  | ,
\eeaa
where 
\bea
 | x  |  = \sqrt{(x^1)^2 + \ldots +(x^n)^2 } .
\eea
For the vector $S$, we have
\beaa
 | S f  |  &=&  | x^{\a} \pa_{\a} f | =  | t \pa_{t} f +\sum_{i=0}^{n} x^{i} \pa_{i} f  |     \\
&\les& (  | t  | +  |x  | ) \cdot  | \pa f  |.
\eeaa
Also, from the definition $ | \pa f  |  $ in wave coordinates, we get
\beaa
 | \pa_{x_\a} f | &\leq&  |\pa f  | .
\eeaa
Consequently, for $Z \in \{Z_{\a\b}, S, \pa_{\a} \} $, $\a, \b \in \{0, 1, \ldots, n  \}$,
\bea
 | Z f  | &\les& ( 1 + | t  |+ | x  | ) \cdot  |\pa f  | \, .
\eea
\end{proof}

\begin{lemma}\label{asymptoticbehaviouratteqzeroforallfields}
If the factor $\gamma$ in the weight is such that $\gamma > \max\{0, \delta -\frac{(n-1)}{2}\} $, then under the bootstrap assumption \eqref{bootstrap}, taken for $k =  |I| +  \lfloor  \frac{n}{2} \rfloor  +1$, we have for all $ | I  | \geq 1$, 
\bea\label{boundonthehyperplanetequalzeroforzeroderivativeoftehfields}
\notag
 | \Lie_{Z^I} A (0,x)  | +   | \Lie_{Z^I}  h^1 (0,x)  |   &\les&  C ( |I| ) \cdot E ( |I| + \lfloor  \frac{n}{2} \rfloor  +1)  \cdot \frac{\eps  }{(1+r)^{1+\gamma-\delta}}  \; ,\\
\eea
and 
\bea\label{blimitatrgoestoinfinityathyperplanetequalzerozeroforzeroderivativeoftehfields}
 \lim_{ r \to \infty }  \big(  | \Lie_{Z^I}  A (0,x)  | +     | \Lie_{Z^I}  h^1 (0,x)  |  \big) &=& 0 \;.
\eea
Also, we choose to take the initial data such that \eqref{boundonthehyperplanetequalzeroforzeroderivativeoftehfields} is also true for $ | I  | = 0$, which implies \eqref{blimitatrgoestoinfinityathyperplanetequalzerozeroforzeroderivativeoftehfields}.

\begin{remark} In addition, for such $\gamma > \max\{0, \delta -\frac{(n-1)}{2}\} $, we also have
 \bea\label{boundonthehyperplanetequalzeroforthefullderivativeoftheLieZfields}
\notag
 | \pa   ( \Lie_{Z^I} A ) A (0,x)  | +   |\pa   ( \Lie_{Z^I} A )  h^1 (0,x)  |   &\les&  C ( |I| ) \cdot E ( |I| + \lfloor  \frac{n}{2} \rfloor  +1)  \cdot \frac{\eps  }{(1+r)^{1+\gamma-\delta}}  \; ,\\
\eea
and 
 \bea
\notag
 \lim_{ r = \to \infty } \big(   |\pa ( \Lie_{Z^I}  A ) (0,x)  | +  |\pa  ( \Lie_{Z^I}   h^1 ) (0,x)  |  \big)  &=&  0  \; .\\
\eea
\end{remark}
\end{lemma}

\begin{proof}
Since $q=r-t$, at $t=0$, we have  $q=r \geq 0$. We have established that for $q\geq 0$, 
 \beaa
 \notag
  |\pa  ( \Lie_{Z^I} A ) (t,x)  |   +  |\pa  ( \Lie_{Z^I}   h^1 ) (t,x)  |       &\leq& C ( |I| ) \cdot E ( |I| + \lfloor  \frac{n}{2} \rfloor  +1)  \cdot \frac{\eps }{(1+t+|q|)^{\frac{(n-1)}{2}-\delta} (1+|q|)^{1+\gamma}} \; .
 \eeaa
 Plugging in $t=0$, we get
  \bea
 \notag
  |\pa   ( \Lie_{Z^I} A ) (0,x)  |     +  |\pa  ( \Lie_{Z^I}   h^1 ) (0,x)  |    &\leq& C ( |I| ) \cdot E ( |I| + \lfloor  \frac{n}{2} \rfloor  +1)  \cdot \frac{\eps }{(1+|q|)^{\frac{(n-1)}{2}+1+\ga-\delta} } \; .\\
      \eea

This means that for $\ga > \delta -1 - \frac{(n-1)}{2} = \de - \frac{(n+1)}{2}$, and $\ga > 0$,
 \bea
 \notag
&& \lim_{ r = | x | \to \infty } \big(   |\pa ( \Lie_{Z^I}  A ) (0,x)  | +  |\pa  ( \Lie_{Z^I}   h^1 ) (0,x)  |  \big)    \\
\notag
   &\leq& \lim_{|q | \to \infty } C ( |I| ) \cdot E ( |I| + \lfloor  \frac{n}{2} \rfloor  +1)  \cdot \frac{\eps  }{(1+|q|)^{\frac{(n+1)}{2}+\ga-\delta}  } = 0 \;. \\
 \eea
In particular, this implies that for $\gamma > \max\{0, \delta -\frac{(n+1)}{2}\} $, 
 \beaa
\lim_{ r = | x | \to \infty } \big( |\pa ( \Lie_{Z^I}   h^1 ) (0,x)  | +  |\pa ( \Lie_{Z^I}  A ) (0,x)  | \big)      = 0 \; .
 \eeaa
 
Now, we would like to estimate the asymptotics of $\Lie_{Z^I} A (0,x)$ and $\Lie_{Z^I} h (0,x)$\,.

For  $\mu, \nu \in (t, x^1, \ldots, x^n)$\,, taking in Lemma \ref{estimateonZderivativeofafuncionbypartialderivativeoff}, $f =\Lie_{Z^I}  A_{\mu} (0,x) $ and then $f= \Lie_{Z^I}  h_{\mu\nu} (0,x)$\,, we obtain for all $I$\,, and for any vector $Z \in {\cal Z}$\,,
 \beaa
 \notag
 && |\ Z \Lie_{Z^I} A_{\mu}  (0,x)  |   +   |\ Z \Lie_{Z^I} h^1_{\mu\nu} (0,x)  |   \\
 &\les&  ( 1 + | t  |+  | x | ) \cdot (  | \pa \Lie_{Z^I}A (0,x)  | +  | \pa \Lie_{Z^I}  h^1 (0,x)  | )\; .
 \eeaa

Since at $t=0$\,, we have $q = | x|$\,, we get for any $\mu, \nu \in (t, x^1, \ldots, x^n)$\,,
 \bea
 \notag
  && |\ Z \Lie_{Z^I} A_{\mu}  (0,x)  |   +   |\ Z \Lie_{Z^I}  h^1_{\mu\nu} (0,x)  |   \\
   \notag
      &\les&  ( 1 +  | q  | ) \cdot (  | \pa \Lie_{Z^I} A (0,x)  |+  | \pa \Lie_{Z^I}  h^1 (0,x)  | ) \\
  \notag
   &\les&  ( 1 +  | q  | ) \cdot C ( |I| ) \cdot E ( |I| + \lfloor  \frac{n}{2} \rfloor  +1)  \cdot \frac{\eps }{(1+|q|)^{\frac{(n-1)}{2}+1+\ga-\delta} } \\
   \notag
    &\les&    C ( |I| ) \cdot E ( |I| + \lfloor  \frac{n}{2} \rfloor  +1)  \cdot \frac{\eps }{ (1+|q|)^{\frac{(n-1)}{2}+\gamma-\delta}}  \; . \\
      \eea

Now, using the fact that a commutation of two vector fields in $\cal Z$ is again a combination of vector fields in $\cal Z$\,, and using the fact that a commutation of a vector field in $\cal Z$ and $\pa_\mu$ gives a linear combination of vectors of the form $\pa_\mu$\,, we get by induction on $|I|$ that for all $I$ such that $| I | \geq 1$\,,
 \bea
 \notag
  && | \Lie_{Z^I} A_{\mu}  (0,x)  |   +   |\Lie_{Z^I}  h^1_{\mu\nu} (0,x)  |   \\
   \notag
    &\les&    C ( |I| ) \cdot E ( |I| + \lfloor  \frac{n}{2} \rfloor  +1)  \cdot \frac{\eps }{ (1+|q|)^{\frac{(n-1)}{2}+\gamma-\delta}}  \, . \\
\eea
      
In particular, this means that if $\gamma > \max\{0, \delta -\frac{(n-1)}{2}\} $\,,
      
\bea
 \lim_{ | x  | \to \infty }  \big(  |  \Lie_{Z^I} A (0,x)  | +  |\Lie_{Z^I}  h^1 (0,x)  | \big) &=& 0 .
\eea

For $|I| = 0$, we take the initial data such that
\beaa
 |A (0,x)  | +  \mid  h^1 (0,x) \mid   &\les&   \frac{\eps }{ (1+|r|)^{\frac{(n-1)}{2}+\gamma-\delta}}
\eeaa
which implies that
\bea
 \lim_{ r \to \infty }  \big( |  A (0,x)  | +   |  h^1 (0,x)  |  \big) &=& 0 .
\eea 
\end{proof}

 \subsection{Estimates on $ \Lie_{Z^I} A $ and $ \Lie_{Z^I} h^1 $ for $t>0$.}\

Now, we will use \eqref{boundonthehyperplanetequalzeroforzeroderivativeoftehfields} in Lemma \ref{asymptoticbehaviouratteqzeroforallfields} to estimate the Lie derivatives in the direction of Minkowski vector fields of the Einstein-Yang-Mills fields $ \Lie_{Z^I} A $ and $ \Lie_{Z^I} h^1 $, for $t > 0$. This will be done by specific integration till we reach the hyperplane prescribed by $t=$ and then use  \eqref{boundonthehyperplanetequalzeroforzeroderivativeoftehfields}.

\begin{lemma}\label{aprioriestimatefrombootstraponzerothderivativeofAandh1}
Under the bootstrap assumption \eqref{bootstrap}, taken for $k =  |I| +  \lfloor  \frac{n}{2} \rfloor  +1$\,, and with $\ga >0$\, and with initial data such that
\beaa
|   A (0,x) |  +  |   h^1 (0,x) |   &\les&   \frac{\eps }{ (1+r)^{\frac{(n-1)}{2}+\gamma-\delta}} \; ,
\eeaa
then, we have for all $ | I  | $\,,
 \bea
 \notag
&&| \Lie_{Z^I}  A  (t,x)  |   +| \Lie_{Z^I}  h^1  (t,x)  |   \\
 \notag 
 &\leq& \begin{cases} c (\gamma) \cdot  C ( |I| ) \cdot E ( |I| +  \lfloor  \frac{n}{2} \rfloor  +1)  \cdot \frac{\eps }{(1+t+|q|)^{\frac{(n-1)}{2}-\delta} (1+|q|)^{\gamma}},\quad\text{when }\quad q>0\; ,\\
       C ( |I| ) \cdot E ( |I| +  \lfloor  \frac{n}{2} \rfloor  +1)  \cdot \frac{\eps \cdot (1+| q |   )^\frac{1}{2} }{(1+t+|q|)^{\frac{(n-1)}{2}-\delta} }  \,\quad\text{when }\quad q<0 \; . \end{cases} \\
      \eea

\end{lemma}

\begin{proof}

Let $ B $ and $C$ be tensors. For a scalar product $<B, C>$ such that
\beaa
\pa_{\mu} <B, C> &=&  <\pa_{\mu}  B, C> + < B, \pa_{\mu} C>
\eeaa
 -- which is the case for the definition for our norms --, we have (see \cite{G2}), that for a vector $X$, 
\bea\label{absolutevalueofderivativeofabsolutevaluelessthanabsolutofderiva}
\mid \pa_{X}| B | \mid &\leq& | \pa_{X}  B |  \; .
\eea

Since $\Sigma_{t=0}$ is diffeomorphic to $\R^n$\,, for each $x \in \Sigma_{0}$\,, there exists $\Om \in \SSS^{n-1}$\,, such that $x =r \cdot \Om $\,.

 \textbf{Case $ q   \geq 0$:}\\
 
 Then, we have $q= | x |  -t \geq 0$ and the point $(t,x)$ is therefore outside the outgoing null cone whose tip is the origin. We then apply the fundamental theorem of calculus by integrating at a fixed $\Om$\,, from $(t, | x | \cdot \Om)$ along the line $(\tau, r \cdot \Om)$ such that $r+\tau =  | x | +t $ till we reach the hyperplane $\tau=0$\,. We obtain,
 
 \bea
\notag
\mid \Lie_{Z^I} A (t, | x | \cdot \Om) \mid &=& \mid \Lie_{Z^I} A \big(0, ( t + | x |) \cdot \Om \big) \mid  + \int_{t + | x | }^{ | x |  } \pa_r \mid \Lie_{Z^I} A (t + | x | - r,  r  \cdot \Om )  \mid dr \\
\notag
&\leq& \mid \Lie_{Z^I} A \big(0, ( t + | x |) \cdot \Om \big) \mid  + | \int_{t + | x | }^{ | x |  } \pa_r \mid \Lie_{Z^I} A (t + | x | - r,  r  \cdot \Om )  \mid dr | \\
\notag
&\leq& \mid \Lie_{Z^I} A \big(0, ( t + | x |) \cdot \Om \big) \mid  +  \int_{ | x | }^{ t + | x |  } | \pa_r \mid \Lie_{Z^I} A (t + | x | - r,  r  \cdot \Om )  \mid |dr  \\
\notag
 &\leq& \mid \Lie_{Z^I} A \big(0, ( t + | x | ) \cdot \Om \big) \mid  + \int_{ | x | }^{ t + | x |  } \mid \pa_r  (\Lie_{Z^I} A (t + | x | - r,  r  \cdot \Om ) ) \mid dr \\
&& \text{(see \eqref{absolutevalueofderivativeofabsolutevaluelessthanabsolutofderiva})}.
\eea
On one hand, we have
\bea
\pa_r = \frac{x^{i}}{r} \pa_{i} \; .
\eea
and thus
\beaa
\notag
 \mid \pa_r  ( \Lie_{Z^I} A (t + | x | - r, \mid r \mid \cdot \Om ) ) \mid &\leq&  \mid \pa ( \Lie_{Z^I} A (t + | x | - r,  r \cdot \Om ) ) \mid \; ,
\eeaa

and on the other hand, for $q\geq 0$, we have
\beaa
 \notag
|\pa (   \Lie_{Z^I} A ) ( \tau,  r  \cdot \Om)  | &\leq& C ( |I| ) \cdot E ( |I| +  \lfloor  \frac{n}{2} \rfloor  +1)  \cdot \frac{\eps }{(1+ \tau +|q|)^{\frac{(n-1)}{2}-\delta} (1+|q|)^{1+\gamma}} \; .
 \eeaa

 Hence, 
   \beaa
 \notag
|\pa   ( \Lie_{Z^I} A ) (t + | x | - r, \mid r \mid \cdot \Om)  | &\leq& C ( |I| ) \cdot E ( |I| +  \lfloor  \frac{n}{2} \rfloor  +1)   \cdot \frac{\eps  }{(1+t + | x | - r + |q|)^{\frac{(n-1)}{2}-\delta} (1+|q|)^{1+\gamma}}
 \eeaa

  with $ q   = r - \tau = r - (t + | x | - r) = 2r - t -| x | =  |q| $ since $q \geq 0$. This means
 $$1+t + | x | - r + |q| = 1+t + | x | - r +2r - t -| x | = 1+ r .$$
This leads to
\bea
 \notag
|\pa  ( \Lie_{Z^I} A ) (t + | x | - r,  r  \cdot \Om)  | &\leq& C ( |I| ) \cdot E ( |I| +  \lfloor  \frac{n}{2} \rfloor  +1)    \cdot \frac{\eps  }{(1+r )^{\frac{(n-1)}{2}-\delta}  (1+2r - t - | x |)^{1+\gamma}} .
\eea

 Since we integrate in the direction $ | x | \leq r $, we have 
    \beaa
|\pa ( \Lie_{Z^I} A ) (t + | x | - r,  r \cdot \Om)  | &\leq& C ( |I| ) \cdot E ( |I| +  \lfloor  \frac{n}{2} \rfloor  +1)  \frac{\eps}{(1+r )^{\frac{(n-1)}{2}-\delta }(1+r - t )^{1+\gamma}} \\
&\leq& \frac{ C ( |I| ) \cdot E ( |I| +  \lfloor  \frac{n}{2} \rfloor  +1)  \cdot\eps}{(1+ | x |)^{\frac{(n-1)}{2}-\delta }(1+r - t )^{1+\gamma}} ,
          \eeaa
and therefore,
    \bea
    \notag
\int_{|x| }^{ t + | x |  } |\pa  ( \Lie_{Z^I} A ) (t + | x | - r,  r \cdot \Om)  | dr  &\leq& \int_{ | x | }^{t + | x | } \frac{C ( |I| ) \cdot E ( |I| +  \lfloor  \frac{n}{2} \rfloor  +1)  \cdot\eps}{(1+ | x |)^{\frac{(n-1)}{2}-\delta }(1+r - t )^{1+\gamma}} dr \\
    \notag
&\leq&\frac{C ( |I| ) \cdot E ( |I| +  \lfloor  \frac{n}{2} \rfloor  +1) \cdot \eps }{(1+| x |)^{\frac{(n-1)}{2}-\delta }}    \big[     \frac{-1}{\gamma (1+r - t )^{\gamma}}  \big]^{t + |x| }_{ | x |} \; .\\
          \eea

Hence,
   \beaa
\int_{ | x | }^{t + | x |  } |\pa ( \Lie_{Z^I} A ) (t + | x | - r,  r  \cdot \Om)  | dr  &\leq&\frac{C ( |I| ) \cdot E ( |I| +  \lfloor  \frac{n}{2} \rfloor  +1) \cdot \eps  }{(1+ | x |)^{\frac{(n-1)}{2}-\delta }}    \big(     \frac{-1}{\gamma (1+ | x |  )^{\gamma}}   + \frac{1}{\gamma (1+ | x |  -t  )^{\gamma}} \big) \\
&\leq&\frac{C ( |I| ) \cdot E ( |I| +  \lfloor  \frac{n}{2} \rfloor  +1)  \cdot \eps  }{(1+ | x |)^{\frac{(n-1)}{2}-\delta }}   \cdot     \frac{1}{\gamma (1+| x | -t  )^{\gamma}} \; .
     \eeaa
However, since at the point $(t,x)$, we have $q= | x |  - t = | q | $ in this region, and therefore, $| x |   = t + | q |$, we obtain
   \bea
   \notag
\int_{| x | }^{ t + | x | } |\pa (  \Lie_{Z^I} A )  (t + | x | - r,  r \cdot \Om) ) | dr  
&\leq&\frac{C ( |I| ) \cdot E ( |I| +  \lfloor  \frac{n}{2} \rfloor  +1)  }{\gamma}  \cdot      \frac{ \eps}{ (1+ t + | q |)^{\frac{(n-1)}{2}-\delta }  (1+| q | )^{\gamma}}  \; .
     \eea

 From what we have previously proved for $t=0$, for $\gamma > \max\{0, \delta -1\} $, we have for the other term on $\tau =0$, the following estimate
 
 \bea
 \notag
\mid \Lie_{Z^I} A (0, ( t +  | x | )  \cdot \Om) \mid  &\leq&   C ( |I| ) \cdot E ( |I| +  \lfloor  \frac{n}{2} \rfloor  +1)  \cdot \frac{\eps }{ (1+ t +  | x | )^{\frac{(n-1)}{2}+\gamma-\delta}} \;.
\eea
Finally, with $q= | x |  - t = | q | $, we have $t +  | x |  = 2t +  | x | -t = 2t + | q | $ and hence
 \beaa
 \notag
\mid \Lie_{Z^I} A (0, ( t +  | x | )  \cdot \Om) \mid  &\leq&   C ( |I| ) \cdot E ( |I| +  \lfloor  \frac{n}{2} \rfloor  +1)  \cdot \frac{\eps }{ (1+ 2t +  | q | )^{\frac{(n-1)}{2}-\delta} (1+ 2t +  | q | )^{\gamma}} \\
&\les&   C ( |I| ) \cdot E ( |I| +  \lfloor  \frac{n}{2} \rfloor  +1)  \cdot \frac{\eps }{ (1+ t +  | q | )^{\frac{(n-1)}{2}-\delta}  (1+   | q | )^{\gamma}} .
\eeaa
Consequently,
 \beaa
\notag
\mid \Lie_{Z^I} A (t,  | x | \cdot \Om) \mid &\les& C ( |I| ) \cdot E ( |I| +  \lfloor  \frac{n}{2} \rfloor  +1) \cdot \eps \\
&& \cdot \big(  \frac{1}{ \gamma(1+ t + | q | )^{\frac{(n-1)}{2}-\delta }  (1+| q |   )^{\gamma}} +   \frac{1}{(1+ t +| q|  )^{\frac{(n-1)}{2}  -\delta } (1+ | q|  )^{\gamma } }  \big)  \\
&\les& c (\gamma) \cdot C ( |I| ) \cdot E ( |I| +  \lfloor  \frac{n}{2} \rfloor  +1) \cdot  \frac{\eps}{ (1+ t + | q | )^{\frac{(n-1)}{2}-\delta }  (1+| q |   )^{\gamma}}   .
\eeaa
Identically, we get the same estimate for $\Lie_{Z^I} h^1 (t,x) $, and hence, for $\gamma > \max\{0, \delta -1\} $, for $q\geq	 0$, 
\bea
\notag
&& \mid \Lie_{Z^I} A (t,  | x | \cdot \Om)  \mid + \mid \Lie_{Z^I} h^1 (t,  | x | \cdot \Om) \mid \\
\notag
&\les& c (\gamma)\cdot  C ( |I| ) \cdot E ( |I| +  \lfloor  \frac{n}{2} \rfloor  +1) \cdot  \frac{\eps}{ (1+ t + | q | )^{\frac{(n-1)}{2}-\delta }  (1+| q |   )^{\gamma}} \;  . \\
\eea

 \textbf{Case $ q < 0$:}\\

 Then, we have $q= | x | -t < 0$ and the point $(t,x)$ is therefore inside the outgoing null cone whose tip is at the origin.  We then apply the fundamental theorem of calculus by integrating at a fixed $\Om \in \SSS^{n-1}$, from $(t, | x | \cdot \Om)$ along the line $(\tau, r\cdot \Om)$ such that $r+\tau =  | x | +t $ till we reach the hyperplane $\tau=0$. We have for all $\mu \in (t, x^1, \ldots, x^n)$,
\beaa
\notag
\mid \Lie_{Z^I} A_{\mu} (t, | x | \cdot \Om) \mid  &\leq& \mid \Lie_{Z^I} A_{\mu} \big(0, ( t + | x |) \cdot \Om \big) \mid  +  \mid \int_{| x | }^{t + | x |  } \pa_r  ( \Lie_{Z^I} A_{\mu} (t +| x | - r,  r \cdot \Om ) )  dr  \mid ,
\eeaa
and we know that for $\gamma > \max\{0, \delta -1\} $,
 \beaa
 \notag
\mid \Lie_{Z^I}  A_{\mu} (0, ( t + | x | )  \cdot \Om) \mid  &\leq& C ( |I| ) \cdot E ( |I| +  \lfloor  \frac{n}{2} \rfloor  +1)  \cdot \frac{\eps  }{(1+ t + | x |)^{\frac{(n-1)}{2}+\gamma-\delta}}  .
\eeaa
In fact 
 \bea
  1+t+ | x | \sim 1+t+|q|  \; .
  \eea 
Consequently, we obtain
 \bea
 \notag
\mid \Lie_{Z^I}  A_{\mu} (0, ( t + | x | )  \cdot \Om) \mid  &\leq& C ( |I| ) \cdot E ( |I| +  \lfloor  \frac{n}{2} \rfloor  +1)  \cdot \frac{\eps  }{(1+ t + | x |)^{\frac{(n-1)}{2}+\gamma-\delta}}  \\
\notag
&\les& C ( |I| ) \cdot E ( |I| +  \lfloor  \frac{n}{2} \rfloor  +1)  \cdot \frac{\eps  }{(1+ t + | q |)^{\frac{(n-1)}{2}+\gamma-\delta}} ,\\
\eea
and thus, we are left out with treating the integral $ \mid \int_{| x | }^{t + | x |  } \pa_r  (\Lie_{Z^I}  A_{\mu} (t +| x | - r,  r  \cdot \Om ) )  dr  \mid$.

 The line $\tau + r = t + | x | $ intersects the outgoing light cone at $\tau = r$, and thus, the intersection point is at $r = \frac{t + \mid x \mid }{2} = \tau$. Computing for any $\mu \in (t, x^1, \ldots, x^n)$,
\beaa
&&  |\int_{| x | }^{t + | x | } \pa_r  (\Lie_{Z^I}  A_{\mu} (t + | x | - r,  r  \cdot \Om ) )  dr    | \\
 &\leq&   | \int_{| x | }^{\frac{t + | x |}{2}  }  \pa_r  (\Lie_{Z^I}  A_{\mu} (t + | x | - r,  r   \cdot \Om ) ) dr   | +   | \int_{\frac{t + | x |}{2}  }^{t + | x |  }  \pa_r  (\Lie_{Z^I}  A_{\mu} (t + | x | - r,  r \cdot \Om ) )  dr   | \; .
\eeaa
Treating the second term, which is an integral on a segment in the region $q\geq 0$, i.e. $r \geq  t$ :

\beaa
   | \int_{\frac{t +  | x |}{2}  }^{t +  | x  |   }  \pa_r  (\Lie_{Z^I}  A_{\mu} (t +  | x | - r, r \cdot \Om ) )  dr   \mid &=&   | \Lie_{Z^I}  A_{\mu} (0 , ( t +  | x | ) \cdot \Om )  -   \Lie_{Z^I}  A_{\mu} \big(\frac{t +  | x |}{2}  , (\frac{t +  | x |}{2}  ) \cdot \Om \big)      |  \\
  &\leq&   | \Lie_{Z^I}  A_{\mu} (0 , ( t +  | x | )  \cdot \Om )   |   +   | \Lie_{Z^I}  A_{\mu} \big(\frac{t +  | x |}{2}  , (\frac{t + | x |}{2}  ) \cdot \Om \big)     | \; .
\eeaa
From what we had proven, we have that the first term, for $\gamma > \max\{0, \delta -1\} $\,, has the following estimate
\beaa
\notag
  \mid  \Lie_{Z^I} A_{\mu} (0 , ( t + | x | ) \cdot \Om )   |   &\les& C ( |I| ) \cdot E ( |I| +  \lfloor  \frac{n}{2} \rfloor  +1)  \cdot \frac{\eps  }{(1+ t + | q |)^{\frac{(n-1)}{2}+\gamma-\delta}}  \;.
\eeaa
For the second term, since it is on $q =0$\,, we can use the estimate that we established for $q \geq 0$ by plugging in it $q=0$ and taking $\tau = \frac{t + | x |}{2}$\,, to obtain 
\beaa
\notag
 | \Lie_{Z^I} A_{\mu} (\frac{t + | x |}{2}  , (\frac{t + | x | }{2}  ) \cdot \Om )      |  &\les& c (\gamma)  \cdot C ( |I| ) \cdot E ( |I| +  \lfloor  \frac{n}{2} \rfloor  +1) \cdot  \frac{\eps}{ (1+ \frac{t + | x |}{2}  )^{\frac{(n-1)}{2}-\delta }  }   \\
  \notag
   &\les& c (\gamma) \cdot C ( |I| ) \cdot E ( |I| +  \lfloor  \frac{n}{2} \rfloor  +1) \cdot  \frac{\eps}{ (1+ t + | x |  )^{\frac{(n-1)}{2}-\delta }  } \\
   \notag
   &\les& c (\gamma) \cdot C ( |I| ) \cdot E ( |I| +  \lfloor  \frac{n}{2} \rfloor  +1) \cdot  \frac{\eps}{ (1+ t + | q |  )^{\frac{(n-1)}{2}-\delta }  } \;.
\eeaa
Thus,
\bea
\notag
  \mid  \int_{\frac{t +  | x |}{2}  }^{t +  | x |  }  \pa_r  ( \Lie_{Z^I} A_{\mu} (t +  | x | - r, r \cdot \Om ) )  dr   \mid  &\les& c (\gamma) \cdot C ( |I| ) \cdot E ( |I| +  \lfloor  \frac{n}{2} \rfloor  +1) \cdot  \frac{\eps}{ (1+ t + | q |  )^{\frac{(n-1)}{2}-\delta }  } \;.\\
\eea
Now, we are left to treat the integral on the region $q <0$\,, that we can estimate using the estimate on the gradient that we already established, \eqref{aprioriestimatesongradientoftheLiederivativesofthefields},

\beaa
 \notag
 && \mid \int_{ | x | }^{\frac{t +  | x |}{2}  } \pa_r (  \Lie_{Z^I} A_{\mu} (t +  | x | - r,  r  \cdot \Om) )  dr  \mid \\
 &\leq&  \mid \int_{ | x |}^{\frac{t +  | x |}{2}  } \pa (  Z^I A_{\mu} (t +  | x | - r, r  \cdot \Om) )  dr  \mid \\
 &\leq& \int_{ | x | }^{\frac{t +  | x |}{2}  } C ( |I| ) \cdot E ( |I| +  \lfloor  \frac{n}{2} \rfloor  +1)  \cdot \frac{\eps  }{(1+t +  | x | - r+|q|)^{\frac{(n-1)}{2}-\delta}(1+|q|)^{\frac{1}{2} }}  dr \; .
\eeaa

In this region of integration, we have $q = r - \tau < 0$\,, and thus we have $|q| = -q = \tau - r  $\,. However, we are integrating along the line $r + \tau = t + | x | $ and thus, $ \tau = t + | x |  - r $\,. Consequently, $|q| = t + | x |  - 2 r $ and $q= 2r - t - | x |  $\,.  Hence,
\beaa
 \notag
 && \mid \int_{ | x | }^{\frac{t +  | x |}{2}  } \pa_r (  \Lie_{Z^I} A_{\mu} (t +  | x | - r,  r  \cdot \Om) )  dr  \mid  \\
 &\leq& \int_{ | x | }^{\frac{t +  | x |}{2}  } C ( |I| ) \cdot E ( |I| +  \lfloor  \frac{n}{2} \rfloor  +1)  \cdot \frac{\eps  }{(1+2t +  2| x | - 3 r)^{\frac{(n-1)}{2}-\delta}(1+|q|)^{\frac{1}{2} }}  dr \\
  &\leq& \frac{ C ( |I| ) \cdot E ( |I| +  \lfloor  \frac{n}{2} \rfloor  +1) }{ (1+2t +  2| x | -  \frac{3 (t +  | x | )}{2})^{\frac{(n-1)}{2}-\delta} } \int_{ | x | }^{\frac{t +  | x |}{2}  } \frac{\eps  }{(1+|q|)^{\frac{1}{2} }}  dr \\
    &\leq& \frac{ C ( |I| ) \cdot E ( |I| +  \lfloor  \frac{n}{2} \rfloor  +1) }{ (1+  \frac{ (t +  | x | )}{2})^{\frac{(n-1)}{2}-\delta} } \int_{ | x | }^{\frac{t +  | x |}{2}  }  \frac{\eps  }{(1 - q)^{\frac{1}{2} }}  dr \\
        &\les& \frac{ C ( |I| ) \cdot E ( |I| +  \lfloor  \frac{n}{2} \rfloor  +1) }{ (1+  t +  | x | ) )^{\frac{(n-1)}{2}-\delta} } \int_{ | x | -t}^{ 0 }  \frac{\eps  }{(1 - q)^{\frac{1}{2} }}  \frac{dq}{2}\; ,
\eeaa
where we made the change of variable $q= 2r - t - | x |  $ with $dq = 2 dr $\,. We get
\bea
 \notag
  && \mid \int_{ | x | }^{\frac{t +  | x |}{2}  } \pa_r ( \Lie_{Z^I} A_{\mu} (t +  | x | - r,  r  \cdot \Om) )  dr  \mid     \\
   \notag
  &\les& \frac{ C ( |I| ) \cdot E ( |I| +  \lfloor  \frac{n}{2} \rfloor  +1) \cdot \eps }{ (1+  t +  | x | ) )^{\frac{(n-1)}{2}-\delta} } \cdot \big[ -  ( 1 - q)^{\frac{1}{2} } \big]_{| x | -t}^{0} \\
   \notag
   &\les& \frac{ C ( |I| ) \cdot E ( |I| +  \lfloor  \frac{n}{2} \rfloor  +1) \cdot \eps }{ (1+  t +  | x | ) )^{\frac{(n-1)}{2}-\delta} } \cdot \big[ -1 +  \big( 1 - (| x | -t) \big)^{\frac{1}{2} }  \big] \\
    \notag
    &\les& \frac{ C ( |I| ) \cdot E ( |I| +  \lfloor  \frac{n}{2} \rfloor  +1) \cdot \eps }{ (1+  t +  | x | ) )^{\frac{(n-1)}{2}-\delta} }  \cdot \big( 1 - (| x | -t) \big)^{\frac{1}{2} }  \\
    &\les& \frac{ C ( |I| ) \cdot E ( |I| +  \lfloor  \frac{n}{2} \rfloor  +1) \cdot \eps }{ (1+  t +  | q | ) )^{\frac{(n-1)}{2}-\delta} } \cdot \big( 1 + | q | ) \big)^{\frac{1}{2} } \;. 
\eea

Putting it all together, we obtain that, for $q<0$\,, we have the following estimate, using the fact that the same argument works also for $\Lie_{Z^I}  h^1$\,,
\bea
\notag
\mid \Lie_{Z^I} A (t,  | x | \cdot \Om)  \mid + \mid \Lie_{Z^I} h^1 (t,  | x | \cdot \Om) \mid &\les&  C ( |I| ) \cdot E ( |I| +  \lfloor  \frac{n}{2} \rfloor  +1) \cdot  \frac{\eps}{ (1+ t + | q | )^{\frac{(n-1)}{2}-\delta }  } (1+| q |   )^{\frac{1}{2} }  \; . \\
\eea

Thus, we get the result.

\end{proof}

We would like now to estimate $ |    \Lie_{Z^K}  g^{\la\mu} \derm_{\la}   \derm_{\mu}  A_{\cal U} (t,x) |$ for $|K| \leq |J|$.

  \begin{lemma}
The Minkowski covariant derivative commutes with the Lie derivative along Minkowski vector fields, that is for any tensor $K$, 
\bea
 \Lie_{Z^I}  \derm  K  =   \derm ( \Lie_{Z^I}   K ) .
\eea
Since  $\derm  K$ is also a tensor, it follows that $\Lie_{Z^I} $ commutes with any product of $\derm$.

Note that the Lie derivatives are not being differentiated in $\derm ( \Lie_{Z^I}   K) $; the differentiation concerns only the tensor $K$.

\end{lemma}

\begin{proof}
In fact, for simplicity, consider $K$ a tensor of order one, $K_\a$. Let $Z \in \cal Z$. We have
\beaa
 \Lie_{Z}  \derm_\a  K_\b  =  Z (  \derm_\a K_\b )  -  \derm_{ \Lie_{Z} e_\a} K_\b -  \derm_\a K_{ \Lie_{Z} e_\b} \;.
\eeaa
Since  $\Lie_{Z}  \derm_\a  K_\b$ is a 2-tensor, we can compute it in wave coordinates $\{x^0, x^1, \ldots, x^n \}$, and if the result we get is also a tensor in $\a, \b$, it would then hold true for any vectors. Let $\a, \b \in \{x^0, x^1, \ldots, x^n \}$: we know by then that $\derm_{e_\a} e_\b = 0$ and therefore we have
\bea
\derm_\a K_\b = \pa_\a K_\b  .
\eea
We also have $\Lie_{Z} e_\a = [Z, e_\a]$. Since $Z$ is a Minkowski vector field, it is either a coordinate vector field (and therefore $[Z, e_\a] = 0$) or it is a rotation or a Lorentz boost and can be written as
\beaa
Z_{\mu\nu} = x_{\nu} \pa_{\mu} - x_{\mu} \pa_{\nu} \,
\eeaa
or it is a scaling vector field and can be written as
\beaa
S =  \sum_{\mu=0}^{3} x^\mu \pa_{\mu} \, .
\eeaa
We have for spatial indices $i, j \in \{1, \ldots, n \}$, that $ e_\a (x_i) = e_\a (x^i)  = \de_{\a i}$, and thus, for rotations,
\bea\label{commuatorrotationandtranslation}
[Z_{ij}, e_\a]  &=& [x_{j} \pa_{i} - x_{i} \pa_{j} , \pa_\a] = \de_{\a j} \pa_{i}  - \de_{\a i} \pa_{j} \, .
\eea
We have for the Lorentz boosts,
\bea\label{commuatorLorentzboostsandtranslation}
[Z_{0j}, e_\a]  &=& [x_{j} \pa_{t} - x_{0} \pa_{j} , \pa_\a] = [x^{j} \pa_{t} + t \pa_{j} , \pa_\a]  = \de_{\a j} \pa_{t}  + \de_{\a 0} \pa_{j} \, .
\eea
For the scaling vector field, we have
\bea\label{commuatorscalingvectorandtranslation}
[ S , e_\a]  &=& [ x^\mu \pa_{\mu} , \pa_\a] = - \de_{\a\mu}  \pa_\mu = - \pa_\a \, .
\eea

Consequently, for all $Z \in {\cal Z}$, we have in wave coordinates that
 \bea
 \derm_{ \Lie_{Z} e_\a} K_\b &=&  \pa_{ \Lie_{Z} e_\a} K_\b \\
   \derm_\a K_{ \Lie_{Z} e_\b} &=&   \pa_\a K_{ \Lie_{Z} e_\b} .
\eea
Hence, for all $Z \in  {\cal Z}$, in wave coordinates, we have
\beaa
 \Lie_{Z}  \derm_\a  K_\b  &=&  Z   \pa_\a K_\b   -  \pa_{ \Lie_{Z} e_\a} K_\b -  \pa_\a K_{ \Lie_{Z} e_\b} \\
 &=&  Z   \pa_\a K_\b   -  \pa_{ [ Z , e_\a ]} K_\b -  \pa_\a K_{ \Lie_{Z} e_\b} .
\eeaa

On the other hand, let’s compute $ \derm_\a (  \Lie_{Z} K_\b )   =   \derm_\a (  \Lie_{Z} K )_\b  $, where the differentiation treats $ \Lie_{Z} K $ as a one-tensor. We have in wave coordinates,
\beaa
  \derm_\a (  \Lie_{Z} K )_\b    =   \pa_\a (  \Lie_{Z} K_\b )  \, .
\eeaa
However, 
\beaa
  \Lie_{Z} K_\b  =   \pa_{Z} K_\b  - K_{ \Lie_{Z} e_\b } \, .
\eeaa
Thus,
\beaa
   \pa_\a (  \Lie_{Z} K_\b )  = \pa_\a (Z K_\b )   -  \pa_\a ( K_{ \Lie_{Z} e_\b } ) .
\eeaa
Yet, 
\beaa
[ Z , e_\a ] K_\b =  Z  \pa_\a K_\b   -  \pa_\a Z  K_\b  
\eeaa
and hence, for all $Z \in \cal Z$, we have
\beaa
  \derm_\a (  \Lie_{Z} K )_\b  &=&  \pa_\a (  \Lie_{Z} K_\b )  = Z  \pa_\a K_\b  -   [ Z , e_\a ]  K_\b   -  \pa_\a  K_{ \Lie_{Z} e_\b }  \\
   &=&  \Lie_{Z}  \derm_\a  K_\b .
\eeaa
By induction on $ |I|$, we get that for all products of Lie derivatives $\Lie_{Z^I}  $, the following equality holds,
 \bea
 \Lie_{Z^I}  \derm_\a  K_\b  =   \derm_\a ( \Lie_{Z^I}   K_\b ) \; .
\eea

\end{proof}

      \begin{lemma}\label{LiederivativesofproductsZofcovariantandcontravariantMinkowskimetrictensor}
      
      The Lie derivative in the direction of the Minkowski vector fields of the Minkowski metric, is either null or proportional to the Minkowski metric, that is for any $Z \in {\cal Z}$,
\bea
 \Lie_{Z}  m_{\mu\nu}  =  c_Z \cdot m_{\mu\nu} 
\eea
and 
\bea
 \Lie_{Z}  m^{\mu\nu}  =  - c_Z \cdot m^{\mu\nu} 
\eea
where $c_Z = 0 $ for all $Z \neq S$ and $c_S = 2$.

Thus,
\bea\label{LiederivativesofproductsZofMinkowskimetric}
 \Lie_{Z^I}  m_{\mu\nu}  =  c (I)  \cdot m_{\mu\nu} \; \\
 \label{LiederivativesofproductsZofcontravariantMinkowskimetric}
 \Lie_{Z^I}  m^{\mu\nu}  = \hat{c} (I) \cdot m^{\mu\nu} \; ,
\eea
where $c(I)$ and $\hat{c}(I)$ are constants that depend on $Z^I$.

      \end{lemma}

  \begin{proof}

  We compute in wave coordinates $\mu, \nu \in \{x^0, x^1, \ldots, x^n \}$
  \beaa
 \Lie_{Z}  m_{\mu\nu}  &=&  Z  m_{\mu\nu} - m( \Lie_{Z} e_\mu, e_\nu ) - m(  e_\mu, \Lie_{Z} e_\nu )  \\
 &=&  - m( [Z,  e_\mu ] , e_\nu ) - m(  e_\mu, [Z,  e_\nu]  )  \;.
 \eeaa

\textbf{Case of rotations:}\\

We already showed in \eqref{commuatorrotationandtranslation}, that 
\beaa
[Z_{ij}, e_\mu ]  &=&  \de_{\mu j} \pa_{i}  - \de_{\mu i} \pa_{j} \; .
\eeaa
Hence,
  \beaa
 \Lie_{Z_{ij}}  m_{\mu\nu}  &=&  - m(  \de_{\mu j} \pa_{i}  - \de_{\mu i} \pa_{j} , e_\nu ) - m(  e_\mu,  \de_{\nu j} \pa_{i}  - \de_{\nu i} \pa_{j}  )  \\
 &=&  - \de_{\mu j}  m_{i \nu } + \de_{\mu i}  m_{ j \nu}  - \de_{\nu j} m_{i \mu}  + \de_{\nu i} m_{ j \mu }  \; .
 \eeaa
 
Consequently, if $i \neq j$ and if $\mu =  j \neq i$ and if $ \nu = j \neq i$,
  \beaa
 \Lie_{Z_{ij}}  m_{\mu\nu}  &=&    -   m_{i \nu }   -  m_{i \mu}    = 0 + 0 \\ 
&& \text{(since $\mu \neq i$ and $\nu \neq i$}) .
 \eeaa
 Now, if $i \neq j$ and if $\mu =  j \neq i$ and if $ \nu = i \neq j$, then
   \beaa
 \Lie_{Z_{ij}}  m_{\mu\nu}  &=&    - m_{i \nu }    +  m_{ j \mu }  = - m_{ii} + m_{jj} \\
 &=& -1 +1 = 0\; .
 \eeaa
 
 Now, if $i \neq j$ and if $\mu =  t$, then clearly
   \beaa
 \Lie_{Z_{ij}}  m_{\mu\nu}   &=&   0 \;.
 \eeaa
 
 Of course, in the case where $i =j $, then $Z_{ij} = 0$ and therefore $\Lie_{Z_{ij}}  m_{\mu\nu}  = 0$.

\textbf{Case of Lorentz boosts:}\\

We showed in \eqref{commuatorLorentzboostsandtranslation}, that
 \beaa
[Z_{0j}, e_\mu ]  &=&  \de_{\mu j} \pa_{t}  + \de_{\mu 0} \pa_{j}  \; .
\eeaa

Thus,
   \beaa
 \Lie_{Z_{0j}}  m_{\mu\nu}  &=&  - m(  \de_{\mu j} \pa_{t}  + \de_{\mu 0} \pa_{j}   , e_\nu ) - m(  e_\mu,  \de_{\nu j} \pa_{t}  + \de_{\nu 0} \pa_{j}   )  \\
 &=&  - \de_{\mu j} m_{t\nu}  -  \de_{\mu 0} m_{j \nu }   -  \de_{\nu j} m_{\mu t}    -  \de_{\nu 0}  m_{\mu j }   \; .
 \eeaa
 Hence, if $\mu = t \neq j$, then
    \beaa
 \Lie_{Z_{0j}}  m_{\mu\nu}   &=&  -   m_{j \nu }   -  \de_{\nu j} m_{t t}   =  -   m_{j \nu }   +  \de_{\nu j}     \; ,
 \eeaa
 and therefore if $\nu = t$ then  $  \Lie_{Z_{0j}}  m_{\mu\nu}   =  0 $ and if $\nu = i$ spatial index, then  $  \Lie_{Z_{0j}}  m_{\mu\nu}   =  - \de_{ji} +  \de_{i j}  = 0 $.
 
 Now considering the case where $\mu = i$, then 
   \beaa
 \Lie_{Z_{0j}}  m_{\mu\nu}  &=&  - \de_{i j} m_{t\nu}       -  \de_{\nu 0}  m_{i j } =  - \de_{i j}  \de_{\nu 0}       -  \de_{\nu 0} \de_{i j}  = 0 \; .
 \eeaa
 
 \textbf{Case of the scaling vector field:}\\
 
 We have shown in \eqref{commuatorscalingvectorandtranslation}, that
 \beaa
[ S , e_\mu ]  &=& - \pa_\mu \; .
\eeaa
Hence,
  \beaa
 \Lie_{S}  m_{\mu\nu}   &=&  m_{\mu\nu} + m_{\mu \nu} = 2 m_{\mu\nu}  \; .
 \eeaa
 
 Since the end result are identities which are tensorial, they are therefore true not only in wave coordinates (yet, we have carried out the computation in wave coordinates).

 \textbf{Case of the contravariant tensor $m^{\mu\nu}$:}\\ 
  We have
  \beaa
  m_{\mu\b} \cdot m^{\b\nu} = \de_{\mu}^{\;\;\; \nu}
  \eeaa
Using the fact that the Lie derivative commutes with contraction, that is
  \beaa
  \Lie_{Z}  ( m_{\mu\b} \cdot m^{\b\nu} ) = 0 \; ,
  \eeaa
yields to
   \beaa
(   \Lie_{Z}   m_{\mu\b} )  \cdot m^{\b\nu}  + m_{\mu\b} \cdot   \Lie_{Z}   m^{\b\nu}   = 0 \; ,
  \eeaa
  and thus,
     \beaa
 c_{Z}   m_{\mu\b}  \cdot m^{\b\nu}  + m_{\mu\b} \cdot   \Lie_{Z}   m^{\b\nu}   = 0 \; ,
  \eeaa
  and therefore,
      \beaa
m_{\mu\b} \cdot   \Lie_{Z}   m^{\b\nu}   = -  c_{Z}   \de_{\mu}^{\;\;\; \nu}    \; .
  \eeaa
  Inverting, this leads to
        \beaa
  \Lie_{Z}   m^{\b\nu}   = -  c_{Z}   m^{\b\nu}    \; .
  \eeaa

  \textbf{Case for higher order Lie derivatives:}\\ 
  
The equalities \eqref{LiederivativesofproductsZofMinkowskimetric} and \eqref{LiederivativesofproductsZofcontravariantMinkowskimetric} follow by trivial induction, from which we get the desired result.
  
  \end{proof}

  \begin{lemma}
 We have for any family of covariant tensors $K^{(1)}, \ldots, K^{(m)}$ of arbitrary order,
 \beaa
 \Lie_{Z^I}  O ( K^{(1)} \cdot \ldots \cdot K^{(m)} ) &= & \sum_{|J_1| + \ldots + |J_m| \leq |I|} O(   \Lie_{Z^{J_1}}  K^{(1)}  \cdot \ldots \cdot    \Lie_{Z^{J_m}}  K^{(m)}   ) \;.
 \eeaa

  \end{lemma}

  \begin{proof}
  We have that $O ( K^{(1)} \cdot \ldots \cdot K^{(m)} )$ is a product of metrics $\textbf m$ and the inverse metric $\textbf m^{-1}$ and the tensors $K^{(1)},\; \ldots, K^{(m)}$, times any polynomial of these. Using chain rule for the Lie derivative and the fact the Lie derivative in the direction of Minkowski vector fields of the metric $\textbf m$ and of the contravariant metric $\textbf m^{-1}$ is proportional to these, we then get that the Lie derivatives of the tensors $K^{(1)},\; \ldots, K^{(m)}$ and $\textbf m$ and of $\textbf m^{-1}$ is contained in the product of all the Lie derivatives of these. The Lie derivatives of any polynomial of  $K^{(1)},\; \ldots, K^{(m)}$ and $\textbf m$ and of $\textbf m^{-1}$ is also a product of the Lie derivatives of these. Given that the Lie derivative of $\textbf m$ is proportional to $\textbf m$ and the Lie derivative of $\textbf m^{-1} $ is proportional to $\textbf m^{-1}$, the multiplication of all these Lie derivatives are contained in the definition of $\sum_{|J_1| + \ldots + |J_m| \leq |I|} O(   \Lie_{Z^{J_1}}  K^{(1)}  \cdot \ldots \cdot    \Lie_{Z^{J_m}}  K^{(m)}   ) $.
  \end{proof}

\begin{lemma}

  In the Lorenz gauge and in wave coordinates, we have for any $V \in \cal U$
  \beaa
   \notag
 &&   \Lie_{Z^I}   ( g^{\la\mu} \derm_{\la}   \derm_{\mu}   A_{V} )    \\
   &=& \sum_{|J| +|K|+|L|+ |M| \leq |I|}  \big( \; O(     \derm (\Lie_{Z^J}  h)  \cdot     \rderm (\Lie_{Z^K} A )   )     + O(   \rderm ( \Lie_{Z^J}  h ) \cdot   \derm ( \Lie_{Z^K} A)  )  \\
       \notag
           && +  O(      \derm ( \Lie_{Z^J} h)  \cdot   \Lie_{Z^K}   A  \cdot     \Lie_{Z^L}  A    )    +   O(  \Lie_{Z^J} A   \cdot     \rderm ( \Lie_{Z^K}  A)   )    +   O(     \Lie_{Z^J}  A_L   \cdot     \derm_V ( \Lie_{Z^K} A ) )  \\
                 \notag
      &&                 + O(    \Lie_{Z^J}   A_{\cal{T}}   \cdot       \derm_V ( \Lie_{Z^K}A_{\cal{T}} ) )   + O(       \Lie_{Z^J} A  \cdot     \Lie_{Z^K} A   \cdot     \Lie_{Z^L} A  ) \\
           \notag
      &&         + O( \Lie_{Z^J}   h \cdot \derm ( \Lie_{Z^K}  h) \cdot  \derm (\Lie_{Z^L}  A) )  +  O( \Lie_{Z^J}  h \cdot    \derm (\Lie_{Z^K}  h ) \cdot  \Lie_{Z^L}  A \cdot \Lie_{Z^M} A)  \\
\notag
&& + O( \Lie_{Z^J} h \cdot   \Lie_{Z^K} A \cdot   \derm (\Lie_{Z^L} A) )    + O( \Lie_{Z^J} h \cdot  \Lie_{Z^K} A \cdot  \Lie_{Z^L} A \cdot \Lie_{Z^M} A) \; \big) \, .
  \eeaa
\end{lemma}

\begin{proof}

  In the Lorenz gauge and in wave coordinates, we have shown that for any $V \in \cal U$
\beaa
   \notag
 &&  g^{\la\mu} \derm_{\la}   \derm_{\mu}   A_{V}     \\
   &=& O(   \derm h  \cdot  \rderm A    )    + O(   \rderm  h  \cdot  \derm A  )  +  O(  \derm  h  \cdot   A^2 )    +   O( A   \cdot   \rderm A  )    \\
       \notag
  &&  +   O(  A_L   \cdot   \derm_V  A  )  + O(   A_{\cal{T}}   \cdot    \derm_V A_{\cal{T}}   )   + O(    A^3 )  \\
    \notag
  && + O( h \cdot  \derm  h \cdot  \derm  A) + O( h \cdot  \derm  h \cdot  A^2) + O( h \cdot  A \cdot \derm  A) + O( h \cdot  A^3) \, .
  \eeaa

Differentiating, we get for any $Z \in \cal Z$,

\beaa
   \notag
 &&   \Lie_{Z}  ( g^{\la\mu} \derm_{\la}   \derm_{\mu}   A_{V} )     \\
    &=& O(   \derm h  \cdot  \rderm A    )    + O(   \rderm  h  \cdot  \derm A  )  +  O(  \derm  h  \cdot   A^2 )    +   O( A   \cdot   \rderm A  )    \\
       \notag
  &&  +   O(  A_L   \cdot   \derm_V  A  )  + O(   A_{\cal{T}}   \cdot    \derm_V A_{\cal{T}}   )   + O(    A^3 )  \\
    \notag
  && + O( h \cdot  \derm  h \cdot  \derm  A) + O( h \cdot  \derm  h \cdot  A^2) + O( h \cdot  A \cdot \derm  A) + O( h \cdot  A^3) \\
   &&+ O(    \Lie_{Z}  \derm h \cdot  \rderm A    )   + O(   \derm h  \cdot   \Lie_{Z} \rderm A    ) + O(  \Lie_{Z}  \rderm  h  \cdot  \derm A  ) \\
   &&+ O(  \rderm  h \cdot   \Lie_{Z} \derm A  )  +  O(    \Lie_{Z} \derm  h \cdot  A^2 ) +  O(   \derm  h  \cdot   \Lie_{Z}  A  \cdot   A   )      \\
   \notag
   &&    +   O(  \Lie_{Z}  A   \cdot   \rderm A  )   +   O( A   \cdot   \Lie_{Z}  \rderm A  )   +   O(    \Lie_{Z} A_L   \cdot   \derm_V  A  )     \\
   \notag
   && +   O(  A_L   \cdot     \Lie_{Z} \derm_V  A  )     + O(    \Lie_{Z} A_{\cal{T}}  \cdot     \derm_V A_{\cal{T}}   )   + O(  A_{\cal{T}}   \cdot      \Lie_{Z} \derm_V A_{\cal{T}}   ) \\
           \notag
          &&    + O(    \Lie_{Z} A   \cdot   A^2 )  + O(  \Lie_{Z}  h \cdot  \derm  h \cdot  \derm  A) + O( h \cdot  \Lie_{Z}  \derm  h \cdot  \derm  A) \\
          \notag
     &&     + O( h \cdot  \derm  h \cdot   \Lie_{Z}  \derm  A)  + O(  \Lie_{Z}  h \cdot  \derm  h \cdot  A^2) + O( h \cdot   \Lie_{Z}  \derm  h \cdot  A^2) \\
     \notag
     && + O( h \cdot  \derm  h \cdot   \Lie_{Z}  A \cdot A)  + O( \Lie_{Z} h \cdot  A \cdot \derm  A)  + O( h \cdot  \Lie_{Z} A \cdot \derm  A) \\
\notag
&&  + O( h \cdot  A \cdot  \Lie_{Z}\derm  A)+ O( \Lie_{Z} h \cdot  A^3) + O( h \cdot \Lie_{Z} A \cdot  A^2) \, .
  \eeaa
 
 By induction, we obtain that for all $Z^I$, we have
\beaa
   \notag
 &&   \Lie_{Z^I}   ( g^{\la\mu} \derm_{\la}   \derm_{\mu}   A_{V} )    \\
   &=& \sum_{|J| +|K|+|L|+ |M| \leq |I|}  \big( \; O(    \Lie_{Z^J}  \derm h  \cdot   \Lie_{Z^K}  \rderm A    )     + O(    \Lie_{Z^J} \rderm  h  \cdot    \Lie_{Z^K} \derm A  )  \\
       \notag
           && +  O(    \Lie_{Z^J}   \derm  h  \cdot   \Lie_{Z^K}   A  \cdot     \Lie_{Z^L}  A    )   +   O(  \Lie_{Z^J} A   \cdot   \Lie_{Z^K}   \rderm A  )    +   O(     \Lie_{Z^J}  A_L   \cdot    \Lie_{Z^K} \derm_V  A  )   \\
            \notag
           &&       + O(    \Lie_{Z^J}   A_{\cal{T}}   \cdot      \Lie_{Z^K}  \derm_V A_{\cal{T}}  )   + O(       \Lie_{Z^J} A  \cdot     \Lie_{Z^K} A   \cdot     \Lie_{Z^L} A  ) \\
           \notag
      &&         + O( \Lie_{Z^J}   h \cdot \Lie_{Z^K}  \derm  h \cdot \Lie_{Z^L}  \derm  A)  +  O( \Lie_{Z^J}  h \cdot  \Lie_{Z^K}   \derm  h \cdot  \Lie_{Z^L}  A \cdot \Lie_{Z^M} A)  \\
\notag
&& + O( \Lie_{Z^J} h \cdot   \Lie_{Z^K} A \cdot  \Lie_{Z^L} \derm  A)    + O( \Lie_{Z^J} h \cdot  \Lie_{Z^K} A \cdot  \Lie_{Z^L} A \cdot \Lie_{Z^M} A) \; \big) \, .
  \eeaa
  
 Using the fact that $\Lie_{Z^I} $ commutes with $\derm $, we obtain the result.
\end{proof}
  
  \begin{lemma}
  In the Lorenz and harmonic gauges, we have the following estimate for the tangential components of the Einstein-Yang-Mills potential,
  \beaa
   \notag
 && |  \Lie_{Z^I}   ( g^{\la\mu} \derm_{\la}   \derm_{\mu}   A_{{\cal T}}  )   | \\
   &\leq& \sum_{|J| +|K|+|L|+ |M| \leq |I|} \big( \; O(  |   \derm (\Lie_{Z^J}  h) | \cdot |    \rderm (\Lie_{Z^K} A ) |  )     +  O(  | \rderm ( \Lie_{Z^J}  h ) | \cdot |  \derm ( \Lie_{Z^K} A) | )  \\
       \notag
           && +  O(   |   \derm ( \Lie_{Z^J} h) | \cdot |  \Lie_{Z^K}   A  | \cdot |    \Lie_{Z^L}  A |   )       +   O( |\Lie_{Z^J} A   | \cdot |    \rderm ( \Lie_{Z^K}  A) |  )    + O(   |    \Lie_{Z^J} A | \cdot   |  \Lie_{Z^K} A  | \cdot  |   \Lie_{Z^L} A | )    \\
           \notag
      &&         + O( | \Lie_{Z^J}   h | \cdot | \derm ( \Lie_{Z^K}  h) | \cdot | \derm (\Lie_{Z^L}  A) | )  +  O( | \Lie_{Z^J}  h | \cdot |   \derm (\Lie_{Z^K}  h ) | \cdot  | \Lie_{Z^L}  A | \cdot |\Lie_{Z^M} A | \\
\notag
&& + O( | \Lie_{Z^J} h  | \cdot |  \Lie_{Z^K} A | \cdot |  \derm (\Lie_{Z^L} A) | )    + O( | \Lie_{Z^J} h | \cdot | \Lie_{Z^K} A | \cdot | \Lie_{Z^L} A | \cdot | \Lie_{Z^M} A | ) \; \big) \, .
  \eeaa

  \end{lemma}
  
  \begin{proof}
   Based on what we have shown, we have for all $Z^I$,

  \beaa
   \notag
 &&   \Lie_{Z^I}   ( g^{\la\mu} \derm_{\la}   \derm_{\mu}   A_{{\cal T}}  )    \\
   &=& \sum_{|J| +|K|+|L|+ |M| \leq |I|}  \big( \; O(     \derm (\Lie_{Z^J}  h)  \cdot     \rderm (\Lie_{Z^K} A )   )     + O(   \rderm ( \Lie_{Z^J}  h ) \cdot   \derm ( \Lie_{Z^K} A)  )  \\
       \notag
           && +  O(      \derm ( \Lie_{Z^J} h)  \cdot   \Lie_{Z^K}   A  \cdot     \Lie_{Z^L}  A    )       +   O(  \Lie_{Z^J} A   \cdot     \rderm ( \Lie_{Z^K}  A)   )    + O(       \Lie_{Z^J} A \cdot     \Lie_{Z^K} A   \cdot     \Lie_{Z^L} A  )    \\
           \notag
                   &&      +   O( \Lie_{Z^I} A   \cdot     \rderm A   )  +   O( A   \cdot   \rderm  ( \Lie_{Z^K}  A)   )  \\
           \notag
      &&         + O( \Lie_{Z^J}   h \cdot \derm ( \Lie_{Z^K}  h) \cdot  \derm (\Lie_{Z^L}  A) )  +  O( \Lie_{Z^J}  h \cdot    \derm (\Lie_{Z^K}  h ) \cdot  \Lie_{Z^L}  A \cdot \Lie_{Z^M} A)  \\
\notag
&& + O( \Lie_{Z^J} h \cdot   \Lie_{Z^K} A \cdot   \derm (\Lie_{Z^L} A) )    + O( \Lie_{Z^J} h \cdot  \Lie_{Z^K} A \cdot  \Lie_{Z^L} A \cdot \Lie_{Z^M} A \; \big) \, .
  \eeaa

  \end{proof}

\section{Studying the structure of the source terms of the coupled non-linear wave equations}

In this section, we study the general structure of the source terms of the coupled non-linear wave equations on the Yang-Mills potential and the metric in the Lorenz gauge and in wave coordinates.

\begin{lemma}\label{structureofthesourcetermsofthewaveequationonpoentialandmetricinLorenzandwavecoordinforEinsteinYangMillssystem}
In the Lorenz gauge, the Yang-Mills potential satisfies the following tensorial equations, where we lower and higher indices with respect to the metric $m$,
  \bea
  \notag
 && g^{\la\mu} \derm_{\la}   \derm_{\mu}   A_{\si}     \\
   \notag
 &=&    ( \derm_{\si}  h^{\a\mu} )  \cdot (  \derm_{\a}A_{\mu} )    \\
 \notag
&&  +   \frac{1}{2}    \big(   \derm^{\mu} h^{\nu}_{\, \, \, \si} +  \derm_\si h^{\nu\mu} -   \derm^{\nu} h^{\mu}_{\,\,\, \si}  \big)   \cdot  \big( \derm_{\mu}A_{\nu} -  \derm_{\nu}A_{\mu}  \big) \\
 \notag
&&     +   \frac{1}{2}    \big(   \derm^{\mu} h^{\nu}_{\, \, \, \si} +  \derm_\si h^{\nu\mu} -   \derm^{\nu} h^{\mu}_{\,\,\, \si}  \big)   \cdot   [A_{\mu},A_{\nu}] \\
 \notag
 && -  \big(  [ A_{\mu}, \derm^{\mu} A_{\si} ]  +    [A^{\mu},  \derm_{\mu}  A_{\si} - \derm_{\si} A_{\mu} ]    +    [A^{\mu}, [A_{\mu},A_{\si}] ]  \big)  \\
 \notag
  && + O( h \cdot  \derm h \cdot  \derm A) + O( h \cdot  \derm h \cdot  A^2) + O( h \cdot  A \cdot \derm A) + O( h \cdot  A^3) \, .\\
  \eea

The perturbations $h$ of the metric $m$, solutions to the Einstein-Yang-Mills equations in theLorenz gauge, satisfy the following tensorial wave equation, where we lower and higher indices with respect to the metric $m$,

  \bea
\notag
 && g^{\alpha\beta}\derm_\alpha \derm_\beta h_{\mu\nu} \\
 \notag
  &=& P(\derm_\mu h,\derm_\nu h)  +  Q_{\mu\nu}(\derm h,\derm h)   + G_{\mu\nu}(h)(\derm h, \derm h)  \\
\notag
 &&   -4     <   \derm_{\mu}A_{\b} - \derm_{\b}A_{\mu}  ,  \derm_{\nu}A^{\b} - \derm^{\b}A_{\nu}  >    \\
 \notag
 &&   + m_{\mu\nu }       \cdot  <  \derm_{\a}A_{\b} - \derm_{\b}A_{\a} , \derm_{\a} A^{\b} - \derm^{\b}A^{\a} >   \\
 \notag
&&           -4  \cdot  \big( <   \derm_{\mu}A_{\b} - \derm_{\b}A_{\mu}  ,  [A_{\nu},A^{\b}] >   + <   [A_{\mu},A_{\b}] ,  \derm_{\nu}A^{\b} - \derm^{\b}A_{\nu}  > \big)  \\
\notag
&& + m_{\mu\nu }    \cdot \big(  <  \derm_{\a}A_{\b} - \derm_{\b}A_{\a} , [A^{\a},A^{\b}] >    +  <  [A_{\a},A_{\b}] , \derm^{\a}A^{\b} - \derm^{\b}A^{\a}  > \big) \\
\notag
 &&  -4     <   [A_{\mu},A_{\b}] ,  [A_{\nu},A^{\b}] >      + m_{\mu\nu }   \cdot   <  [A_{\a},A_{\b}] , [A^{\a},A^{\b}] >  \\
 \notag
     && + O \big(h \cdot  (\derm A)^2 \big)   + O \big(  h  \cdot  A^2 \cdot \derm A \big)     + O \big(  h   \cdot  A^4 \big)  \,  , \\
\eea
where $P$\;, $Q$ and $G$ are defined in \eqref{definitionofthetermbigPinsourcetermsforeinstein}, \eqref{definitionofthetermbigQinsourcetermsforeinstein} and \eqref{definitionofthetermbigGinsourcetermsforeinstein}.
\begin{remark}
As a reminder, $m$ is defined to be the Minkowski metric in wave coordinates.
\end{remark}

    \end{lemma}
    
    \begin{proof}

We showed in Lemma \ref{EYMsystemashyperbolicPDE}, that in the Lorenz gauge and in wave coordinates, in other words for indices running only over wave coordinates, i.e. $\la, \mu, \si, \b, \nu, \a \in \{t, x^1, \ldots, x^n \} $, the Yang-Mills potential satisfies 
          \beaa
   \notag
g^{\la\mu} \pa_{\la}   \pa_{\mu}   A_{\si}      &=&  m^{\a\ga} m ^{\mu\la} (   \pa_{\si}  h_{\ga\la} )   \pa_{\a}A_{\mu}       +   \frac{1}{2}  m^{\a\mu}m^{\b\nu}   \big(   \pa_\a h_{\b\si} + \pa_\si h_{\b\a}- \pa_\b h_{\a\si}  \big)   \cdot  \big( \pa_{\mu}A_{\nu} - \pa_{\nu}A_{\mu}  \big) \\
 \notag
&& +      \frac{1}{2}  m^{\a\mu}m^{\b\nu}   \big(   \pa_\a h_{\b\si} + \pa_\si h_{\b\a}- \pa_\b h_{\a\si}  \big)   \cdot   [A_{\mu},A_{\nu}] \\
 \notag
 && -  m^{\a\mu} \big(  [ A_{\mu}, \pa_{\a} A_{\si} ]  +    [A_{\alpha},  \pa_{\mu}  A_{\si} - \pa_{\si} A_{\mu} ]    +    [A_{\alpha}, [A_{\mu},A_{\si}] ]  \big)  \\
  && + O( h \cdot  \pa h \cdot  \pa A) + O( h \cdot  \pa h \cdot  A^2) + O( h \cdot  A \cdot \pa A) + O( h \cdot  A^3) \, .
  \eeaa

 Since the Christoffel symbols for the connection $\derm$ are vanishing in wave coordinates, we could then write,
            \beaa
   \notag
g^{\la\mu} \pa_{\la}   \pa_{\mu}   A_{\si}     &=&  m^{\a\ga} m ^{\mu\la}  ( \derm_{\si}  h_{\ga\la} )   \derm_{\a}A_{\mu}     \\
&&  +   \frac{1}{2}  m^{\a\mu}m^{\b\nu}   \big(   \derm_\a h_{\b\si} +  \derm_\si h_{\b\a}-  \derm_\b h_{\a\si}  \big)   \cdot  \big( \derm_{\mu}A_{\nu} -  \derm_{\nu}A_{\mu}  \big) \\
 \notag
&& +      \frac{1}{2}  m^{\a\mu}m^{\b\nu}   \big(   \derm_\a h_{\b\si} +  \derm_\si h_{\b\a}-  \derm_\b h_{\a\si}  \big)   \cdot   [A_{\mu},A_{\nu}] \\
 \notag
 && -  m^{\a\mu} \big(  [ A_{\mu}, \derm_{\a} A_{\si} ]  +    [A_{\alpha},  \derm_{\mu}  A_{\si} - \derm_{\si} A_{\mu} ]    +    [A_{\alpha}, [A_{\mu},A_{\si}] ]  \big)  \\
  && + O( h \cdot  \pa h \cdot  \pa A) + O( h \cdot  \pa h \cdot  A^2) + O( h \cdot  A \cdot \pa A) + O( h \cdot  A^3) \, .
  \eeaa
 Also, we have
 \beaa
 g^{\la\mu} \pa_{\la}   \pa_{\mu}   A_{\si} = g^{\la\mu} \derm_{\la}   \derm_{\mu}   A_{\si} ,
 \eeaa
 
 which is a tensor in $\si$. Thus, the right hand side and the left hand side of the following equation is a tensor in $\si$ and corresponds to a full tensorial contraction on all other indices and hence the expression does not depend on the system of coordinates that we choose. 
 
 By lowering and highering indices with respect to the metric $m$, defined to be the Minkowski metric in wave coordinates, we get the result for wave equation satisfied for $A_{\si}$.

Similarly, we showed that in wave coordinates, the metric solution to the Einstein-Yang-Mills equations in the Lorenz gauge satisfies the following equation,

  \beaa
\notag
 && g^{\alpha\beta}\derm_\alpha \derm_\beta h_{\mu\nu} \\
  &=& P(\pa_\mu h,\pa_\nu h)  +  Q_{\mu\nu}(\pa h,\pa h)   + G_{\mu\nu}(h)(\pa h,\pa h)  \\
\notag
 &&   -4   m^{\si\b} \cdot  <   \derm_{\mu}A_{\b} - \derm_{\b}A_{\mu}  ,  \derm_{\nu}A_{\si} - \derm_{\si}A_{\nu}  >    \\
 \notag
 &&   + m_{\mu\nu }  m^{\si\b}  m^{\a\la}    \cdot  <  \derm_{\a}A_{\b} - \derm_{\b}A_{\a} , \derm_{\la}A_{\si} - \derm_{\si}A_{\la} >   \\
 \notag
&&           -4 m^{\si\b}  \cdot  \big( <   \derm_{\mu}A_{\b} - \derm_{\b}A_{\mu}  ,  [A_{\nu},A_{\si}] >   + <   [A_{\mu},A_{\b}] ,  \derm_{\nu}A_{\si} - \derm_{\si}A_{\nu}  > \big)  \\
\notag
&& + m_{\mu\nu }  m^{\si\b}  m^{\a\la}    \cdot \big(  <  \derm_{\a}A_{\b} - \derm_{\b}A_{\a} , [A_{\la},A_{\si}] >    +  <  [A_{\a},A_{\b}] , \derm_{\la}A_{\si} - \derm_{\si}A_{\la}  > \big) \\
\notag
 &&  -4 m^{\si\b}  \cdot   <   [A_{\mu},A_{\b}] ,  [A_{\nu},A_{\si}] >      + m_{\mu\nu }  m^{\si\b}  m^{\a\la}   \cdot   <  [A_{\a},A_{\b}] , [A_{\la},A_{\si}] >  \\
     && + O \big(h \cdot  (\pa A)^2 \big)   + O \big(  h  \cdot  A^2 \cdot \pa A \big)     + O \big(  h   \cdot  A^4 \big)  \,  .
\eeaa

Again, by lowering and highering indices with respect to the metric $m$, we obtain the result for the wave equation satisfied for $h_{\mu\nu}$.

    \end{proof}

\begin{lemma}\label{structureofLiederivativeZofthesourcestermsforAv}

  In the Lorenz gauge, we have for any $V \in \{\frac{\pa}{\pa x_\mu} \, \, | \, \,   \mu \in \{0, 1, \ldots, n \} \}$,
  \beaa
   \notag
 &&   \Lie_{Z^I}   ( g^{\la\mu} \derm_{\la}   \derm_{\mu}   A_{V} )    \\
   &=& \sum_{|J| +|K|+|L|+ |M| \leq |I|}  \big( \; O(     \derm (\Lie_{Z^J}  h)  \cdot     \derm (\Lie_{Z^K} A )   )     +  O(      \derm ( \Lie_{Z^J} h)  \cdot   \Lie_{Z^K}   A  \cdot     \Lie_{Z^L}  A    )    \\
        \notag
   &&  +   O(  \Lie_{Z^J} A   \cdot     \derm ( \Lie_{Z^K}  A)   )    + O(       \Lie_{Z^J} A  \cdot     \Lie_{Z^K} A   \cdot     \Lie_{Z^L} A  ) \\
           \notag
      &&         + O( \Lie_{Z^J}   h \cdot \derm ( \Lie_{Z^K}  h) \cdot  \derm (\Lie_{Z^L}  A) )  +  O( \Lie_{Z^J}  h \cdot    \derm (\Lie_{Z^K}  h ) \cdot  \Lie_{Z^L}  A \cdot \Lie_{Z^M} A)  \\
\notag
&& + O( \Lie_{Z^J} h \cdot   \Lie_{Z^K} A \cdot   \derm (\Lie_{Z^L} A) )    + O( \Lie_{Z^J} h \cdot  \Lie_{Z^K} A \cdot  \Lie_{Z^L} A \cdot \Lie_{Z^M} A) \; \big) \, ,
  \eeaa
  and therefore,
    \beaa
   \notag
 && |  \Lie_{Z^I}   ( g^{\la\mu} \derm_{\la}   \derm_{\mu}   A  )   | \\
   &\leq& \sum_{|J| +|K|+|L|+ |M| \leq |I|} \big( \; O(  |   \derm (\Lie_{Z^J}  h) | \cdot |    \derm (\Lie_{Z^K} A ) |  )    +  O(   |   \derm ( \Lie_{Z^J} h) | \cdot |  \Lie_{Z^K}   A  | \cdot |    \Lie_{Z^L}  A |   )     \\
           \notag
      &&       +   O( |\Lie_{Z^J} A   | \cdot |    \derm ( \Lie_{Z^K}  A) |  )    + O(   |    \Lie_{Z^J} A | \cdot   |  \Lie_{Z^K} A  | \cdot  |   \Lie_{Z^L} A | )    \\
           \notag
      &&         + O( | \Lie_{Z^J}   h | \cdot | \derm ( \Lie_{Z^K}  h) | \cdot | \derm (\Lie_{Z^L}  A) | )  +  O( | \Lie_{Z^J}  h | \cdot |   \derm (\Lie_{Z^K}  h ) | \cdot  | \Lie_{Z^L}  A | \cdot  | \Lie_{Z^M} A | \\
\notag
&& + O( | \Lie_{Z^J} h  | \cdot |  \Lie_{Z^K} A | \cdot |  \derm (\Lie_{Z^L} A) | )    + O( | \Lie_{Z^J} h | \cdot | \Lie_{Z^K} A | \cdot | \Lie_{Z^L} A | \cdot | \Lie_{Z^M} A | ) \; \big) \, .
  \eeaa

\end{lemma}

\begin{proof}

  In the Lorenz gauge, we have shown that we have,

  \beaa
  \notag
 g^{\la\mu} \derm_{\la}   \derm_{\mu}   A_{\si}     &=&    ( \derm_{\si}  h^{\a\mu} )  \cdot  \derm_{\a}A_{\mu}     \\
 \notag
&&  +   \frac{1}{2}    \big(   \derm^{\mu} h^{\nu}_{\, \, \, \si} +  \derm_\si h^{\nu\mu} -   \derm^{\nu} h^{\mu}_{\,\,\, \si}  \big)   \cdot  \big( \derm_{\mu}A_{\nu} -  \derm_{\nu}A_{\mu}  \big) \\
 \notag
&&     +   \frac{1}{2}    \big(   \derm^{\mu} h^{\nu}_{\, \, \, \si} +  \derm_\si h^{\nu\mu} -   \derm^{\nu} h^{\mu}_{\,\,\, \si}  \big)   \cdot   [A_{\mu},A_{\nu}] \\
 \notag
 && -  \big(  [ A_{\mu}, \derm^{\mu} A_{\si} ]  +    [A^{\mu},  \derm_{\mu}  A_{\si} - \derm_{\si} A_{\mu} ]    +    [A^{\mu}, [A_{\mu},A_{\si}] ]  \big)  \\
 \notag
  && + O( h \cdot  \pa h \cdot  \pa A) + O( h \cdot  \pa h \cdot  A^2) + O( h \cdot  A \cdot \pa A) + O( h \cdot  A^3) \, , \\
   &=& O(  \derm h \cdot  \derm A)  + O(  \derm h \cdot   A^2 ) + O(  A \cdot   \derm A ) + O(  A^3 ) \\
  && + O( h \cdot  \derm h \cdot  \derm A) + O( h \cdot  \derm h \cdot  A^2) + O( h \cdot  A \cdot \derm A) + O( h \cdot  A^3) \, .
  \eeaa

Differentiating, and using Definition \ref{definitionofbigOforLiederivatives}, we get for any $Z \in \cal Z$, and for any wave coordinate vector $V$, and using Definition \ref{definitionofbigOforLiederivatives}

\beaa
   \notag
 &&   \Lie_{Z}  ( g^{\la\mu} \derm_{\la}   \derm_{\mu}   A_{V} )     \\
 &=& O(  \derm h \cdot  \derm A)  + O(  \derm h \cdot   A^2 ) + O(  A \cdot   \derm A ) + O(  A^3 ) \\
  && + O( h \cdot  \derm h \cdot  \derm A) + O( h \cdot  \derm h \cdot  A^2) + O( h \cdot  A \cdot \derm A) + O( h \cdot  A^3) \\
   &+& O(    \Lie_{Z}  \derm h \cdot  \derm A    )   + O(   \derm h  \cdot   \Lie_{Z} \derm A    )  \\
       \notag
           && +  O(    \Lie_{Z} \derm  h \cdot  A^2 ) +  O(   \derm  h  \cdot   \Lie_{Z}  A  \cdot   A   )          +   O(  \Lie_{Z}  A   \cdot   \derm A  )   +   O( A   \cdot   \Lie_{Z}  \derm A  )   \\
       \notag
  &&     + O(    \Lie_{Z} A   \cdot   A^2 )  + O(  \Lie_{Z}  h \cdot  \derm  h \cdot  \derm  A) + O( h \cdot  \Lie_{Z}  \derm  h \cdot  \derm  A) \\
          \notag
     &&     + O( h \cdot  \derm  h \cdot   \Lie_{Z}  \derm  A)  + O(  \Lie_{Z}  h \cdot  \derm  h \cdot  A^2) + O( h \cdot   \Lie_{Z}  \derm  h \cdot  A^2) + O( h \cdot  \derm  h \cdot   \Lie_{Z}  A \cdot A) \\
&& + O( \Lie_{Z} h \cdot  A \cdot \derm  A)  + O( h \cdot  \Lie_{Z} A \cdot \derm  A)   + O( h \cdot  A \cdot  \Lie_{Z}\derm  A)  \\
&& + O( \Lie_{Z} h \cdot  A^3) + O( h \cdot \Lie_{Z} A \cdot  A^2) \, .
  \eeaa
 
 By induction, we obtain that for all $Z^I$, we have
\beaa
   \notag
 &&   \Lie_{Z^I}   ( g^{\la\mu} \derm_{\la}   \derm_{\mu}   A_{V} )    \\
   &=& \sum_{|J| +|K|+|L|+ |M| \leq |I|}  \big( \; O(    \Lie_{Z^J}  \derm h  \cdot   \Lie_{Z^K}  \derm A    )    +  O(    \Lie_{Z^J}   \derm  h  \cdot   \Lie_{Z^K}   A  \cdot     \Lie_{Z^L}  A    )\\
   \notag
   &&   +   O(  \Lie_{Z^J} A   \cdot   \Lie_{Z^K}   \derm A  )          + O(       \Lie_{Z^J} A  \cdot     \Lie_{Z^K} A   \cdot     \Lie_{Z^L} A  ) \\
           \notag
      &&         + O( \Lie_{Z^J}   h \cdot \Lie_{Z^K}  \derm  h \cdot \Lie_{Z^L}  \derm  A)  +  O( \Lie_{Z^J}  h \cdot  \Lie_{Z^K}   \derm  h \cdot  \Lie_{Z^L}  A \cdot \Lie_{Z^M} A)  \\
\notag
&& + O( \Lie_{Z^J} h \cdot   \Lie_{Z^K} A \cdot  \Lie_{Z^L} \derm  A)    + O( \Lie_{Z^J} h \cdot  \Lie_{Z^K} A \cdot  \Lie_{Z^L} A \cdot \Lie_{Z^M} A) \; \big) \, .
  \eeaa
  
 Using the fact that $\Lie_{Z^I} $ commutes with $\derm $, we obtain the result.
\end{proof}

\begin{lemma}\label{structureofLiederivativeZofthesourcestermsforhuv}

We have for any $U, V \in \{\frac{\pa}{\pa x_\mu} \, \, | \, \,   \mu \in \{0, 1, \ldots, n \} \}$,
  \beaa
   \notag
 &&   \Lie_{Z^I}   ( g^{\la\mu} \derm_{\la}   \derm_{\mu}    h_{UV} )    \\
   &=& \sum_{|J| +|K|+|L|+ |M| \leq |I|}  \big( \; O(     \derm (\Lie_{Z^J}  h)  \cdot     \derm (\Lie_{Z^K} h )   )         + O( \Lie_{Z^J}   h \cdot \derm ( \Lie_{Z^K}  h) \cdot  \derm (\Lie_{Z^L}  h) ) \\
       \notag
           &&  +   O(   \derm ( \Lie_{Z^J} A  ) \cdot     \derm ( \Lie_{Z^K}  A)   )    +  O(       \Lie_{Z^K}   A  \cdot     \Lie_{Z^L}  A   \cdot  \derm ( \Lie_{Z^J} A)    )      + O(       \Lie_{Z^J} A  \cdot     \Lie_{Z^K} A   \cdot     \Lie_{Z^L} A \cdot   \Lie_{Z^M} A   ) \\
           \notag
      &&     +  O( \Lie_{Z^J}  h \cdot    \derm (\Lie_{Z^K} A ) \cdot \derm ( \Lie_{Z^L}  A ) )  \\
\notag
&& + O( \Lie_{Z^J} h \cdot   \Lie_{Z^K} A \cdot   \Lie_{Z^L} A \cdot   \derm (\Lie_{Z^M} A) )    + O( \Lie_{Z^J} h \cdot  \Lie_{Z^K} A \cdot  \Lie_{Z^L} A \cdot \Lie_{Z^M} A \cdot \Lie_{Z^N} A) \; \big) \, ,
  \eeaa
  and therefore,
  \beaa
   \notag
 &&  | \Lie_{Z^I}   ( g^{\la\mu} \derm_{\la}   \derm_{\mu}    h ) |   \\
   &=& \sum_{|J| +|K|+|L|+ |M| \leq |I|}  \big( \; O(  |   \derm (\Lie_{Z^J}  h) | \cdot  |   \derm (\Lie_{Z^K} h )  | )         + O( | \Lie_{Z^J}   h | \cdot | \derm ( \Lie_{Z^K}  h) | \cdot | \derm (\Lie_{Z^L}  h) | ) \\
       \notag
           &&  +   O(  | \derm ( \Lie_{Z^J} A  ) | \cdot   |  \derm ( \Lie_{Z^K}  A)  | )    +  O(   |    \Lie_{Z^K}   A | \cdot   |  \Lie_{Z^L}  A |  \cdot | \derm ( \Lie_{Z^J} A) |   )   \\
           &&    + O(     |  \Lie_{Z^J} A | \cdot  |   \Lie_{Z^K} A |  \cdot   |  \Lie_{Z^L} A | \cdot |  \Lie_{Z^M} A  | )     +  O(|  \Lie_{Z^J}  h| \cdot  |  \derm (\Lie_{Z^K} A ) | \cdot | \derm ( \Lie_{Z^L}  A ) | )  \\
\notag
&& + O(|  \Lie_{Z^J} h| \cdot |  \Lie_{Z^K} A | \cdot  | \Lie_{Z^L} A | \cdot |  \derm (\Lie_{Z^M} A)| )    + O( | \Lie_{Z^J} h | \cdot  | \Lie_{Z^K} A | \cdot | \Lie_{Z^L} A |\cdot | \Lie_{Z^M} A |\cdot | \Lie_{Z^N} A |) \; \big) \,.
 \eeaa

\end{lemma}

  \begin{proof}
  We showed that
  \beaa
\notag
 && g^{\alpha\beta}\derm_\alpha \derm_\beta h_{\mu\nu} \\
 \notag
  &=& P(\pa_\mu h,\pa_\nu h)  +  Q_{\mu\nu}(\pa h,\pa h)   + G_{\mu\nu}(h)(\pa h,\pa h)  \\
\notag
 &&   -4     <   \derm_{\mu}A_{\b} - \derm_{\b}A_{\mu}  ,  \derm_{\nu}A^{\b} - \derm^{\b}A_{\nu}  >    \\
 \notag
 &&   + m_{\mu\nu }       \cdot  <  \derm_{\a}A_{\b} - \derm_{\b}A_{\a} , \derm_{\a} A^{\b} - \derm^{\b}A^{\a} >   \\
 \notag
&&           -4  \cdot  \big( <   \derm_{\mu}A_{\b} - \derm_{\b}A_{\mu}  ,  [A_{\nu},A^{\b}] >   + <   [A_{\mu},A_{\b}] ,  \derm_{\nu}A^{\b} - \derm^{\b}A_{\nu}  > \big)  \\
\notag
&& + m_{\mu\nu }    \cdot \big(  <  \derm_{\a}A_{\b} - \derm_{\b}A_{\a} , [A^{\a},A^{\b}] >    +  <  [A_{\a},A_{\b}] , \derm^{\a}A^{\b} - \derm^{\b}A^{\a}  > \big) \\
\notag
 &&  -4     <   [A_{\mu},A_{\b}] ,  [A_{\nu},A^{\b}] >      + m_{\mu\nu }   \cdot   <  [A_{\a},A_{\b}] , [A^{\a},A^{\b}] >  \\
 \notag
     && + O \big(h \cdot  (\pa A)^2 \big)   + O \big(  h  \cdot  A^2 \cdot \pa A \big)     + O \big(  h   \cdot  A^4 \big)  \,  . \\
     &=& P(\derm_\mu h,\derm_\nu h)  +  Q_{\mu\nu}(\derm h,\derm h)   + G_{\mu\nu}(h)(\derm h,\derm h)  \\
\notag
 &&+  O \big( (\derm A )^2  \big)  + O \big( A^2 \cdot \derm A  \big)   + O (A^4)   \\
\notag
     && + O \big(h \cdot  (\derm A)^2 \big)   + O \big(  h  \cdot  A^2 \cdot \derm A \big)     + O \big(  h   \cdot  A^4 \big)  \,  . \\
\eeaa

As stated previously, Lindblad and Rodnianski showed in Proposition 3.1 in \cite{LR10}, that
\beaa
 P(\pa_\mu h,\pa_\nu h) &=& \frac{1}{4} m^{\alpha\alpha^\prime}\pa_\mu h_{\alpha\alpha^\prime} \, m^{\beta\beta^\prime}\pa_\nu h_{\beta\beta^\prime}  -\frac{1}{2} m^{\alpha\alpha^\prime}m^{\beta\beta^\prime} \pa_\mu h_{\alpha\beta}\, \pa_\nu h_{\alpha^\prime\beta^\prime} \\
&=& \frac{1}{4} \derm_\mu h_{\a}^{\;\;\; \a} \cdot \derm_\nu  h_{\b}^{\;\;\; \b}  -\frac{1}{2}  \derm_\mu h_{\alpha\beta} \cdot \derm_\nu h^{\a\b} \\
&=& O ( (\derm h )^2 ) \, ,\\
\eeaa

and
 
\beaa
\notag
&& Q_{\mu\nu}(\pa h,\pa h) \\
\notag
&=& \pa_{\alpha} h_{\beta\mu}\, \, m^{\alpha\alpha^\prime}m^{\beta\beta^\prime} \pa_{\alpha^\prime} h_{\beta^\prime\nu} -m^{\alpha\alpha^\prime}m^{\beta\beta^\prime} \big(\pa_{\alpha} h_{\beta\mu}\,\,\pa_{\beta^\prime} h_{\alpha^\prime \nu} -\pa_{\beta^\prime} h_{\beta\mu}\,\,\pa_{\alpha} h_{\alpha^\prime\nu}\big)\\
\notag
&& +m^{\a\a'}m^{\b\b'}\big (\pa_\mu h_{\a'\b'}\, \pa_\a h_{\b\nu}- \pa_\a h_{\a'\b'} \,\pa_\mu h_{\b\nu}\big )  + m^{\a\a'}m^{\b\b'}\big (\pa_\nu h_{\a'\b'} \,\pa_\a h_{\b\mu} - \pa_\a h_{\a'\b'}\, \pa_\nu h_{\b\mu}\big )\\
\notag
&& +\frac 12 m^{\a\a'}m^{\b\b'}\big (\pa_{\b'} h_{\a\a'}\, \pa_\mu h_{\b\nu} - \pa_{\mu} h_{\a\a'}\, \pa_{\b'} h_{\b\nu} \big ) +\frac 12 m^{\a\a'}m^{\b\b'} \big (\pa_{\b'} h_{\a\a'}\, \pa_\nu h_{\b\mu} - \pa_{\nu} h_{\a\a'} \,\pa_{\b'} h_{\b\mu} \big ) \, ,\\
&=& \derm_{\alpha} h_{\beta\mu}\, \cdot \derm^{\a} h^{\b}_{\;\;\;\nu} -  \derm_{\alpha} h_{\beta\mu} \cdot \derm^{\b} h^{\a}_{\;\;\;\nu} +\derm^{\b} h_{\b\mu} \cdot \derm_{\alpha} h^{\a}_{\;\;\;\nu} \\
\notag
&& + \derm_\mu h^{\a\b}\cdot \derm_\a h_{\b\nu}- \derm_\a h^{\a\b} \cdot \derm_\mu h_{\b\nu} \\
&&   + \derm_\nu h^{\a\b} \cdot \derm_\a h_{\b\mu} - \derm_\a h^{\a\b}\, \derm_\nu h_{\b\mu} \\
\notag
&& +\derm^{\b} h_{\a}^{\;\;\; \a}\, \derm_\mu h_{\b\nu} - \frac{1}{2}  \derm_{\mu} h_{\a}^{\;\;\; \a} \cdot \derm^{\b} h_{\b\nu}  \\
&& +\frac{1}{2} \derm^{\b} h_{\a}^{\;\;\; \a} \cdot \derm_\nu h_{\b\mu} -\frac{1}{2}  \derm_{\nu} h_{\a}^{\;\;\; \a} \cdot \derm^{\b} h_{\b\mu} \, ,\\
&=& O ( (\derm h )^2 ) \, ,
\eeaa
and
\bea
G_{\mu\nu}(h)(\pa h,\pa h) =  O  (h \cdot (\derm h)^2) \, .
\eea

Thus,

  \beaa
\notag
 && g^{\alpha\beta}\derm_\alpha \derm_\beta h_{\mu\nu} \\
 \notag
     &=& O ( (\derm h )^2 )   + O  (h \cdot (\derm h)^2)  \\
\notag
 &&+  O \big( (\derm A )^2  \big)  + O \big( A^2 \cdot \derm A  \big)   + O (A^4)   \\
\notag
     && + O \big(h \cdot  (\derm A)^2 \big)   + O \big(  h  \cdot  A^2 \cdot \derm A \big)     + O \big(  h   \cdot  A^4 \big)  \,  . 
\eeaa

Differentiating the equation above and using the fact that the $\Lie_{Z^I} $ commutes with $\derm $, we obtain the result.

  \end{proof}

\section{Using the bootstrap assumption to exhibit the structure of the source terms of the Einstein-Yang-Mills system}

\subsection{Using the bootstrap assumption to exhibit the structure of the source terms for the Yang-Mills potential}\

Now, we want to use the bootstrap assumption to exhibit the structure of the source term for the wave equation on the Yang-Mills potential in the Lorenz gauge and in wave coordinates, depending also on the space-dimension $n$.

In fact, we would like to estimate the term $\int_0^t\int_{\Si_{\tau}}  | g^{\la\a} \derm_{\la}   \derm_{\a}  ( \Lie_{Z^I}  A) | \cdot |\derm ( \Lie_{Z^I}  A )| \, w  \cdot dx_1 \ldots dx_n d\tau $.
Using the inequality $a\cdot b \les a^2 + b^2$, we get,
\beaa
&& \int_0^t\int_{\Si_{\tau}} \sqrt{(1+\tau )^{1+\la}} \sqrt{w} \cdot | g^{\la\a} \derm_{\la}   \derm_{\a}  ( \Lie_{Z^I}  A) |\cdot  \frac{1}{\sqrt{(1+\tau )^{1+\la}}} \cdot |\derm ( \Lie_{Z^I}  A )| \, \sqrt{w}  \cdot dx_1 \ldots dx_n d\tau \\
&\les&\int_0^t\int_{\Si_{\tau}} \frac{ |\derm  (\Lie_{Z^I}  A) |^{2}}{(1+\tau)^{1+\la}} \cdot dx_1 \ldots dx_n d\tau \\
&&+  \int_0^t\int_{\Si_{\tau}}  (1+\tau )^{1+\la} \cdot | g^{\la\a} \derm_{\la}   \derm_{\a}  ( \Lie_{Z^I}  A) |^2 \, w  \cdot dx_1 \ldots dx_n d\tau \;
\eeaa
where one can choose $\la > 0$ so that $\int \frac{1}{(1+\tau)^{1+\la}} d\tau$ is integrable.

Yet, we have
\beaa
|  g^{\la\a} \derm_{\la}   \derm_{\a}  ( \Lie_{Z^I}  A)|  &\leq&  | \Lie_{Z^I} ( g^{\la\mu} \derm_{\la}   \derm_{\mu}   A  ) | \\
&&+|  g^{\la\mu} \derm_{\la}   \derm_{\mu}   ( \Lie_{ Z^I}   A ) - \Lie_{Z^I}  (g^{\la\mu} \derm_{\la}   \derm_{\mu}   A )| \, , 
\eeaa
and consequently, we have

\beaa
(1+\tau )\cdot |  g^{\la\a} \derm_{\la}   \derm_{\a}  ( \Lie_{Z^I}  A)|^2  &\leq& (1+\tau ) \cdot | \Lie_{Z^I} ( g^{\la\mu} \derm_{\la}   \derm_{\mu}   A  ) |^2 \\
&&+ (1+\tau ) \cdot |  g^{\la\mu} \derm_{\la}   \derm_{\mu}   ( \Lie_{ Z^I}   A ) - \Lie_{Z^I}  (g^{\la\mu} \derm_{\la}   \derm_{\mu}   A )|^2 \, .
\eeaa

\begin{lemma}\label{usingbootstraoptoshowstructorforsourcesonYangMillspotentialA}
We have
      \beaa
   \notag
 && |  \Lie_{Z^I}   ( g^{\la\mu} \derm_{\la}   \derm_{\mu}   A  )   | \\
   \notag
 &\les&   \sum_{|K|\leq |I|}  \Big( |    \derm (\Lie_{Z^K} A ) | \Big) \cdot E (   \lfloor \frac{|I|}{2} \rfloor + \lfloor  \frac{n}{2} \rfloor  + 1)\\
 \notag
 && \cdot \Big( \Big( \begin{cases}  \frac{\eps }{(1+t+|q|)^{\frac{(n-1)}{2}-\delta} (1+|q|)^{1+\gamma}},\quad\text{when }\quad q>0,\\
           \notag
      \frac{\eps  }{(1+t+|q|)^{\frac{(n-1)}{2}-\delta}(1+|q|)^{\frac{1}{2} }}  \,\quad\text{when }\quad q<0 . \end{cases} \Big) \\
      \notag
      && +   \Big( \begin{cases}  \frac{\eps }{(1+t+|q|)^{\frac{(n-1)}{2}-\delta} (1+|q|)^{\gamma}},\quad\text{when }\quad q>0,\\
           \notag
      \frac{\eps  \cdot (1+|q|)^{\frac{1}{2} } }{(1+t+|q|)^{\frac{(n-1)}{2}-\delta}}  \,\quad\text{when }\quad q<0 . \end{cases} \Big) \\
         \notag
         && + \Big( \begin{cases}  \frac{\eps }{(1+t+|q|)^{\frac{(n-1)}{2}-\delta} (1+|q|)^{\frac{(n-1)}{2}-\delta + 1+2\gamma}},\quad\text{when }\quad q>0,\\
           \notag
      \frac{\eps  }{ (1+t+|q|)^{\frac{(n-1)}{2}-\delta} \cdot (1+ |q|)^{\frac{(n-1)}{2}-\delta}  }  \,\quad\text{when }\quad q<0 . \end{cases} \Big) \\
      \notag
   && +    \Big(  \begin{cases}  \frac{\eps }{(1+t+|q|)^{\frac{(n-1)}{2}-\delta} (1+|q|)^{\frac{(n-1)}{2}-\delta +2\gamma}},\quad\text{when }\quad q>0,\\
           \notag
      \frac{\eps \cdot (1+ |q|)  }{(1+t+|q|)^{\frac{(n-1)}{2}-\delta} \cdot (1+ |q|)^{\frac{(n-1)}{2}-\delta} } \,\quad\text{when }\quad q<0 . \end{cases} \Big) \Big) 
               \notag
  \eeaa
    \beaa
   \notag
 &&  + \sum_{|K|\leq |I|}  \Big( |   \Lie_{Z^K} A  | \Big) \cdot E (   \lfloor \frac{|I|}{2} \rfloor + \lfloor  \frac{n}{2} \rfloor  + 1) \\
 \notag
 &&  \cdot \Big( \Big(  \begin{cases}  \frac{\eps }{(1+t+|q|)^{\frac{(n-1)}{2}-\delta} (1+|q|)^{\frac{(n-1)}{2}-\delta+1+2\gamma}},\quad\text{when }\quad q>0,\\
           \notag
      \frac{\eps  }{(1+t+|q|)^{\frac{(n-1)}{2}-\delta} \cdot ( 1+|q| )^{\frac{(n-1)}{2}-\delta} } \,\quad\text{when }\quad q<0 . \end{cases} \Big) \\
      \notag 
      && + \Big(  \begin{cases}  \frac{\eps }{(1+t+|q|)^{\frac{(n-1)}{2}-\delta} (1+|q|)^{1+\gamma}},\quad\text{when }\quad q>0,\\
           \notag
      \frac{\eps  }{(1+t+|q|)^{\frac{(n-1)}{2}-\delta}(1+|q|)^{\frac{1}{2} }}  \,\quad\text{when }\quad q<0 . \end{cases} \Big) \\
         \notag
       && +  \Big(  \begin{cases}  \frac{\eps }{(1+t+|q|)^{\frac{(n-1)}{2}-\delta} (1+|q|)^{\frac{(n-1)}{2}-\delta + 2\gamma}},\quad\text{when }\quad q>0,\\
           \notag
      \frac{\eps \cdot (1+ |q|)  }{(1+t+|q|)^{\frac{(n-1)}{2}-\delta } \cdot (1+|q|)^{\frac{(n-1)}{2}-\delta}} \,\quad\text{when }\quad q<0 . \end{cases} \Big)  \\
               \notag
     && +             \Big(  \begin{cases}  \frac{\eps }{(1+t+|q|)^{\frac{(n-1)}{2}-\delta} \cdot (1+|q|)^{(n-1)-2\delta + 1+3\gamma}},\quad\text{when }\quad q>0,\\
           \notag
      \frac{\eps  \cdot(1+|q|)^{\frac{1}{2} } }{(1+t+|q|)^{\frac{(n-1)}{2}-\delta} \cdot (1+|q|)^{(n-1)-2\delta}}  \,\quad\text{when }\quad q<0 . \end{cases} \Big)  \\
      \notag
  &&     +    \Big(  \begin{cases}  \frac{\eps }{(1+t+|q|)^{\frac{(n-1)}{2}-\delta} (1+|q|)^{\frac{(n-1)}{2}-\delta +1+2\gamma}},\quad\text{when }\quad q>0,\\
           \notag
      \frac{\eps  }{(1+t+|q|)^{\frac{(n-1)}{2}-\delta} \cdot (1+|q|)^{\frac{(n-1)}{2}-\delta}}  \,\quad\text{when }\quad q<0 . \end{cases} \Big)  \\
      \notag
  && +      \Big(  \begin{cases}  \frac{\eps }{(1+t+|q|)^{\frac{(n-1)}{2}-\delta} (1+|q|)^{(n-1)-2\delta + 3\gamma}},\quad\text{when }\quad q>0,\\
           \notag
      \frac{\eps \cdot (1+|q|)^{\frac{3}{2} }  }{(1+t+|q|)^{\frac{(n-1)}{2}-\delta} \cdot (1+|q|)^{(n-1)-2\delta}}  \,\quad\text{when }\quad q<0 . \end{cases} \Big) \Big)
               \notag
  \eeaa
        \beaa
   \notag
&&  + \Big( \sum_{|K|\leq |I|}   |    \derm (\Lie_{Z^K} h ) | \Big) \cdot E (   \lfloor \frac{|I|}{2} \rfloor + \lfloor  \frac{n}{2} \rfloor  + 1)  \\
 \notag
 && \cdot  \Big( \Big( \begin{cases}  \frac{\eps }{(1+t+|q|)^{\frac{(n-1)}{2}-\delta} (1+|q|)^{1+\gamma}},\quad\text{when }\quad q>0,\\
           \notag
      \frac{\eps  }{(1+t+|q|)^{\frac{(n-1)}{2}-\delta}(1+|q|)^{\frac{1}{2} }}  \,\quad\text{when }\quad q<0 . \end{cases} \Big) \\
      \notag
    && +   \Big(  \begin{cases}  \frac{\eps }{(1+t+|q|)^{\frac{(n-1)}{2}-\delta} (1+|q|)^{\frac{(n-1)}{2}-\delta + 2\gamma}},\quad\text{when }\quad q>0,\\
           \notag
      \frac{\eps \cdot (1+ |q|)  }{(1+t+|q|)^{\frac{(n-1)}{2}-\delta} \cdot (1+ |q|)^{\frac{(n-1)}{2}-\delta} }  \,\quad\text{when }\quad q<0 . \end{cases} \Big) \\
      \notag
&& +        \Big( \begin{cases}  \frac{\eps }{(1+t+|q|)^{\frac{(n-1)}{2}-\delta} (1+|q|)^{\frac{(n-1)}{2}-\delta+1+2\gamma}},\quad\text{when }\quad q>0,\\
           \notag
      \frac{\eps  }{(1+t+|q|)^{\frac{(n-1)}{2}-\delta} \cdot (1+|q|)^{\frac{(n-1)}{2}-\delta}  } ,\quad\text{when }\quad q<0 . \end{cases} \Big) \\
     \notag
     && + \Big( \begin{cases}  \frac{\eps }{(1+t+|q|)^{\frac{(n-1)}{2}-\delta} (1+|q|)^{(n-1) - 2\de + 3\gamma}},\quad\text{when }\quad q>0,\\
           \notag
      \frac{\eps \cdot (1+|q|)^{\frac{3}{2} }  }{(1+t+|q|)^{\frac{(n-1)}{2}-\delta} \cdot (1+|q|)^{(n-1)-2\delta}}  \,\quad\text{when }\quad q<0 . \end{cases} \Big) \Big) 
               \notag
  \eeaa
       \beaa
   \notag
 && + \Big( \sum_{|K|\leq |I|}   |   \Lie_{Z^K} h  | \Big) \cdot E (   \lfloor \frac{|I|}{2} \rfloor + \lfloor  \frac{n}{2} \rfloor  + 1)  \\
 \notag
 && \cdot \Big( \Big(  \begin{cases}  \frac{\eps }{(1+t+|q|)^{(n-1)-2\delta} (1+|q|)^{2+2\gamma}},\quad\text{when }\quad q>0,\\
           \notag
      \frac{\eps  }{(1+t+|q|)^{(n-1)-2\delta}(1+|q|)}  \,\quad\text{when }\quad q<0 . \end{cases} \Big)  \\
      \notag
  && +      \Big( \begin{cases}  \frac{\eps }{(1+t+|q|)^{(n-1)-2\delta} \cdot (1+|q|)^{\frac{(n-1)}{2}-\delta + 1+3\gamma}},\quad\text{when }\quad q>0,\\
           \notag
      \frac{\eps  \cdot(1+|q|)^{\frac{1}{2} } }{(1+t+|q|)^{(n-1)-2\delta} \cdot (1+|q|)^{\frac{(n-1)}{2}-\delta} }  \,\quad\text{when }\quad q<0 . \end{cases} \Big)  \\
      \notag
   && +   \Big(  \begin{cases}  \frac{\eps }{(1+t+|q|)^{(n-1)-2\delta} (1+|q|)^{1+2\gamma}},\quad\text{when }\quad q>0,\\
           \notag
      \frac{\eps  }{(1+t+|q|)^{(n-1)-2\delta}}  \,\quad\text{when }\quad q<0 . \end{cases} \Big) \\
      \notag
   && +    \Big(  \begin{cases}  \frac{\eps }{(1+t+|q|)^{(n-1)-2\delta} (1+|q|)^{\frac{(n-1)}{2}-\delta +3\gamma}},\quad\text{when }\quad q>0,\\
           \notag
      \frac{\eps \cdot (1+|q|)^{\frac{3}{2} }  }{(1+t+|q|)^{(n-1)-2\delta} \cdot (1+ |q|)^{\frac{(n-1)}{2}-\delta} }  \,\quad\text{when }\quad q<0 . \end{cases} \Big) \Big) \; .
               \notag
  \eeaa

\end{lemma}

\begin{proof}

We have already estimated
    \beaa
   \notag
 && |  \Lie_{Z^I}   ( g^{\la\mu} \derm_{\la}   \derm_{\mu}   A  )   | \\
   &\leq& \sum_{|J| +|K|+|L|+ |M| \leq |I|} \big( \; O(  |   \derm (\Lie_{Z^J}  h) | \cdot |    \derm (\Lie_{Z^K} A ) |  )      +  O(   |   \derm ( \Lie_{Z^J} h) | \cdot |  \Lie_{Z^K}   A  | \cdot |    \Lie_{Z^L}  A |   )    \\
   \notag
   &&   +   O( |\Lie_{Z^J} A   | \cdot |    \derm ( \Lie_{Z^K}  A) |  )    + O(   |    \Lie_{Z^J} A | \cdot   |  \Lie_{Z^K} A  | \cdot  |   \Lie_{Z^L} A | )    \\
           \notag
      &&         + O( | \Lie_{Z^J}   h | \cdot | \derm ( \Lie_{Z^K}  h) | \cdot | \derm (\Lie_{Z^L}  A) | )  +  O( | \Lie_{Z^J}  h | \cdot |   \derm (\Lie_{Z^K}  h ) | \cdot  | \Lie_{Z^L}  A | \cdot \Lie_{Z^M} A | \\
\notag
&& + O( | \Lie_{Z^J} h  | \cdot |  \Lie_{Z^K} A | \cdot |  \derm (\Lie_{Z^L} A) | )    + O( | \Lie_{Z^J} h | \cdot | \Lie_{Z^K} A | \cdot | \Lie_{Z^L} A | \cdot | \Lie_{Z^M} A | ) \; \big) \, .
  \eeaa

Yet, we can now look at each term one by one.

\textbf{Terms of the type $ |   \derm (\Lie_{Z^J}  h) | \cdot |    \derm (\Lie_{Z^K} A ) |$:}\\

We have
      \beaa
   \notag
&& \sum_{|J| +|K|\leq |I|}   |   \derm (\Lie_{Z^J}  h) | \cdot |    \derm (\Lie_{Z^K} A ) |   \\
 &\leq&  \sum_{|J| \leq  \lfloor \frac{|I|}{2} \rfloor , \; |K|\leq |I|}    |   \derm (\Lie_{Z^J}  h) | \cdot |    \derm (\Lie_{Z^K} A ) |   \\
 &&  +   \sum_{|J| \leq |I| , \; |K| \leq  \lfloor \frac{|I|}{2} \rfloor  }   |   \derm (\Lie_{Z^J}  h) | \cdot |    \derm (\Lie_{Z^K} A ) |   \; .
  \eeaa
  
  However, for $|J|, \; |K| \leq \lfloor  \frac{|I|}{2} \rfloor  $, based on what we have proved in Lemma \ref{aprioriestimatesongradientoftheLiederivativesofthefields}, we have
        \beaa
   \notag
&& |   \derm (\Lie_{Z^J}  h) |  + |   \derm (\Lie_{Z^K}  A) | \les E (   \lfloor \frac{|I|}{2} \rfloor + \lfloor  \frac{n}{2} \rfloor  +1 )  \cdot \begin{cases}  \frac{\eps }{(1+t+|q|)^{\frac{(n-1)}{2}-\delta} (1+|q|)^{1+\gamma}},\quad\text{when }\quad q>0,\\
           \notag
      \frac{\eps  }{(1+t+|q|)^{\frac{(n-1)}{2}-\delta}(1+|q|)^{\frac{1}{2} }}  \,\quad\text{when }\quad q<0 . \end{cases} \\
      \eeaa

Thus, we could write
     \bea
   \notag
&& \sum_{|J| +|K|\leq |I|}   |   \derm (\Lie_{Z^J}  h) | \cdot |    \derm (\Lie_{Z^K} A ) |   \\
   \notag
 &\leq& \Big( E (   \lfloor \frac{|I|}{2} \rfloor + \lfloor  \frac{n}{2} \rfloor  + 1)  \cdot \begin{cases}  \frac{\eps }{(1+t+|q|)^{\frac{(n-1)}{2}-\delta} (1+|q|)^{1+\gamma}},\quad\text{when }\quad q>0,\\
           \notag
      \frac{\eps  }{(1+t+|q|)^{\frac{(n-1)}{2}-\delta}(1+|q|)^{\frac{1}{2} }}  \,\quad\text{when }\quad q<0 . \end{cases} \Big) \\
      && \cdot \sum_{|K|\leq |I|}  \Big( |    \derm (\Lie_{Z^K} h ) | +  |    \derm (\Lie_{Z^K} A ) |\Big) \; . 
  \eea
  
  \textbf{Terms of the type $ |   \derm ( \Lie_{Z^J} h) | \cdot |  \Lie_{Z^K}   A  | \cdot |    \Lie_{Z^L}  A |$:}\\

Similarly,
\beaa
&& \sum_{|J| +|K| + |L| \leq |I|}    |   \derm ( \Lie_{Z^J} h) | \cdot |  \Lie_{Z^K}   A  | \cdot |    \Lie_{Z^L}  A | \\
  &\les&   \sum_{|J| \leq  \lfloor \frac{|I|}{2} \rfloor  , \; |K| \leq \lfloor \frac{|I|}{2} \rfloor   , \; |L| \leq |I|}        |   \derm ( \Lie_{Z^J} h) | \cdot |  \Lie_{Z^K}   A  | \cdot |    \Lie_{Z^L}  A | \\
  && + \sum_{|J| \leq   \lfloor \frac{|I|}{2} \rfloor  ,  \; |K|  \leq |I|, \; |L| \leq  \lfloor \frac{|I|}{2} \rfloor }    |   \derm ( \Lie_{Z^J} h) | \cdot |  \Lie_{Z^K}   A  | \cdot |    \Lie_{Z^L}  A | \\
   && + \sum_{|J| \leq  |I|  , \; |K| \leq \lfloor \frac{|I|}{2} \rfloor  , \; |L| \leq \lfloor \frac{|I|}{2} \rfloor }   |   \derm ( \Lie_{Z^J} h) | \cdot |  \Lie_{Z^K}   A  | \cdot |    \Lie_{Z^L}  A | \; .
\eeaa  
Again using that $E (   \lfloor \frac{|I|}{2} \rfloor + \lfloor  \frac{n}{2} \rfloor  + 1) \leq 1 $,  and $\eps \leq 1$, so that $E^2 \cdot \eps^2 \leq E \cdot \eps $\;, we have based on what we showed in Lemma \ref{aprioriestimatesongradientoftheLiederivativesofthefields}, that for $|J|,\; |K|, \; |L| \leq \lfloor \frac{|I|}{2} \rfloor $,
 \beaa
    |   \derm ( \Lie_{Z^J} h) | \cdot |  \Lie_{Z^K}   A  | &\les&  \Big( E (   \lfloor \frac{|I|}{2} \rfloor + \lfloor  \frac{n}{2} \rfloor  + 1)  \cdot \begin{cases}  \frac{\eps }{(1+t+|q|)^{\frac{(n-1)}{2}-\delta} (1+|q|)^{1+\gamma}},\quad\text{when }\quad q>0,\\
           \notag
      \frac{\eps  }{(1+t+|q|)^{\frac{(n-1)}{2}-\delta}(1+|q|)^{\frac{1}{2} }}  \,\quad\text{when }\quad q<0 . \end{cases} \Big) \\
      && \cdot  \Big( E (   \lfloor \frac{|I|}{2} \rfloor + \lfloor  \frac{n}{2} \rfloor  + 1)  \cdot \begin{cases}  \frac{\eps }{(1+t+|q|)^{\frac{(n-1)}{2}-\delta} (1+|q|)^{\gamma}},\quad\text{when }\quad q>0,\\
           \notag
      \frac{\eps  }{(1+t+|q|)^{\frac{(n-1)}{2}-\delta}(1+|q|)^{\frac{-1}{2} }}  \,\quad\text{when }\quad q<0 . \end{cases} \Big) \\
        &\les &  \Big( E (   \lfloor \frac{|I|}{2} \rfloor + \lfloor  \frac{n}{2} \rfloor  + 1)  \cdot \begin{cases}  \frac{\eps }{(1+t+|q|)^{(n-1)-2\delta} (1+|q|)^{1+2\gamma}},\quad\text{when }\quad q>0,\\
           \notag
      \frac{\eps  }{(1+t+|q|)^{(n-1)-2\delta}}  \,\quad\text{when }\quad q<0 . \end{cases} \Big) \; ,
 \eeaa
 and
  \beaa
     |  \Lie_{Z^K}   A  | \cdot |  \Lie_{Z^L}   A  | &\les&  \Big( E (   \lfloor \frac{|I|}{2} \rfloor + \lfloor  \frac{n}{2} \rfloor  + 1)  \cdot \begin{cases}  \frac{\eps }{(1+t+|q|)^{\frac{(n-1)}{2}-\delta} (1+|q|)^{\gamma}},\quad\text{when }\quad q>0,\\
           \notag
      \frac{\eps  }{(1+t+|q|)^{\frac{(n-1)}{2}-\delta}(1+|q|)^{\frac{-1}{2} }}  \,\quad\text{when }\quad q<0 . \end{cases} \Big)^2 \\
        &\les &  \Big( E (   \lfloor \frac{|I|}{2} \rfloor + \lfloor  \frac{n}{2} \rfloor  + 1)  \cdot \begin{cases}  \frac{\eps }{(1+t+|q|)^{(n-1)-2\delta} (1+|q|)^{2\gamma}},\quad\text{when }\quad q>0,\\
           \notag
      \frac{\eps \cdot (1+ |q|)  }{(1+t+|q|)^{(n-1)-2\delta}}  \,\quad\text{when }\quad q<0 . \end{cases} \Big) \; .
 \eeaa
 Consequently,
 \bea
    \notag
&& \sum_{|J| +|K| + |L| \leq |I|}    |   \derm ( \Lie_{Z^J} h) | \cdot |  \Lie_{Z^K}   A  | \cdot |    \Lie_{Z^L}  A | \\
   \notag
  &\les&     \Big( E (   \lfloor \frac{|I|}{2} \rfloor + \lfloor  \frac{n}{2} \rfloor  + 1)  \cdot \begin{cases}  \frac{\eps }{(1+t+|q|)^{(n-1)-2\delta} (1+|q|)^{1+2\gamma}},\quad\text{when }\quad q>0,\\
           \notag
      \frac{\eps  }{(1+t+|q|)^{(n-1)-2\delta}}  \,\quad\text{when }\quad q<0 . \end{cases} \Big) \cdot   \sum_{|K| \leq |I|}    |  \Lie_{Z^K}   A  | \\
         \notag
      && +   \Big( E (   \lfloor \frac{|I|}{2} \rfloor + \lfloor  \frac{n}{2} \rfloor  + 1)  \cdot \begin{cases}  \frac{\eps }{(1+t+|q|)^{(n-1)-2\delta} (1+|q|)^{2\gamma}},\quad\text{when }\quad q>0,\\
           \notag
      \frac{\eps \cdot (1+ |q|)  }{(1+t+|q|)^{(n-1)-2\delta}}  \,\quad\text{when }\quad q<0 . \end{cases} \Big)  \cdot    \sum_{|K| \leq |I|}    |   \derm ( \Lie_{Z^K} h) | \; .\\
 \eea

  \textbf{Terms of the type $  |\Lie_{Z^J} A   | \cdot |    \derm ( \Lie_{Z^K}  A) |$:} \\
  
Similarly,
\bea
         \notag
 && \sum_{|J| +|K|\leq |I|}   |\Lie_{Z^J} A   | \cdot |    \derm ( \Lie_{Z^K}  A) | \\
 \notag
  &\les& \Big( E (   \lfloor \frac{|I|}{2} \rfloor + \lfloor  \frac{n}{2} \rfloor  + 1)  \cdot \begin{cases}  \frac{\eps }{(1+t+|q|)^{\frac{(n-1)}{2}-\delta} (1+|q|)^{1+\gamma}},\quad\text{when }\quad q>0,\\
           \notag
      \frac{\eps  }{(1+t+|q|)^{\frac{(n-1)}{2}-\delta}(1+|q|)^{\frac{1}{2} }}  \,\quad\text{when }\quad q<0 . \end{cases} \Big) \\
         \notag
      && \cdot \sum_{|K|\leq |I|}   |  \Lie_{Z^K} A | \\
        \notag
      && +  \Big( E (   \lfloor \frac{|I|}{2} \rfloor + \lfloor  \frac{n}{2} \rfloor  + 1)  \cdot \begin{cases}  \frac{\eps }{(1+t+|q|)^{\frac{(n-1)}{2}-\delta} (1+|q|)^{\gamma}},\quad\text{when }\quad q>0,\\
           \notag
      \frac{\eps \cdot (1+|q|)^{\frac{1}{2} }  }{(1+t+|q|)^{\frac{(n-1)}{2}-\delta}}  \,\quad\text{when }\quad q<0 . \end{cases} \Big) \\
         \notag
      && \cdot \sum_{|K|\leq |I|}   | \derm ( \Lie_{Z^K} A ) | \; . \\
  \eea

  \textbf{Terms of the type $ |    \Lie_{Z^J} A | \cdot   |  \Lie_{Z^K} A  | \cdot  |   \Lie_{Z^L} A |  $:} \\
  
Also, 
\bea  
         \notag
 && \sum_{|J| +|K| + |L| \leq |I|}  |    \Lie_{Z^J} A | \cdot   |  \Lie_{Z^K} A  | \cdot  |   \Lie_{Z^L} A |  \\
          \notag
 &\les&  \Big( E (   \lfloor \frac{|I|}{2} \rfloor + \lfloor  \frac{n}{2} \rfloor  + 1)  \cdot \begin{cases}  \frac{\eps }{(1+t+|q|)^{\frac{(n-1)}{2}-\delta} (1+|q|)^{\gamma}},\quad\text{when }\quad q>0,\\
           \notag
      \frac{\eps  }{(1+t+|q|)^{\frac{(n-1)}{2}-\delta}(1+|q|)^{\frac{-1}{2} }}  \,\quad\text{when }\quad q<0 . \end{cases} \Big)^2 \\
               \notag
      && \cdot \sum_{|K|\leq |I|}   | \Lie_{Z^K} A  | \;  \\
               \notag
       &\les& \Big( E (   \lfloor \frac{|I|}{2} \rfloor + \lfloor  \frac{n}{2} \rfloor  + 1)  \cdot \begin{cases}  \frac{\eps }{(1+t+|q|)^{(n-1)-2\delta} (1+|q|)^{2\gamma}},\quad\text{when }\quad q>0,\\
           \notag
      \frac{\eps \cdot (1+ |q|)  }{(1+t+|q|)^{(n-1)-2\delta}}  \,\quad\text{when }\quad q<0 . \end{cases} \Big)  \\
               \notag
      && \cdot \sum_{|K|\leq |I|}   |  \Lie_{Z^K} A  | \; . \\
  \eea

  \textbf{Terms of the type $   | \Lie_{Z^J}   h | \cdot | \derm ( \Lie_{Z^K}  h) | \cdot | \derm (\Lie_{Z^L}  A) | $:} \\ 
  
We have,
\beaa  
         \notag
  && \sum_{|J| +|K| + |L| \leq |I|}   | \Lie_{Z^J}   h | \cdot | \derm ( \Lie_{Z^K}  h) | \cdot | \derm (\Lie_{Z^L}  A) |  \\
  &\les&     \Big( E (   \lfloor \frac{|I|}{2} \rfloor + \lfloor  \frac{n}{2} \rfloor  + 1)  \cdot \begin{cases}  \frac{\eps }{(1+t+|q|)^{(n-1)-2\delta} (1+|q|)^{1+2\gamma}},\quad\text{when }\quad q>0,\\
           \notag
      \frac{\eps  }{(1+t+|q|)^{(n-1)-2\delta}}  \,\quad\text{when }\quad q<0 . \end{cases} \Big) \cdot   \sum_{|K| \leq |I|}    | \derm (\Lie_{Z^K}  A) | \\
         \notag
&&       +    \Big( E (   \lfloor \frac{|I|}{2} \rfloor + \lfloor  \frac{n}{2} \rfloor  + 1)  \cdot \begin{cases}  \frac{\eps }{(1+t+|q|)^{(n-1)-2\delta} (1+|q|)^{1+2\gamma}},\quad\text{when }\quad q>0,\\
           \notag
      \frac{\eps  }{(1+t+|q|)^{(n-1)-2\delta}}  \,\quad\text{when }\quad q<0 . \end{cases} \Big) \cdot   \sum_{|K| \leq |I|}    | \derm (\Lie_{Z^K}  h) | \\
        \notag
 && +         \Big( E (   \lfloor \frac{|I|}{2} \rfloor + \lfloor  \frac{n}{2} \rfloor  + 1)  \cdot \begin{cases}  \frac{\eps }{(1+t+|q|)^{\frac{(n-1)}{2}-\delta} (1+|q|)^{1+\gamma}},\quad\text{when }\quad q>0,\\
           \notag
      \frac{\eps  }{(1+t+|q|)^{\frac{(n-1)}{2}-\delta}(1+|q|)^{\frac{1}{2} }}  \,\quad\text{when }\quad q<0 . \end{cases} \Big)^2   \cdot   \sum_{|K| \leq |I|}    | \Lie_{Z^K}  h | \;.
      \notag
          \eeaa
 
 Thus,
 
 \bea  
         \notag
  && \sum_{|J| +|K| + |L| \leq |I|}   | \Lie_{Z^J}   h | \cdot | \derm ( \Lie_{Z^K}  h) | \cdot | \derm (\Lie_{Z^L}  A) |  \\
  &\les&    \Big( E (   \lfloor \frac{|I|}{2} \rfloor + \lfloor  \frac{n}{2} \rfloor  + 1)  \cdot \begin{cases}  \frac{\eps }{(1+t+|q|)^{(n-1)-2\delta} (1+|q|)^{1+2\gamma}},\quad\text{when }\quad q>0,\\
           \notag
      \frac{\eps  }{(1+t+|q|)^{(n-1)-2\delta}}  \,\quad\text{when }\quad q<0 . \end{cases} \Big) \\
      \notag
     &&   \cdot   \sum_{|K| \leq |I|} \Big(   | \derm (\Lie_{Z^K}  A) | +    | \derm (\Lie_{Z^K}  h) | \Big) \\
        \notag
 && +         \Big( E (   \lfloor \frac{|I|}{2} \rfloor + \lfloor  \frac{n}{2} \rfloor  + 1)  \cdot \begin{cases}  \frac{\eps }{(1+t+|q|)^{(n-1)-2\delta} (1+|q|)^{2+2\gamma}},\quad\text{when }\quad q>0,\\
           \notag
      \frac{\eps  }{(1+t+|q|)^{(n-1)-2\delta}(1+|q|)}  \,\quad\text{when }\quad q<0 . \end{cases} \Big)   \cdot   \sum_{|K| \leq |I|}    | \Lie_{Z^K}  h | \, . \\
          \eea
 
    \textbf{Terms of the type $ | \Lie_{Z^J}  h | \cdot |   \derm (\Lie_{Z^K}  h ) | \cdot  | \Lie_{Z^L}  A | \cdot \Lie_{Z^M} A |  $:} \\ 
        
We have,
\beaa
\notag
 && \sum_{|J| +|K| + |L| + |M| \leq |I|}    | \Lie_{Z^J}  h | \cdot |   \derm (\Lie_{Z^K}  h ) | \cdot  | \Lie_{Z^L}  A | \cdot \Lie_{Z^M} A | \\
 \notag
   &\les&     \Big( E (   \lfloor \frac{|I|}{2} \rfloor + \lfloor  \frac{n}{2} \rfloor  + 1)  \cdot \begin{cases}  \frac{\eps }{(1+t+|q|)^{(n-1)-2\delta} (1+|q|)^{1+2\gamma}},\quad\text{when }\quad q>0,\\
           \notag
      \frac{\eps  }{(1+t+|q|)^{(n-1)-2\delta}}  \,\quad\text{when }\quad q<0 . \end{cases} \Big)  \\
      \notag
 && \cdot     \Big( E (   \lfloor \frac{|I|}{2} \rfloor + \lfloor  \frac{n}{2} \rfloor  + 1)  \cdot \begin{cases}  \frac{\eps }{(1+t+|q|)^{\frac{(n-1)}{2}-\delta} (1+|q|)^{\gamma}},\quad\text{when }\quad q>0,\\
           \notag
      \frac{\eps  }{(1+t+|q|)^{\frac{(n-1)}{2}-\delta}(1+|q|)^{\frac{-1}{2} }}  \,\quad\text{when }\quad q<0 . \end{cases} \Big) \\
        \notag
  &&      \cdot   \sum_{|K| \leq |I|}    | \Lie_{Z^K}  h | \\
  && +    \Big( E (   \lfloor \frac{|I|}{2} \rfloor + \lfloor  \frac{n}{2} \rfloor  + 1)  \cdot \begin{cases}  \frac{\eps }{(1+t+|q|)^{(n-1)-2\delta} (1+|q|)^{1+2\gamma}},\quad\text{when }\quad q>0,\\
           \notag
      \frac{\eps  }{(1+t+|q|)^{(n-1)-2\delta}}  \,\quad\text{when }\quad q<0 . \end{cases} \Big)  \\
      \notag
 && \cdot     \Big( E (   \lfloor \frac{|I|}{2} \rfloor + \lfloor  \frac{n}{2} \rfloor  + 1)  \cdot \begin{cases}  \frac{\eps }{(1+t+|q|)^{\frac{(n-1)}{2}-\delta} (1+|q|)^{\gamma}},\quad\text{when }\quad q>0,\\
           \notag
      \frac{\eps  }{(1+t+|q|)^{\frac{(n-1)}{2}-\delta}(1+|q|)^{\frac{-1}{2} }}  \,\quad\text{when }\quad q<0 . \end{cases} \Big) \\
        \notag
  &&      \cdot   \sum_{|K| \leq |I|}    | \Lie_{Z^K}  A | \\
  \notag
  && +   \Big( E (   \lfloor \frac{|I|}{2} \rfloor + \lfloor  \frac{n}{2} \rfloor  + 1)  \cdot \begin{cases}  \frac{\eps }{(1+t+|q|)^{\frac{(n-1)}{2}-\delta} (1+|q|)^{\gamma}},\quad\text{when }\quad q>0,\\
           \notag
      \frac{\eps  }{(1+t+|q|)^{\frac{(n-1)}{2}-\delta}(1+|q|)^{\frac{-1}{2} }}  \,\quad\text{when }\quad q<0 . \end{cases} \Big)^3 \\
               \notag
      && \cdot \sum_{|K|\leq |I|}   | \derm ( \Lie_{Z^K} h ) | \; . \\
  \eeaa
 Hence,
 \bea
\notag
 && \sum_{|J| +|K| + |L| + |M| \leq |I|}    | \Lie_{Z^J}  h | \cdot |   \derm (\Lie_{Z^K}  h ) | \cdot  | \Lie_{Z^L}  A | \cdot \Lie_{Z^M} A | \\
 \notag
   &\les&     \Big( E (   \lfloor \frac{|I|}{2} \rfloor + \lfloor  \frac{n}{2} \rfloor  + 1)  \cdot \begin{cases}  \frac{\eps }{(1+t+|q|)^{\frac{3(n-1)}{2}-3\delta} (1+|q|)^{1+3\gamma}},\quad\text{when }\quad q>0,\\
           \notag
      \frac{\eps  \cdot(1+|q|)^{\frac{1}{2} } }{(1+t+|q|)^{\frac{3(n-1)}{2}-3\delta}}  \,\quad\text{when }\quad q<0 . \end{cases} \Big)  \\
      \notag
  &&      \cdot   \sum_{|K| \leq |I|}  \Big(   | \Lie_{Z^K}  h | +   | \Lie_{Z^K}  A | \Big) \\
  \notag
  && +   \Big( E (   \lfloor \frac{|I|}{2} \rfloor + \lfloor  \frac{n}{2} \rfloor  + 1)  \cdot \begin{cases}  \frac{\eps }{(1+t+|q|)^{\frac{3(n-1)}{2}-3\delta} (1+|q|)^{3\gamma}},\quad\text{when }\quad q>0,\\
           \notag
      \frac{\eps \cdot (1+|q|)^{\frac{3}{2} }  }{(1+t+|q|)^{\frac{3(n-1)}{2}-3\delta}}  \,\quad\text{when }\quad q<0 . \end{cases} \Big) \\
               \notag
      && \cdot \sum_{|K|\leq |I|}   | \derm ( \Lie_{Z^K} h ) | \; . \\
  \eea

     \textbf{Terms of the type $  | \Lie_{Z^J} h  | \cdot |  \Lie_{Z^K} A | \cdot |  \derm (\Lie_{Z^L} A) |   $:} \\ 
     
 Also,
 \bea
 \notag
 && \sum_{|J| +|K| + |L| | \leq |I|}   | \Lie_{Z^J} h  | \cdot |  \Lie_{Z^K} A | \cdot |  \derm (\Lie_{Z^L} A) |   \\
 \notag
&\les&    \Big( E (   \lfloor \frac{|I|}{2} \rfloor + \lfloor  \frac{n}{2} \rfloor  + 1)  \cdot \begin{cases}  \frac{\eps }{(1+t+|q|)^{(n-1)-2\delta} (1+|q|)^{2\gamma}},\quad\text{when }\quad q>0,\\
           \notag
      \frac{\eps \cdot (1+ |q|)  }{(1+t+|q|)^{(n-1)-2\delta}}  \,\quad\text{when }\quad q<0 . \end{cases} \Big)  \\
               \notag
      && \cdot \sum_{|K|\leq |I|}   | \derm ( \Lie_{Z^K} A ) | \\
   \notag
 && +    \Big( E (   \lfloor \frac{|I|}{2} \rfloor + \lfloor  \frac{n}{2} \rfloor  + 1)  \cdot \begin{cases}  \frac{\eps }{(1+t+|q|)^{(n-1)-2\delta} (1+|q|)^{1+2\gamma}},\quad\text{when }\quad q>0,\\
           \notag
      \frac{\eps  }{(1+t+|q|)^{(n-1)-2\delta}}  \,\quad\text{when }\quad q<0 . \end{cases} \Big)  \\
      \notag
   &&   \cdot   \sum_{|K|\leq |I|} \Big(  | \Lie_{Z^K} A  |  +   | \Lie_{Z^K} h  | \Big)  \; . \\
 \eea

      \textbf{Terms of the type $ | \Lie_{Z^J} h | \cdot | \Lie_{Z^K} A | \cdot | \Lie_{Z^L} A | \cdot | \Lie_{Z^M} A |   $:} \\ 
      
We have,

 \bea
 \notag
&& | \Lie_{Z^J} h | \cdot | \Lie_{Z^K} A | \cdot | \Lie_{Z^L} A | \cdot | \Lie_{Z^M} A | \\
\notag
&\les&   \Big( E (   \lfloor \frac{|I|}{2} \rfloor + \lfloor  \frac{n}{2} \rfloor  + 1)  \cdot \begin{cases}  \frac{\eps }{(1+t+|q|)^{\frac{3(n-1)}{2}-3\delta} (1+|q|)^{3\gamma}},\quad\text{when }\quad q>0,\\
           \notag
      \frac{\eps \cdot (1+|q|)^{\frac{3}{2} }  }{(1+t+|q|)^{\frac{3(n-1)}{2}-3\delta}}  \,\quad\text{when }\quad q<0 . \end{cases} \Big) \\
               \notag
      && \cdot \sum_{|K|\leq |I|} \Big(  | \Lie_{Z^K} A | +  | \Lie_{Z^K} h  \Big) \; . \\
  \eea
 
   \textbf{The whole term  $|  \Lie_{Z^I}   ( g^{\la\mu} \derm_{\la}   \derm_{\mu}   A  )   |$:} \\ 
   
 Putting all together, we get
    \beaa
   \notag
 && |  \Lie_{Z^I}   ( g^{\la\mu} \derm_{\la}   \derm_{\mu}   A  )   | \\
   \notag
 &\les &   \sum_{|K|\leq |I|}   |    \derm (\Lie_{Z^K} A ) | \\
 \notag
 && \cdot \Big( E (   \lfloor \frac{|I|}{2} \rfloor + \lfloor  \frac{n}{2} \rfloor  + 1)  \cdot \begin{cases}  \frac{\eps }{(1+t+|q|)^{\frac{(n-1)}{2}-\delta} (1+|q|)^{1+\gamma}},\quad\text{when }\quad q>0,\\
           \notag
      \frac{\eps  }{(1+t+|q|)^{\frac{(n-1)}{2}-\delta}(1+|q|)^{\frac{1}{2} }}  \,\quad\text{when }\quad q<0 . \end{cases} \Big) \\
      \notag
      && +   \Big( E (   \lfloor \frac{|I|}{2} \rfloor + \lfloor  \frac{n}{2} \rfloor  + 1)  \cdot \begin{cases}  \frac{\eps }{(1+t+|q|)^{\frac{(n-1)}{2}-\delta} (1+|q|)^{\gamma}},\quad\text{when }\quad q>0,\\
           \notag
      \frac{\eps  \cdot (1+|q|)^{\frac{1}{2} } }{(1+t+|q|)^{\frac{(n-1)}{2}-\delta}}  \,\quad\text{when }\quad q<0 . \end{cases} \Big) \\
         \notag
         && + \Big( E (   \lfloor \frac{|I|}{2} \rfloor + \lfloor  \frac{n}{2} \rfloor  + 1)  \cdot \begin{cases}  \frac{\eps }{(1+t+|q|)^{(n-1)-2\delta} (1+|q|)^{1+2\gamma}},\quad\text{when }\quad q>0,\\
           \notag
      \frac{\eps  }{(1+t+|q|)^{(n-1)-2\delta}}  \,\quad\text{when }\quad q<0 . \end{cases} \Big) \\
      \notag
   && +    \Big( E (   \lfloor \frac{|I|}{2} \rfloor + \lfloor  \frac{n}{2} \rfloor  + 1)  \cdot \begin{cases}  \frac{\eps }{(1+t+|q|)^{(n-1)-2\delta} (1+|q|)^{2\gamma}},\quad\text{when }\quad q>0,\\
           \notag
      \frac{\eps \cdot (1+ |q|)  }{(1+t+|q|)^{(n-1)-2\delta}}  \,\quad\text{when }\quad q<0 . \end{cases} \Big) \Big) 
               \notag
  \eeaa
    \beaa
   \notag
 &&  + \sum_{|K|\leq |I|}   |   \Lie_{Z^K} A  | \\
 \notag
 &&  \cdot \Big( \Big( E (   \lfloor \frac{|I|}{2} \rfloor + \lfloor  \frac{n}{2} \rfloor  + 1)  \cdot \begin{cases}  \frac{\eps }{(1+t+|q|)^{(n-1)-2\delta} (1+|q|)^{1+2\gamma}},\quad\text{when }\quad q>0,\\
           \notag
      \frac{\eps  }{(1+t+|q|)^{(n-1)-2\delta}}  \,\quad\text{when }\quad q<0 . \end{cases} \Big) \\
      \notag 
      && + \Big( E (   \lfloor \frac{|I|}{2} \rfloor + \lfloor  \frac{n}{2} \rfloor  + 1)  \cdot \begin{cases}  \frac{\eps }{(1+t+|q|)^{\frac{(n-1)}{2}-\delta} (1+|q|)^{1+\gamma}},\quad\text{when }\quad q>0,\\
           \notag
      \frac{\eps  }{(1+t+|q|)^{\frac{(n-1)}{2}-\delta}(1+|q|)^{\frac{1}{2} }}  \,\quad\text{when }\quad q<0 . \end{cases} \Big) \\
         \notag
       && +  \Big( E (   \lfloor \frac{|I|}{2} \rfloor + \lfloor  \frac{n}{2} \rfloor  + 1)  \cdot \begin{cases}  \frac{\eps }{(1+t+|q|)^{(n-1)-2\delta} (1+|q|)^{2\gamma}},\quad\text{when }\quad q>0,\\
           \notag
      \frac{\eps \cdot (1+ |q|)  }{(1+t+|q|)^{(n-1)-2\delta}}  \,\quad\text{when }\quad q<0 . \end{cases} \Big)  \\
               \notag
     && +             \Big( E (   \lfloor \frac{|I|}{2} \rfloor + \lfloor  \frac{n}{2} \rfloor  + 1)  \cdot \begin{cases}  \frac{\eps }{(1+t+|q|)^{\frac{3(n-1)}{2}-3\delta} (1+|q|)^{1+3\gamma}},\quad\text{when }\quad q>0,\\
           \notag
      \frac{\eps  \cdot(1+|q|)^{\frac{1}{2} } }{(1+t+|q|)^{\frac{3(n-1)}{2}-3\delta}}  \,\quad\text{when }\quad q<0 . \end{cases} \Big)  \\
      \notag
  &&     +    \Big( E (   \lfloor \frac{|I|}{2} \rfloor + \lfloor  \frac{n}{2} \rfloor  + 1)  \cdot \begin{cases}  \frac{\eps }{(1+t+|q|)^{(n-1)-2\delta} (1+|q|)^{1+2\gamma}},\quad\text{when }\quad q>0,\\
           \notag
      \frac{\eps  }{(1+t+|q|)^{(n-1)-2\delta}}  \,\quad\text{when }\quad q<0 . \end{cases} \Big)  \\
      \notag
  && +      \Big( E (   \lfloor \frac{|I|}{2} \rfloor + \lfloor  \frac{n}{2} \rfloor  + 1)  \cdot \begin{cases}  \frac{\eps }{(1+t+|q|)^{\frac{3(n-1)}{2}-3\delta} (1+|q|)^{3\gamma}},\quad\text{when }\quad q>0,\\
           \notag
      \frac{\eps \cdot (1+|q|)^{\frac{3}{2} }  }{(1+t+|q|)^{\frac{3(n-1)}{2}-3\delta}}  \,\quad\text{when }\quad q<0 . \end{cases} \Big) 
               \notag
  \eeaa

     \beaa
   \notag
&& + \sum_{|K|\leq |I|}   |    \derm (\Lie_{Z^K} h ) | \\
 \notag
 && \cdot \Big( \Big( E (   \lfloor \frac{|I|}{2} \rfloor + \lfloor  \frac{n}{2} \rfloor  + 1)  \cdot \begin{cases}  \frac{\eps }{(1+t+|q|)^{\frac{(n-1)}{2}-\delta} (1+|q|)^{1+\gamma}},\quad\text{when }\quad q>0,\\
           \notag
      \frac{\eps  }{(1+t+|q|)^{\frac{(n-1)}{2}-\delta}(1+|q|)^{\frac{1}{2} }}  \,\quad\text{when }\quad q<0 . \end{cases} \Big) \\
      \notag
    && +   \Big( E (   \lfloor \frac{|I|}{2} \rfloor + \lfloor  \frac{n}{2} \rfloor  + 1)  \cdot \begin{cases}  \frac{\eps }{(1+t+|q|)^{(n-1)-2\delta} (1+|q|)^{2\gamma}},\quad\text{when }\quad q>0,\\
           \notag
      \frac{\eps \cdot (1+ |q|)  }{(1+t+|q|)^{(n-1)-2\delta}}  \,\quad\text{when }\quad q<0 . \end{cases} \Big) \\
      \notag
&& +        \Big( E (   \lfloor \frac{|I|}{2} \rfloor + \lfloor  \frac{n}{2} \rfloor  + 1)  \cdot \begin{cases}  \frac{\eps }{(1+t+|q|)^{(n-1)-2\delta} (1+|q|)^{1+2\gamma}},\quad\text{when }\quad q>0,\\
           \notag
      \frac{\eps  }{(1+t+|q|)^{(n-1)-2\delta}}  \,\quad\text{when }\quad q<0 . \end{cases} \Big) \\
     \notag
     && + \Big( E (   \lfloor \frac{|I|}{2} \rfloor + \lfloor  \frac{n}{2} \rfloor  + 1)  \cdot \begin{cases}  \frac{\eps }{(1+t+|q|)^{\frac{3(n-1)}{2}-3\delta} (1+|q|)^{3\gamma}},\quad\text{when }\quad q>0,\\
           \notag
      \frac{\eps \cdot (1+|q|)^{\frac{3}{2} }  }{(1+t+|q|)^{\frac{3(n-1)}{2}-3\delta}}  \,\quad\text{when }\quad q<0 . \end{cases} \Big) \Big) 
               \notag
  \eeaa
     \beaa
   \notag
 && +  \sum_{|K|\leq |I|}   |   \Lie_{Z^K} h  | \\
 \notag
 && \cdot \Big( \Big( E (   \lfloor \frac{|I|}{2} \rfloor + \lfloor  \frac{n}{2} \rfloor  + 1)  \cdot \begin{cases}  \frac{\eps }{(1+t+|q|)^{(n-1)-2\delta} (1+|q|)^{2+2\gamma}},\quad\text{when }\quad q>0,\\
           \notag
      \frac{\eps  }{(1+t+|q|)^{(n-1)-2\delta}(1+|q|)}  \,\quad\text{when }\quad q<0 . \end{cases} \Big)  \\
      \notag
  && +      \Big( E (   \lfloor \frac{|I|}{2} \rfloor + \lfloor  \frac{n}{2} \rfloor  + 1)  \cdot \begin{cases}  \frac{\eps }{(1+t+|q|)^{\frac{3(n-1)}{2}-3\delta} (1+|q|)^{1+3\gamma}},\quad\text{when }\quad q>0,\\
           \notag
      \frac{\eps  \cdot(1+|q|)^{\frac{1}{2} } }{(1+t+|q|)^{\frac{3(n-1)}{2}-3\delta}}  \,\quad\text{when }\quad q<0 . \end{cases} \Big)  \\
      \notag
   && +   \Big( E (   \lfloor \frac{|I|}{2} \rfloor + \lfloor  \frac{n}{2} \rfloor  + 1)  \cdot \begin{cases}  \frac{\eps }{(1+t+|q|)^{(n-1)-2\delta} (1+|q|)^{1+2\gamma}},\quad\text{when }\quad q>0,\\
           \notag
      \frac{\eps  }{(1+t+|q|)^{(n-1)-2\delta}}  \,\quad\text{when }\quad q<0 . \end{cases} \Big) \\
      \notag
   && +    \Big( E (   \lfloor \frac{|I|}{2} \rfloor + \lfloor  \frac{n}{2} \rfloor  + 1)  \cdot \begin{cases}  \frac{\eps }{(1+t+|q|)^{\frac{3(n-1)}{2}-3\delta} (1+|q|)^{3\gamma}},\quad\text{when }\quad q>0,\\
           \notag
      \frac{\eps \cdot (1+|q|)^{\frac{3}{2} }  }{(1+t+|q|)^{\frac{3(n-1)}{2}-3\delta}}  \,\quad\text{when }\quad q<0 . \end{cases} \Big) \Big) \; .
               \notag
  \eeaa

  \end{proof}

\subsection{Using the bootstrap assumption to exhibit the structure of the source terms for the metric}\

We aim to use the bootstrap assumption to exhibit the structure of the source term for the wave equation on the metric in wave coordinates coupled to the Yang-Mills potential in the Lorenz gauge, depending on the space-dimension $n$.

Recall that the weighted energy for $h^0$ is in fact infinite; thus, the energy was defined for $h^1 = h - h^0$. Thus, the equation that we are interested in, is the wave equation for $h^1$. We have for all $U, V \in \{\frac{\pa}{\pa x_\mu} \, \, | \, \,   \mu \in \{0, 1, \ldots, n \} \}$,
  \beaa
   \notag
  && g^{\la\mu} \derm_{\la}   \derm_{\mu}    h^1_{UV}   \\
   \notag
    &=&     g^{\la\mu} \derm_{\la}   \derm_{\mu}    h_{UV}  -   g^{\la\mu} \derm_{\la}   \derm_{\mu}    h^0_{UV}   \\
 \notag
  &=& P(\pa_\mu h,\pa_\nu h)  +  Q_{\mu\nu}(\pa h,\pa h)   + G_{\mu\nu}(h)(\pa h,\pa h)  \\
\notag
 &&   -4     <   \derm_{\mu}A_{\b} - \derm_{\b}A_{\mu}  ,  \derm_{\nu}A^{\b} - \derm^{\b}A_{\nu}  >    \\
 \notag
 &&   + m_{\mu\nu }       \cdot  <  \derm_{\a}A_{\b} - \derm_{\b}A_{\a} , \derm_{\a} A^{\b} - \derm^{\b}A^{\a} >   \\
 \notag
&&           -4  \cdot  \big( <   \derm_{\mu}A_{\b} - \derm_{\b}A_{\mu}  ,  [A_{\nu},A^{\b}] >   + <   [A_{\mu},A_{\b}] ,  \derm_{\nu}A^{\b} - \derm^{\b}A_{\nu}  > \big)  \\
\notag
&& + m_{\mu\nu }    \cdot \big(  <  \derm_{\a}A_{\b} - \derm_{\b}A_{\a} , [A^{\a},A^{\b}] >    +  <  [A_{\a},A_{\b}] , \derm^{\a}A^{\b} - \derm^{\b}A^{\a}  > \big) \\
\notag
 &&  -4     <   [A_{\mu},A_{\b}] ,  [A_{\nu},A^{\b}] >      + m_{\mu\nu }   \cdot   <  [A_{\a},A_{\b}] , [A^{\a},A^{\b}] >  \\
 \notag
     && + O \big(h \cdot  (\pa A)^2 \big)   + O \big(  h  \cdot  A^2 \cdot \pa A \big)     + O \big(  h   \cdot  A^4 \big)  \, \\
     && -   g^{\la\mu} \derm_{\la}   \derm_{\mu}    h^0_{UV} \; .
\eeaa
  Thus, for all $U, V \in \{\frac{\pa}{\pa x_\mu} \, \, | \, \,   \mu \in \{0, 1, \ldots, n \} \}$,
  \beaa
   \notag
 &&   \Lie_{Z^I}   ( g^{\la\mu} \derm_{\la}   \derm_{\mu}    h^1_{UV} )    \\
   &=& \sum_{|J| +|K|+|L|+ |M| \leq |I|}  \big( \; O(     \derm (\Lie_{Z^J}  h)  \cdot     \derm (\Lie_{Z^K} h )   )         + O( \Lie_{Z^J}   h \cdot \derm ( \Lie_{Z^K}  h) \cdot  \derm (\Lie_{Z^L}  h) ) \\
       \notag
           &&  +   O(   \derm ( \Lie_{Z^J} A  ) \cdot     \derm ( \Lie_{Z^K}  A)   )    +  O(       \Lie_{Z^K}   A  \cdot     \Lie_{Z^L}  A   \cdot  \derm ( \Lie_{Z^J} A)    )      + O(       \Lie_{Z^J} A  \cdot     \Lie_{Z^K} A   \cdot     \Lie_{Z^L} A \cdot   \Lie_{Z^M} A   ) \\
           \notag
      &&     +  O( \Lie_{Z^J}  h \cdot    \derm (\Lie_{Z^K} A ) \cdot \derm ( \Lie_{Z^L}  A ) )  \\
\notag
&& + O( \Lie_{Z^J} h \cdot   \Lie_{Z^K} A \cdot   \Lie_{Z^L} A \cdot   \derm (\Lie_{Z^M} A) )    + O( \Lie_{Z^J} h \cdot  \Lie_{Z^K} A \cdot  \Lie_{Z^L} A \cdot \Lie_{Z^M} A \cdot \Lie_{Z^N} A) \; \big) \, \\
\notag
&& + O (  \Lie_{Z^I}   ( g^{\la\mu} \derm_{\la}   \derm_{\mu}    h^0 )  ) \; .
  \eeaa

To estimate the term $\int_0^t\int_{\Si_{\tau}}  | g^{\la\a} \derm_{\la}   \derm_{\a}  ( \Lie_{Z^I}  h^1)) | \cdot |\derm ( \Lie_{Z^I}  h^1) )| \, w  \cdot dx_1 \ldots dx_n d\tau $, we use the inequality $a\cdot b \les a^2 + b^2$, to write
\beaa
&& \int_0^t\int_{\Si_{\tau}} \sqrt{(1+\tau )^{1+\la}} \sqrt{w} \cdot | g^{\la\a} \derm_{\la}   \derm_{\a}  ( \Lie_{Z^I}  h^1)) |\cdot  \frac{1}{\sqrt{(1+\tau )^{1+\la}}} \cdot |\derm ( \Lie_{Z^I}  h^1) )| \, \sqrt{w}  \cdot dx_1 \ldots dx_n d\tau \\
&\les&\int_0^t\int_{\Si_{\tau}} \frac{ |\derm  (\Lie_{Z^I}  h^1)) |^{2}}{(1+\tau)^{1+\la}} \cdot dx_1 \ldots dx_n d\tau \\
&&+  \int_0^t\int_{\Si_{\tau}}  (1+\tau )^{1+\la} \cdot | g^{\la\a} \derm_{\la}   \derm_{\a}  ( \Lie_{Z^I}  h^1)) |^2 \, w  \cdot dx_1 \ldots dx_n d\tau \; .
\eeaa

However, we have
\beaa
 && g^{\la\a} \derm_{\la}   \derm_{\a}  ( \Lie_{Z^I}  h^1) \\
    &= &  g^{\la\mu} \derm_{\la}   \derm_{\mu}(  \Lie_{Z^I}  h)      -   g^{\la\mu} \derm_{\la}   \derm_{\mu}  (\Lie_{Z^I}  h^0 )    \\
&= & \Lie_{Z^I}  (  g^{\la\mu} \derm_{\la}   \derm_{\mu}  h )  + g^{\la\mu} \derm_{\la}   \derm_{\mu}   ( \Lie_{ Z^I}   h ) - \Lie_{Z^I}  (g^{\la\mu} \derm_{\la}   \derm_{\mu}   h ) \\
&& -   \Lie_{Z^I}  g^{\la\mu} \derm_{\la}   \derm_{\mu} ( \Lie_{Z^I}  h^0 ) \, .
\eeaa

Using the triangular inequality, we have
\beaa
 && | g^{\la\a} \derm_{\la}   \derm_{\a}  ( \Lie_{Z^I}  h^1) | \\
&\leq& |  \Lie_{Z^I}  (g^{\la\mu} \derm_{\la}   \derm_{\mu}   h ) |+  | g^{\la\mu} \derm_{\la}   \derm_{\mu}   ( \Lie_{ Z^I}   h ) - \Lie_{Z^I}  (g^{\la\mu} \derm_{\la}   \derm_{\mu}   h ) | \\
&& +   | \Lie_{Z^I} ( g^{\la\mu} \derm_{\la}   \derm_{\mu} h^0 ) | \, ,
\eeaa
and therefore, we get
\bea
\notag
 && (1+\tau ) \cdot |  g^{\la\a} \derm_{\la}   \derm_{\a}  ( \Lie_{Z^I}  h^1)|^2  \\
 \notag
 &\leq&  (1+\tau ) \cdot | \Lie_{Z^I} ( g^{\la\mu} \derm_{\la}   \derm_{\mu}   h  ) |^2 \\
 \notag
&&+  (1+\tau ) \cdot |  g^{\la\mu} \derm_{\la}   \derm_{\mu}   ( \Lie_{ Z^I}   h ) - \Lie_{Z^I}  (g^{\la\mu} \derm_{\la}   \derm_{\mu}   h )|^2  \\
&&+  (1+\tau ) \cdot | \Lie_{ Z^I}  ( g^{\la\mu} \derm_{\la}   \derm_{\mu}    h^0 )  |^2 \, . 
\eea

We would like to estimate the term  $ | \Lie_{Z^I} ( g^{\la\mu} \derm_{\la}   \derm_{\mu}   h  ) |$\;.
\begin{lemma}
We have

\beaa
   \notag
 &&  | \Lie_{Z^I}   ( g^{\la\mu} \derm_{\la}   \derm_{\mu}    h ) |   \\
\notag
 &\les &\Big( \sum_{|K|\leq |I|}    | \derm ( \Lie_{Z^K} A) |\Big) \cdot  E (   \lfloor \frac{|I|}{2} \rfloor + \lfloor  \frac{n}{2} \rfloor  + 1) \\
 \notag
 &&\cdot  \Big(  \Big(\begin{cases}  \frac{\eps }{(1+t+|q|)^{\frac{(n-1)}{2}-\delta} (1+|q|)^{1+\gamma}},\quad\text{when }\quad q>0,\\
           \notag
      \frac{\eps  }{(1+t+|q|)^{\frac{(n-1)}{2}-\delta}(1+|q|)^{\frac{1}{2} }}  \,\quad\text{when }\quad q<0 . \end{cases} \Big) \\
      \notag
      && + \Big( \begin{cases}  \frac{\eps }{(1+t+|q|)^{\frac{(n-1)}{2}-\delta} (1+|q|)^{\frac{(n-1)}{2}-\delta +2\gamma}},\quad\text{when }\quad q>0,\\
           \notag
      \frac{\eps \cdot (1+ |q|)  }{(1+t+|q|)^{\frac{(n-1)}{2}-\delta} \cdot (1+|q|)^{\frac{(n-1)}{2}-\delta}}  \,\quad\text{when }\quad q<0 . \end{cases} \Big) \\
      \notag
    &&  +  \Big( \begin{cases}  \frac{\eps }{(1+t+|q|)^{\frac{(n-1)}{2}-\delta} (1+|q|)^{\frac{(n-1)}{2}-\delta+1+2\gamma}},\quad\text{when }\quad q>0,\\
           \notag
      \frac{\eps  }{(1+t+|q|)^{\frac{(n-1)}{2}-\delta} \cdot (1+|q|)^{\frac{(n-1)}{2}-\delta}}  \,\quad\text{when }\quad q<0 . \end{cases} \Big) \\
      \notag
   && + \Big(  \begin{cases}  \frac{\eps }{(1+t+|q|)^{\frac{(n-1)}{2}-\delta} (1+|q|)^{(n-1)-2\delta +3\gamma}},\quad\text{when }\quad q>0,\\
           \notag
      \frac{\eps \cdot (1+|q|)^{\frac{3}{2} }  }{(1+t+|q|)^{\frac{(n-1)}{2}-\delta} \cdot (1+|q|)^{(n-1)-2\delta}}  \,\quad\text{when }\quad q<0 . \end{cases} \Big) \Big) 
               \notag
\eeaa

\beaa
\notag
 &&+ \Big( \sum_{|K|\leq |I|}    |  \Lie_{Z^K} A | \Big) \cdot E (   \lfloor \frac{|I|}{2} \rfloor + \lfloor  \frac{n}{2} \rfloor  + 1)   \\
 \notag
 &&   \cdot \Big(   \Big(  \begin{cases}  \frac{\eps }{(1+t+|q|)^{(n-1)-2\delta} (1+|q|)^{1+2\gamma}},\quad\text{when }\quad q>0,\\
           \notag
      \frac{\eps  }{(1+t+|q|)^{(n-1)-2\delta}}  \,\quad\text{when }\quad q<0 . \end{cases} \Big) \\
      \notag
  &&  +   \Big(  \begin{cases}  \frac{\eps }{(1+t+|q|)^{(n-1)-2\delta} (1+|q|)^{\frac{(n-1)}{2}-\delta + 3\gamma}},\quad\text{when }\quad q>0,\\
           \notag
      \frac{\eps \cdot (1+|q|)^{\frac{3}{2} }  }{(1+t+|q|)^{(n-1)-2\delta } \cdot (1+|q|)^{\frac{(n-1)}{2}-\delta }}  \,\quad\text{when }\quad q<0 . \end{cases} \Big) \\
               \notag
 &&   +   \Big( \begin{cases}  \frac{\eps }{(1+t+|q|)^{(n-1)-2\delta} (1+|q|)^{\frac{(n-1)}{2}-\delta +1+3\gamma}},\quad\text{when }\quad q>0,\\
           \notag
      \frac{\eps  \cdot(1+|q|)^{\frac{1}{2} } }{(1+t+|q|)^{(n-1)-2\delta} \cdot (1+|q|)^{\frac{(n-1)}{2}-\delta}}  \,\quad\text{when }\quad q<0 . \end{cases} \Big)  \\
      \notag
   &&  + \Big(  \begin{cases}  \frac{\eps }{(1+t+|q|)^{(n-1)-2\delta} \cdot (1+|q|)^{(n-1)-2\delta +4\gamma}},\quad\text{when }\quad q>0,\\
           \notag
      \frac{\eps \cdot (1+|q|)^{2 }  }{(1+t+|q|)^{(n-1)-2\delta} \cdot (1+|q|)^{(n-1)-2\delta} }  \,\quad\text{when }\quad q<0 . \end{cases} \Big) \Big)  
               \notag
 \eeaa
 
\beaa
\notag
 &&+ \Big( \sum_{|K|\leq |I|}    | \derm(  \Lie_{Z^K} h ) | \Big) \cdot E (   \lfloor \frac{|I|}{2} \rfloor + \lfloor  \frac{n}{2} \rfloor  + 1) \\
 \notag
 && \cdot \Big(   \Big(  \begin{cases}  \frac{\eps }{(1+t+|q|)^{\frac{(n-1)}{2}-\delta} (1+|q|)^{1+\gamma}},\quad\text{when }\quad q>0,\\
 \notag
      \frac{\eps  }{(1+t+|q|)^{\frac{(n-1)}{2}-\delta}(1+|q|)^{\frac{1}{2} }}  \,\quad\text{when }\quad q<0 . \end{cases} \Big) \\
      \notag
   && + \Big(  \begin{cases}  \frac{\eps }{(1+t+|q|)^{\frac{(n-1)}{2}-\delta} (1+|q|)^{\frac{(n-1)}{2}-\delta+1+2\gamma}},\quad\text{when }\quad q>0,\\
           \notag
      \frac{\eps  }{(1+t+|q|)^{\frac{(n-1)}{2}-\delta} \cdot (1+|q|)^{\frac{(n-1)}{2}-\delta}}  \,\quad\text{when }\quad q<0 . \end{cases} \Big) \Big) 
      \notag
 \eeaa

\beaa
\notag
 &&+ \Big(   \sum_{|K|\leq |I|}    |   \Lie_{Z^K} h  |\Big) \cdot  E (   \lfloor \frac{|I|}{2} \rfloor + \lfloor  \frac{n}{2} \rfloor  + 1) \\
 \notag
   &&\Big(   \Big(  \begin{cases}  \frac{\eps }{(1+t+|q|)^{(n-1)-2\delta} (1+|q|)^{2+2\gamma}},\quad\text{when }\quad q>0,\\
           \notag
      \frac{\eps  }{(1+t+|q|)^{(n-1)-2\delta}(1+|q|)}  \,\quad\text{when }\quad q<0 . \end{cases} \Big) \\
      \notag
      && + \Big(  \begin{cases}  \frac{\eps }{(1+t+|q|)^{(n-1)-2\delta} (1+|q|)^{\frac{(n-1)}{2}-\delta+1+3\gamma}},\quad\text{when }\quad q>0,\\
           \notag
      \frac{\eps  \cdot(1+|q|)^{\frac{1}{2} } }{(1+t+|q|)^{(n-1)-2\delta} \cdot (1+|q|)^{\frac{(n-1)}{2}-\delta}}  \,\quad\text{when }\quad q<0 . \end{cases} \Big)  \\
      \notag
  &&  + \Big(  \begin{cases}  \frac{\eps }{(1+t+|q|)^{(n-1)-2\delta} \cdot (1+|q|)^{(n-1)-2\delta +4\gamma}},\quad\text{when }\quad q>0,\\
           \notag
      \frac{\eps \cdot (1+|q|)^{2 }  }{(1+t+|q|)^{(n-1)-2\delta} \cdot (1+|q|)^{(n-1)-2\delta} }  \,\quad\text{when }\quad q<0 . \end{cases} \Big) \Big)  \; .\\
               \notag
 \eeaa

\end{lemma}

\begin{proof}
We already showed in Lemma \ref{structureofLiederivativeZofthesourcestermsforhuv}, that
  \beaa
   \notag
 &&  | \Lie_{Z^I}   ( g^{\la\mu} \derm_{\la}   \derm_{\mu}    h ) |   \\
   &=& \sum_{|J| +|K|+|L|+ |M| \leq |I|}  \big( \; O(  |   \derm (\Lie_{Z^J}  h) | \cdot  |   \derm (\Lie_{Z^K} h )  | )         + O( | \Lie_{Z^J}   h | \cdot | \derm ( \Lie_{Z^K}  h) | \cdot | \derm (\Lie_{Z^L}  h) | ) \\
       \notag
           &&  +   O(  | \derm ( \Lie_{Z^J} A  ) | \cdot   |  \derm ( \Lie_{Z^K}  A)  | )    +  O(   |    \Lie_{Z^K}   A | \cdot   |  \Lie_{Z^L}  A |  \cdot | \derm ( \Lie_{Z^J} A) |   )   \\
           &&    + O(     |  \Lie_{Z^J} A | \cdot  |   \Lie_{Z^K} A |  \cdot   |  \Lie_{Z^L} A | \cdot |  \Lie_{Z^M} A  | )     +  O(|  \Lie_{Z^J}  h| \cdot  |  \derm (\Lie_{Z^K} A ) | \cdot | \derm ( \Lie_{Z^L}  A ) | )  \\
\notag
&& + O(|  \Lie_{Z^J} h| \cdot |  \Lie_{Z^K} A | \cdot  | \Lie_{Z^L} A | \cdot |  \derm (\Lie_{Z^M} A)| )    + O( | \Lie_{Z^J} h | \cdot  | \Lie_{Z^K} A | \cdot | \Lie_{Z^L} A |\cdot | \Lie_{Z^M} A |\cdot | \Lie_{Z^N} A |) \; \big) \,.\\
 \eeaa

        \textbf{Terms of the type $  |   \derm (\Lie_{Z^J}  h) | \cdot |    \derm (\Lie_{Z^K} h ) |    $:} \\ 
        We have
   \bea
   \notag
&& \sum_{|J| +|K|\leq |I|}   |   \derm (\Lie_{Z^J}  h) | \cdot |    \derm (\Lie_{Z^K} h ) |   \\
   \notag
 &\leq& \Big( E (   \lfloor \frac{|I|}{2} \rfloor + \lfloor  \frac{n}{2} \rfloor  + 1)  \cdot \begin{cases}  \frac{\eps }{(1+t+|q|)^{\frac{(n-1)}{2}-\delta} (1+|q|)^{1+\gamma}},\quad\text{when }\quad q>0,\\
           \notag
      \frac{\eps  }{(1+t+|q|)^{\frac{(n-1)}{2}-\delta}(1+|q|)^{\frac{1}{2} }}  \,\quad\text{when }\quad q<0 . \end{cases} \Big) \\
      && \cdot \sum_{|K|\leq |I|}  |    \derm (\Lie_{Z^K} h ) |  \; . 
  \eea

   \textbf{Terms of the type $  | \Lie_{Z^J}   h | \cdot | \derm ( \Lie_{Z^K}  h) | \cdot | \derm (\Lie_{Z^L}  h) |  $:} \\

 \bea  
         \notag
  && \sum_{|J| +|K| + |L| \leq |I|}   | \Lie_{Z^J}   h | \cdot | \derm ( \Lie_{Z^K}  h) | \cdot | \derm (\Lie_{Z^L}  h) |  \\
  &\les&    \Big( E (   \lfloor \frac{|I|}{2} \rfloor + \lfloor  \frac{n}{2} \rfloor  + 1)  \cdot \begin{cases}  \frac{\eps }{(1+t+|q|)^{(n-1)-2\delta} (1+|q|)^{1+2\gamma}},\quad\text{when }\quad q>0,\\
           \notag
      \frac{\eps  }{(1+t+|q|)^{(n-1)-2\delta}}  \,\quad\text{when }\quad q<0 . \end{cases} \Big) \\
      \notag
     &&   \cdot   \sum_{|K| \leq |I|}     | \derm (\Lie_{Z^K}  h) |  \\
        \notag
 && +         \Big( E (   \lfloor \frac{|I|}{2} \rfloor + \lfloor  \frac{n}{2} \rfloor  + 1)  \cdot \begin{cases}  \frac{\eps }{(1+t+|q|)^{(n-1)-2\delta} (1+|q|)^{2+2\gamma}},\quad\text{when }\quad q>0,\\
           \notag
      \frac{\eps  }{(1+t+|q|)^{(n-1)-2\delta}(1+|q|)}  \,\quad\text{when }\quad q<0 . \end{cases} \Big)   \cdot   \sum_{|K| \leq |I|}    | \Lie_{Z^K}  h | \, . \\
          \eea

             \textbf{Terms of the type $   |   \derm (\Lie_{Z^J}  A) | \cdot |    \derm (\Lie_{Z^K} A ) |  $:} \\ 
             
    \bea
   \notag
&& \sum_{|J| +|K|\leq |I|}   |   \derm (\Lie_{Z^J}  A) | \cdot |    \derm (\Lie_{Z^K} A ) |   \\
   \notag
 &\leq& \Big( E (   \lfloor \frac{|I|}{2} \rfloor + \lfloor  \frac{n}{2} \rfloor  + 1)  \cdot \begin{cases}  \frac{\eps }{(1+t+|q|)^{\frac{(n-1)}{2}-\delta} (1+|q|)^{1+\gamma}},\quad\text{when }\quad q>0,\\
           \notag
      \frac{\eps  }{(1+t+|q|)^{\frac{(n-1)}{2}-\delta}(1+|q|)^{\frac{1}{2} }}  \,\quad\text{when }\quad q<0 . \end{cases} \Big) \\
      && \cdot \sum_{|K|\leq |I|}  |    \derm (\Lie_{Z^K} A ) |  \; . 
  \eea

      \textbf{Terms of the type $   |   \derm ( \Lie_{Z^J} A) | \cdot |  \Lie_{Z^K}   A  | \cdot |    \Lie_{Z^L}  A | $:} \\ 
         We have
  \bea
    \notag
&& \sum_{|J| +|K| + |L| \leq |I|}    |   \derm ( \Lie_{Z^J} A) | \cdot |  \Lie_{Z^K}   A  | \cdot |    \Lie_{Z^L}  A | \\
   \notag
  &\les&     \Big( E (   \lfloor \frac{|I|}{2} \rfloor + \lfloor  \frac{n}{2} \rfloor  + 1)  \cdot \begin{cases}  \frac{\eps }{(1+t+|q|)^{(n-1)-2\delta} (1+|q|)^{1+2\gamma}},\quad\text{when }\quad q>0,\\
           \notag
      \frac{\eps  }{(1+t+|q|)^{(n-1)-2\delta}}  \,\quad\text{when }\quad q<0 . \end{cases} \Big) \cdot   \sum_{|K| \leq |I|}    |  \Lie_{Z^K}   A  | \\
         \notag
      && +   \Big( E (   \lfloor \frac{|I|}{2} \rfloor + \lfloor  \frac{n}{2} \rfloor  + 1)  \cdot \begin{cases}  \frac{\eps }{(1+t+|q|)^{(n-1)-2\delta} (1+|q|)^{2\gamma}},\quad\text{when }\quad q>0,\\
           \notag
      \frac{\eps \cdot (1+ |q|)  }{(1+t+|q|)^{(n-1)-2\delta}}  \,\quad\text{when }\quad q<0 . \end{cases} \Big)  \cdot    \sum_{|K| \leq |I|}    |   \derm ( \Lie_{Z^K} A) | \; .\\
 \eea         

      \textbf{Terms of the type $  | \Lie_{Z^J} A | \cdot | \Lie_{Z^K} A | \cdot | \Lie_{Z^L} A | \cdot | \Lie_{Z^M} A | $:} \\ 
      
 \bea
 \notag
&&  \sum_{|J| +|K| + |L| + |M| \leq |I|}   | \Lie_{Z^J} A | \cdot | \Lie_{Z^K} A | \cdot | \Lie_{Z^L} A | \cdot | \Lie_{Z^M} A | \\
\notag
&\les&   \Big( E (   \lfloor \frac{|I|}{2} \rfloor + \lfloor  \frac{n}{2} \rfloor  + 1)  \cdot \begin{cases}  \frac{\eps }{(1+t+|q|)^{\frac{3(n-1)}{2}-3\delta} (1+|q|)^{3\gamma}},\quad\text{when }\quad q>0,\\
           \notag
      \frac{\eps \cdot (1+|q|)^{\frac{3}{2} }  }{(1+t+|q|)^{\frac{3(n-1)}{2}-3\delta}}  \,\quad\text{when }\quad q<0 . \end{cases} \Big) \\
               \notag
      && \cdot \sum_{|K|\leq |I|}   | \Lie_{Z^K} A |  \; . \\
  \eea
  
      \textbf{Terms of the type $  | \Lie_{Z^J}   h | \cdot | \derm ( \Lie_{Z^K}  A) | \cdot | \derm (\Lie_{Z^L}  A) |$:} \\  
   \bea  
         \notag
  && \sum_{|J| +|K| + |L| \leq |I|}   | \Lie_{Z^J}   h | \cdot | \derm ( \Lie_{Z^K}  A) | \cdot | \derm (\Lie_{Z^L}  A) |  \\
  &\les&    \Big( E (   \lfloor \frac{|I|}{2} \rfloor + \lfloor  \frac{n}{2} \rfloor  + 1)  \cdot \begin{cases}  \frac{\eps }{(1+t+|q|)^{(n-1)-2\delta} (1+|q|)^{1+2\gamma}},\quad\text{when }\quad q>0,\\
           \notag
      \frac{\eps  }{(1+t+|q|)^{(n-1)-2\delta}}  \,\quad\text{when }\quad q<0 . \end{cases} \Big) \\
      \notag
     &&   \cdot   \sum_{|K| \leq |I|}     | \derm (\Lie_{Z^K}  A) |  \\
        \notag
 && +         \Big( E (   \lfloor \frac{|I|}{2} \rfloor + \lfloor  \frac{n}{2} \rfloor  + 1)  \cdot \begin{cases}  \frac{\eps }{(1+t+|q|)^{(n-1)-2\delta} (1+|q|)^{2+2\gamma}},\quad\text{when }\quad q>0,\\
           \notag
      \frac{\eps  }{(1+t+|q|)^{(n-1)-2\delta}(1+|q|)}  \,\quad\text{when }\quad q<0 . \end{cases} \Big)   \cdot   \sum_{|K| \leq |I|}    | \Lie_{Z^K}  h | \, . \\
          \eea

   \textbf{Terms of the type $   |  \Lie_{Z^J} h| \cdot |  \Lie_{Z^K} A | \cdot  | \Lie_{Z^L} A | \cdot |  \derm (\Lie_{Z^M} A)| $:} \\  
   We have
     \bea
\notag
 && \sum_{|J| +|K| + |L| + |M| \leq |I|}   |  \Lie_{Z^J} h| \cdot |  \Lie_{Z^K} A | \cdot  | \Lie_{Z^L} A | \cdot |  \derm (\Lie_{Z^M} A)| \\
 \notag
   &\les&     \Big( E (   \lfloor \frac{|I|}{2} \rfloor + \lfloor  \frac{n}{2} \rfloor  + 1)  \cdot \begin{cases}  \frac{\eps }{(1+t+|q|)^{\frac{3(n-1)}{2}-3\delta} (1+|q|)^{1+3\gamma}},\quad\text{when }\quad q>0,\\
           \notag
      \frac{\eps  \cdot(1+|q|)^{\frac{1}{2} } }{(1+t+|q|)^{\frac{3(n-1)}{2}-3\delta}}  \,\quad\text{when }\quad q<0 . \end{cases} \Big)  \\
      \notag
  &&      \cdot   \sum_{|K| \leq |I|}  \Big(   | \Lie_{Z^K}  h | +   | \Lie_{Z^K}  A | \Big) \\
  \notag
  && +   \Big( E (   \lfloor \frac{|I|}{2} \rfloor + \lfloor  \frac{n}{2} \rfloor  + 1)  \cdot \begin{cases}  \frac{\eps }{(1+t+|q|)^{\frac{3(n-1)}{2}-3\delta} (1+|q|)^{3\gamma}},\quad\text{when }\quad q>0,\\
           \notag
      \frac{\eps \cdot (1+|q|)^{\frac{3}{2} }  }{(1+t+|q|)^{\frac{3(n-1)}{2}-3\delta}}  \,\quad\text{when }\quad q<0 . \end{cases} \Big) \\
               \notag
      && \cdot \sum_{|K|\leq |I|}   | \derm ( \Lie_{Z^K} A) | \; . \\
  \eea
  
   \textbf{Terms of the type $  | \Lie_{Z^J} h | \cdot  | \Lie_{Z^K} A | \cdot | \Lie_{Z^L} A |\cdot | \Lie_{Z^M} A |\cdot | \Lie_{Z^N} A |$:} \\  
   
 \bea
\notag
     &&  \sum_{|J| +|K| + |L| + |M| + |N| \leq |I|}   | \Lie_{Z^J} h | \cdot  | \Lie_{Z^K} A | \cdot | \Lie_{Z^L} A |\cdot | \Lie_{Z^M} A |\cdot | \Lie_{Z^N} A |      \\
       &\les& \Big( E (   \lfloor \frac{|I|}{2} \rfloor + \lfloor  \frac{n}{2} \rfloor  + 1)  \cdot \begin{cases}  \frac{\eps }{(1+t+|q|)^{2(n-1)-4\delta} (1+|q|)^{4\gamma}},\quad\text{when }\quad q>0,\\
           \notag
      \frac{\eps \cdot (1+|q|)^{2 }  }{(1+t+|q|)^{2(n-1)-4\delta}}  \,\quad\text{when }\quad q<0 . \end{cases} \Big) \\
               \notag
      && \cdot \sum_{|K|\leq |I|} \Big(  | \Lie_{Z^K} A | +  | \Lie_{Z^K} h |  \Big) \; . \\
  \eea

   \textbf{The whole term $ | \Lie_{Z^I}   ( g^{\la\mu} \derm_{\la}   \derm_{\mu}    h ) |$:} \\  
   Putting the terms together, we obtain
\beaa
   \notag
 &&  | \Lie_{Z^I}   ( g^{\la\mu} \derm_{\la}   \derm_{\mu}    h ) |   \\
 \notag
 &\les& \sum_{|K|\leq |I|}    | \derm ( \Lie_{Z^K} A) | \\
 \notag
 &&\cdot  \Big(  \Big( E (   \lfloor \frac{|I|}{2} \rfloor + \lfloor  \frac{n}{2} \rfloor  + 1)  \cdot \begin{cases}  \frac{\eps }{(1+t+|q|)^{\frac{(n-1)}{2}-\delta} (1+|q|)^{1+\gamma}},\quad\text{when }\quad q>0,\\
           \notag
      \frac{\eps  }{(1+t+|q|)^{\frac{(n-1)}{2}-\delta}(1+|q|)^{\frac{1}{2} }}  \,\quad\text{when }\quad q<0 . \end{cases} \Big) \\
      \notag
      && + \Big( E (   \lfloor \frac{|I|}{2} \rfloor + \lfloor  \frac{n}{2} \rfloor  + 1)  \cdot \begin{cases}  \frac{\eps }{(1+t+|q|)^{(n-1)-2\delta} (1+|q|)^{2\gamma}},\quad\text{when }\quad q>0,\\
           \notag
      \frac{\eps \cdot (1+ |q|)  }{(1+t+|q|)^{(n-1)-2\delta}}  \,\quad\text{when }\quad q<0 . \end{cases} \Big) \\
      \notag
    &&  +  \Big( E (   \lfloor \frac{|I|}{2} \rfloor + \lfloor  \frac{n}{2} \rfloor  + 1)  \cdot \begin{cases}  \frac{\eps }{(1+t+|q|)^{(n-1)-2\delta} (1+|q|)^{1+2\gamma}},\quad\text{when }\quad q>0,\\
           \notag
      \frac{\eps  }{(1+t+|q|)^{(n-1)-2\delta}}  \,\quad\text{when }\quad q<0 . \end{cases} \Big) \\
      \notag
   && + \Big( E (   \lfloor \frac{|I|}{2} \rfloor + \lfloor  \frac{n}{2} \rfloor  + 1)  \cdot \begin{cases}  \frac{\eps }{(1+t+|q|)^{\frac{3(n-1)}{2}-3\delta} (1+|q|)^{3\gamma}},\quad\text{when }\quad q>0,\\
           \notag
      \frac{\eps \cdot (1+|q|)^{\frac{3}{2} }  }{(1+t+|q|)^{\frac{3(n-1)}{2}-3\delta}}  \,\quad\text{when }\quad q<0 . \end{cases} \Big) \\
               \notag
   &&        +     \Big( E (   \lfloor \frac{|I|}{2} \rfloor + \lfloor  \frac{n}{2} \rfloor  + 1)  \cdot \begin{cases}  \frac{\eps }{(1+t+|q|)^{\frac{3(n-1)}{2}-3\delta} (1+|q|)^{3\gamma}},\quad\text{when }\quad q>0,\\
           \notag
      \frac{\eps \cdot (1+|q|)^{\frac{3}{2} }  }{(1+t+|q|)^{\frac{3(n-1)}{2}-3\delta}}  \,\quad\text{when }\quad q<0 . \end{cases} \Big) \Big)  
               \notag
\eeaa
\beaa
\notag
 && + \sum_{|K|\leq |I|}    |  \Lie_{Z^K} A |  \\
 \notag
 &&   \cdot \Big(   \Big( E (   \lfloor \frac{|I|}{2} \rfloor + \lfloor  \frac{n}{2} \rfloor  + 1)  \cdot \begin{cases}  \frac{\eps }{(1+t+|q|)^{(n-1)-2\delta} (1+|q|)^{1+2\gamma}},\quad\text{when }\quad q>0,\\
           \notag
      \frac{\eps  }{(1+t+|q|)^{(n-1)-2\delta}}  \,\quad\text{when }\quad q<0 . \end{cases} \Big) \\
      \notag
  &&  +   \Big( E (   \lfloor \frac{|I|}{2} \rfloor + \lfloor  \frac{n}{2} \rfloor  + 1)  \cdot \begin{cases}  \frac{\eps }{(1+t+|q|)^{\frac{3(n-1)}{2}-3\delta} (1+|q|)^{3\gamma}},\quad\text{when }\quad q>0,\\
           \notag
      \frac{\eps \cdot (1+|q|)^{\frac{3}{2} }  }{(1+t+|q|)^{\frac{3(n-1)}{2}-3\delta}}  \,\quad\text{when }\quad q<0 . \end{cases} \Big) \\
               \notag
 &&   +   \Big( E (   \lfloor \frac{|I|}{2} \rfloor + \lfloor  \frac{n}{2} \rfloor  + 1)  \cdot \begin{cases}  \frac{\eps }{(1+t+|q|)^{\frac{3(n-1)}{2}-3\delta} (1+|q|)^{1+3\gamma}},\quad\text{when }\quad q>0,\\
           \notag
      \frac{\eps  \cdot(1+|q|)^{\frac{1}{2} } }{(1+t+|q|)^{\frac{3(n-1)}{2}-3\delta}}  \,\quad\text{when }\quad q<0 . \end{cases} \Big)  \\
      \notag
   &&  + \Big( E (   \lfloor \frac{|I|}{2} \rfloor + \lfloor  \frac{n}{2} \rfloor  + 1)  \cdot \begin{cases}  \frac{\eps }{(1+t+|q|)^{2(n-1)-4\delta} (1+|q|)^{4\gamma}},\quad\text{when }\quad q>0,\\
           \notag
      \frac{\eps \cdot (1+|q|)^{2 }  }{(1+t+|q|)^{2(n-1)-4\delta}}  \,\quad\text{when }\quad q<0 . \end{cases} \Big) \Big) 
               \notag
 \eeaa
\beaa
\notag
 && + \sum_{|K|\leq |I|}    | \derm(  \Lie_{Z^K} h ) |  \\
 \notag
 && \cdot \Big(   \Big( E (   \lfloor \frac{|I|}{2} \rfloor + \lfloor  \frac{n}{2} \rfloor  + 1)  \cdot \begin{cases}  \frac{\eps }{(1+t+|q|)^{\frac{(n-1)}{2}-\delta} (1+|q|)^{1+\gamma}},\quad\text{when }\quad q>0,\\
 \notag
      \frac{\eps  }{(1+t+|q|)^{\frac{(n-1)}{2}-\delta}(1+|q|)^{\frac{1}{2} }}  \,\quad\text{when }\quad q<0 . \end{cases} \Big) \\
      \notag
   && + \Big( E (   \lfloor \frac{|I|}{2} \rfloor + \lfloor  \frac{n}{2} \rfloor  + 1)  \cdot \begin{cases}  \frac{\eps }{(1+t+|q|)^{(n-1)-2\delta} (1+|q|)^{1+2\gamma}},\quad\text{when }\quad q>0,\\
           \notag
      \frac{\eps  }{(1+t+|q|)^{(n-1)-2\delta}}  \,\quad\text{when }\quad q<0 . \end{cases} \Big) \Big) 
      \notag
 \eeaa
\beaa
\notag
 && +\sum_{|K|\leq |I|}    |   \Lie_{Z^K} h  | \\
 \notag
   &&\Big(   \Big( E (   \lfloor \frac{|I|}{2} \rfloor + \lfloor  \frac{n}{2} \rfloor  + 1)  \cdot \begin{cases}  \frac{\eps }{(1+t+|q|)^{(n-1)-2\delta} (1+|q|)^{2+2\gamma}},\quad\text{when }\quad q>0,\\
           \notag
      \frac{\eps  }{(1+t+|q|)^{(n-1)-2\delta}(1+|q|)}  \,\quad\text{when }\quad q<0 . \end{cases} \Big) \\
      \notag
     &&   + \Big( E (   \lfloor \frac{|I|}{2} \rfloor + \lfloor  \frac{n}{2} \rfloor  + 1)  \cdot \begin{cases}  \frac{\eps }{(1+t+|q|)^{(n-1)-2\delta} (1+|q|)^{2+2\gamma}},\quad\text{when }\quad q>0,\\
           \notag
      \frac{\eps  }{(1+t+|q|)^{(n-1)-2\delta}(1+|q|)}  \,\quad\text{when }\quad q<0 . \end{cases} \Big) \\
      && + \Big( E (   \lfloor \frac{|I|}{2} \rfloor + \lfloor  \frac{n}{2} \rfloor  + 1)  \cdot \begin{cases}  \frac{\eps }{(1+t+|q|)^{\frac{3(n-1)}{2}-3\delta} (1+|q|)^{1+3\gamma}},\quad\text{when }\quad q>0,\\
           \notag
      \frac{\eps  \cdot(1+|q|)^{\frac{1}{2} } }{(1+t+|q|)^{\frac{3(n-1)}{2}-3\delta}}  \,\quad\text{when }\quad q<0 . \end{cases} \Big)  \\
      \notag
  &&  +  \Big( E (   \lfloor \frac{|I|}{2} \rfloor + \lfloor  \frac{n}{2} \rfloor  + 1)  \cdot \begin{cases}  \frac{\eps }{(1+t+|q|)^{2(n-1)-4\delta} (1+|q|)^{4\gamma}},\quad\text{when }\quad q>0,\\
           \notag
      \frac{\eps \cdot (1+|q|)^{2 }  }{(1+t+|q|)^{2(n-1)-4\delta}}  \,\quad\text{when }\quad q<0 . \end{cases} \Big) \Big) \; .
               \notag
  \eeaa

  \end{proof}

 \begin{lemma}
 We have
 \beaa
 \notag
&& \frac{(1+t )}{\eps} \cdot |  \Lie_{Z^I}   ( g^{\la\mu} \derm_{\la}   \derm_{\mu}   A  )   |^2 \, \\
   \notag
 &\les&   \sum_{|K|\leq |I|}  \Big( |    \derm (\Lie_{Z^K} A ) |^2 \Big) \cdot E (   \lfloor \frac{|I|}{2} \rfloor + \lfloor  \frac{n}{2} \rfloor  + 1)\\
 \notag
 && \cdot \Big( \begin{cases}  \frac{\eps }{(1+t+|q|)^{(n-2)-2\delta} (1+|q|)^{2+2\gamma}},\quad\text{when }\quad q>0,\\
           \notag
      \frac{\eps  }{(1+t+|q|)^{(n-2)-2\delta}(1+|q|)}  \,\quad\text{when }\quad q<0 . \end{cases} \Big) \\
      \notag
      && +   \Big( \begin{cases}  \frac{\eps }{(1+t+|q|)^{(n-2)-2\delta} (1+|q|)^{2\gamma}},\quad\text{when }\quad q>0,\\
           \notag
      \frac{\eps  \cdot (1+|q|)^{ } }{(1+t+|q|)^{(n-2)-2\delta}}  \,\quad\text{when }\quad q<0 . \end{cases} \Big) \\
         \notag
         && + \Big( \begin{cases}  \frac{\eps }{(1+t+|q|)^{(n-2)-2\delta} (1+|q|)^{(n-1)-2\delta + 2+4\gamma}},\quad\text{when }\quad q>0,\\
           \notag
      \frac{\eps  }{ (1+t+|q|)^{(n-2)-2\delta} \cdot (1+ |q|)^{(n-1)-2\delta}  }  \,\quad\text{when }\quad q<0 . \end{cases} \Big) \\
      \notag
   && +    \Big(  \begin{cases}  \frac{\eps }{(1+t+|q|)^{(n-2)-2\delta} (1+|q|)^{(n-1)-2\delta +4\gamma}},\quad\text{when }\quad q>0,\\
           \notag
      \frac{\eps \cdot (1+ |q|)  }{(1+t+|q|)^{(n-2)-2\delta} \cdot (1+ |q|)^{(n-1)-2\delta} } \,\quad\text{when }\quad q<0 . \end{cases} \Big) \Big) 
               \notag
  \eeaa

       \beaa
   \notag
 &&  + \sum_{|K|\leq |I|}  \Big( |   \Lie_{Z^K} A  |^2 \Big) \cdot E (   \lfloor \frac{|I|}{2} \rfloor + \lfloor  \frac{n}{2} \rfloor  + 1) \\
 \notag
 &&  \cdot \Big( \Big(  \begin{cases}  \frac{\eps }{(1+t+|q|)^{(n-2)-2\delta} (1+|q|)^{(n-1)-2\delta+2+4\gamma}},\quad\text{when }\quad q>0,\\
           \notag
      \frac{\eps  }{(1+t+|q|)^{(n-2)-2\delta} \cdot ( 1+|q| )^{(n-1)-2\delta} } \,\quad\text{when }\quad q<0 . \end{cases} \Big) \\
      \notag 
      && + \Big(  \begin{cases}  \frac{\eps }{(1+t+|q|)^{(n-2)-2\delta} (1+|q|)^{2+2\gamma}},\quad\text{when }\quad q>0,\\
           \notag
      \frac{\eps  }{(1+t+|q|)^{(n-2)-2\delta}(1+|q|)^{ }}  \,\quad\text{when }\quad q<0 . \end{cases} \Big) \\
         \notag
       && +  \Big(  \begin{cases}  \frac{\eps }{(1+t+|q|)^{(n-2)-2\delta} (1+|q|)^{(n-1)-2\delta + 4\gamma}},\quad\text{when }\quad q>0,\\
           \notag
      \frac{\eps \cdot (1+ |q|)^2  }{(1+t+|q|)^{(n-2)-2\delta } \cdot (1+|q|)^{(n-1)-2\delta}} \,\quad\text{when }\quad q<0 . \end{cases} \Big)  \\
               \notag
     && +             \Big(  \begin{cases}  \frac{\eps }{(1+t+|q|)^{(n-2)-2\delta} \cdot (1+|q|)^{2(n-1)-4\delta + 2+6\gamma}},\quad\text{when }\quad q>0,\\
           \notag
      \frac{\eps  \cdot(1+|q|)^{ } }{(1+t+|q|)^{(n-2)-2\delta} \cdot (1+|q|)^{2(n-1)-4\delta}}  \,\quad\text{when }\quad q<0 . \end{cases} \Big)  \\
      \notag
  &&     +    \Big(  \begin{cases}  \frac{\eps }{(1+t+|q|)^{(n-2)-2\delta} (1+|q|)^{(n-1)-2\delta +2+4\gamma}},\quad\text{when }\quad q>0,\\
           \notag
      \frac{\eps  }{(1+t+|q|)^{(n-2)-2\delta} \cdot (1+|q|)^{(n-1)-2\delta}}  \,\quad\text{when }\quad q<0 . \end{cases} \Big)  \\
      \notag
  && +      \Big(  \begin{cases}  \frac{\eps }{(1+t+|q|)^{(n-2)-2\delta} (1+|q|)^{2(n-1)-4\delta + 6\gamma}},\quad\text{when }\quad q>0,\\
           \notag
      \frac{\eps \cdot (1+|q|)^{3 }  }{(1+t+|q|)^{(n-2)-2\delta} \cdot (1+|q|)^{2(n-1)-4\delta}}  \,\quad\text{when }\quad q<0 . \end{cases} \Big) 
               \notag
  \eeaa
 
           \beaa
   \notag
&&  + \Big( \sum_{|K|\leq |I|}   |    \derm (\Lie_{Z^K} h ) |^2 \Big) \cdot E (   \lfloor \frac{|I|}{2} \rfloor + \lfloor  \frac{n}{2} \rfloor  + 1)  \\
 \notag
 && \cdot  \Big( \Big( \begin{cases}  \frac{\eps }{(1+t+|q|)^{(n-2)-2\delta} (1+|q|)^{2+2\gamma}},\quad\text{when }\quad q>0,\\
           \notag
      \frac{\eps  }{(1+t+|q|)^{(n-2)-2\delta}(1+|q|)^{ }}  \,\quad\text{when }\quad q<0 . \end{cases} \Big) \\
      \notag
    && +   \Big(  \begin{cases}  \frac{\eps }{(1+t+|q|)^{(n-2)-2\delta} (1+|q|)^{(n-1)-2\delta + 4\gamma}},\quad\text{when }\quad q>0,\\
           \notag
      \frac{\eps \cdot (1+ |q|)^2  }{(1+t+|q|)^{(n-2)-2\delta} \cdot (1+ |q|)^{(n-1)-2\delta} }  \,\quad\text{when }\quad q<0 . \end{cases} \Big) \\
      \notag
&& +        \Big( \begin{cases}  \frac{\eps }{(1+t+|q|)^{(n-2)-2\delta} (1+|q|)^{(n-1)-2\delta+2+4\gamma}},\quad\text{when }\quad q>0,\\
           \notag
      \frac{\eps  }{(1+t+|q|)^{(n-2)-2\delta} \cdot (1+|q|)^{(n-1)-2\delta}  } ,\quad\text{when }\quad q<0 . \end{cases} \Big) \\
     \notag
     && + \Big( \begin{cases}  \frac{\eps }{(1+t+|q|)^{(n-2)-2\delta} (1+|q|)^{2(n-1) - 4\de + 6\gamma}},\quad\text{when }\quad q>0,\\
           \notag
      \frac{\eps \cdot (1+|q|)^{3 }  }{(1+t+|q|)^{(n-2)-2\delta} \cdot (1+|q|)^{2(n-1)-4\delta}}  \,\quad\text{when }\quad q<0 . \end{cases} \Big) \Big) 
               \notag
  \eeaa
 
       \beaa
   \notag
 && + \Big( \sum_{|K|\leq |I|}   |   \Lie_{Z^K} h  |^2 \Big) \cdot E (   \lfloor \frac{|I|}{2} \rfloor + \lfloor  \frac{n}{2} \rfloor  + 1)  \\
 \notag
 && \cdot \Big( \Big(  \begin{cases}  \frac{\eps }{(1+t+|q|)^{2(n-\frac{3}{2})-4\delta} (1+|q|)^{4+4\gamma}},\quad\text{when }\quad q>0,\\
           \notag
      \frac{\eps  }{(1+t+|q|)^{2(n-\frac{3}{2})-4\delta}(1+|q|)^{2}}  \,\quad\text{when }\quad q<0 . \end{cases} \Big)  \\
      \notag
  && +      \Big( \begin{cases}  \frac{\eps }{(1+t+|q|)^{2(n-\frac{3}{2})-4\delta} \cdot (1+|q|)^{(n-1)-2\delta + 2+6\gamma}},\quad\text{when }\quad q>0,\\
           \notag
      \frac{\eps  \cdot(1+|q|)^{ } }{(1+t+|q|)^{2(n-\frac{3}{2})-4\delta} \cdot (1+|q|)^{(n-1)-2\delta} }  \,\quad\text{when }\quad q<0 . \end{cases} \Big)  \\
      \notag
   && +   \Big(  \begin{cases}  \frac{\eps }{(1+t+|q|)^{2(n-\frac{3}{2})-4\delta} (1+|q|)^{2+4\gamma}},\quad\text{when }\quad q>0,\\
           \notag
      \frac{\eps  }{(1+t+|q|)^{2(n-\frac{3}{2})-4\delta}}  \,\quad\text{when }\quad q<0 . \end{cases} \Big) \\
      \notag
   && +    \Big(  \begin{cases}  \frac{\eps }{(1+t+|q|)^{2(n-\frac{3}{2})-4\delta} (1+|q|)^{(n-1)-2\delta +6\gamma}},\quad\text{when }\quad q>0,\\
           \notag
      \frac{\eps \cdot (1+|q|)^{3 }  }{(1+t+|q|)^{2(n-\frac{3}{2})-4\delta} \cdot (1+ |q|)^{(n-1)-2\delta} }  \,\quad\text{when }\quad q<0 . \end{cases} \Big) \Big) \; .
               \notag
  \eeaa

 \end{lemma}
 
 \begin{proof}
 We have
      \beaa
   \notag
 &&  |  \Lie_{Z^I}   ( g^{\la\mu} \derm_{\la}   \derm_{\mu}   A  )   |^2 \\
   \notag
 &\les&   \sum_{|K|\leq |I|}  \Big( |    \derm (\Lie_{Z^K} A ) |^2 \Big) \cdot E (   \lfloor \frac{|I|}{2} \rfloor + \lfloor  \frac{n}{2} \rfloor  + 1)\\
 \notag
 && \cdot \Big( \begin{cases}  \frac{\eps^2 }{(1+t+|q|)^{(n-1)-2\delta} (1+|q|)^{2+2\gamma}},\quad\text{when }\quad q>0,\\
           \notag
      \frac{\eps^2  }{(1+t+|q|)^{(n-1)-2\delta}(1+|q|)}  \,\quad\text{when }\quad q<0 . \end{cases} \Big) \\
      \notag
      && +   \Big( \begin{cases}  \frac{\eps^2 }{(1+t+|q|)^{(n-1)-2\delta} (1+|q|)^{2\gamma}},\quad\text{when }\quad q>0,\\
           \notag
      \frac{\eps  \cdot (1+|q|)^{ } }{(1+t+|q|)^{(n-1)-2\delta}}  \,\quad\text{when }\quad q<0 . \end{cases} \Big) \\
         \notag
         && + \Big( \begin{cases}  \frac{\eps^2 }{(1+t+|q|)^{(n-1)-2\delta} (1+|q|)^{(n-1)-2\delta + 2+4\gamma}},\quad\text{when }\quad q>0,\\
           \notag
      \frac{\eps  }{ (1+t+|q|)^{(n-1)-2\delta} \cdot (1+ |q|)^{(n-1)-2\delta}  }  \,\quad\text{when }\quad q<0 . \end{cases} \Big) \\
      \notag
   && +    \Big(  \begin{cases}  \frac{\eps^2 }{(1+t+|q|)^{(n-1)-2\delta} (1+|q|)^{(n-1)-2\delta +4\gamma}},\quad\text{when }\quad q>0,\\
           \notag
      \frac{\eps \cdot (1+ |q|)  }{(1+t+|q|)^{(n-1)-2\delta} \cdot (1+ |q|)^{(n-1)-2\delta} } \,\quad\text{when }\quad q<0 . \end{cases} \Big) \Big) 
               \notag
  \eeaa

      \beaa
   \notag
 &&  + \sum_{|K|\leq |I|}  \Big( |   \Lie_{Z^K} A  |^2 \Big) \cdot E (   \lfloor \frac{|I|}{2} \rfloor + \lfloor  \frac{n}{2} \rfloor  + 1) \\
 \notag
 &&  \cdot \Big( \Big(  \begin{cases}  \frac{\eps^2 }{(1+t+|q|)^{(n-1)-2\delta} (1+|q|)^{(n-1)-2\delta+2+4\gamma}},\quad\text{when }\quad q>0,\\
           \notag
      \frac{\eps^2  }{(1+t+|q|)^{(n-1)-2\delta} \cdot ( 1+|q| )^{(n-1)-2\delta} } \,\quad\text{when }\quad q<0 . \end{cases} \Big) \\
      \notag 
      && + \Big(  \begin{cases}  \frac{\eps^2 }{(1+t+|q|)^{(n-1)-2\delta} (1+|q|)^{2+2\gamma}},\quad\text{when }\quad q>0,\\
           \notag
      \frac{\eps^2  }{(1+t+|q|)^{(n-1)-2\delta}(1+|q|)^{ }}  \,\quad\text{when }\quad q<0 . \end{cases} \Big) \\
         \notag
       && +  \Big(  \begin{cases}  \frac{\eps^2 }{(1+t+|q|)^{(n-1)-2\delta} (1+|q|)^{(n-1)-2\delta + 4\gamma}},\quad\text{when }\quad q>0,\\
           \notag
      \frac{\eps^2 \cdot (1+ |q|)^2  }{(1+t+|q|)^{(n-1)-2\delta } \cdot (1+|q|)^{(n-1)-2\delta}} \,\quad\text{when }\quad q<0 . \end{cases} \Big)  \\
               \notag
     && +             \Big(  \begin{cases}  \frac{\eps^2 }{(1+t+|q|)^{(n-1)-2\delta} \cdot (1+|q|)^{2(n-1)-4\delta + 2+6\gamma}},\quad\text{when }\quad q>0,\\
           \notag
      \frac{\eps^2  \cdot(1+|q|)^{ } }{(1+t+|q|)^{(n-1)-2\delta} \cdot (1+|q|)^{2(n-1)-4\delta}}  \,\quad\text{when }\quad q<0 . \end{cases} \Big)  \\
      \notag
  &&     +    \Big(  \begin{cases}  \frac{\eps^2 }{(1+t+|q|)^{(n-1)-2\delta} (1+|q|)^{(n-1)-2\delta +2+4\gamma}},\quad\text{when }\quad q>0,\\
           \notag
      \frac{\eps^2  }{(1+t+|q|)^{(n-1)-2\delta} \cdot (1+|q|)^{(n-1)-2\delta}}  \,\quad\text{when }\quad q<0 . \end{cases} \Big)  \\
      \notag
  && +      \Big(  \begin{cases}  \frac{\eps^2 }{(1+t+|q|)^{(n-1)-2\delta} (1+|q|)^{2(n-1)-4\delta + 6\gamma}},\quad\text{when }\quad q>0,\\
           \notag
      \frac{\eps^2 \cdot (1+|q|)^{3 }  }{(1+t+|q|)^{(n-1)-2\delta} \cdot (1+|q|)^{2(n-1)-4\delta}}  \,\quad\text{when }\quad q<0 . \end{cases} \Big) 
               \notag
  \eeaa
    
          \beaa
   \notag
&&  + \Big( \sum_{|K|\leq |I|}   |    \derm (\Lie_{Z^K} h ) |^2 \Big) \cdot E (   \lfloor \frac{|I|}{2} \rfloor + \lfloor  \frac{n}{2} \rfloor  + 1)  \\
 \notag
 && \cdot  \Big( \Big( \begin{cases}  \frac{\eps^2 }{(1+t+|q|)^{(n-1)-2\delta} (1+|q|)^{2+2\gamma}},\quad\text{when }\quad q>0,\\
           \notag
      \frac{\eps^2  }{(1+t+|q|)^{(n-1)-2\delta}(1+|q|)^{ }}  \,\quad\text{when }\quad q<0 . \end{cases} \Big) \\
      \notag
    && +   \Big(  \begin{cases}  \frac{\eps^2 }{(1+t+|q|)^{(n-1)-2\delta} (1+|q|)^{(n-1)-2\delta + 4\gamma}},\quad\text{when }\quad q>0,\\
           \notag
      \frac{\eps^2 \cdot (1+ |q|)^2  }{(1+t+|q|)^{(n-1)-2\delta} \cdot (1+ |q|)^{(n-1)-2\delta} }  \,\quad\text{when }\quad q<0 . \end{cases} \Big) \\
      \notag
&& +        \Big( \begin{cases}  \frac{\eps^2 }{(1+t+|q|)^{(n-1)-2\delta} (1+|q|)^{(n-1)-2\delta+2+4\gamma}},\quad\text{when }\quad q>0,\\
           \notag
      \frac{\eps^2  }{(1+t+|q|)^{(n-1)-2\delta} \cdot (1+|q|)^{(n-1)-2\delta}  } ,\quad\text{when }\quad q<0 . \end{cases} \Big) \\
     \notag
     && + \Big( \begin{cases}  \frac{\eps^2 }{(1+t+|q|)^{(n-1)-2\delta} (1+|q|)^{2(n-1) - 4\de + 6\gamma}},\quad\text{when }\quad q>0,\\
           \notag
      \frac{\eps^2 \cdot (1+|q|)^{3 }  }{(1+t+|q|)^{(n-1)-2\delta} \cdot (1+|q|)^{2(n-1)-4\delta}}  \,\quad\text{when }\quad q<0 . \end{cases} \Big) \Big) 
               \notag
  \eeaa

       \beaa
   \notag
 && + \Big( \sum_{|K|\leq |I|}   |   \Lie_{Z^K} h  |^2 \Big) \cdot E (   \lfloor \frac{|I|}{2} \rfloor + \lfloor  \frac{n}{2} \rfloor  + 1)  \\
 \notag
 && \cdot \Big( \Big(  \begin{cases}  \frac{\eps^2 }{(1+t+|q|)^{2(n-1)-4\delta} (1+|q|)^{4+4\gamma}},\quad\text{when }\quad q>0,\\
           \notag
      \frac{\eps  }{(1+t+|q|)^{2(n-1)-4\delta}(1+|q|)^{2}}  \,\quad\text{when }\quad q<0 . \end{cases} \Big)  \\
      \notag
  && +      \Big( \begin{cases}  \frac{\eps^2 }{(1+t+|q|)^{2(n-1)-4\delta} \cdot (1+|q|)^{(n-1)-2\delta + 2+6\gamma}},\quad\text{when }\quad q>0,\\
           \notag
      \frac{\eps^2  \cdot(1+|q|)^{ } }{(1+t+|q|)^{2(n-1)-4\delta} \cdot (1+|q|)^{(n-1)-2\delta} }  \,\quad\text{when }\quad q<0 . \end{cases} \Big)  \\
      \notag
   && +   \Big(  \begin{cases}  \frac{\eps^2 }{(1+t+|q|)^{2(n-1)-4\delta} (1+|q|)^{2+4\gamma}},\quad\text{when }\quad q>0,\\
           \notag
      \frac{\eps^2  }{(1+t+|q|)^{2(n-1)-4\delta}}  \,\quad\text{when }\quad q<0 . \end{cases} \Big) \\
      \notag
   && +    \Big(  \begin{cases}  \frac{\eps^2 }{(1+t+|q|)^{2(n-1)-4\delta} (1+|q|)^{(n-1)-2\delta +6\gamma}},\quad\text{when }\quad q>0,\\
           \notag
      \frac{\eps^2 \cdot (1+|q|)^{3 }  }{(1+t+|q|)^{2(n-1)-4\delta} \cdot (1+ |q|)^{(n-1)-2\delta} }  \,\quad\text{when }\quad q<0 . \end{cases} \Big) \Big) \; .
               \notag
  \eeaa
 \end{proof}

\begin{lemma}\label{boostraptostudystructureofsourcetermsonthemetricperturbationh}
We have
\beaa
   \notag
 &&  \frac{(1+t )}{\eps} \cdot  | \Lie_{Z^I}   ( g^{\la\mu} \derm_{\la}   \derm_{\mu}    h ) |^2   \\
\notag
 &\les &\Big( \sum_{|K|\leq |I|}    | \derm ( \Lie_{Z^K} A) |^2 \Big) \cdot  E (   \lfloor \frac{|I|}{2} \rfloor + \lfloor  \frac{n}{2} \rfloor  + 1) \\
 \notag
 &&\cdot  \Big(  \Big(\begin{cases}  \frac{\eps }{(1+t+|q|)^{(n-2)-2\delta} (1+|q|)^{2+2\gamma}},\quad\text{when }\quad q>0,\\
           \notag
      \frac{\eps  }{(1+t+|q|)^{(n-2)-2\delta}(1+|q|)^{ }}  \,\quad\text{when }\quad q<0 . \end{cases} \Big) \\
      \notag
      && + \Big( \begin{cases}  \frac{\eps }{(1+t+|q|)^{(n-2)-2\delta} (1+|q|)^{(n-1)-2\delta +4\gamma}},\quad\text{when }\quad q>0,\\
           \notag
      \frac{\eps \cdot (1+ |q|)^2  }{(1+t+|q|)^{(n-2)-2\delta} \cdot (1+|q|)^{(n-1)-2\delta}}  \,\quad\text{when }\quad q<0 . \end{cases} \Big) \\
      \notag
    &&  +  \Big( \begin{cases}  \frac{\eps }{(1+t+|q|)^{(n-2)-2\delta} (1+|q|)^{(n-1)-2\delta+2+4\gamma}},\quad\text{when }\quad q>0,\\
           \notag
      \frac{\eps  }{(1+t+|q|)^{(n-2)-2\delta} \cdot (1+|q|)^{(n-1)-2\delta}}  \,\quad\text{when }\quad q<0 . \end{cases} \Big) \\
      \notag
   && + \Big(  \begin{cases}  \frac{\eps }{(1+t+|q|)^{(n-2)-2\delta} (1+|q|)^{2(n-1)-4\delta +6\gamma}},\quad\text{when }\quad q>0,\\
           \notag
      \frac{\eps \cdot (1+|q|)^{3 }  }{(1+t+|q|)^{(n-2)-2\delta} \cdot (1+|q|)^{2(n-1)-4\delta}}  \,\quad\text{when }\quad q<0 . \end{cases} \Big) \Big) 
               \notag
\eeaa

 \beaa
\notag
 &&+ \Big( \sum_{|K|\leq |I|}    |  \Lie_{Z^K} A |^2 \Big) \cdot E (   \lfloor \frac{|I|}{2} \rfloor + \lfloor  \frac{n}{2} \rfloor  + 1)   \\
 \notag
 &&   \cdot \Big(   \Big(  \begin{cases}  \frac{\eps }{(1+t+|q|)^{2(n-\frac{3}{2})-4\delta} (1+|q|)^{2+4\gamma}},\quad\text{when }\quad q>0,\\
           \notag
      \frac{\eps  }{(1+t+|q|)^{2(n-\frac{3}{2})-4\delta}}  \,\quad\text{when }\quad q<0 . \end{cases} \Big) \\
      \notag
  &&  +   \Big(  \begin{cases}  \frac{\eps }{(1+t+|q|)^{2(n-\frac{3}{2})-4\delta} (1+|q|)^{(n-1)-2\delta + 6\gamma}},\quad\text{when }\quad q>0,\\
           \notag
      \frac{\eps \cdot (1+|q|)^{3 }  }{(1+t+|q|)^{2(n-\frac{3}{2})-4\delta } \cdot (1+|q|)^{(n-1)-2\delta }}  \,\quad\text{when }\quad q<0 . \end{cases} \Big) \\
               \notag
 &&   +   \Big( \begin{cases}  \frac{\eps }{(1+t+|q|)^{2(n-\frac{3}{2})-4\delta} (1+|q|)^{(n-1)-2\delta +2+6\gamma}},\quad\text{when }\quad q>0,\\
           \notag
      \frac{\eps  \cdot(1+|q|)^{ } }{(1+t+|q|)^{2(n-\frac{3}{2})-4\delta} \cdot (1+|q|)^{(n-1)-2\delta}}  \,\quad\text{when }\quad q<0 . \end{cases} \Big)  \\
      \notag
   &&  + \Big(  \begin{cases}  \frac{\eps }{(1+t+|q|)^{2(n-\frac{3}{2})-4\delta} \cdot (1+|q|)^{2(n-1)-4\delta +8\gamma}},\quad\text{when }\quad q>0,\\
           \notag
      \frac{\eps \cdot (1+|q|)^{4 }  }{(1+t+|q|)^{2(n-\frac{3}{2})-4\delta} \cdot (1+|q|)^{2(n-1)-4\delta} }  \,\quad\text{when }\quad q<0 . \end{cases} \Big) \Big)  
               \notag
 \eeaa
 
 \beaa
\notag
 &&+ \Big( \sum_{|K|\leq |I|}    | \derm(  \Lie_{Z^K} h ) |^2 \Big) \cdot E (   \lfloor \frac{|I|}{2} \rfloor + \lfloor  \frac{n}{2} \rfloor  + 1) \\
 \notag
 && \cdot \Big(   \Big(  \begin{cases}  \frac{\eps }{(1+t+|q|)^{(n-2)-2\delta} (1+|q|)^{2+2\gamma}},\quad\text{when }\quad q>0,\\
 \notag
      \frac{\eps  }{(1+t+|q|)^{(n-2)-2\delta}(1+|q|)^{ }}  \,\quad\text{when }\quad q<0 . \end{cases} \Big) \\
      \notag
   && + \Big(  \begin{cases}  \frac{\eps }{(1+t+|q|)^{(n-2)-2\delta} (1+|q|)^{(n-1)-2\delta+2+4\gamma}},\quad\text{when }\quad q>0,\\
           \notag
      \frac{\eps  }{(1+t+|q|)^{(n-2)-2\delta} \cdot (1+|q|)^{(n-1)-2\delta}}  \,\quad\text{when }\quad q<0 . \end{cases} \Big) \Big) 
      \notag
 \eeaa

 \beaa
\notag
 &&+ \Big(   \sum_{|K|\leq |I|}    |   \Lie_{Z^K} h  |^2 \Big) \cdot  E (   \lfloor \frac{|I|}{2} \rfloor + \lfloor  \frac{n}{2} \rfloor  + 1) \\
 \notag
   &&\Big(   \Big(  \begin{cases}  \frac{\eps }{(1+t+|q|)^{2(n-\frac{3}{2})-4\delta} (1+|q|)^{4+4\gamma}},\quad\text{when }\quad q>0,\\
           \notag
      \frac{\eps  }{(1+t+|q|)^{2(n-\frac{3}{2})-4\delta}(1+|q|)^2}  \,\quad\text{when }\quad q<0 . \end{cases} \Big) \\
      \notag
      && + \Big(  \begin{cases}  \frac{\eps }{(1+t+|q|)^{2(n-\frac{3}{2})-4\delta} (1+|q|)^{(n-1)-2\delta+2+6\gamma}},\quad\text{when }\quad q>0,\\
           \notag
      \frac{\eps  \cdot(1+|q|)^{ } }{(1+t+|q|)^{2(n-\frac{3}{2})-4\delta} \cdot (1+|q|)^{(n-1)-2\delta}}  \,\quad\text{when }\quad q<0 . \end{cases} \Big)  \\
      \notag
  &&  + \Big(  \begin{cases}  \frac{\eps }{(1+t+|q|)^{2(n-\frac{3}{2})-4\delta} \cdot (1+|q|)^{2(n-1)-4\delta +8\gamma}},\quad\text{when }\quad q>0,\\
           \notag
      \frac{\eps \cdot (1+|q|)^{4 }  }{(1+t+|q|)^{2(n-\frac{3}{2})-4\delta} \cdot (1+|q|)^{2(n-1)-4\delta} }  \,\quad\text{when }\quad q<0 . \end{cases} \Big) \Big)  \; .\\
               \notag
 \eeaa
 
\end{lemma}

\begin{proof}

We have

\beaa
   \notag
 &&  | \Lie_{Z^I}   ( g^{\la\mu} \derm_{\la}   \derm_{\mu}    h ) |^2   \\
\notag
 &\les &\Big( \sum_{|K|\leq |I|}    | \derm ( \Lie_{Z^K} A) |^2 \Big) \cdot  E (   \lfloor \frac{|I|}{2} \rfloor + \lfloor  \frac{n}{2} \rfloor  + 1) \\
 \notag
 &&\cdot  \Big(  \Big(\begin{cases}  \frac{\eps^2 }{(1+t+|q|)^{(n-1)-2\delta} (1+|q|)^{2+2\gamma}},\quad\text{when }\quad q>0,\\
           \notag
      \frac{\eps^2  }{(1+t+|q|)^{(n-1)-2\delta}(1+|q|)^{ }}  \,\quad\text{when }\quad q<0 . \end{cases} \Big) \\
      \notag
      && + \Big( \begin{cases}  \frac{\eps^2 }{(1+t+|q|)^{(n-1)-2\delta} (1+|q|)^{(n-1)-2\delta +4\gamma}},\quad\text{when }\quad q>0,\\
           \notag
      \frac{\eps^2 \cdot (1+ |q|)^2  }{(1+t+|q|)^{(n-1)-2\delta} \cdot (1+|q|)^{(n-1)-2\delta}}  \,\quad\text{when }\quad q<0 . \end{cases} \Big) \\
      \notag
    &&  +  \Big( \begin{cases}  \frac{\eps^2 }{(1+t+|q|)^{(n-1)-2\delta} (1+|q|)^{(n-1)-2\delta+2+4\gamma}},\quad\text{when }\quad q>0,\\
           \notag
      \frac{\eps^2  }{(1+t+|q|)^{(n-1)-2\delta} \cdot (1+|q|)^{(n-1)-2\delta}}  \,\quad\text{when }\quad q<0 . \end{cases} \Big) \\
      \notag
   && + \Big(  \begin{cases}  \frac{\eps^2 }{(1+t+|q|)^{(n-1)-2\delta} (1+|q|)^{2(n-1)-4\delta +6\gamma}},\quad\text{when }\quad q>0,\\
           \notag
      \frac{\eps^2 \cdot (1+|q|)^{3 }  }{(1+t+|q|)^{(n-1)-2\delta} \cdot (1+|q|)^{2(n-1)-4\delta}}  \,\quad\text{when }\quad q<0 . \end{cases} \Big) \Big) 
               \notag
\eeaa

\beaa
\notag
 &&+ \Big( \sum_{|K|\leq |I|}    |  \Lie_{Z^K} A |^2 \Big) \cdot E (   \lfloor \frac{|I|}{2} \rfloor + \lfloor  \frac{n}{2} \rfloor  + 1)   \\
 \notag
 &&   \cdot \Big(   \Big(  \begin{cases}  \frac{\eps^2 }{(1+t+|q|)^{2(n-1)-4\delta} (1+|q|)^{2+4\gamma}},\quad\text{when }\quad q>0,\\
           \notag
      \frac{\eps^2  }{(1+t+|q|)^{2(n-1)-4\delta}}  \,\quad\text{when }\quad q<0 . \end{cases} \Big) \\
      \notag
  &&  +   \Big(  \begin{cases}  \frac{\eps^2 }{(1+t+|q|)^{2(n-1)-4\delta} (1+|q|)^{(n-1)-2\delta + 6\gamma}},\quad\text{when }\quad q>0,\\
           \notag
      \frac{\eps^2 \cdot (1+|q|)^{3 }  }{(1+t+|q|)^{2(n-1)-4\delta } \cdot (1+|q|)^{(n-1)-2\delta }}  \,\quad\text{when }\quad q<0 . \end{cases} \Big) \\
               \notag
 &&   +   \Big( \begin{cases}  \frac{\eps^2 }{(1+t+|q|)^{2(n-1)-4\delta} (1+|q|)^{(n-1)-2\delta +2+6\gamma}},\quad\text{when }\quad q>0,\\
           \notag
      \frac{\eps^2  \cdot(1+|q|)^{ } }{(1+t+|q|)^{2(n-1)-4\delta} \cdot (1+|q|)^{(n-1)-2\delta}}  \,\quad\text{when }\quad q<0 . \end{cases} \Big)  \\
      \notag
   &&  + \Big(  \begin{cases}  \frac{\eps^2 }{(1+t+|q|)^{2(n-1)-4\delta} \cdot (1+|q|)^{2(n-1)-4\delta +8\gamma}},\quad\text{when }\quad q>0,\\
           \notag
      \frac{\eps^2 \cdot (1+|q|)^{4 }  }{(1+t+|q|)^{2(n-1)-4\delta} \cdot (1+|q|)^{2(n-1)-4\delta} }  \,\quad\text{when }\quad q<0 . \end{cases} \Big) \Big)  
               \notag
 \eeaa
 
\beaa
\notag
 &&+ \Big( \sum_{|K|\leq |I|}    | \derm(  \Lie_{Z^K} h ) |^2 \Big) \cdot E (   \lfloor \frac{|I|}{2} \rfloor + \lfloor  \frac{n}{2} \rfloor  + 1) \\
 \notag
 && \cdot \Big(   \Big(  \begin{cases}  \frac{\eps^2 }{(1+t+|q|)^{(n-1)-2\delta} (1+|q|)^{2+2\gamma}},\quad\text{when }\quad q>0,\\
 \notag
      \frac{\eps^2  }{(1+t+|q|)^{(n-1)-2\delta}(1+|q|)^{ }}  \,\quad\text{when }\quad q<0 . \end{cases} \Big) \\
      \notag
   && + \Big(  \begin{cases}  \frac{\eps^2 }{(1+t+|q|)^{(n-1)-2\delta} (1+|q|)^{(n-1)-2\delta+2+4\gamma}},\quad\text{when }\quad q>0,\\
           \notag
      \frac{\eps^2  }{(1+t+|q|)^{(n-1)-2\delta} \cdot (1+|q|)^{(n-1)-2\delta}}  \,\quad\text{when }\quad q<0 . \end{cases} \Big) \Big) 
      \notag
 \eeaa

\beaa
\notag
 &&+ \Big(   \sum_{|K|\leq |I|}    |   \Lie_{Z^K} h  |^2 \Big) \cdot  E (   \lfloor \frac{|I|}{2} \rfloor + \lfloor  \frac{n}{2} \rfloor  + 1) \\
 \notag
   &&\Big(   \Big(  \begin{cases}  \frac{\eps^2 }{(1+t+|q|)^{2(n-1)-4\delta} (1+|q|)^{4+4\gamma}},\quad\text{when }\quad q>0,\\
           \notag
      \frac{\eps^2  }{(1+t+|q|)^{2(n-1)-4\delta}(1+|q|)^2}  \,\quad\text{when }\quad q<0 . \end{cases} \Big) \\
      \notag
      && + \Big(  \begin{cases}  \frac{\eps^2 }{(1+t+|q|)^{2(n-1)-4\delta} (1+|q|)^{(n-1)-2\delta+2+6\gamma}},\quad\text{when }\quad q>0,\\
           \notag
      \frac{\eps^2  \cdot(1+|q|)^{ } }{(1+t+|q|)^{2(n-1)-4\delta} \cdot (1+|q|)^{(n-1)-2\delta}}  \,\quad\text{when }\quad q<0 . \end{cases} \Big)  \\
      \notag
  &&  + \Big(  \begin{cases}  \frac{\eps^2 }{(1+t+|q|)^{2(n-1)-4\delta} \cdot (1+|q|)^{2(n-1)-4\delta +8\gamma}},\quad\text{when }\quad q>0,\\
           \notag
      \frac{\eps^2 \cdot (1+|q|)^{4 }  }{(1+t+|q|)^{2(n-1)-4\delta} \cdot (1+|q|)^{2(n-1)-4\delta} }  \,\quad\text{when }\quad q<0 . \end{cases} \Big) \Big)  \; .\\
               \notag
 \eeaa

\end{proof}

\subsection{The source terms for $n\geq 5$}\

\begin{lemma}\label{ttimesthesquareofthesourcetermsforAforngeq5}
For $n \geq 5$, $\eps \leq 1$, we have
 \beaa
 \notag
&& (1+t ) \cdot |  \Lie_{Z^I}   ( g^{\la\mu} \derm_{\la}   \derm_{\mu}   A  )   |^2 \, \\
   \notag
   &\les&  \sum_{|K|\leq |I|}  \Big( |    \derm (\Lie_{Z^K} A ) |^2 +  |    \derm (\Lie_{Z^K} h ) |^2 \Big) \cdot E (   \lfloor \frac{|I|}{2} \rfloor + \lfloor  \frac{n}{2} \rfloor  + 1)\\
 \notag
   && \cdot \Big(  \begin{cases}  \frac{\eps }{(1+t+|q|)^{3-2\delta} (1+|q|)^{2\gamma}},\quad\text{when }\quad q>0,\\
           \notag
      \frac{\eps  }{(1+t+|q|)^{3-2\delta} \cdot ( 1+|q| )^{-1} } \,\quad\text{when }\quad q<0 . \end{cases} \Big) 
      \notag 
   \eeaa
     \beaa
     \notag
   && + \sum_{|K|\leq |I|}  \Big( |   \Lie_{Z^K} A  |^2 +|   \Lie_{Z^K} h  |^2 \Big) \cdot E (   \lfloor \frac{|I|}{2} \rfloor + \lfloor  \frac{n}{2} \rfloor  + 1) \\
 \notag
   && \cdot \Big(  \begin{cases}  \frac{\eps }{(1+t+|q|)^{3-2\delta} (1+|q|)^{2+2\gamma}},\quad\text{when }\quad q>0,\\
           \notag
      \frac{\eps  }{(1+t+|q|)^{3-2\delta} \cdot ( 1+|q| )^{} } \,\quad\text{when }\quad q<0 . \end{cases} \Big) \; .
      \notag 
   \eeaa

\end{lemma}

\begin{proof}

For $n \geq 5$, we examine one by one the terms in $ (1+t ) \cdot |  \Lie_{Z^I}   ( g^{\la\mu} \derm_{\la}   \derm_{\mu}   A  )   |^2 $\;. We get

\beaa
   \notag
 &&   \sum_{|K|\leq |I|}  \Big( |    \derm (\Lie_{Z^K} A ) |^2 \Big) \cdot E (   \lfloor \frac{|I|}{2} \rfloor + \lfloor  \frac{n}{2} \rfloor  + 1)\\
 \notag
 && \cdot \Big( \begin{cases}  \frac{\eps }{(1+t+|q|)^{3-2\delta} (1+|q|)^{2+2\gamma}},\quad\text{when }\quad q>0,\\
           \notag
      \frac{\eps  }{(1+t+|q|)^{3-2\delta}(1+|q|)}  \,\quad\text{when }\quad q<0 . \end{cases} \Big) \\
      \notag
      && +   \Big( \begin{cases}  \frac{\eps }{(1+t+|q|)^{3-2\delta} (1+|q|)^{2\gamma}},\quad\text{when }\quad q>0,\\
           \notag
      \frac{\eps  \cdot (1+|q|)^{ } }{(1+t+|q|)^{3-2\delta}}  \,\quad\text{when }\quad q<0 . \end{cases} \Big) \\
         \notag
         && + \Big( \begin{cases}  \frac{\eps }{(1+t+|q|)^{3-2\delta} (1+|q|)^{6 + 2(\ga-\delta) +2\gamma}},\quad\text{when }\quad q>0,\\
           \notag
      \frac{\eps  }{ (1+t+|q|)^{3-2\delta} \cdot (1+ |q|)^{4-2\delta}  }  \,\quad\text{when }\quad q<0 . \end{cases} \Big) \\
      \notag
   && +    \Big(  \begin{cases}  \frac{\eps }{(1+t+|q|)^{3-2\delta} (1+|q|)^{4+2(\ga-\delta)+2\gamma}},\quad\text{when }\quad q>0,\\
           \notag
      \frac{\eps \cdot (1+ |q|)  }{(1+t+|q|)^{3-2\delta} \cdot (1+ |q|)^{4-2\delta} } \,\quad\text{when }\quad q<0 . \end{cases} \Big) \Big) 
               \notag
  \eeaa

     \beaa
     \notag
   &\les&  \sum_{|K|\leq |I|}  \Big( |    \derm (\Lie_{Z^K} A ) |^2 \Big) \cdot E (   \lfloor \frac{|I|}{2} \rfloor + \lfloor  \frac{n}{2} \rfloor  + 1)\\
 \notag
   && \cdot \Big(  \begin{cases}  \frac{\eps }{(1+t+|q|)^{3-2\delta} (1+|q|)^{2\gamma}},\quad\text{when }\quad q>0,\\
           \notag
      \frac{\eps  }{(1+t+|q|)^{3-2\delta} \cdot ( 1+|q| )^{-1} } \,\quad\text{when }\quad q<0 . \end{cases} \Big) \;.
      \notag 
   \eeaa

And,
       \beaa
   \notag
 &&   \sum_{|K|\leq |I|}  \Big( |   \Lie_{Z^K} A  |^2 \Big) \cdot E (   \lfloor \frac{|I|}{2} \rfloor + \lfloor  \frac{n}{2} \rfloor  + 1) \\
 \notag
 &&  \cdot \Big( \Big(  \begin{cases}  \frac{\eps }{(1+t+|q|)^{3-2\delta} (1+|q|)^{6+2(\ga-\delta)+2\gamma}},\quad\text{when }\quad q>0,\\
           \notag
      \frac{\eps  }{(1+t+|q|)^{3-2\delta} \cdot ( 1+|q| )^{4-2\delta} } \,\quad\text{when }\quad q<0 . \end{cases} \Big) \\
      \notag 
      && + \Big(  \begin{cases}  \frac{\eps }{(1+t+|q|)^{3-2\delta} (1+|q|)^{2+2\gamma}},\quad\text{when }\quad q>0,\\
           \notag
      \frac{\eps  }{(1+t+|q|)^{3-2\delta}(1+|q|)^{ }}  \,\quad\text{when }\quad q<0 . \end{cases} \Big) \\
         \notag
       && +  \Big(  \begin{cases}  \frac{\eps }{(1+t+|q|)^{3-2\delta} (1+|q|)^{4+2(\ga-\delta)+ 2\gamma}},\quad\text{when }\quad q>0,\\
           \notag
      \frac{\eps \cdot (1+ |q|)^2  }{(1+t+|q|)^{3-2\delta } \cdot (1+|q|)^{4-2\delta}} \,\quad\text{when }\quad q<0 . \end{cases} \Big)  \\
               \notag
     && +             \Big(  \begin{cases}  \frac{\eps }{(1+t+|q|)^{3-2\delta} \cdot (1+|q|)^{10+4(\ga-\delta) +2\gamma}},\quad\text{when }\quad q>0,\\
           \notag
      \frac{\eps  \cdot(1+|q|)^{ } }{(1+t+|q|)^{3-2\delta} \cdot (1+|q|)^{8-4\delta}}  \,\quad\text{when }\quad q<0 . \end{cases} \Big)  \\
      \notag
  &&     +    \Big(  \begin{cases}  \frac{\eps }{(1+t+|q|)^{3-2\delta} (1+|q|)^{6+2(\ga-\delta) +2\gamma}},\quad\text{when }\quad q>0,\\
           \notag
      \frac{\eps  }{(1+t+|q|)^{3-2\delta} \cdot (1+|q|)^{4-2\delta}}  \,\quad\text{when }\quad q<0 . \end{cases} \Big)  \\
      \notag
  && +      \Big(  \begin{cases}  \frac{\eps }{(1+t+|q|)^{3-2\delta} (1+|q|)^{8+4(\ga-\delta) + 2\gamma}},\quad\text{when }\quad q>0,\\
           \notag
      \frac{\eps \cdot (1+|q|)^{3 }  }{(1+t+|q|)^{3-2\delta} \cdot (1+|q|)^{8-4\delta}}  \,\quad\text{when }\quad q<0 . \end{cases} \Big) 
               \notag
  \eeaa
 
     \beaa
     \notag
   &\les&  \sum_{|K|\leq |I|}  \Big( |   \Lie_{Z^K} A  |^2 \Big) \cdot E (   \lfloor \frac{|I|}{2} \rfloor + \lfloor  \frac{n}{2} \rfloor  + 1) \\
 \notag
   && \cdot \Big(  \begin{cases}  \frac{\eps }{(1+t+|q|)^{3-2\delta} (1+|q|)^{2+2\gamma}},\quad\text{when }\quad q>0,\\
           \notag
      \frac{\eps  }{(1+t+|q|)^{3-2\delta} \cdot ( 1+|q| )^{} } \,\quad\text{when }\quad q<0 , \end{cases} \Big) 
      \notag 
   \eeaa
       (where we used the fact that $\ga \geq \delta$).
   
And,
 
           \beaa
   \notag
&&   \Big( \sum_{|K|\leq |I|}   |    \derm (\Lie_{Z^K} h ) |^2 \Big) \cdot E (   \lfloor \frac{|I|}{2} \rfloor + \lfloor  \frac{n}{2} \rfloor  + 1)  \\
 \notag
 && \cdot  \Big( \Big( \begin{cases}  \frac{\eps }{(1+t+|q|)^{3-2\delta} (1+|q|)^{2+2\gamma}},\quad\text{when }\quad q>0,\\
           \notag
      \frac{\eps  }{(1+t+|q|)^{3-2\delta}(1+|q|)^{ }}  \,\quad\text{when }\quad q<0 . \end{cases} \Big) \\
      \notag
    && +   \Big(  \begin{cases}  \frac{\eps }{(1+t+|q|)^{3-2\delta} (1+|q|)^{4+2(\ga-\delta) + 2\gamma}},\quad\text{when }\quad q>0,\\
           \notag
      \frac{\eps \cdot (1+ |q|)^2  }{(1+t+|q|)^{3-2\delta} \cdot (1+ |q|)^{4-2\delta} }  \,\quad\text{when }\quad q<0 . \end{cases} \Big) \\
      \notag
&& +        \Big( \begin{cases}  \frac{\eps }{(1+t+|q|)^{3-2\delta} (1+|q|)^{6+2(\ga-\delta)+2\gamma}},\quad\text{when }\quad q>0,\\
           \notag
      \frac{\eps  }{(1+t+|q|)^{3-2\delta} \cdot (1+|q|)^{4-2\delta}  } ,\quad\text{when }\quad q<0 . \end{cases} \Big) \\
     \notag
     && + \Big( \begin{cases}  \frac{\eps }{(1+t+|q|)^{3-2\delta} (1+|q|)^{8 +4(\ga-\delta)+ 2\gamma}},\quad\text{when }\quad q>0,\\
           \notag
      \frac{\eps \cdot (1+|q|)^{3 }  }{(1+t+|q|)^{3-2\delta} \cdot (1+|q|)^{8-4\delta}}  \,\quad\text{when }\quad q<0 . \end{cases} \Big) \Big) 
               \notag
  \eeaa
      \beaa
     \notag
   &\les&  \Big( \sum_{|K|\leq |I|}   |    \derm (\Lie_{Z^K} h ) |^2 \Big) \cdot E (   \lfloor \frac{|I|}{2} \rfloor + \lfloor  \frac{n}{2} \rfloor  + 1)  \\
 \notag
   && \cdot \Big(  \begin{cases}  \frac{\eps }{(1+t+|q|)^{3-2\delta} (1+|q|)^{2+2\gamma}},\quad\text{when }\quad q>0,\\
           \notag
      \frac{\eps  }{(1+t+|q|)^{3-2\delta} \cdot ( 1+|q| )^{} } \,\quad\text{when }\quad q<0 , \end{cases} \Big) 
      \notag 
   \eeaa
      (using the fact that $\ga \geq \delta$).

Also,
  
       \beaa
   \notag
 &&  \Big( \sum_{|K|\leq |I|}   |   \Lie_{Z^K} h  |^2 \Big) \cdot E (   \lfloor \frac{|I|}{2} \rfloor + \lfloor  \frac{n}{2} \rfloor  + 1)  \\
 \notag
 && \cdot \Big( \Big(  \begin{cases}  \frac{\eps }{(1+t+|q|)^{7-4\delta} (1+|q|)^{4+4\gamma}},\quad\text{when }\quad q>0,\\
           \notag
      \frac{\eps  }{(1+t+|q|)^{7-4\delta}(1+|q|)^{2}}  \,\quad\text{when }\quad q<0 . \end{cases} \Big)  \\
      \notag
  && +      \Big( \begin{cases}  \frac{\eps }{(1+t+|q|)^{7-4\delta} \cdot (1+|q|)^{6+2(\ga-\delta) +4\gamma}},\quad\text{when }\quad q>0,\\
           \notag
      \frac{\eps  \cdot(1+|q|)^{ } }{(1+t+|q|)^{7-4\delta} \cdot (1+|q|)^{4-2\delta} }  \,\quad\text{when }\quad q<0 . \end{cases} \Big)  \\
      \notag
   && +   \Big(  \begin{cases}  \frac{\eps }{(1+t+|q|)^{7-4\delta} (1+|q|)^{2+4\gamma}},\quad\text{when }\quad q>0,\\
           \notag
      \frac{\eps  }{(1+t+|q|)^{7-4\delta}}  \,\quad\text{when }\quad q<0 . \end{cases} \Big) \\
      \notag
   && +    \Big(  \begin{cases}  \frac{\eps }{(1+t+|q|)^{7-4\delta} (1+|q|)^{4+2(\ga-\delta) +4\gamma}},\quad\text{when }\quad q>0,\\
           \notag
      \frac{\eps \cdot (1+|q|)^{3 }  }{(1+t+|q|)^{7-4\delta} \cdot (1+ |q|)^{4-2\delta} }  \,\quad\text{when }\quad q<0 . \end{cases} \Big) \Big) 
               \notag
  \eeaa
     \beaa
     \notag
   &\les&  \Big( \sum_{|K|\leq |I|}   |   \Lie_{Z^K} h  |^2 \Big) \cdot E (   \lfloor \frac{|I|}{2} \rfloor + \lfloor  \frac{n}{2} \rfloor  + 1)   \\
 \notag
   && \cdot \Big(  \begin{cases}  \frac{\eps }{(1+t+|q|)^{7-4\delta} (1+|q|)^{2+4\gamma}},\quad\text{when }\quad q>0,\\
           \notag
      \frac{\eps  }{(1+t+|q|)^{7-4\delta} } \,\quad\text{when }\quad q<0 . \end{cases} \Big) \;.
      \notag 
   \eeaa
  \end{proof}

\begin{lemma}\label{ttimesthesquareofthesourcetermsforthemetrichforngeq5}
For $n \geq 5$, 
\beaa
   \notag
 &&  (1+t ) \cdot  | \Lie_{Z^I}   ( g^{\la\mu} \derm_{\la}   \derm_{\mu}    h ) |^2   \\
     \notag
   &\les&  \sum_{|K|\leq |I|}  \Big( |    \derm (\Lie_{Z^K} A ) |^2 + |    \derm (\Lie_{Z^K} h ) |^2 \Big) \cdot E (   \lfloor \frac{|I|}{2} \rfloor + \lfloor  \frac{n}{2} \rfloor  + 1)\\
 \notag
   && \cdot \Big(  \begin{cases}  \frac{\eps }{(1+t+|q|)^{3-2\delta} (1+|q|)^{2+2\gamma}},\quad\text{when }\quad q>0,\\
           \notag
      \frac{\eps  }{(1+t+|q|)^{3-2\delta} \cdot ( 1+|q| )^{} } \,\quad\text{when }\quad q<0 . \end{cases} \Big) \\
      \notag 
   &&+  \sum_{|K|\leq |I|}  \Big( |   \Lie_{Z^K} A  |^2 + |   \Lie_{Z^K} h  |^2 \Big) \cdot E (   \lfloor \frac{|I|}{2} \rfloor + \lfloor  \frac{n}{2} \rfloor  + 1)\\
 \notag
   && \cdot \Big(  \begin{cases}  \frac{\eps }{(1+t+|q|)^{7-4\delta} (1+|q|)^{2+4\gamma}},\quad\text{when }\quad q>0,\\
           \notag
      \frac{\eps  }{(1+t+|q|)^{7-4\delta}  } \,\quad\text{when }\quad q<0 . \end{cases} \Big) \; .
      \notag 
   \eeaa

\end{lemma}

\begin{proof}

For $n \geq 5$, we examine the terms in $(1+t ) \cdot  | \Lie_{Z^I}   ( g^{\la\mu} \derm_{\la}   \derm_{\mu}    h ) |^2 $, one by one. We have
\beaa
   \notag
&& \Big( \sum_{|K|\leq |I|}    | \derm ( \Lie_{Z^K} A) |^2 \Big) \cdot  E (   \lfloor \frac{|I|}{2} \rfloor + \lfloor  \frac{n}{2} \rfloor  + 1) \\
 \notag
 &&\cdot  \Big(  \Big(\begin{cases}  \frac{\eps }{(1+t+|q|)^{3-2\delta} (1+|q|)^{2+2\gamma}},\quad\text{when }\quad q>0,\\
           \notag
      \frac{\eps  }{(1+t+|q|)^{3-2\delta}(1+|q|)^{ }}  \,\quad\text{when }\quad q<0 . \end{cases} \Big) \\
      \notag
      && + \Big( \begin{cases}  \frac{\eps }{(1+t+|q|)^{3-2\delta} (1+|q|)^{4+2(\ga-\delta) +2\gamma}},\quad\text{when }\quad q>0,\\
           \notag
      \frac{\eps \cdot (1+ |q|)^2  }{(1+t+|q|)^{3-2\delta} \cdot (1+|q|)^{4-2\delta}}  \,\quad\text{when }\quad q<0 . \end{cases} \Big) \\
      \notag
    &&  +  \Big( \begin{cases}  \frac{\eps }{(1+t+|q|)^{3-2\delta} (1+|q|)^{6+2(\ga-\delta)+2\gamma}},\quad\text{when }\quad q>0,\\
           \notag
      \frac{\eps  }{(1+t+|q|)^{3-2\delta} \cdot (1+|q|)^{4-2\delta}}  \,\quad\text{when }\quad q<0 . \end{cases} \Big) \\
      \notag
   && + \Big(  \begin{cases}  \frac{\eps }{(1+t+|q|)^{3-2\delta} (1+|q|)^{8+4(\ga-\delta) +2\gamma}},\quad\text{when }\quad q>0,\\
           \notag
      \frac{\eps \cdot (1+|q|)^{3 }  }{(1+t+|q|)^{3-2\delta} \cdot (1+|q|)^{8-4\delta}}  \,\quad\text{when }\quad q<0 . \end{cases} \Big) \Big) 
               \notag
\eeaa

     \beaa
     \notag
   &\les&  \sum_{|K|\leq |I|}  \Big( |    \derm (\Lie_{Z^K} A ) |^2 \Big) \cdot E (   \lfloor \frac{|I|}{2} \rfloor + \lfloor  \frac{n}{2} \rfloor  + 1)\\
 \notag
   && \cdot \Big(  \begin{cases}  \frac{\eps }{(1+t+|q|)^{3-2\delta} (1+|q|)^{2+2\gamma}},\quad\text{when }\quad q>0,\\
           \notag
      \frac{\eps  }{(1+t+|q|)^{3-2\delta} \cdot ( 1+|q| )^{} } \,\quad\text{when }\quad q<0 . \end{cases} \Big) \; .
      \notag 
   \eeaa

And,
 \beaa
\notag
 && \Big( \sum_{|K|\leq |I|}    |  \Lie_{Z^K} A |^2 \Big) \cdot E (   \lfloor \frac{|I|}{2} \rfloor + \lfloor  \frac{n}{2} \rfloor  + 1)   \\
 \notag
 &&   \cdot \Big(   \Big(  \begin{cases}  \frac{\eps }{(1+t+|q|)^{7-4\delta} (1+|q|)^{2+4\gamma}},\quad\text{when }\quad q>0,\\
           \notag
      \frac{\eps  }{(1+t+|q|)^{7-4\delta}}  \,\quad\text{when }\quad q<0 . \end{cases} \Big) \\
      \notag
  &&  +   \Big(  \begin{cases}  \frac{\eps }{(1+t+|q|)^{7-4\delta} (1+|q|)^{4+2(\ga-\delta) + 4\gamma}},\quad\text{when }\quad q>0,\\
           \notag
      \frac{\eps \cdot (1+|q|)^{3 }  }{(1+t+|q|)^{7-4\delta } \cdot (1+|q|)^{4-2\delta }}  \,\quad\text{when }\quad q<0 . \end{cases} \Big) \\
               \notag
 &&   +   \Big( \begin{cases}  \frac{\eps }{(1+t+|q|)^{7-4\delta} (1+|q|)^{6+2(\ga-\delta) +4\gamma}},\quad\text{when }\quad q>0,\\
           \notag
      \frac{\eps  \cdot(1+|q|)^{ } }{(1+t+|q|)^{7-4\delta} \cdot (1+|q|)^{4-2\delta}}  \,\quad\text{when }\quad q<0 . \end{cases} \Big)  \\
      \notag
   &&  + \Big(  \begin{cases}  \frac{\eps }{(1+t+|q|)^{7-4\delta} \cdot (1+|q|)^{8+4(\ga-\delta) +4\gamma}},\quad\text{when }\quad q>0,\\
           \notag
      \frac{\eps \cdot (1+|q|)^{4 }  }{(1+t+|q|)^{7-4\delta} \cdot (1+|q|)^{8-4\delta} }  \,\quad\text{when }\quad q<0 . \end{cases} \Big) \Big)  
               \notag
 \eeaa

     \beaa
     \notag
   &\les&  \sum_{|K|\leq |I|}  \Big( |   \Lie_{Z^K} A  |^2 \Big) \cdot E (   \lfloor \frac{|I|}{2} \rfloor + \lfloor  \frac{n}{2} \rfloor  + 1)\\
 \notag
   && \cdot \Big(  \begin{cases}  \frac{\eps }{(1+t+|q|)^{7-4\delta} (1+|q|)^{2+4\gamma}},\quad\text{when }\quad q>0,\\
           \notag
      \frac{\eps  }{(1+t+|q|)^{7-4\delta}  } \,\quad\text{when }\quad q<0 . \end{cases} \Big) \; .
      \notag 
   \eeaa

And,
 \beaa
\notag
 && \Big( \sum_{|K|\leq |I|}    | \derm(  \Lie_{Z^K} h ) |^2 \Big) \cdot E (   \lfloor \frac{|I|}{2} \rfloor + \lfloor  \frac{n}{2} \rfloor  + 1) \\
 \notag
 && \cdot \Big(   \Big(  \begin{cases}  \frac{\eps }{(1+t+|q|)^{3-2\delta} (1+|q|)^{2+2\gamma}},\quad\text{when }\quad q>0,\\
 \notag
      \frac{\eps  }{(1+t+|q|)^{3-2\delta}(1+|q|)^{ }}  \,\quad\text{when }\quad q<0 . \end{cases} \Big) \\
      \notag
   && + \Big(  \begin{cases}  \frac{\eps }{(1+t+|q|)^{3-2\delta} (1+|q|)^{6+2(\ga-\delta)+2\gamma}},\quad\text{when }\quad q>0,\\
           \notag
      \frac{\eps  }{(1+t+|q|)^{3-2\delta} \cdot (1+|q|)^{4-2\delta}}  \,\quad\text{when }\quad q<0 . \end{cases} \Big) \Big) 
      \notag
 \eeaa

     \beaa
     \notag
   &\les&  \sum_{|K|\leq |I|}  \Big( |    \derm (\Lie_{Z^K} h ) |^2 \Big) \cdot E (   \lfloor \frac{|I|}{2} \rfloor + \lfloor  \frac{n}{2} \rfloor  + 1)\\
 \notag
   && \cdot \Big(  \begin{cases}  \frac{\eps }{(1+t+|q|)^{3-2\delta} (1+|q|)^{2+2\gamma}},\quad\text{when }\quad q>0,\\
           \notag
      \frac{\eps  }{(1+t+|q|)^{3-2\delta} \cdot ( 1+|q| )^{} } \,\quad\text{when }\quad q<0 . \end{cases} \Big) \; .
      \notag 
   \eeaa

Also,
 \beaa
\notag
 && \Big(   \sum_{|K|\leq |I|}    |   \Lie_{Z^K} h  |^2 \Big) \cdot  E (   \lfloor \frac{|I|}{2} \rfloor + \lfloor  \frac{n}{2} \rfloor  + 1) \\
 \notag
   &&\Big(   \Big(  \begin{cases}  \frac{\eps }{(1+t+|q|)^{7-4\delta} (1+|q|)^{4+4\gamma}},\quad\text{when }\quad q>0,\\
           \notag
      \frac{\eps  }{(1+t+|q|)^{7-4\delta}(1+|q|)^2}  \,\quad\text{when }\quad q<0 . \end{cases} \Big) \\
      \notag
      && + \Big(  \begin{cases}  \frac{\eps }{(1+t+|q|)^{7-4\delta} (1+|q|)^{6+2(\ga-\delta)+4\gamma}},\quad\text{when }\quad q>0,\\
           \notag
      \frac{\eps  \cdot(1+|q|)^{ } }{(1+t+|q|)^{7-4\delta} \cdot (1+|q|)^{4-2\delta}}  \,\quad\text{when }\quad q<0 . \end{cases} \Big)  \\
      \notag
  &&  + \Big(  \begin{cases}  \frac{\eps }{(1+t+|q|)^{7-4\delta} \cdot (1+|q|)^{8+4(\ga-\delta) +4\gamma}},\quad\text{when }\quad q>0,\\
           \notag
      \frac{\eps \cdot (1+|q|)^{4 }  }{(1+t+|q|)^{7-4\delta} \cdot (1+|q|)^{8-4\delta} }  \,\quad\text{when }\quad q<0 . \end{cases} \Big) \Big)  
               \notag
 \eeaa

     \beaa
     \notag
   &\les&  \sum_{|K|\leq |I|}  \Big( |   \Lie_{Z^K} h  |^2 \Big) \cdot E (   \lfloor \frac{|I|}{2} \rfloor + \lfloor  \frac{n}{2} \rfloor  + 1)\\
 \notag
   && \cdot \Big(  \begin{cases}  \frac{\eps }{(1+t+|q|)^{7-4\delta} (1+|q|)^{4+4\gamma}},\quad\text{when }\quad q>0,\\
           \notag
      \frac{\eps  }{(1+t+|q|)^{7-4\delta} \cdot ( 1+|q| )^{2} } \,\quad\text{when }\quad q<0 . \end{cases} \Big) \; .
      \notag 
   \eeaa

\end{proof}

\section{Energy estimates}

\begin{definition}\label{defwidehatw}
We define $\widehat{w}$ by 
\beaa
\widehat{w}(q)&:=&\begin{cases} (1+|q|)^{1+2\gamma} \quad\text{when }\quad q>0 , \\
        (1+|q|)^{2\mu}  \,\quad\text{when }\quad q<0 , \end{cases} \\
        &=&\begin{cases} (1+ q)^{1+2\gamma} \quad\text{when }\quad q>0 , \\
      (1 - q)^{2\mu}  \,\quad\text{when }\quad q<0 ,\end{cases} 
\eeaa

for $\ga > 0$ and $\mu < 0$. Note that the definition of $\widehat{w}$\,, is so that on one hand, for $\ga \neq - \frac{1}{2} $ and $\mu \neq 0$ (which is assumed here), we would have 
\beaa
\widehat{w}^{\prime}(q) \sim \frac{\widehat{w}(q)}{(1+|q|)} \; ,
\eeaa
(see Lemma \ref{derivativeoftildwandrelationtotildew}). On the other hand, we want that for $q<0$, the derivative $\frac{\pa \widehat{w}}{\pa q}$ to be non-vanishing.
\end{definition}

\begin{remark}
We take $\mu < 0$ (instead of $\mu > 0$), because we want the derivative $\frac{\pa \widehat{w}}{\pa q} > 0$\,, as we will see that this is what we need in order to obtain an energy estimate on the fields (see Lemma \ref{Theenerhyestimatefornequalfive}). In other words, $\mu < 0$ is a necessary condition to ensure that $\widehat{w}^{\prime} (q)$ enters with the right sign in the energy estimate.
\end{remark}

\begin{definition}\label{defwidetildew}
We define $\widetilde{w}$ by 
\beaa
\widetilde{w} ( q)&:=&  \widehat{w}(q) + w(q) \\
&:=&\begin{cases} 2 (1+|q|)^{1+2\gamma} \quad\text{when }\quad q>0 , \\
       1+  (1+|q|)^{2\mu}  \,\quad\text{when }\quad q<0 . \end{cases} \\    
\eeaa
Note that the definition of $\widetilde{w}$ is constructed so that Lemma \ref{equaivalenceoftildewandtildeandofderivativeoftildwandderivativeofhatw} holds.
\end{definition}

\begin{lemma}\label{equaivalenceoftildewandtildeandofderivativeoftildwandderivativeofhatw}
We have
\beaa
\widetilde{w}^{\prime}  &\sim & \widehat{w}^{\prime } \; .
\eeaa
Furthermore, for $\mu < 0$, we have
\beaa
\widetilde{w} ( q)& \sim & w(q) \; .
\eeaa

\end{lemma}

\begin{proof}
We compute the derivative with respect to $q$\,,
\beaa
\widetilde{w}^{\prime} &=&   \widehat{w}^{\prime} ( q) + w^{\prime} ( q) \\
&=& \begin{cases} 2 \cdot \widehat{w}^{\prime} ( q) \quad\text{when }\quad q>0 , \\
        \widehat{w}^{\prime }( q)  \,\quad\text{when }\quad q<0 . \end{cases} \\ 
\eeaa
Consequently,
\beaa
\widetilde{w}^{\prime}  &\sim & \widehat{w}^{\prime } \; .
\eeaa

Now, on one hand, since $ \widehat{w} \geq 0$, we have 
\beaa
\widetilde{w} ( q)&\geq&w(q) \; .
\eeaa
On the other hand, since $\mu < 0$\,, we have
\beaa
\widetilde{w} ( q)& = & \begin{cases} 2 (1+|q|)^{1+2\gamma} \quad\text{when }\quad q>0 , \\
       1+  (1+|q|)^{2\mu}  \,\quad\text{when }\quad q<0 . \end{cases} \\    
       &\leq&\begin{cases} 2 (1+|q|)^{1+2\gamma} \quad\text{when }\quad q>0 , \\
       2  \,\quad\text{when }\quad q<0 . \end{cases} \\    
       &\leq& 2 w(q)
\eeaa
Thus
\beaa
\widetilde{w} ( q)& \sim & w(q) \; .
\eeaa
\end{proof}

 \begin{definition}\label{definitionofthescalarproductoftwopartialderivatives}
 Let $\Phi$ be a tensor of any order, say a 2-tensor $\Phi_{\mu\nu}$\;, either valued in the Lie algebra ${\cal G}$\;, or a a scalar. For any $\a\, ,\, \b \in \{ r, t, x^1, \ldots, x^n \}$\,, we define the following scalar product by
\bea
\notag
 < \pa_\a  \Phi ,\pa_\b \Phi > &:=& \sum_{\mu,\, \nu \in \{ t, x^1, \ldots, x^n \}  } < \pa_\a  \Phi_{\mu\nu} ,\pa_\b \Phi_{\mu\nu} > \; . \\
 \eea
 \end{definition}

\begin{lemma}\label{Conservationlawwithweightforwaveequations}
Let $ \Phi_{\mu\nu}$ be a tensor solution of the following tensorial wave equation
\bea
 g^{\la\a} \derm_{\la}   \derm_{\a}   \Phi_{\mu\nu}= S_{\mu\nu} \, , 
\eea
where $S_{\mu\nu}$ is the source term, with a sufficiently smooth metric $g$. Assume that the field is decaying fast enough at spatial infinity for all time $t$, such that in wave coordinates $\{t, x^1, \ldots, x^n \}$, we have for $j$ running over spatial indices $\{ x^1, \ldots, x^n \}$,
\bea
\lim_{r \to \infty }   \int_{\SSS^{n} }  g^{rj}  \cdot < \pa_t  \Phi ,\pa_j \Phi > \cdot w \cdot r^{n-1} d\si^{n-1} =0 \;, \\
\lim_{r \to \infty }   \int_{\SSS^{n} }  g^{tr}  \cdot < \pa_t  \Phi ,\pa_t \Phi > \cdot w \cdot r^{n-1} d\si^{n-1} =0 \;.
\eea
Then, we have the following
\bea
\notag
&& \int_{\Si_{t}} \Big( - (m^{t t} + H^{t t} )  < \pa_t \Phi , \pa_t \Phi > +  (m^{ij} + H^{ij} )< \pa_i \Phi , \pa_j \Phi >  \Big) \cdot w  \\
\notag
&=&  \int_{\Si_{t=0}} \Big( - (m^{t t} + H^{t t} ) < \pa_t \Phi , \pa_t \Phi > +   (m^{ij} + H^{ij} ) < \pa_i \Phi , \pa_j \Phi >  \Big) \cdot w\\
\notag
&+& \int_0^t \int_{\Si_{t}} - 2 < \pa_t \Phi , S >  \cdot w \\
\notag
&+& \int_0^t \int_{\Si_{t}}  \Big((- \pa_t H^{t t} ) \cdot  < \pa_t \Phi , \pa_t \Phi > +  ( \pa_t  H^{ij} ) \cdot  < \pa_i \Phi , \pa_j \Phi > \\
\notag
&& -  2 ( \pa_i H^{ij} ) \cdot < \pa_t  \Phi ,\pa_j \Phi >  - 2 \pa_j ( H^{tj} ) \cdot < \pa_t \Phi,  \pa_t \Phi > ) \Big) \cdot w  \\
\notag
&+& \int_0^t  \int_{\Si_{t}} \Big(-  2 H^{ij}   < \pa_t  \Phi , \pa_j \Phi >  \cdot ( \frac{x_i}{r}) - 2 H^{tj} < \pa_t \Phi,  \pa_t \Phi >  \cdot ( \frac{x_j}{r})  \\
\notag
&&   +  H^{t t}  < \pa_t \Phi , \pa_t \Phi >   -  H^{ij}  < \pa_i \Phi , \pa_j \Phi > \Big) \cdot w^{\prime} (q)  \\
\notag
&-&  \int_0^t   \int_{\Si_{t}} \Big(   < \pa_t  \Phi + \pa_r \Phi, \pa_t  \Phi + \pa_r \Phi >     +  \de^{ij}  < \pa_i \Phi - \frac{x_i}{r} \pa_{r}  \Phi , \pa_j \Phi - \frac{x_j}{r} \pa_{r}  \Phi >  \Big) \cdot w^{\prime} (q)   \; ,
\eea

where the integration on $\Sigma_t$ is taken with respect to the measure $dx^1 \ldots dx^n$, and the integration in $t$ is taken with respect to the measure $dt$\; and where the scalar product is taken as in Definition \ref{definitionofthescalarproductoftwopartialderivatives}.

\end{lemma}

\begin{proof}

Let $d^{n} x := dx^{1} \ldots dx^{n}$ and let $w^\prime (q) := \frac{\pa}{\pa q } w(q)$. We denote by $i$ and $j$, spatial indices running only over $\{1, 2, \ldots, n \}$. We compute, on one hand,
\beaa
&& \frac{\d }{dt} \int_{\Si_{t}} \Big( g^{\a \b} < \pa_\a \Phi , \pa_\b \Phi >  -2 g^{tt} < \pa_t \Phi , \pa_t \Phi > -  2 g^{tj} < \pa_t \Phi , \pa_j \Phi > \Big) \cdot w \cdot dx^{1} \ldots dx^{n} \\
&=& \frac{\d }{dt} \int_{\Si_{t}} \Big( -g^{t t} < \pa_t \Phi , \pa_t \Phi > +  g^{ij} < \pa_i \Phi , \pa_j \Phi >  \Big) \cdot w \cdot d^{n}x \\
&& \text{(using the symmetry of metric $g$)} \\
&=& \int_{\Si_{t}} \Big( (- \pa_t g^{t t} ) \cdot  < \pa_t \Phi , \pa_t \Phi > +  ( \pa_t  g^{ij} ) \cdot  < \pa_i \Phi , \pa_j \Phi >  \Big) \cdot w \cdot d^{n}x \\
&& +  \int_{\Si_{t}} \Big( -2 g^{t t}  < \pa^2_t \Phi , \pa_t \Phi > +  2g^{ij} < \pa_t \pa_i \Phi , \pa_j \Phi >  \Big) \cdot w \cdot d^{n}x \\
&& + \int_{\Si_{t}} \Big( -g^{t t} < \pa_t \Phi , \pa_t \Phi > +  g^{ij} < \pa_i \Phi , \pa_j \Phi >  \Big) \cdot ( \frac{\pa}{\pa q } w(q)) \cdot (\pa_t q )\cdot d^{n}x \\
&& \text{(using again the symmetry of metric $g$)} \\
&:=& I_1 + I_2 + I_3 \\
&& \text{(where $I_1, \; I_2, \; I_3$ are defined respectively as the last three integrals).} \\
\eeaa

On the other hand, since we would like to get rid of the second order derivatives, or express them in terms of $ g^{\la\a} \derm_{\la}   \derm_{\a} \Phi = S$, we compute independently, 
\beaa
 < \pa_t \Phi , S  > &=& < \pa_t \Phi , g^{t t}  \pa^2_t \Phi  +  g^{ij}  \pa_i \pa_j \Phi +  2 g^{tj} \pa_t  \pa_j \Phi >  \\
  &=&g^{t t} < \pa_t \Phi ,   \pa^2_t \Phi > +  g^{ij} < \pa_t \Phi , \pa_i \pa_j \Phi > +  2 g^{tj}< \pa_t \Phi , \pa_t  \pa_j \Phi >  \; .
\eeaa
In order to write $I_2$ in that form, we integrate by parts,
\beaa
I_2 &:=&   \int_{\Si_{t}} \Big( -2 g^{t t}  < \pa^2_t \Phi , \pa_t \Phi > +  2g^{ij} < \pa_t \pa_i \Phi , \pa_j \Phi >  \Big) \cdot w \cdot d^{n}x \\
&=&  \int_{\Si_{t}} \Big( -2 g^{t t}  < \pa^2_t \Phi , \pa_t \Phi > -  2g^{ij} < \pa_t  \Phi ,\pa_i \pa_j \Phi > -  2 ( \pa_i g^{ij} ) \cdot < \pa_t  \Phi ,\pa_j \Phi >  \Big) \cdot w \cdot d^{n}x \\
&& +   \int_{\Si_{t}} \Big(-  2g^{ij} < \pa_t  \Phi , \pa_j \Phi >  \Big) \cdot w^{\prime} (q) \cdot ( \pa_i q) \cdot d^{n}x +2 \lim_{r \to \infty }   \int_{\SSS^{n} }  g^{rj}  \cdot < \pa_t  \Phi ,\pa_j \Phi > \cdot w \cdot r^{n-1} d\si^{n-1}  \\
&& \text{(where $d\si^{n-1}$ is the volume form on the unit $(n-1)$-sphere)} \\
&=&  \int_{\Si_{t}} \Big( -2 g^{t t}  < \pa^2_t \Phi , \pa_t \Phi > -  2g^{ij} < \pa_t  \Phi ,\pa_i \pa_j \Phi >   - 4 g^{tj}< \pa_t \Phi , \pa_t  \pa_j \Phi > \\
&& -  2 ( \pa_i g^{ij} ) \cdot < \pa_t  \Phi ,\pa_j \Phi >  + 4 g^{tj}< \pa_t \Phi , \pa_t  \pa_j \Phi > \Big) \cdot w \cdot d^{n}x \\
&& +   \int_{\Si_{t}} \Big(-  2g^{ij} < \pa_t  \Phi , \pa_j \Phi >  \Big) \cdot w^{\prime} (q) \cdot ( \pa_i q) \cdot d^{n}x \\
&& \text{(where we used the fact that the boundary terms vanish)} \\
&=&  \int_{\Si_{t}} - 2 < \pa_t \Phi ,  S >  \cdot w \cdot d^{n}x  +  \int_{\Si_{t}}  \Big( -  2 ( \pa_i g^{ij} ) \cdot < \pa_t  \Phi ,\pa_j \Phi >  + 4 g^{tj}< \pa_t \Phi , \pa_t  \pa_j \Phi > \Big) \cdot w \cdot d^{n}x \\
&& +   \int_{\Si_{t}} \Big(-  2g^{ij} < \pa_t  \Phi , \pa_j \Phi >  \Big) \cdot w^{\prime} (q) \cdot ( \pa_i q) \cdot d^{n}x \; .
\eeaa
However, we notice that
\beaa
2 < \pa_t \Phi , \pa_t  \pa_j \Phi > = \pa_j (< \pa_t \Phi,  \pa_t \Phi > ) \;.
\eeaa
Thus, integrating by parts using the fact the that the boundary term vanishes, i.e.
\bea
\notag
\lim_{r \to \infty }   \int_{\SSS^{n} }  g^{tr}  \cdot < \pa_t  \Phi ,\pa_t \Phi > \cdot w \cdot r^{n-1} d\si^{n-1} =0 \;,
\eea
since the fields are decaying fast at spatial infinity, we obtain
\beaa
I_2 &=&  \int_{\Si_{t}} - 2 < \pa_t \Phi , S   >  \cdot w \cdot d^{n}x  +  \int_{\Si_{t}}  \Big( -  2 ( \pa_i g^{ij} ) \cdot < \pa_t  \Phi ,\pa_j \Phi >  + 2 g^{tj} \pa_j (< \pa_t \Phi,  \pa_t \Phi > ) \Big) \cdot w \cdot d^{n}x \\
&& +   \int_{\Si_{t}} \Big(-  2g^{ij} < \pa_t  \Phi , \pa_j \Phi >  \Big) \cdot w^{\prime} (q) \cdot ( \pa_i q) \cdot d^{n}x \\
&=&  \int_{\Si_{t}} - 2 < \pa_t \Phi , S   >  \cdot w \cdot d^{n}x  +  \int_{\Si_{t}}  \Big( -  2 ( \pa_i g^{ij} ) \cdot < \pa_t  \Phi ,\pa_j \Phi >  - 2 \pa_j ( g^{tj} ) \cdot < \pa_t \Phi,  \pa_t \Phi > ) \Big) \cdot w \cdot d^{n}x \\
&& +   \int_{\Si_{t}} \Big(-  2g^{ij} < \pa_t  \Phi , \pa_j \Phi >  \cdot ( \pa_i q) - 2  g^{tj} < \pa_t \Phi,  \pa_t \Phi >  \cdot ( \pa_j q) \Big) \cdot w^{\prime} (q) \cdot d^{n}x \;. 
\eeaa

Putting together, we get
\beaa
&& \frac{\d }{dt} \int_{\Si_{t}} \Big( -g^{t t} < \pa_t \Phi , \pa_t \Phi > +  g^{ij} < \pa_i \Phi , \pa_j \Phi >  \Big) \cdot w \cdot d^{n}x \\
&=&  \int_{\Si_{t}} - 2 < \pa_t \Phi , S  >  \cdot w \cdot d^{n}x \\
&+&  \int_{\Si_{t}}  \Big((- \pa_t g^{t t} ) \cdot  < \pa_t \Phi , \pa_t \Phi > +  ( \pa_t  g^{ij} ) \cdot  < \pa_i \Phi , \pa_j \Phi > \\
&& -  2 ( \pa_i g^{ij} ) \cdot < \pa_t  \Phi ,\pa_j \Phi >  - 2 \pa_j ( g^{tj} ) \cdot < \pa_t \Phi,  \pa_t \Phi > ) \Big) \cdot w \cdot d^{n}x \\
&+&   \int_{\Si_{t}} \Big(-  2g^{ij} < \pa_t  \Phi , \pa_j \Phi >  \cdot ( \pa_i q) - 2  g^{tj} < \pa_t \Phi,  \pa_t \Phi >  \cdot ( \pa_j q)  \\
&&   -g^{t t} < \pa_t \Phi , \pa_t \Phi >  \cdot (\pa_t q ) +  g^{ij} < \pa_i \Phi , \pa_j \Phi > \cdot (\pa_t q ) \Big) \cdot w^{\prime} (q)  \cdot d^{n}x \, .
\eeaa
 
 Since $q = r - t$, we have
 \beaa
 \pa_t q = -1  \; ,
 \eeaa
 and 
 \beaa
 \pa_j q = \pa_j r = \pa_j  \Big( \sum_{k=1}^{n} (x_k)^{2} \Big)^{\frac{1}{2}} = \frac{x_j}{r} \,
\eeaa
and since be definition
\beaa
H^{\mu\nu} &=& g^{\mu\nu}-m^{\mu\nu} \; ,
\eeaa
and therefore, for all $\a \in \{\frac{\pa}{\pa x_\mu} \;,\; \mu \in \{0, 1, \ldots, n \} \}$
\beaa
\pa_\a H^{\mu\nu} &=& \pa_\a g^{\mu\nu}- \pa_a m^{\mu\nu} =  \pa_\a g^{\mu\nu} \; ,
\eeaa
we get
\beaa
&& \frac{\d }{dt} \int_{\Si_{t}} \Big( -g^{t t} < \pa_t \Phi , \pa_t \Phi > +  g^{ij} < \pa_i \Phi , \pa_j \Phi >  \Big) \cdot w \cdot d^{n}x \\
&=&  \int_{\Si_{t}} - 2 < \pa_t \Phi ,S   >  \cdot w \cdot d^{n}x \\
&+&  \int_{\Si_{t}}  \Big((- \pa_t H^{t t} ) \cdot  < \pa_t \Phi , \pa_t \Phi > +  ( \pa_t  H^{ij} ) \cdot  < \pa_i \Phi , \pa_j \Phi > \\
&& -  2 ( \pa_i H^{ij} ) \cdot < \pa_t  \Phi ,\pa_j \Phi >  - 2 \pa_j ( H^{tj} ) \cdot < \pa_t \Phi,  \pa_t \Phi > ) \Big) \cdot w \cdot d^{n}x \\
&+&   \int_{\Si_{t}} \Big(-  2g^{ij} < \pa_t  \Phi , \pa_j \Phi >  \cdot ( \frac{x_i}{r}) - 2  g^{tj} < \pa_t \Phi,  \pa_t \Phi >  \cdot ( \frac{x_j}{r})  \\
&&   + g^{t t} < \pa_t \Phi , \pa_t \Phi >   -  g^{ij} < \pa_i \Phi , \pa_j \Phi > \Big) \cdot w^{\prime} (q)  \cdot d^{n}x \; .
\eeaa
Thus,
\bea
\notag
&& \frac{\d }{dt} \int_{\Si_{t}} \Big( -g^{t t} < \pa_t \Phi , \pa_t \Phi > +  g^{ij} < \pa_i \Phi , \pa_j \Phi >  \Big) \cdot w \cdot d^{n}x \\
\notag
&=&  \int_{\Si_{t}} - 2 < \pa_t \Phi , S  >  \cdot w \cdot d^{n}x \\
\notag
&+&  \int_{\Si_{t}}  \Big((- \pa_t H^{t t} ) \cdot  < \pa_t \Phi , \pa_t \Phi > +  ( \pa_t  H^{ij} ) \cdot  < \pa_i \Phi , \pa_j \Phi > \\
\notag
&& -  2 ( \pa_i H^{ij} ) \cdot < \pa_t  \Phi ,\pa_j \Phi >  - 2 \pa_j ( H^{tj} ) \cdot < \pa_t \Phi,  \pa_t \Phi > ) \Big) \cdot w \cdot d^{n}x \\
\notag
&+&   \int_{\Si_{t}} \Big(-  2 (g^{ij} - m^{ij} )  < \pa_t  \Phi , \pa_j \Phi >  \cdot ( \frac{x_i}{r}) - 2  ( g^{tj} -  m^{tj} ) < \pa_t \Phi,  \pa_t \Phi >  \cdot ( \frac{x_j}{r})  \\
\notag
&&   + ( g^{t t} - m^{t t}  ) < \pa_t \Phi , \pa_t \Phi >   -  ( g^{ij} -m^{ij} ) < \pa_i \Phi , \pa_j \Phi > \Big) \cdot w^{\prime} (q)  \cdot d^{n}x \\
\notag
&+&   \int_{\Si_{t}} \Big(-  2m^{ij} < \pa_t  \Phi , \pa_j \Phi >  \cdot ( \frac{x_i}{r}) - 2  m^{tj} < \pa_t \Phi,  \pa_t \Phi >  \cdot ( \frac{x_j}{r})  \\
\notag
&&   + m^{t t} < \pa_t \Phi , \pa_t \Phi >   -  m^{ij} < \pa_i \Phi , \pa_j \Phi > \Big) \cdot w^{\prime} (q)  \cdot d^{n}x \; .\\
\eea
Now, we would like to compute, 
\beaa
&& \int_{\Si_{t}} \Big(-  2m^{ij} < \pa_t  \Phi , \pa_j \Phi >  \cdot ( \frac{x_i}{r}) - 2  m^{tj} < \pa_t \Phi,  \pa_t \Phi >  \cdot ( \frac{x_j}{r})  \\
&&   + m^{t t} < \pa_t \Phi , \pa_t \Phi >   -  m^{ij} < \pa_i \Phi , \pa_j \Phi > \Big) \cdot w^{\prime} (q)  \cdot d^{n}x \\
 &=&  \int_{\Si_{t}} \Big(-  2 < \pa_t  \Phi , \pa_j \Phi >  \cdot ( \frac{x^j}{r}) +0  -  < \pa_t \Phi , \pa_t \Phi >   -  \de^{ij} < \pa_i \Phi , \pa_j \Phi > \Big) \cdot w^{\prime} (q)  \cdot d^{n}x \; .
\eeaa
However,
\beaa
\pa_r = \sum_{j=1}^{n} \pa_j r \cdot \pa_j = m^{ij} \pa_i r \cdot  \pa_j = \frac{x^j}{r} \cdot  \pa_j \;.
\eeaa 
Thus,
\beaa
\pa_r \Phi = \frac{x^j}{r} \cdot  \pa_j  \Phi \;.
\eeaa 
We consider the derivatives restricted on the $n$-spheres
\beaa
\notag
 \pa_i - E ( \pa_{i} ,  \pa_{r}  ) \cdot \pa_{r} &=& \pa_i - E ( \pa_{i} ,  \frac{x^j}{r} \pa_{j}  ) \cdot  \pa_{r} \\
&=& \pa_i - \frac{x_i}{r} \pa_{r} \, .
\eeaa
We have
\beaa
&& \de^{ij}  < ( \pa_i - \frac{x_i}{r} \pa_{r}  )\Phi , (\pa_j - \frac{x_j}{r} \pa_{r} ) \Phi >\\
 &=& \de^{ij}  < \pa_i \Phi , \pa_j \Phi > - 2\de^{ij} < \frac{x_i}{r} \pa_{r} \Phi, \pa_j \Phi > +  \de^{ij}   \frac{x_i}{r}  \frac{x_j}{r} < \pa_{r}  \Phi, \pa_{r}  \Phi > \\
&=& \de^{ij}  < \pa_i \Phi , \pa_j \Phi > - 2 < \frac{x^j}{r} \pa_{r} \Phi, \pa_j \Phi > +  \frac{r^2}{r^2} < \pa_{r}  \Phi, \pa_{r}  \Phi > \\
&=& \de^{ij}  < \pa_i \Phi , \pa_j \Phi > - 2 <  \pa_{r} \Phi, \pa_r \Phi > +  < \pa_{r}  \Phi, \pa_{r}  \Phi > \;.
\eeaa
Hence,
\bea
\notag
\de^{ij}  < \pa_i \Phi , \pa_j \Phi > = \de^{ij}  < ( \pa_i - \frac{x_i}{r} \pa_{r}  )\Phi , (\pa_j - \frac{x_j}{r} \pa_{r} ) \Phi > +  < \pa_{r}  \Phi, \pa_{r}  \Phi > \;.
\eea
Injecting, we obtain
\beaa
&& \int_{\Si_{t}} \Big(-  2m^{ij} < \pa_t  \Phi , \pa_j \Phi >  \cdot ( \frac{x_i}{r}) - 2  m^{tj} < \pa_t \Phi,  \pa_t \Phi >  \cdot ( \frac{x_j}{r})  \\
&&   + m^{t t} < \pa_t \Phi , \pa_t \Phi >   -  m^{ij} < \pa_i \Phi , \pa_j \Phi > \Big) \cdot w^{\prime} (q)  \cdot d^{n}x \\
 &=&  \int_{\Si_{t}} \Big(-  2 < \pa_t  \Phi , \pa_r \Phi >   -  < \pa_t \Phi , \pa_t \Phi >   \\
 && -  \de^{ij}  < ( \pa_i - \frac{x_i}{r} \pa_{r}  )\Phi , (\pa_j - \frac{x_j}{r} \pa_{r} ) \Phi >  -  < \pa_{r}  \Phi, \pa_{r}  \Phi >\Big) \cdot w^{\prime} (q)  \cdot d^{n}x \\
  &=&  \int_{\Si_{t}} \Big(  - < \pa_t  \Phi + \pa_r \Phi, \pa_t  \Phi + \pa_r \Phi >     -  \de^{ij}  < ( \pa_i - \frac{x_i}{r} \pa_{r}  )\Phi , (\pa_j - \frac{x_j}{r} \pa_{r} ) \Phi >  \Big) \cdot w^{\prime} (q)  \cdot d^{n}x \; .
\eeaa

As a result, we obtain
\bea
\notag
&& \frac{\d }{dt} \int_{\Si_{t}} \Big( -g^{t t} < \pa_t \Phi , \pa_t \Phi > +  g^{ij} < \pa_i \Phi , \pa_j \Phi >  \Big) \cdot w \cdot d^{n}x \\
\notag
&=&  \int_{\Si_{t}} - 2 < \pa_t \Phi , S  >  \cdot w \cdot d^{n}x \\
\notag
&+&  \int_{\Si_{t}}  \Big((- \pa_t H^{t t} ) \cdot  < \pa_t \Phi , \pa_t \Phi > +  ( \pa_t  H^{ij} ) \cdot  < \pa_i \Phi , \pa_j \Phi > \\
\notag
&& -  2 ( \pa_i H^{ij} ) \cdot < \pa_t  \Phi ,\pa_j \Phi >  - 2 \pa_j ( H^{tj} ) \cdot < \pa_t \Phi,  \pa_t \Phi > ) \Big) \cdot w \cdot d^{n}x \\
\notag
&+&   \int_{\Si_{t}} \Big(-  2 H^{ij}   < \pa_t  \Phi , \pa_j \Phi >  \cdot ( \frac{x_i}{r}) - 2 H^{tj} < \pa_t \Phi,  \pa_t \Phi >  \cdot ( \frac{x_j}{r})  \\
\notag
&&   +  H^{t t}  < \pa_t \Phi , \pa_t \Phi >   -  H^{ij}  < \pa_i \Phi , \pa_j \Phi > \Big) \cdot w^{\prime} (q)  \cdot d^{n}x \\
\notag
&-&     \int_{\Si_{t}} \Big(   < \pa_t  \Phi + \pa_r \Phi, \pa_t  \Phi + \pa_r \Phi >     +  \de^{ij}  < ( \pa_i - \frac{x_i}{r} \pa_{r}  )\Phi , (\pa_j - \frac{x_j}{r} \pa_{r} ) \Phi >  \Big) \cdot w^{\prime} (q)  \cdot d^{n}x \; .
\eea
Integrating in time $t$, we obtain the result.
\end{proof}

\begin{lemma}\label{howtogetthedesirednormintheexpressionofenergyestimate}
Assume that the perturbation of the Minkowski metric is such that $ H^{\mu\nu} = g^{\mu\nu}-m^{\mu\nu}$ is bounded by a constant $C < \frac{1}{n}$\;, i.e.
\bea\label{AssumptiononHforgettingthenormintheexpressionofenergyestimate}
| H| \leq C < \frac{1}{n} \; ,
\eea
then we have
 \beaa
 | \pa \Phi|^2 \sim  - (m^{t t} + H^{t t} ) < \pa_t \Phi , \pa_t \Phi > +   (m^{ij} + H^{ij} ) < \pa_i \Phi , \pa_j \Phi > \; ,
  \eeaa
where the scalar product of the partial derivatives is as in Definition \ref{definitionofthescalarproductoftwopartialderivatives}.
\begin{remark}
The assumption on $H$ in \ref{AssumptiononHforgettingthenormintheexpressionofenergyestimate} is satisfied under the bootstrap argument for initial data small enough.
\end{remark}
\end{lemma}

 \begin{proof}
For each $\mu, \nu \in \{t, x^1, \ldots, x^n \}$, we have
\beaa
 H^{\mu\nu} \leq | H| = \Big( E_{\la\si } E_{\a\b } H^{\la\a} H^{\si\b}  \Big)^\frac{1}{2}\leq C 
\eeaa
and therefore
\beaa
- (m^{t t} + H^{t t} ) &=& - (-1  + H^{t t} ) = 1 + H^{t t} \leq 1+C \\
- (m^{t t} + H^{t t} ) &\geq& 1- C 
\eeaa
and
\beaa
  (m^{ij} + H^{ij} ) &\leq& \de^{ij} + C  \;, \\
    (m^{ij} + H^{ij} ) &\geq&  \de^{ij} - C  \; .
\eeaa
Now, let $C^{ij} = C$ for all $i,  j$ spatial indices. We get
\beaa
 && (1- C )   \cdot < \pa_t \Phi , \pa_t \Phi > + (\de^{ij} - C^{ij} )  \cdot < \pa_i \Phi , \pa_j \Phi > \\
  &\leq&  - (m^{t t} + H^{t t} )  \cdot < \pa_t \Phi , \pa_t \Phi > +   (m^{ij} + H^{ij} )  \cdot < \pa_i \Phi , \pa_j \Phi > \\
   &\leq& (1+C)   \cdot < \pa_t \Phi , \pa_t \Phi > + (\de^{ij} + C^{ij} )   \cdot< \pa_i \Phi , \pa_j \Phi > \; .
\eeaa
As a result, we have
\beaa
 && (1- C ) \cdot | \pa \Phi|^2    - C  \cdot \sum_{i\neq j }< \pa_i \Phi , \pa_j \Phi > \\
  &\leq&  - (m^{t t} + H^{t t} )  \cdot < \pa_t \Phi , \pa_t \Phi > +   (m^{ij} + H^{ij} )  \cdot < \pa_i \Phi , \pa_j \Phi > \\
   &\leq& (1+C) \cdot   | \pa \Phi|^2 + C   \cdot \sum_{i\neq j }  < \pa_i \Phi , \pa_j \Phi > \; .
\eeaa

Using $|a| \cdot |b| \leq   \frac{1}{2} a^2 +   \frac{1}{2} b^2$, we get
\beaa
| < \pa_i \Phi,  \pa_j \Phi> | &\leq& | < \pa_i \Phi,  \pa_i \Phi> |^\frac{1}{2} \cdot| < \pa_j \Phi,  \pa_j \Phi> |^\frac{1}{2} \\
&\leq&  \frac{1}{2}  \cdot <\pa_i \Phi , \pa_i \Phi > +  \frac{1}{2}  \cdot < \pa_j \Phi, \pa_j \Phi > \;,
 \eeaa
 and therefore, we have
 \beaa
&& C  \cdot \sum_{i, j, \; i\neq j }  | < \pa_i \Phi,  \pa_j \Phi> | \\
&\leq&  \frac{C}{2}  \cdot \sum_{i, j, \; i\neq j }  \Big(   <\pa_i \Phi , \pa_i \Phi > +   < \pa_j \Phi, \pa_j \Phi >  \Big)  \\
&\leq&   \frac{C}{2} \cdot \Big( 2(n-1) \cdot |\pa \Phi |^2  \Big) \\
&& \text{(where we used, in counting the sum, the fact that $\sum_{i =1}^n     <\pa_i \Phi , \pa_i \Phi > \leq |\pa \Phi |^2$ )} \\
&\leq&  C \cdot (n-1)  \cdot |\pa \Phi |^2 \; .
 \eeaa
 As a result,
 \beaa
&& (1- C ) \cdot | \pa \Phi|^2    -  C \cdot (n-1) \cdot  |\pa \Phi |^2 \\
  &\leq&  - (m^{t t} + H^{t t} ) \cdot  < \pa_t \Phi , \pa_t \Phi > +   (m^{ij} + H^{ij} )  \cdot < \pa_i \Phi , \pa_j \Phi > \\
    &\leq& (1+C) \cdot   | \pa \Phi|^2 +    C \cdot (n-1)  \cdot | \pa \Phi|^2 \;.
\eeaa
 Consequently,
 \beaa
 && (1 - C \cdot  n) ) \cdot | \pa \Phi|^2 \\
   &\leq&  - (m^{t t} + H^{t t} )  \cdot < \pa_t \Phi , \pa_t \Phi > +   (m^{ij} + H^{ij} )  \cdot< \pa_i \Phi , \pa_j \Phi > \\
   &\leq&    (1 + C \cdot  n  ) \cdot | \pa \Phi|^2 \; .
  \eeaa
    
    \end{proof}

\begin{lemma}\label{Theenerhyestimatefornequalfive}

Let $ \Phi_{\mu\nu}$ be a tensor solution of the following tensorial wave equation
\bea
 g^{\la\a} \derm_{\la}   \derm_{\a}   \Phi_{\mu\nu}= S_{\mu\nu} \, , 
\eea
where $S_{\mu\nu}$ is the source term, with a sufficiently smooth metric $g$. Assume that $ H^{\mu\nu} = g^{\mu\nu}-m^{\mu\nu}$ satisfies 
\beaa
| H| \leq C < \frac{1}{n} \; , \quad \text{where $n$ is the space dimension},
\eeaa
and assume that the field $\Phi$ is decaying fast enough at spatial infinity for all time $t$, such that in wave coordinates $\{t, x^1, \ldots, x^n \}$, we have for $j$ running over spatial indices $\{ x^1, \ldots, x^n \}$, 
\bea
\lim_{r \to \infty }   \int_{\SSS^{n} }  g^{rj}  \cdot < \pa_t  \Phi ,\pa_j \Phi > \cdot w \cdot r^{n-1} d\si^{n-1} =0 \\
\lim_{r \to \infty }   \int_{\SSS^{n} }  g^{tr}  \cdot < \pa_t  \Phi ,\pa_t \Phi > \cdot w \cdot r^{n-1} d\si^{n-1} =0 \;.
\eea

Then, we have the following

\bea
\notag
 &&\int_{\Si_{t}}  | \pa \Phi|^2 \cdot w \\
 \notag
 &\leq&  \int_{\Si_{t=0}}  | \pa \Phi|^2  \cdot w + \int_0^t \int_{\Si_{t}} - 2 < \pa_t \Phi , S  >  \cdot w \\
\notag
&&+  \int_0^t \int_{\Si_{t}}   | \derm H | \cdot |\derm \Phi |^2  \cdot  w  + \int_0^t  \int_{\Si_{t}}  | H |  \cdot |\derm \Phi |^2 \cdot | w^{\prime} (q) |  \\
\notag
&&-  \int_0^t   \int_{\Si_{t}} \Big(   | \pa_t  \Phi + \pa_r \Phi |^2     + \sum_{i=1}^{n}  | ( \pa_i - \frac{x_i}{r} \pa_{r}  )\Phi |^2 \Big) \cdot w^{\prime} (q)   \; .
\eea

where the integration on $\Sigma_t$ is taken with respect to the measure $dx^1 \ldots dx^n$, and the integration in $t$ is taken with respect to the measure $dt$.\\

\end{lemma}

\begin{proof}

We examine the term
 \beaa
 && \int_0^t \int_{\Si_{t}}  \Big((- \pa_t H^{t t} ) \cdot  < \pa_t \Phi , \pa_t \Phi > +  ( \pa_t  H^{ij} ) \cdot  < \pa_i \Phi , \pa_j \Phi > \\
\notag
&& -  2 ( \pa_i H^{ij} ) \cdot < \pa_t  \Phi ,\pa_j \Phi >  - 2 \pa_j ( H^{tj} ) \cdot < \pa_t \Phi,  \pa_t \Phi > ) \Big) \cdot w  \; .
\eeaa

Given the definition of the norms computed in wave coordinates $\{ t, x^i \quad i \in \{1, \ldots, n \} \}$, 
 \beaa
| (- \pa_t H^{t t} ) \cdot  < \pa_t \Phi , \pa_t \Phi > | &\leq& | \derm_t H^{tt} | \cdot |\derm_t \Phi |^2 \leq | \derm H | \cdot |\derm \Phi |^2 \; , \\
| ( \pa_t  H^{ij} ) \cdot  < \pa_i \Phi , \pa_j \Phi >  |&\leq&| \derm_t H^{ij} | \cdot |\derm_i \Phi | \cdot |\derm_j \Phi | \leq | \derm H | \cdot |\derm \Phi |^2  \; , \\
| ( \pa_i H^{ij} ) \cdot < \pa_t  \Phi ,\pa_j \Phi > |  &\leq&| \derm_i H^{ij} | \cdot |\derm_t \Phi | \cdot |\derm_j \Phi | \leq | \derm H | \cdot |\derm \Phi |^2 \; , \\
|  \pa_j ( H^{tj} ) \cdot < \pa_t \Phi,  \pa_t \Phi >  |&\leq& | \derm_j H^{tj} | \cdot |\derm_t \Phi |^2 \leq | \derm H | \cdot |\derm \Phi |^2 \; .
\eeaa
Consequently, we get
 \bea\label{estimateforconservationlawwithweightfortermswithderivativesofH}
 \notag
 && \int_0^t \int_{\Si_{t}}  \Big((- \pa_t H^{t t} ) \cdot  < \pa_t \Phi , \pa_t \Phi > +  ( \pa_t  H^{ij} ) \cdot  < \pa_i \Phi , \pa_j \Phi > \\
  \notag
&& -  2 ( \pa_i H^{ij} ) \cdot < \pa_t  \Phi ,\pa_j \Phi >  - 2 \pa_j ( H^{tj} ) \cdot < \pa_t \Phi,  \pa_t \Phi > ) \Big) \cdot w   \\
&\leq & \int_0^t \int_{\Si_{t}}   | \derm H | \cdot |\derm \Phi |^2  \cdot |w|  \;. 
\eea

We look at the term
\beaa
&& \int_0^t  \int_{\Si_{t}} \Big(-  2 H^{ij}   \cdot < \pa_t  \Phi , \pa_j \Phi >  \cdot ( \frac{x_i}{r}) - 2 H^{tj}  \cdot < \pa_t \Phi,  \pa_t \Phi >  \cdot ( \frac{x_j}{r})  \\
\notag
&&   +  H^{t t}   \cdot < \pa_t \Phi , \pa_t \Phi >   -  H^{ij}  \cdot  < \pa_i \Phi , \pa_j \Phi > \Big) \cdot w^{\prime} (q)  \;.
\eeaa

Using the fact that $  |x_i| \leq r $ for all $i$, spatial index, we get
\beaa
 | H^{ij}   < \pa_t  \Phi , \pa_j \Phi >  \cdot ( \frac{x_i}{r}) | &\leq& | H^{ij} |  \cdot |\derm_t \Phi | \cdot |\derm_j \Phi |  \cdot  \frac{ |x_i| }{r} \leq  | H |  \cdot |\derm \Phi |^2 \; ,\\
 | H^{tj} < \pa_t \Phi,  \pa_t \Phi >  \cdot ( \frac{x_j}{r}) | &\leq&  | H |  \cdot |\derm \Phi |^2 \; , \\
 | H^{t t}  < \pa_t \Phi , \pa_t \Phi > | &\leq&  | H |  \cdot |\derm \Phi |^2 \; , \\
 | H^{ij}  < \pa_i \Phi , \pa_j \Phi > | &\leq&  | H |  \cdot |\derm \Phi |^2 \; .
\eeaa
Thus,
\bea\label{estimateforconservationlawwithweightfortermswithHanderiavtiveofweight}
\notag
&& \int_0^t  \int_{\Si_{t}} \Big(-  2 H^{ij}   \cdot < \pa_t  \Phi , \pa_j \Phi >  \cdot ( \frac{x_i}{r}) - 2 H^{tj}  \cdot < \pa_t \Phi,  \pa_t \Phi >  \cdot ( \frac{x_j}{r})  \\
\notag
&&   +  H^{t t}   \cdot < \pa_t \Phi , \pa_t \Phi >   -  H^{ij}   \cdot < \pa_i \Phi , \pa_j \Phi > \Big) \cdot w^{\prime} (q)  \\
&\leq & \int_0^t  \int_{\Si_{t}}  | H |  \cdot |\derm \Phi |^2 \cdot | w^{\prime} (q) | \; .
\eea
Using what we showed in Lemma \ref{Conservationlawwithweightforwaveequations} and injecting the estimates \eqref{estimateforconservationlawwithweightfortermswithderivativesofH} and \eqref{estimateforconservationlawwithweightfortermswithHanderiavtiveofweight} that we proved, and using Lemma \ref{howtogetthedesirednormintheexpressionofenergyestimate}, we get the result.
\end{proof}

\section{A Hardy type inequality}

We will prove a Hardy type inequality with the weight $w$ that we defined in Definition \ref{defw}. However, since we will need a Hardy type inequality for a more general weight for the case of lower space-dimensions (which we will treat papers that follow), we will prove a Hardy type inequality for a more general weight $\widehat{w}$ which we will define in what follows (the weight $w$ corresponds to the case of $\mu = 0$ in $\widehat{w}$).

\begin{definition}\label{defwidehatw}
We define $\widehat{w}$ by 
\beaa
\widehat{w}(q)&:=&\begin{cases} (1+|q|)^{1+2\gamma} \quad\text{when }\quad q>0 , \\
        (1+|q|)^{2\mu}  \,\quad\text{when }\quad q<0 , \end{cases} \\
        &=&\begin{cases} (1+ q)^{1+2\gamma} \quad\text{when }\quad q>0 , \\
      (1 - q)^{2\mu}  \,\quad\text{when }\quad q<0 ,\end{cases} 
\eeaa

for $\ga > 0$ being the same as in Definition \ref{defw}, and for $\mu \in \R$ which could be restricted later in paper the follow for the lower dimensions. In other words, we will finally take in this paper $\mu = 0$, however we will perform our calculations with a general $\mu \neq \frac{1}{2}$ as will be pointed out later when it is needed.

\end{definition}

\begin{lemma}\label{HardytypeinequalityforintegralstartingatROmwithgeneralhatw}
Let $\widehat{w}$ defined as in Definition \ref{defwidehatw}.
Let  $\Phi$ a tensor that decays fast enough at spatial infinity for all time $t$\,, such that
\bea
 \int_{\SSS^{n-1}} \lim_{r \to \infty} \Big( \frac{r^{n-1}}{(1+t+r)^{a} \cdot (1+|q|) } \widehat{w}(q) \cdot <\Phi, \Phi>  \Big)  d\si^{n-1} (t ) &=& 0 \; .
\eea
Let $R(\Om)  \geq 0 $\,, be a function of $\Om \in \SSS^{n-1}$\,. Then, for $\ga \neq 0$ and $\mu \neq \frac{1}{2}$\,, $0 \leq a \leq n-1$\,, we have 
\bea
\notag
 &&   \int_{\SSS^{n-1}} \int_{r=R(\Om)}^{r=\infty} \frac{r^{n-1}}{(1+t+r)^{a}} \cdot   \frac{\widehat{w} (q)}{(1+|q|)^2} \cdot <\Phi, \Phi>   \cdot dr  \cdot d\si^{n-1}  \\
 \notag
 &\leq& c(\ga, \mu) \cdot  \int_{\SSS^{n-1}} \int_{r=R(\Om)}^{r=\infty}  \frac{ r^{n-1}}{(1+t+r)^{a}} \cdot \widehat{w}(q) \cdot <\pa_r\Phi, \pa_r \Phi>  \cdot  dr  \cdot d\si^{n-1}  \; , \\
 \eea
 where the constant $c(\ga, \mu)$ does not depend on $R(\Om)$\,.
 \end{lemma}
 
 \begin{proof}
Let 
\beaa
m(q)&:=&\frac{\widehat{w}(q)}{(1+|q|)} = \begin{cases} (1+|q|)^{2\gamma} \quad\text{when }\quad q>0 \;, \\
        (1+|q|)^{2\mu-1}  \,\quad\text{when }\quad q<0\;, \end{cases} \\
        &=&\begin{cases} (1+ q)^{2\gamma} \quad\text{when }\quad q>0\; , \\
        (1 - q)^{2\mu-1}  \,\quad\text{when }\quad q<0\; , \end{cases} \\
\eeaa
and we compute,
\beaa
m^{\prime}(q):=\begin{cases} 2\gamma (1+|q|)^{2\gamma-1} \quad\text{when }\quad q>0 \;, \\
         (1 - 2\mu) (1+|q|)^{2\mu-1} \,\quad\text{when }\quad q<0\; . \end{cases}
\eeaa

We want to prove the following Hardy type inequality for $a$\,, such that $0 \leq a \leq n-1$\,,
\beaa
 \int_{\SSS^{n-1}} \int_{r=R(\Om)}^{\infty }\frac{|\phi|^2}{(1+|q|)^2} \cdot \frac{ \widehat{w}(q) \, r^{n-1} \cdot dr \cdot d\si^{n-1} }{(1+t+|q|)^{a}} \les \int_{\SSS^{n-1}}  \int_{r=R(\Om)}^{\infty }  |\pa \phi|^2\, \frac{ \widehat{w}(q) \, r^{n-1} \cdot dr \cdot d\si^{n-1} }{(1+t+|q|)^{a}}
\eeaa
which means that we need to prove that
\beaa
\int_{\SSS^{n-1}}  \int_{r=R(\Om)}^{\infty }  \frac{|\phi|^2}{(1+|q|)}\cdot \frac{ m(q) \, r^{n-1} \cdot dr \cdot  d\si^{n-1} }{(1+t+|q|)^{a}} \les\int_{\SSS^{n-1}}  \int_{r=R(\Om)}^{\infty }  |\pa \phi|^2\, \frac{ (1+|q|) \cdot m(q) \, r^{n-1} \cdot dr \cdot d\si^{n-1} }{(1+t+|q|)^{a}} \; .
\eeaa

Since the term $\frac{(R(\Om))^{n-1}}{(1+t+r)^{a} \cdot (1+|q|) }  w(q) \cdot <\Phi, \Phi>$ is non-negative, we have
\bea
\notag
&& \int_{\SSS^{n-1}} \int_{r=R(\Om)}^{r=\infty} \pa_r \Big( \frac{r^{n-1}}{(1+t+r)^{a}} m(q) \cdot <\Phi, \Phi>  \Big)  dr d\si^{n-1} \\
\notag
 &=& \int_{\SSS^{n-1}}  \lim_{r \to \infty} \Big( \frac{r^{n-1}}{(1+t+r)^{a}} m(q) \cdot <\Phi, \Phi>  \Big)d\si^{n-1} - \int_{\SSS^{n-1}}  \Big( \frac{(R(\Om))^{n-1}}{(1+t+R(\Om))^{a}} m(q) \cdot <\Phi, \Phi>  \Big)d\si^{n-1} \\
  &\leq& \int_{\SSS^{n-1}}  \lim_{r \to \infty} \Big( \frac{r^{n-1}}{(1+t+r)^{a}} m(q) \cdot <\Phi, \Phi>  \Big)d\si^{n-1} \; .
\eea
We assume that $\Phi$ decays fast enough at spatial infinity for all time $t$, so that
\bea
 \int_{\SSS^{n-1}} \lim_{r \to \infty} \Big( \frac{r^{n-1}}{(1+t+r)^{a} \cdot (1+|q|) } \widehat{w}(q) \cdot <\Phi, \Phi>  \Big)  d\si^{n-1} (t ) &=& 0 \; ,
\eea
and therefore
\bea
\notag
 \int_{\SSS^{n-1}} \int_{r=0}^{r=\infty} \pa_r \Big( \frac{r^{n-1}}{(1+t+r)^{a}  \cdot (1+|q|)} \widehat{w}(q) \cdot <\Phi, \Phi>  \Big)  dr d\si^{n-1}  &\leq& 0 \; .\\
\eea

We compute
\beaa
&& \pa_r \Big( \frac{r^{n-1}}{(1+t+r)^{a}} \cdot m(q) \cdot <\Phi, \Phi>  \Big) \\
&=&  \pa_r \Big( \frac{r^{n-1}}{(1+t+r)^{a}} m(q)\Big) \cdot  <\Phi, \Phi> +  2 \frac{r^{n-1}}{(1+t+r)^{a}} m(q) \cdot <\pa_r \Phi, \Phi>  \;  .
\eeaa
We evaluate the term 
\beaa
&& \pa_r \Big( \frac{r^{n-1}}{(1+t+r)^{a}} m(q)\Big) \\
 &=& (n-1) r^{n-2} ( 1+t+r)^{-a}  m(q) + r^{n-1} \cdot (-a)  ( 1+t+r)^{-a-1}  m(q) +  \frac{r^{n-1}}{(1+t+r)^{a}}  m^{\prime} (q) \cdot (\pa_r q) \\
&=& \frac{r^{n-1}}{(1+t+r)^{a}} \cdot \Big( \frac{(n-1)}{r} \cdot  m(q) -    \frac{a}{( 1+t+r)} \cdot  m(q) +  m^{\prime} (q)  \Big) \\
&& \text{(since $q= r -t$)} \\
 &=& \frac{r^{n-1}}{(1+t+r)^{a}} \cdot \Big( m(q)  \cdot (  \frac{(n-1)}{r}  -    \frac{a}{( 1+t+r)}  ) +  m^{\prime} (q)  \Big) \; .
\eeaa
Since   $- \frac{a}{( 1+t+r)} \geq  -\frac{a}{r}$, and since $ (n-1) - a  \geq 0$, we get
\beaa
\pa_r \Big( \frac{r^{n-1}}{(1+t+r)^{a}} m(q)\Big)   &\geq& \frac{r^{n-1}}{(1+t+r)^{a}} \cdot \Big( m(q)  \cdot  \frac{ \big((n-1) - a \big) }{r}  +  m^{\prime} (q)  \Big) \\
&\geq& \frac{r^{n-1}}{(1+t+r)^{a}} \cdot   m^{\prime} (q)  \; .
\eeaa

Therefore, using also that $<\Phi, \Phi> \geq 0$, we get
\beaa
&& \pa_r \Big( \frac{r^{n-1}}{(1+t+r)^{a}} m(q) \cdot <\Phi, \Phi>  \Big) \\
&=&  \pa_r \Big( \frac{r^{n-1}}{(1+t+r)^{a}} m(q)\Big) \cdot <\Phi, \Phi> +  2 \frac{r^{n-1}}{(1+t+r)^{a}} m(q) \cdot <\pa_r \Phi, \Phi>   \\
&\geq& \frac{r^{n-1}}{(1+t+r)^{a}} \cdot   m^{\prime} (q) \cdot <\Phi, \Phi> +  2 \frac{r^{n-1}}{(1+t+r)^{a}} m(q) \cdot <\pa_r \Phi, \Phi>   \;.
\eeaa
By integrating and using the fact that the integral of the left hand side of the above inequality is non-positive, we obtain
\beaa
 && \int_{\SSS^{n-1}} \int_{r=R(\Om)}^{r=\infty} \frac{r^{n-1}}{(1+t+r)^{a}} \cdot   m^{\prime} (q) \cdot <\Phi, \Phi>  dr d\si^{n-1} \\
  &\leq&   \int_{\SSS^{n-1}} \int_{r=R(\Om)}^{r=\infty}  \frac{-2 r^{n-1}}{(1+t+r)^{a}} m(q) \cdot <\pa_r \Phi, \Phi>  dr d\si^{n-1}   \; .
\eeaa
Using Cauchy-Schwarz inequality, and the fact that $m^{\prime}(q) \neq 0 $ for all $q$ (since $\ga \neq 0$ and $\mu \neq \frac{1}{2}$), we obtain 
\beaa
&& \int_{\SSS^{n-1}} \int_{r=R(\Om)}^{r=\infty}  \frac{-2 r^{n-1}}{(1+t+r)^{a}} m(q) \cdot <\pa_r \Phi, \Phi>  dr d\si^{n-1} \\
 &\leq& 2\int_{\SSS^{n-1}} \int_{r=R(\Om)}^{r=\infty} \sqrt{  \frac{ r^{n-1}}{(1+t+r)^{a}} } \sqrt{m^{\prime}(q)} \cdot \sqrt{< \Phi, \Phi>} \cdot  \sqrt{  \frac{ r^{n-1}}{(1+t+r)^{a}} }  \frac{m(q)}{\sqrt{m^{\prime}(q)} }\sqrt{< \pa_r\Phi, \pa_r \Phi>}  dr d\si^{n-1} \\
 &\leq& 2 \Big( \int_{\SSS^{n-1}} \int_{r=R(\Om)}^{r=\infty}  \frac{ r^{n-1}}{(1+t+r)^{a}} m^{\prime} (q) \cdot < \Phi,  \Phi>  dr d\si^{n-1} \Big)^{\frac{1}{2}} \\
 && \cdot  \Big( \int_{\SSS^{n-1}} \int_{r=R(\Om)}^{r=\infty}  \frac{ r^{n-1}}{(1+t+r)^{a}} \frac{(m(q))^2}{m^{\prime}(q) } \cdot <\pa_r\Phi, \pa_r \Phi>  dr d\si^{n-1}   \Big)^{\frac{1}{2}} \; .
 \eeaa

Consequently,
\bea
\notag
 &&  \Big( \int_{\SSS^{n-1}} \int_{r=R(\Om)}^{r=\infty} \frac{r^{n-1}}{(1+t+r)^{a}} \cdot   m^{\prime} (q) \cdot <\Phi, \Phi>  dr d\si^{n-1} \Big)^{\frac{1}{2}} \\
 \notag
 &\leq& 2  \Big( \int_{\SSS^{n-1}} \int_{r=R(\Om)}^{r=\infty}  \frac{ r^{n-1}}{(1+t+r)^{a}} \frac{(m(q))^2}{m^{\prime}(q) } \cdot <\pa_r\Phi, \pa_r \Phi>  dr d\si^{n-1}   \Big)^{\frac{1}{2}} \; . \\
 \eea

We have

\beaa
m^{\prime}(q) &=& \begin{cases} 2\gamma (1+|q|)^{2\gamma-1} \quad\text{when }\quad q>0 , \\
         (1 - 2\mu) (1+|q|)^{2\mu-1} \,\quad\text{when }\quad q<0 . \end{cases} \\
          &=&\begin{cases}        2\gamma   \frac{m(q)}{(1+|q|)}   \quad\text{when }\quad q>0 , \\
          ( 1-  2\mu )     \frac{m(q)}{(1+|q|)}    \,\quad\text{when }\quad q<0 . \end{cases} \\
\eeaa
Thus,
\beaa
\min \{ 2\gamma, (1-2\mu) \}  \cdot \frac{m(q)}{(1+|q|)} 
 \leq m^{\prime}(q) \leq \max \{ 2\gamma, (1-2\mu) \} \cdot \frac{m(q)}{(1+|q|)} \;.
\eeaa
For $\ga \neq 0$ and $\mu \neq \frac{1}{2}$, we have $ \min \{ 2\gamma, (1-2\mu) \}  \neq 0 $ and $\max \{ 2\gamma, (1-2\mu) \}  \neq 0$, and therefore, we get
\bea
m^{\prime}(q) \sim \frac{m(q)}{(1+|q|)} \; .
\eea

As a result,
\beaa
 &&  \Big( \int_{\SSS^{n-1}} \int_{r=R(\Om)}^{r=\infty} \frac{r^{n-1}}{(1+t+r)^{a}} \cdot   \frac{m (q)}{(1+|q|)} \cdot <\Phi, \Phi>  dr d\si^{n-1} \Big)^{\frac{1}{2}} \\
 &\leq& c(\ga, \mu) \cdot  \Big( \int_{\SSS^{n-1}} \int_{r=R(\Om)}^{r=\infty}  \frac{ r^{n-1}}{(1+t+r)^{a}} \cdot (1+|q|) \cdot m(q) \cdot <\pa_r\Phi, \pa_r \Phi>  dr d\si^{n-1}   \Big)^{\frac{1}{2}} \; .
 \eeaa
 
Therefore,
\beaa
 &&  \Big( \int_{\SSS^{n-1}} \int_{r=R(\Om)}^{r=\infty} \frac{r^{n-1}}{(1+t+r)^{a}} \cdot   \frac{\widehat{w} (q)}{(1+|q|)^2} \cdot <\Phi, \Phi>  dr d\si^{n-1} \Big)^{\frac{1}{2}} \\
 &\leq& c(\ga, \mu) \cdot  \Big( \int_{\SSS^{n-1}} \int_{r=R(\Om)}^{r=\infty}  \frac{ r^{n-1}}{(1+t+r)^{a}} \cdot \widehat{w}(q) \cdot <\pa_r\Phi, \pa_r \Phi>  dr d\si^{n-1}   \Big)^{\frac{1}{2}} \; .
 \eeaa

 \end{proof}

\begin{corollary}\label{HardytypeinequalityforintegralstartingatROm}
Let $w$ defined as in Definition \ref{defw}, where $\ga > 0$.
Let  $\Phi$ a tensor that decays fast enough at spatial infinity for all time $t$\,, such that
\bea
 \int_{\SSS^{n-1}} \lim_{r \to \infty} \Big( \frac{r^{n-1}}{(1+t+r)^{a} \cdot (1+|q|) } w(q) \cdot <\Phi, \Phi>  \Big)  d\si^{n-1} (t ) &=& 0 \; .
\eea
Let $R(\Om)  \geq 0 $\,, be a function of $\Om \in \SSS^{n-1}$\,. Then, since $\ga \neq 0$\,, we have for $0 \leq a \leq n-1$\,, that 
\bea
\notag
 &&   \int_{\SSS^{n-1}} \int_{r=R(\Om)}^{r=\infty} \frac{r^{n-1}}{(1+t+r)^{a}} \cdot   \frac{w (q)}{(1+|q|)^2} \cdot <\Phi, \Phi>   \cdot dr  \cdot d\si^{n-1}  \\
 \notag
 &\leq& c(\ga) \cdot  \int_{\SSS^{n-1}} \int_{r=R(\Om)}^{r=\infty}  \frac{ r^{n-1}}{(1+t+r)^{a}} \cdot w(q) \cdot <\pa_r\Phi, \pa_r \Phi>  \cdot  dr  \cdot d\si^{n-1}  \; , \\
 \eea
 where the constant $c(\ga)$ does not depend on $R(\Om)$\,.
 \end{corollary}

 \begin{proof}
 By taking in Lemma \ref{HardytypeinequalityforintegralstartingatROmwithgeneralhatw}, on one hand $\mu = 0 $ (which satisfies the assumption $\mu \neq \frac{1}{2}$) and and on the other hand $\ga > 0$\,, as considered in Definition \ref{defw} (which in particular satisfies the assumption $\ga \neq 0$), we obtain the result.
 \end{proof}

\section{The commutator term for $n\geq 4$}

\begin{lemma}\label{aprioriestimateonthezeroLiederivativeZforgradiantbigHandbigH}

      We have for all $|I|$,  $\delta  \leq \frac{(n-2)}{2}$,\,$\eps \leq 1$,

              \beaa
 \notag
 |\derm   H (t,x)  |    &\les& \begin{cases}  E (  \lfloor  \frac{n}{2} \rfloor  +1)  \cdot \frac{\eps }{(1+t+|q|)^{\frac{(n-1)}{2}-\delta} (1+|q|)^{1+\ga}},\quad\text{when }\quad q>0,\\
        E (     \lfloor  \frac{n}{2} \rfloor  +1)  \cdot \frac{\eps  }{(1+t+|q|)^{\frac{(n-1)}{2}-\delta}(1+|q|)^{\frac{1}{2} }}  \,\quad\text{when }\quad q<0 , \end{cases} \\
      \eeaa
      and
       \beaa
 \notag
|   H (t,x)  | &\les& \begin{cases} c (\delta) \cdot c (\gamma) \cdot  E (    \lfloor  \frac{n}{2} \rfloor  +1) \cdot  \frac{\eps}{ (1+ t + | q | )^{\frac{(n-1)}{2}-\delta }  (1+| q |   )^{\ga}}  ,\quad\text{when }\quad q>0,\\
\notag
    E (    \lfloor  \frac{n}{2} \rfloor  +1) \cdot  \frac{\eps}{ (1+ t + | q | )^{\frac{(n-1)}{2}-\delta }  } (1+| q |   )^{\frac{1}{2} }  , \,\quad\text{when }\quad q<0 . \end{cases} \\ 
    \eeaa

\end{lemma}

\begin{proof}

We showed in Lemma \ref{aprioriestimatefrombootstraponzerothderivativeofAandh1}, that 
 \beaa
 \notag
| \Lie_{Z^I}  h^1  (t,x)  |    &\leq& \begin{cases} c (\gamma) \cdot  C ( |I| ) \cdot E ( |I| +  \lfloor  \frac{n}{2} \rfloor  +1)  \cdot \frac{\eps }{(1+t+|q|)^{\frac{(n-1)}{2}-\delta} (1+|q|)^{\gamma}},\quad\text{when }\quad q>0,\\
       C ( |I| ) \cdot E ( |I| +  \lfloor  \frac{n}{2} \rfloor  +1)  \cdot \frac{\eps \cdot (1+| q |   )^\frac{1}{2} }{(1+t+|q|)^{\frac{(n-1)}{2}-\delta} }  \,\quad\text{when }\quad q<0 . \end{cases} \\
      \eeaa
      
However, for $n\geq 4$, we have $h = h^{1}$. In addition, we know from Lemma \ref{BigHintermsofsmallh}, that
  \beaa
H^{\mu\nu}=-h^{\mu\nu}+ O^{\mu\nu}(h^2) \; .
\eeaa

Since for all $|I|$,

 \bea
 \notag
|   \Lie_{Z^I} h (t,x)  | &\leq& \begin{cases} c (\delta) \cdot c (\gamma) \cdot C ( |I| ) \cdot E (|I| +  \lfloor  \frac{n}{2} \rfloor  +1) \cdot  \frac{\eps}{ (1+ t + | q | )^{\frac{(n-1)}{2}-\de }  (1+| q |   )^{\ga}}  ,\quad\text{when }\quad q>0,\\
\notag
    C ( |I| ) \cdot E (|I| +  \lfloor  \frac{n}{2} \rfloor  +1) \cdot  \frac{\eps}{ (1+ t + | q | )^{\frac{(n-1)}{2}-\delta }  } (1+| q |   )^{\frac{1}{2} }  , \,\quad\text{when }\quad q<0 , \end{cases} 
      \eea
      
       we get
  \beaa
 \notag
&& |   H (t,x)  | \\
  &\les& \begin{cases} c (\delta) \cdot c (\gamma)  \cdot E (   \lfloor  \frac{n}{2} \rfloor  +1) \cdot \Big( \frac{\eps}{ (1+ t + | q | )^{\frac{(n-1)}{2}-\de }  (1+| q |   )^{\ga}}  + O(  \frac{\eps^2}{ (1+ t + | q | )^{(n-1)-2\delta }  (1+| q |   )^{2\ga}}  ) \Big),\quad\text{when }\quad q>0,\\
\notag
     E (   \lfloor  \frac{n}{2} \rfloor  +1) \cdot \Big( \frac{\eps}{ (1+ t + | q | )^{\frac{(n-1)}{2}-\delta }  } (1+| q |   )^{\frac{1}{2} }  + O(  \frac{\eps^2}{ (1+ t + | q | )^{(n-1)-2\delta }  } (1+| q |   )  ) \Big), \,\quad\text{when }\quad q<0 , \end{cases} \\
      &\les& \begin{cases} c (\delta) \cdot c (\gamma) \cdot E (    \lfloor  \frac{n}{2} \rfloor  +1) \cdot  \frac{\eps}{ (1+ t + | q | )^{\frac{(n-1)}{2}-\delta }  (1+| q |   )^{\ga}}  ,\quad\text{when }\quad q>0,\\
\notag
    E (   \lfloor  \frac{n}{2} \rfloor  +1) \cdot \Big( \frac{\eps}{ (1+ t + | q | )^{\frac{(n-1)}{2}-\delta }  } (1+| q |   )^{\frac{1}{2} }  \\
\quad\quad\quad\quad\quad\quad     + O(  \frac{\eps}{ (1+ t + | q | )^{\frac{(n-1)}{2}-\delta }  } (1+| q |   )^{\frac{1}{2} }     \cdot  \frac{\eps \cdot (1+| q |   )^{\frac{1}{2} } }{ (1+ t + | q | )^{\frac{(n-1)}{2}-\delta }  }   \Big), \,\quad\text{when }\quad q<0 .\end{cases} 
      \eeaa
      
      Thus,
        \bea
 \notag
&& |   H (t,x)  | \\
      &\les& \begin{cases} c (\delta) \cdot c (\gamma)  \cdot E (   \lfloor  \frac{n}{2} \rfloor  +1) \cdot  \frac{\eps}{ (1+ t + | q | )^{\frac{(n-1)}{2}-\delta }  (1+| q |   )^{\ga}}  ,\quad\text{when }\quad q>0,\\
\notag
 E (   \lfloor  \frac{n}{2} \rfloor  +1) \cdot \frac{\eps}{ (1+ t + | q | )^{\frac{(n-1)}{2}-\delta }  } (1+| q |   )^{\frac{1}{2} }  , \,\quad\text{when }\quad q<0 .\end{cases} \\
      \eea
However, given the fact that in the expression      
  \beaa
H^{\mu\nu}=-h^{\mu\nu}+ O^{\mu\nu}(h^2) \;,
\eeaa
here the $O^{\mu\nu}(h^2)$ happen to be a product of tensors of $m$ with $h^2$, we then also have that
  \bea
\derm_\a H^{\mu\nu}=- \derm_\a h^{\mu\nu}+ O_\a^{\, \,\,\, \mu\nu} (h \cdot \derm h )\;. 
\eea      

Since for all $|I|$,
               \bea
 \notag
&& |   \Lie_{Z^I} h (t,x) |  \cdot |\derm  ( \Lie_{Z^I} h ) (t,x)  |  \\
 \notag
  &\leq& \begin{cases} c (\delta) \cdot  c (\gamma) \cdot C ( |I| ) \cdot E ( |I| +  \lfloor  \frac{n}{2} \rfloor  +1)  \cdot \frac{\eps^2 }{(1+t+|q|)^{(n-1)-2\delta} (1+|q|)^{1+2\ga}},\quad\text{when }\quad q>0,\\
       C ( |I| ) \cdot E ( |I| +  \lfloor  \frac{n}{2} \rfloor  +1)  \cdot \frac{\eps^2  }{(1+t+|q|)^{(n-1)-2\delta} }  \,\quad\text{when }\quad q<0 , \end{cases} \\
      \eea
      
we obtain,
              \beaa
 \notag
 && |\derm ( \Lie_{Z^I}  h ) (t,x)  | +  |   \Lie_{Z^I} h (t,x) |  \cdot |\derm  ( \Lie_{Z^I} h ) (t,x)  |  \\
  &\leq& \begin{cases} c (\delta) \cdot  c (\gamma) \cdot C ( |I| ) \cdot E ( |I| +  \lfloor  \frac{n}{2} \rfloor  +1)  \cdot \Big( \frac{\eps }{(1+t+|q|)^{\frac{(n-1)}{2}-\delta} (1+|q|)^{1+\ga}} \\
\quad\quad\quad\quad\quad\quad\quad\quad\quad\quad\quad\quad\quad\quad\quad\quad\quad  + \frac{\eps^2 }{(1+t+|q|)^{(n-1)-2\delta} (1+|q|)^{1+2\ga}}\Big),\quad\text{when }\quad q>0,\\
       C ( |I| ) \cdot E ( |I| +  \lfloor  \frac{n}{2} \rfloor  +1)  \cdot \big( \frac{\eps  }{(1+t+|q|)^{\frac{(n-1)}{2}-\delta}(1+|q|)^{\frac{1}{2} }} +  \frac{\eps^2  }{(1+t+|q|)^{(n-1)-2\delta} } \big) \,\quad\text{when }\quad q<0 . \end{cases} \\
       \eeaa

            Thus, if $\eps \leq 1$ and if $\frac{(n-1)}{2} - \de \geq \frac{1}{2}$\;, which means if  $\delta \leq \frac{(n-1)}{2} - \frac{1}{2} \leq \frac{(n-2)}{2}$, we get

       \beaa
       \notag
 && |\derm  ( \Lie_{Z^I} h ) (t,x)  | +  |   \Lie_{Z^I} h (t,x) |  \cdot |\derm  ( \Lie_{Z^I} h ) (t,x)  |  \\
 \notag
         &\leq& \begin{cases} c (\delta) \cdot  c (\gamma) \cdot C ( |I| ) \cdot E (|I| +  \lfloor  \frac{n}{2} \rfloor  +1)  \cdot \frac{\eps }{(1+t+|q|)^{\frac{(n-1)}{2}-\delta} (1+|q|)^{1+\ga}} ,\quad\text{when }\quad q>0,\\
       C ( |I| ) \cdot E ( |I| +  \lfloor  \frac{n}{2} \rfloor  +1)  \cdot  \frac{\eps  }{(1+t+|q|)^{\frac{(n-1)}{2}-\delta}(1+|q|)^{\frac{1}{2} }}   \,\quad\text{when }\quad q<0 , \end{cases} \\
      \eeaa
which gives the result for  $|\derm   H (t,x)  |$.

\end{proof}

      \begin{lemma}\label{EstimateonLiederivativeZofthemetricbigH}
      We have for all $|I|$, $\delta  \leq \frac{(n-2)}{2}$,\,$\eps \leq 1$,

              \beaa
 \notag
 |\derm ( \Lie_{Z^I} H ) (t,x)  |    &\leq& \begin{cases} C ( |I| ) \cdot E (|I| +  \lfloor  \frac{n}{2} \rfloor  +1)  \cdot \frac{\eps }{(1+t+|q|)^{\frac{(n-1)}{2}-\delta} (1+|q|)^{1+\ga}},\quad\text{when }\quad q>0,\\
       C ( |I| ) \cdot E (|I| +  \lfloor  \frac{n}{2} \rfloor  +1)  \cdot \frac{\eps  }{(1+t+|q|)^{\frac{(n-1)}{2}-\delta}(1+|q|)^{\frac{1}{2} }}  \,\quad\text{when }\quad q<0 , \end{cases} \\
      \eeaa
      and
       \beaa
 \notag
|  \Lie_{Z^I} H (t,x)  | &\leq& \begin{cases} c (\delta) \cdot c (\gamma) \cdot C ( |I| ) \cdot E ( |I| +  \lfloor  \frac{n}{2} \rfloor  +1) \cdot  \frac{\eps}{ (1+ t + | q | )^{\frac{(n-1)}{2}-\delta }  (1+| q |   )^{\ga}}  ,\quad\text{when }\quad q>0,\\
\notag
    C ( |I| ) \cdot E (|I| +  \lfloor  \frac{n}{2} \rfloor  +1) \cdot  \frac{\eps}{ (1+ t + | q | )^{\frac{(n-1)}{2}-\delta }  } (1+| q |   )^{\frac{1}{2} }  , \,\quad\text{when }\quad q<0 . \end{cases} \\ 
    \eeaa

\end{lemma}

\begin{proof}
We have already showed in Lemma \ref{BigHintermsofsmallh}, that 
  \beaa
H^{\mu\nu}=-h^{\mu\nu}+ O^{\mu\nu}(h^2) .
\eeaa
Hence, for $\mu, \nu \in \{x^0, x^1, \ldots, x^n \}$, we have
  \beaa
H_{\mu\nu}=-h_{\mu\nu}+ O_{\mu\nu}(h^2) .
\eeaa
 
Using again that here that $O_{\mu\nu}(h^2)$ and $O_{\a\mu\nu} (h \cdot \pa h )$ are in fact product of Minkowski metric with $h$ and $\derm h$, and using Lemma \ref{LiederivativesofproductsZofcovariantandcontravariantMinkowskimetrictensor}, as well as the Leibniz rule for Lie derivatives, we obtain that for all $Z \in \cal Z$, 
  \beaa
 \Lie_{Z}  H_{\mu\nu} &=& -  \Lie_{Z}  h_{\mu\nu}+ O_{\mu\nu}( h \cdot  \Lie_{Z} h ) \\
 \derm_\a  ( \Lie_{Z}    H)_{\mu\nu} &=& -  \derm_\a  ( \Lie_{Z}  h)_{\mu\nu}+ O_{\a\mu\nu} (  \derm  h \cdot  \Lie_{Z}     h ) + O_{\a\mu\nu} (    h \cdot   \derm (\Lie_{Z}     h)  ). 
\eeaa
Since $|h|$ and  $ | \Lie_{Z^I} h|$ obey the same estimate, and since also $|\derm h|$ and  $| \derm ( \Lie_{Z^I} h ) |$ obey the same estimate, we then derive the same estimate for $|\Lie_{Z} H^{\mu\nu}|$ as for  $| H^{\mu\nu}|$ and the same estimate for $|\derm ( \Lie_{Z} H )^{\mu\nu}|$ as for  $| \derm H^{\mu\nu}|$. By induction, we get the result for all $|I|$.

\end{proof}

We now look at the commutator term for $n \geq 4$.

\begin{lemma}\label{thecommutatortermusingthebootstrap}
For $\Phi = H$ or $\Phi = A$, using the bootstrap assumption on $\Phi$, we have
   \beaa
\notag
 && | g^{\la\mu} \derm_{\la}   \derm_{\mu}   ( \Lie_{ Z^I}   \Phi ) - \Lie_{Z^I}  ( g^{\la\mu} \derm_{\la}   \derm_{\mu}     \Phi ) | \\
 &\les&  \Big( \sum_{|J|\leq |I|} |\Lie_{Z^{J}} H|  \Big)   \cdot \begin{cases} C ( |I| ) \cdot E (\lfloor \frac{|I|}{2} \rfloor +  \lfloor  \frac{n}{2} \rfloor  +1)  \cdot \frac{\eps }{(1+t+|q|)^{\frac{(n-1)}{2}-\delta} (1+|q|)^{2+\ga}},\quad\text{for } q>0,\\
       C ( |I| ) \cdot E (\lfloor \frac{|I|}{2} \rfloor +  \lfloor  \frac{n}{2} \rfloor  +1)  \cdot \frac{\eps  }{(1+t+|q|)^{\frac{(n-1)}{2}-\delta}(1+|q|)^{\frac{3}{2} }} , \,\quad\text{for } q<0 , \end{cases} 
\\
&& +  \Big( \sum_{|J|\leq |I|} |\derm ( \Lie_{Z^{K}} \Phi )  |  \Big)  \\
&&\times \begin{cases} c (\delta) \cdot c (\gamma) \cdot C ( |I| ) \cdot E ( \lfloor \frac{|I|}{2} \rfloor+  \lfloor  \frac{n}{2} \rfloor  +2) \cdot  \frac{\eps}{ (1+ t + | q | )^{\frac{(n-1)}{2}-\delta }  (1+| q |   )^{1+\ga}}  ,\quad\text{for } q>0,\\
\notag
    C ( |I| ) \cdot E (\lfloor \frac{|I|}{2} \rfloor +  \lfloor  \frac{n}{2} \rfloor  +2) \cdot  \frac{\eps}{ (1+ t + | q | )^{\frac{(n-1)}{2}-\delta } \cdot (1+| q |   )^{\frac{1}{2} } }   , \,\quad\text{for } q<0 . \end{cases}  \\
&& +\sum_{|K| \leq |I| -1}  | \Lie_{Z^K}  g^{\la\mu} \derm_{\la}   \derm_{\mu}     \Phi | \;.
\eeaa

\end{lemma}

\begin{proof}

Let  $\Phi_{\mu\nu}$ be a tensor valued either in the Lie algebra (which could be the one tensor Yang-Mills potential) or a two tensor valued a s a scalar (the two tensor of the metric $h^1$), satisfying the following tensorial wave equation
\beaa
 g^{\la\a} \derm_{\la}   \derm_{\a}   \Phi_{\mu\nu}= S_{\mu\nu} \, , 
\eeaa
where $S_{\mu\nu}$ is the source term. Based on a more refined estimate that we will prove in a paper that follows that deals with the case $n=3$, (see also \cite{Lind}, \cite{LT} and \cite{LR10}), we have
\beaa
\notag
&& | g^{\la\mu} \derm_{\la}   \derm_{\mu}   ( \Lie_{ Z^I}   \Phi ) - \Lie_{Z^I}  ( g^{\la\mu} \derm_{\la}   \derm_{\mu}     \Phi ) | \\
\notag
 \les&& \frac 1{1+t+|q|}
\,\,\,\sum_{|K|\leq |I|,}\,\, \sum_{|J|+(|K|-1)_+\le |I|} \,\,\,
|\Lie_{ Z^{J}} H|\,\cdot {|\derm ( \Lie_{Z^{K}}  \Phi ) |} \\
\notag
& +& \frac 1{1+|q|}
 \sum_{|K|\leq |I|,} \,\,  \sum_{|J|+(|K|-1)_+\leq |I|} \!\!\!\!\!| (\Lie_{Z^{J}} H)_{LL} | \cdot {|\derm \Lie_{Z^{K}} \Phi |} \\
&& +\sum_{|K| \leq  |I| -1}  | \Lie_{Z^K}  g^{\la\mu} \derm_{\la}   \derm_{\mu}     \Phi |\;,
\eeaa
where $(|K|-1)_+=|K|-1$ if $|K|\geq 1$ and $(|K|-1)_+=0$ if
$|K|=0$.
Therefore,

\beaa
\notag
 && | g^{\la\mu} \derm_{\la}   \derm_{\mu}   ( \Lie_{ Z^I}   \Phi ) - \Lie_{Z^I}  ( g^{\la\mu} \derm_{\la}   \derm_{\mu}     \Phi ) | \\
 &\les& \frac {1}{1+|q|}  \sum_{|K|\leq |I|} \Big( \sum_{|J|+(|K|-1)_+\leq |I|} |\Lie_{Z^{J}} H|  \Big)   \cdot {|\derm ( \Lie_{Z^{K}} \Phi ) |} \\
&& +\sum_{|K| < |I|}  | \Lie_{Z^I}  g^{\la\mu} \derm_{\la}   \derm_{\mu}     \Phi | \; .
\eeaa
Thus,

\beaa
\notag
 && | g^{\la\mu} \derm_{\la}   \derm_{\mu}   ( \Lie_{ Z^I}   \Phi ) - \Lie_{Z^I}  ( g^{\la\mu} \derm_{\la}   \derm_{\mu}     \Phi ) | \\
 &\les& \frac {1}{1+|q|}  \sum_{|K|\leq \lfloor \frac{|I|}{2} \rfloor } \Big( \sum_{|J|\leq |I|} |\Lie_{Z^{J}} H|  \Big)   \cdot {|\derm  ( \Lie_{Z^{K}} \Phi ) |} \\
&& + \frac {1}{1+|q|}  \sum_{\lfloor \frac{|I|}{2} \rfloor \leq |K|\leq |I|  } \Big( \sum_{|J|\leq \lfloor \frac{|I|}{2} \rfloor + 1} |\Lie_{Z^{J}} H|  \Big)   \cdot {|\derm ( \Lie_{Z^{K}} \Phi ) |} \\
&& +\sum_{|K| \leq |I| -1}  | \Lie_{Z^K}  g^{\la\mu} \derm_{\la}   \derm_{\mu}     \Phi | \;.
\eeaa

Yet, for $|K|\leq \lfloor \frac{|I|}{2} \rfloor$, and for either $\Phi = H$ or $\Phi = A$, using the bootstrap assumption, we obtain
\beaa
|\derm  ( \Lie_{Z^{K}} \Phi )  | &\leq& \begin{cases} C ( |I| ) \cdot E (\lfloor \frac{|I|}{2} \rfloor +  \lfloor  \frac{n}{2} \rfloor  +1)  \cdot \frac{\eps }{(1+t+|q|)^{\frac{(n-1)}{2}-\delta} (1+|q|)^{1+\ga}},\quad\text{when }\quad q>0,\\
       C ( |I| ) \cdot E (\lfloor \frac{|I|}{2} \rfloor +  \lfloor  \frac{n}{2} \rfloor  +1)  \cdot \frac{\eps  }{(1+t+|q|)^{\frac{(n-1)}{2}-\delta}(1+|q|)^{\frac{1}{2} }}  \,\quad\text{when }\quad q<0 , \end{cases} 
\eeaa
and for  $|J|\leq \lfloor \frac{|I|}{2} \rfloor + 1$,

      \beaa
 \notag
|  \Lie_{Z^J} H (t,x)  | &\leq& \begin{cases} c (\delta) \cdot c (\gamma) \cdot C ( |I| ) \cdot E ( \lfloor \frac{|I|}{2} \rfloor+  \lfloor  \frac{n}{2} \rfloor  +2) \cdot  \frac{\eps}{ (1+ t + | q | )^{\frac{(n-1)}{2}-\delta }  (1+| q |   )^{\ga}}  ,\quad\text{when }\quad q>0,\\
\notag
    C ( |I| ) \cdot E (\lfloor \frac{|I|}{2} \rfloor +  \lfloor  \frac{n}{2} \rfloor  +2) \cdot  \frac{\eps}{ (1+ t + | q | )^{\frac{(n-1)}{2}-\delta }  } (1+| q |   )^{\frac{1}{2} }  , \,\quad\text{when }\quad q<0 . \end{cases} 
    \eeaa
   Consequently,
   
   \beaa
\notag
 && | g^{\la\mu} \derm_{\la}   \derm_{\mu}   ( \Lie_{ Z^I}   \Phi ) - \Lie_{Z^I}  ( g^{\la\mu} \derm_{\la}   \derm_{\mu}     \Phi ) | \\
 &\les&  \Big( \sum_{|J|\leq |I|} |\Lie_{Z^{J}} H|  \Big)   \cdot \begin{cases} C ( |I| ) \cdot E (\lfloor \frac{|I|}{2} \rfloor +  \lfloor  \frac{n}{2} \rfloor  +1)  \cdot \frac{\eps }{(1+t+|q|)^{\frac{(n-1)}{2}-\delta} (1+|q|)^{2+\ga}},\quad\text{for } q>0,\\
       C ( |I| ) \cdot E (\lfloor \frac{|I|}{2} \rfloor +  \lfloor  \frac{n}{2} \rfloor  +1)  \cdot \frac{\eps  }{(1+t+|q|)^{\frac{(n-1)}{2}-\delta}(1+|q|)^{\frac{3}{2} }} , \,\quad\text{for } q<0 , \end{cases} 
\\
&& +  \Big( \sum_{|J|\leq |I|} |\derm ( \Lie_{Z^{K}} \Phi )  |  \Big) \\
&& \times \begin{cases} c (\delta) \cdot c (\gamma) \cdot C ( |I| ) \cdot E ( \lfloor \frac{|I|}{2} \rfloor+  \lfloor  \frac{n}{2} \rfloor  +2) \cdot  \frac{\eps}{ (1+ t + | q | )^{\frac{(n-1)}{2}-\delta }  (1+| q |   )^{1+\ga}}  ,\quad\text{for } q>0,\\
\notag
    C ( |I| ) \cdot E (\lfloor \frac{|I|}{2} \rfloor +  \lfloor  \frac{n}{2} \rfloor  +2) \cdot  \frac{\eps}{ (1+ t + | q | )^{\frac{(n-1)}{2}-\delta } \cdot (1+| q |   )^{\frac{1}{2} } }   , \,\quad\text{for } q<0 . \end{cases}  \\
&& +\sum_{|K| \leq |I| -1}  | \Lie_{Z^K}  g^{\la\mu} \derm_{\la}   \derm_{\mu}     \Phi | \;.
\eeaa
\end{proof}

\begin{lemma}\label{estimatingtheequareofthesourcetermsusingbootsrapassumption}
For $n\geq 4$\,, $\de = 0$\,, $\eps \leq 1$\,, for either $\Phi = H$ or $\Phi = A$\,, using the bootstrap assumption on $\Lie_{Z^{K}} \Phi$ for $|K|\leq \lfloor \frac{|I|}{2} \rfloor$\,, we have
   \beaa
\notag
 && (1+t )^{1+\la}  \cdot  | g^{\a\b} \derm_{\a}   \derm_{\b}   ( \Lie_{ Z^I}   \Phi ) |^2  \\
&\les&  \Big( \sum_{|J|\leq |I|} |\Lie_{Z^{J}} H|^2  \Big)   \cdot     C ( |I| ) \cdot E ( \lfloor \frac{|I|}{2} \rfloor+  \lfloor  \frac{n}{2} \rfloor  +1)  \cdot \frac{\eps  }{(1+t+|q|)^{2-\la}(1+|q|)^{3 }} \\
\notag
&& + \Big( \sum_{|J|\leq |I|} |\derm ( \Lie_{Z^{K}} \Phi  ) |^2  \Big) \cdot   C ( |I| ) \cdot E ( \lfloor \frac{|I|}{2} \rfloor+  \lfloor  \frac{n}{2} \rfloor  +2) \cdot   \frac{\eps}{ (1+ t + | q | )^{2-\la } \cdot (1+| q |   ) }  \\
\notag
&& + (1+t )^{1+\la}  \cdot \sum_{|K| \leq |I| }  | \Lie_{Z^K}  g^{\a\b} \derm_{\a}   \derm_{\b}     \Phi |^2 \;.
\eeaa

\end{lemma}

\begin{proof}
We showed in Lemma \ref{thecommutatortermusingthebootstrap}, that
   \beaa
\notag
 && | g^{\a\b} \derm_{\a}   \derm_{\b}   ( \Lie_{ Z^I}   \Phi ) - \Lie_{Z^I}  ( g^{\a\b} \derm_{\a}   \derm_{\b}     \Phi ) | \\
 &\les&  \Big( \sum_{|J|\leq |I|} |\Lie_{Z^{J}} H|  \Big)   \cdot \begin{cases} C ( |I| ) \cdot E (\lfloor \frac{|I|}{2} \rfloor +  \lfloor  \frac{n}{2} \rfloor  +1)  \cdot \frac{\eps }{(1+t+|q|)^{\frac{(n-1)}{2}-\delta} (1+|q|)^{2+\ga}},\quad\text{for } q>0,\\
       C ( |I| ) \cdot E (\lfloor \frac{|I|}{2} \rfloor +  \lfloor  \frac{n}{2} \rfloor  +1)  \cdot \frac{\eps  }{(1+t+|q|)^{\frac{(n-1)}{2}-\delta}(1+|q|)^{\frac{3}{2} }} , \,\quad\text{for } q<0 , \end{cases} 
\\
&& +  \Big( \sum_{|J|\leq |I|} |\derm  ( \Lie_{Z^{K}} \Phi )  |  \Big) \\
&& \times \begin{cases} c (\delta) \cdot c (\gamma) \cdot C ( |I| ) \cdot E ( \lfloor \frac{|I|}{2} \rfloor+  \lfloor  \frac{n}{2} \rfloor  +2) \cdot  \frac{\eps}{ (1+ t + | q | )^{\frac{(n-1)}{2}-\delta }  (1+| q |   )^{1+\ga}}  ,\quad\text{for } q>0,\\
\notag
    C ( |I| ) \cdot E (\lfloor \frac{|I|}{2} \rfloor +  \lfloor  \frac{n}{2} \rfloor  +2) \cdot  \frac{\eps}{ (1+ t + | q | )^{\frac{(n-1)}{2}-\delta } \cdot (1+| q |   )^{\frac{1}{2} } }   , \,\quad\text{for } q<0 . \end{cases}  \\
&& +\sum_{|K| \leq |I| -1}  | \Lie_{Z^K}  g^{\a\b} \derm_{\a}   \derm_{\b}     \Phi | \;.
\eeaa
Taking $\de = 0$, $\ga \geq - \frac{1}{2}$, and for $n\geq4$, we obtain
   \beaa
\notag
 && | g^{\a\b} \derm_{\a}   \derm_{\b}   ( \Lie_{ Z^I}   \Phi ) - \Lie_{Z^I}  ( g^{\a\b} \derm_{\a}   \derm_{\b}     \Phi ) | \\
 &\les&  \Big( \sum_{|J|\leq |I|} |\Lie_{Z^{J}} H|  \Big)   \cdot \begin{cases} C ( |I| ) \cdot E ( \lfloor \frac{|I|}{2} \rfloor+  \lfloor  \frac{n}{2} \rfloor  +1)  \cdot \frac{\eps }{(1+t+|q|)^{\frac{3}{2}} (1+|q|)^{2+\ga}},\quad\text{for } q>0,\\
       C ( |I| ) \cdot E ( \lfloor \frac{|I|}{2} \rfloor+  \lfloor  \frac{n}{2} \rfloor  +1)  \cdot \frac{\eps  }{(1+t+|q|)^{\frac{3}{2}}(1+|q|)^{\frac{3}{2} }} , \,\quad\text{for } q<0 , \end{cases} 
\\
&& +  \Big( \sum_{|J|\leq |I|} |\derm \Lie_{Z^{K}} \Phi |  \Big) \cdot \begin{cases} c (\delta) \cdot c (\gamma) \cdot C ( |I| ) \cdot E ( \lfloor \frac{|I|}{2} \rfloor+  \lfloor  \frac{n}{2} \rfloor  +2) \cdot  \frac{\eps}{ (1+ t + | q | )^{\frac{3}{2} }  (1+| q |   )^{1+\ga}}  ,\quad\text{for } q>0,\\
\notag
    C ( |I| ) \cdot E ( \lfloor \frac{|I|}{2} \rfloor+  \lfloor  \frac{n}{2} \rfloor  +2) \cdot  \frac{\eps}{ (1+ t + | q | )^{\frac{3}{2} } \cdot (1+| q |   )^{\frac{1}{2} } }   , \,\quad\text{for } q<0 . \end{cases}  \\
&\les&  \Big( \sum_{|J|\leq |I|} |\Lie_{Z^{J}} H|  \Big)   \cdot     C ( |I| ) \cdot E ( \lfloor \frac{|I|}{2} \rfloor+  \lfloor  \frac{n}{2} \rfloor  +1)  \cdot \frac{\eps  }{(1+t+|q|)^{\frac{3}{2}}(1+|q|)^{\frac{3}{2} }} \\
\notag
&& + \Big( \sum_{|J|\leq |I|} |\derm ( \Lie_{Z^{K}} \Phi )  |  \Big) \cdot    c (\delta) \cdot c (\gamma) \cdot C ( |I| ) \cdot E ( \lfloor \frac{|I|}{2} \rfloor+  \lfloor  \frac{n}{2} \rfloor  +2) \cdot   \frac{\eps}{ (1+ t + | q | )^{\frac{3}{2}} \cdot (1+| q |   )^{\frac{1}{2} } }  \\
\notag
&& +\sum_{|K| \leq |I| -1}  | \Lie_{Z^K}  g^{\a\b} \derm_{\a}   \derm_{\b}     \Phi | \;.
\eeaa

Hence,
   \beaa
\notag
 && | g^{\a\b} \derm_{\a}   \derm_{\b}   ( \Lie_{ Z^I}   \Phi ) |  \\
&\les&  \Big( \sum_{|J|\leq |I|} |\Lie_{Z^{J}} H|  \Big)   \cdot     C ( |I| ) \cdot E ( \lfloor \frac{|I|}{2} \rfloor+  \lfloor  \frac{n}{2} \rfloor  +1)  \cdot \frac{\eps  }{(1+t+|q|)^{\frac{3}{2}}(1+|q|)^{\frac{3}{2} }} \\
\notag
&& + \Big( \sum_{|J|\leq |I|} |\derm (  \Lie_{Z^{K}} \Phi )  |  \Big) \cdot   C ( |I| ) \cdot E ( \lfloor \frac{|I|}{2} \rfloor+  \lfloor  \frac{n}{2} \rfloor  +2) \cdot   \frac{\eps}{ (1+ t + | q | )^{\frac{3}{2} } \cdot (1+| q |   )^{\frac{1}{2} } }  \\
\notag
&& +\sum_{|K| \leq |I| }  | \Lie_{Z^K}  g^{\a\b} \derm_{\a}   \derm_{\b}     \Phi | \;.
\eeaa
Consequently,
   \beaa
\notag
 &&  | g^{\a\b} \derm_{\a}   \derm_{\b}   ( \Lie_{ Z^I}   \Phi ) |^2  \\
&\les&  \Big( \sum_{|J|\leq |I|} |\Lie_{Z^{J}} H|^2  \Big)   \cdot     C ( |I| ) \cdot E ( \lfloor \frac{|I|}{2} \rfloor+  \lfloor  \frac{n}{2} \rfloor  +1)  \cdot \frac{\eps^2  }{(1+t+|q|)^{3}(1+|q|)^{3 }} \\
\notag
&& + \Big( \sum_{|J|\leq |I|} |\derm (  \Lie_{Z^{K}} \Phi )  |^2  \Big) \cdot   C ( |I| ) \cdot E ( \lfloor \frac{|I|}{2} \rfloor+  \lfloor  \frac{n}{2} \rfloor  +2) \cdot   \frac{\eps^2}{ (1+ t + | q | )^{3 } \cdot (1+| q |   ) }  \\
\notag
&& +  \sum_{|K| \leq |I| }  | \Lie_{Z^K}  g^{\a\b} \derm_{\a}   \derm_{\b}     \Phi |^2 \;.
\eeaa

\end{proof}

 \subsection{Using the Hardy type inequality to estimate the commutator term}\

\begin{lemma}\label{Hardytypeinequalitywithintegralonthewholespaceslice}
Let $w$ be defined as in Definition \ref{defw}, where $\ga > 0$.
Let  $\Phi$ a tensor that decays fast enough at spatial infinity for all time $t$, such that
\bea
\notag
 \int_{\SSS^{n-1}} \lim_{r \to \infty} \Big( \frac{r^{n-1}}{(1+t+r)^{a} \cdot (1+|q|) }\cdot w \cdot <\Phi, \Phi>  \Big)  d\si^{n-1} (t ) &=& 0 \; . \\
\eea
Then, since $\ga \neq 0$, we have for $0 \leq a \leq n-1$, 
  \bea
\notag
  \int_{\Si_{t}} \frac{1}{(1+t+|q| )^{a}(1+|q|)^2} \cdot |\Phi |^2 \cdot   w       &\leq& c(\ga) \cdot  \int_{\Si_{t}}  \frac{1}{(1+t+|q|)^{a}}  \cdot | \derm_r \Phi |^2  \cdot   w  \; . \\
 \eea
 
 \end{lemma}
 
 \begin{proof}
 
 Based on the Hardy type inequality that we showed in Corollary \ref{HardytypeinequalityforintegralstartingatROm}, by taking $R(\Om) = 0$ for all $\Om \in \SSS^{n-1}$, we have for $\ga \neq 0$ and for $0 \leq a \leq n-1$, that if
\bea
 \int_{\SSS^{n-1}} \lim_{r \to \infty} \Big( \frac{r^{n-1}}{(1+t+r)^{a} \cdot (1+|q|) } w(q) \cdot <\Phi, \Phi>  \Big)  d\si^{n-1} (t ) &=& 0 \;,
\eea
then,
\bea
\notag
 && \int_{\SSS^{n-1}} \int_{r=0}^{r=\infty} \frac{1}{(1+t+r)^{a}(1+|q|)^2} \cdot <\Phi, \Phi> \cdot   w  \cdot r^{n-1} dr d\si^{n-1}  \\
 \notag
 &\leq& c(\ga) \cdot  \int_{\SSS^{n-1}} \int_{r=0}^{r=\infty}  \frac{1}{(1+t+r)^{a}}  \cdot <\pa_r\Phi, \pa_r \Phi>  \cdot   w \cdot r^{n-1} dr d\si^{n-1}   \; . \\
 \eea

Thus, 
\bea
\notag
  \int_{\Si_{t}} \frac{1}{(1+t+r)^{a}(1+|q|)^2} \cdot <\Phi, \Phi> \cdot  w   &\leq& c(\ga, \mu) \cdot  \int_{\Si_{t}}  \frac{1}{(1+t+r)^{a}}  \cdot <\pa_r\Phi, \pa_r \Phi>  \cdot   w    \; . \\
 \eea
We have
  \bea
  1+t+r \sim 1+t+|q|  \; .
  \eea
  Therefore,
  \beaa
\notag
  && \int_{\Si_{t}} \frac{1}{(1+t+|q| )^{a}(1+|q|)^2} \cdot <\Phi, \Phi> \cdot  w \\
      &\leq&  \int_{\Si_{t}} \frac{1}{(1+t+r )^{a}(1+|q|)^2} \cdot <\Phi, \Phi> \cdot   w  \\
  &\leq& c(\ga) \cdot  \int_{\Si_{t}}  \frac{1}{(1+t+r)^{a}}  \cdot <\pa_r\Phi, \pa_r \Phi>  \cdot  w   \\
    &\leq& c(\ga) \cdot  \int_{\Si_{t}}  \frac{1}{(1+t+|q|)^{a}}  \cdot <\derm_r\Phi, \derm_r \Phi>  \cdot w  \; .
 \eeaa
 \end{proof}

 We will use now the Hardy type inequality to estimate the commutator term.
 \begin{lemma}\label{estimatethespace-timeintegralofthecommutatortermneededfortheenergyestimate}
For $n\geq 4$\,, let $H$ such that for all time $t$\,, for $\ga \neq 0$ and $ 0 < \la \leq \frac{1}{2}$\,, 
\bea
 \int_{\SSS^{n-1}} \lim_{r \to \infty} \Big( \frac{r^{n-1}}{(1+t+r)^{2-\la} \cdot (1+|q|) } \cdot  w(q) \cdot |H|^2  \Big)  d\si^{n-1} (t ) &=& 0 \; ,
\eea
and let $h$ such that for all time $t$\,, for all $|K| \leq |I|$\,,
\bea
 \int_{\SSS^{n-1}} \lim_{r \to \infty} \Big( \frac{r^{n-1}}{ (1+|q|) }  \cdot  w(q) \cdot  | \Lie_{Z^K} h  |^2  \Big)  d\si^{n-1} (t ) &=& 0 \; ,
\eea
then, for $\de = 0$\,, for either $\Phi = H$ or $\Phi = A$\,, using the bootstrap assumption on $\Phi$\,, we have
   \bea
      \notag
  &&   \int_0^t \Big(  \int_{\Si_{t}}  (1+t )^{1+\la}  \cdot  | g^{\a\b} \derm_{\a}   \derm_{\b}   ( \Lie_{ Z^I}   \Phi ) |^2  \cdot w  \cdot dx^1 \ldots dx^n  \Big) \cdot dt  \\
 \notag
&\les&  \int_0^t   \frac{\eps}{(1+t)^{2-\la}}  \cdot  C ( |I| ) \cdot E ( \lfloor \frac{|I|}{2} \rfloor+  \lfloor  \frac{n}{2} \rfloor  +1)  \cdot  c(\ga)  \\
      \notag
&&  \times \Big(  \sum_{|J|\leq |I|}   \int_{\Si_{t}} \big(  | \derm ( \Lie_{Z^{J}} h ) |^2 + |\derm ( \Lie_{Z^{K}} \Phi )  |^2 \big) \cdot   w  \cdot dx^1 \ldots dx^n   \Big) \cdot dt   \\
      \notag
&& + \int_0^t  \Big( \int_{\Si_{t}}  (1+t )^{1+\la}  \cdot \sum_{|K| \leq |I| }  | \Lie_{Z^K}  g^{\a\b} \derm_{\a}   \derm_{\b}     \Phi |^2   \cdot w   \cdot dx^1 \ldots dx^n  \Big) \cdot dt  \; . \\
\eea

 \end{lemma}

 \begin{proof}
 
Based on what we have shown in Lemma \ref{Hardytypeinequalitywithintegralonthewholespaceslice}, for $H$ decaying fast enough at spatial infinity, for $\ga \neq 0$ and $0 \leq a \leq n-1$, we have 
\bea
\notag
 && \int_{\SSS^{n-1}} \int_{r=0}^{r=\infty} \frac{1}{(1+t+|q|)^{a}(1+|q|)^2} \cdot |\Lie_{Z^{J}} H|^2 \cdot   w  \cdot r^{n-1} dr d\si^{n-1}  \\
 \notag
 &\leq& c(\ga) \cdot  \int_{\SSS^{n-1}} \int_{r=0}^{r=\infty}  \frac{1}{(1+t+|q|)^{a}}  \cdot | \derm_r ( \Lie_{Z^{J}} H ) |^2  \cdot   w  \cdot r^{n-1} dr d\si^{n-1}   \; . \\
 \eea

Hence, for $n \geq 4$,  for $ 0 < \la \leq \frac{1}{2}$, we have $2-\la < 2 \leq n-1 $ and therefore
\beaa
&&  C ( |I| ) \cdot E ( \lfloor \frac{|I|}{2} \rfloor+  \lfloor  \frac{n}{2} \rfloor  +1)   \cdot \sum_{|J|\leq |I|}  \int_{\Si_{t}}  \Big( |\Lie_{Z^{J}} H|^2  \Big)   \cdot     \frac{\eps  }{(1+t+|q|)^{2-\la}(1+|q|)^{3 }}  \cdot w \\
 &\leq& C ( |I| ) \cdot E ( \lfloor \frac{|I|}{2} \rfloor+  \lfloor  \frac{n}{2} \rfloor  +1)  \cdot    \sum_{|J|\leq |I|}  c(\ga) \cdot   \int_{\Si_{t}} \frac{\eps}{(1+t+|q|)^{2-\la}}  \cdot | \derm_r ( \Lie_{Z^{J}} H ) |^2  \cdot   w  \\
  &\leq& C ( |I| ) \cdot E ( \lfloor \frac{|I|}{2} \rfloor+  \lfloor  \frac{n}{2} \rfloor  +1)  \cdot  c(\ga)  \cdot   \frac{\eps}{(1+t)^{2-\la}}  \sum_{|J|\leq |I|}    \int_{\Si_{t}} \cdot | \derm ( \Lie_{Z^{J}} H ) |^2  \cdot   w \; . 
\eeaa

Based on what we showed in Lemma \ref{estimatingtheequareofthesourcetermsusingbootsrapassumption}, we have for $n\geq 4$\,, $\de = 0$\,, $ 0 < \la \leq \frac{1}{2}$\,,
   \beaa
\notag
 &&\int_{\Si_{t}}  (1+t )^{1+\la} \cdot  | g^{\a\b} \derm_{\a}   \derm_{\b}   ( \Lie_{ Z^I}   \Phi ) |^2  \cdot w  \\
&\les& \int_{\Si_{t}}  \Big( \sum_{|J|\leq |I|} |\Lie_{Z^{J}} H|^2  \Big)   \cdot     C ( |I| ) \cdot E ( \lfloor \frac{|I|}{2} \rfloor+  \lfloor  \frac{n}{2} \rfloor  +1)  \cdot \frac{\eps  }{(1+t+|q|)^{2-\la}(1+|q|)^{3 }}  \cdot w \\
\notag
&& + \int_{\Si_{t}} \Big( \sum_{|J|\leq |I|} |\derm (  \Lie_{Z^{K}} \Phi )  |^2  \Big) \cdot   C ( |I| ) \cdot E ( \lfloor \frac{|I|}{2} \rfloor+  \lfloor  \frac{n}{2} \rfloor  +2) \cdot   \frac{\eps}{ (1+ t + | q | )^{2-\la } \cdot (1+| q |   ) }  \cdot w \\
\notag
&& +\int_{\Si_{t}}  (1+t )^{1+\la}  \cdot \sum_{|K| \leq |I| }  | \Lie_{Z^K}  g^{\a\b} \derm_{\a}   \derm_{\b}     \Phi |^2   \cdot w \\
&\les& C ( |I| ) \cdot E ( \lfloor \frac{|I|}{2} \rfloor+  \lfloor  \frac{n}{2} \rfloor  +1)  \cdot  c(\ga)  \cdot   \frac{\eps}{(1+t)^{2-\la}}  \cdot  \sum_{|J|\leq |I|}   \int_{\Si_{t}}  | \derm ( \Lie_{Z^{J}} H ) |^2  \cdot   w    \\
\notag
&& +C ( |I| ) \cdot E ( \lfloor \frac{|I|}{2} \rfloor+  \lfloor  \frac{n}{2} \rfloor  +1)    \cdot   \frac{\eps}{(1+t)^{2-\la}} \cdot  \sum_{|J|\leq |I|}    \int_{\Si_{t}}  |\derm ( \Lie_{Z^{K}} \Phi )  |^2    \cdot   w    \\
\notag
&& +\int_{\Si_{t}}  (1+t )^{1+\la}  \cdot \sum_{|K| \leq |I| }  | \Lie_{Z^K}  g^{\a\b} \derm_{\a}   \derm_{\b}     \Phi |^2   \cdot w \; .
\eeaa

However, based on what we showed in Lemma \ref{BigHintermsofsmallh}, and by lowering indices with respect to the metric $m$, we have
  \beaa
H_{\mu\nu}=-h_{\mu\nu}+ O_{\mu\nu}(h^2) .
\eeaa
thus, for all $|I|$,
  \beaa
 \Lie_{Z^I}  H_{\mu\nu} &=& -   \Lie_{Z^I}  h_{\mu\nu}+   \sum_{|J| + |K| \leq |I|} O_{\mu\nu}(  \Lie_{Z^J}  h  \cdot  \Lie_{Z^K} h ) \\
 \eeaa
 and hence,
   \beaa
\derm_\a ( \Lie_{Z^I}  H_{\mu\nu} )  &=& -  \derm_a \Lie_{Z^I}  h_{\mu\nu}+   \sum_{|J| + |K| \leq |I|} O_{\mu\nu}( \derm_\a ( \Lie_{Z^J}  h ) \cdot  \Lie_{Z^K} h ) \; .
\eeaa
Consequently,
   \beaa
| \derm_\a ( \Lie_{Z^I}  H ) |  &\leq& | \derm_\a ( \Lie_{Z^I}  h ) | +   \sum_{|J| + |K| \leq |I|}  | \derm_\a ( \Lie_{Z^J}  h ) | \cdot  | \Lie_{Z^K} h  | \; .
\eeaa
We obtain,
\beaa
&& | \derm ( \Lie_{Z^I}  H ) |  \\
&\les&  | \derm ( \Lie_{Z^I}  h ) | +   \sum_{|J| + |K| \leq |I|}  | \derm ( \Lie_{Z^J}  h ) | \cdot  | \Lie_{Z^K} h  | \\
&\les &  | \derm ( \Lie_{Z^I}  h ) |  +    \sum_{|J| \leq  \lfloor \frac{|I|}{2} \rfloor  \; , \; |K|\leq |I|}   | \derm ( \Lie_{Z^J}  h ) | \cdot  | \Lie_{Z^K} h  |   +   \sum_{|J| \leq |I| \; , \; |K| \leq  \lfloor \frac{|I|}{2} \rfloor  }   | \derm ( \Lie_{Z^J}  h ) | \cdot  | \Lie_{Z^K} h  |  \; .
  \eeaa

  We have shown in Lemma \ref{aprioriestimatesongradientoftheLiederivativesofthefields} and Lemma \ref{aprioriestimatefrombootstraponzerothderivativeofAandh1}, that for $n\geq 4$,  $\de = 0$, we have for all  $|I|$,
                 \beaa
 \notag
 |\derm ( \Lie_{Z^I} h) (t,x)  |    &\leq&       C ( |I| ) \cdot E (|I| +  \lfloor  \frac{n}{2} \rfloor  +1)   \cdot \frac{\eps  }{(1+t+|q|)^{\frac{3}{2}}(1+|q|)^{\frac{1}{2} }} \; ,\\
 \notag
|  \Lie_{Z^I} h (t,x)  |      &\leq&    C ( |I| ) \cdot E (|I| +  \lfloor  \frac{n}{2} \rfloor  +1) \cdot  \frac{\eps}{ (1+ t + | q | )^{ \frac{3}{2}} } \cdot (1+| q |   )^{\frac{1}{2} }  \; .
    \eeaa

Hence,
\beaa
&& | \derm ( \Lie_{Z^I}  H ) |  \\
\notag
&\les &  | \derm ( \Lie_{Z^I}  h ) |  +    \sum_{  |K|\leq |I|}   | \Lie_{Z^K} h  | \cdot   C (  \lfloor \frac{|I|}{2} \rfloor   ) \cdot E ( \lfloor \frac{|I|}{2} \rfloor   +  \lfloor  \frac{n}{2} \rfloor  +1)   \cdot \frac{\eps  }{(1+t+|q|)^{\frac{3}{2}}(1+|q|)^{\frac{1}{2} }} \\
&&   +   \sum_{|J| \leq |I|  }   | \derm ( \Lie_{Z^J}  h ) | \cdot   C ( \lfloor \frac{|I|}{2} \rfloor  ) \cdot E ( \lfloor \frac{|I|}{2} \rfloor  +  \lfloor  \frac{n}{2} \rfloor  +1) \cdot  \frac{\eps}{ (1+ t + | q | )^{\frac{3}{2} } } \cdot (1+| q |   )^{\frac{1}{2} } \\
&\les & \sum_{|J| \leq |I|  }   \Big( 1+ C ( \lfloor \frac{|I|}{2} \rfloor  ) \cdot E ( \lfloor \frac{|I|}{2} \rfloor  +  \lfloor  \frac{n}{2} \rfloor  +1) \cdot \eps \Big) \cdot | \derm ( \Lie_{Z^J}  h ) |  \\
\notag
&& +    \sum_{  |K|\leq |I|}     C (  \lfloor \frac{|I|}{2} \rfloor   ) \cdot E ( \lfloor \frac{|I|}{2} \rfloor   +  \lfloor  \frac{n}{2} \rfloor  +1)   \cdot \frac{\eps  }{(1+t+|q|)^{\frac{3}{2}} (1+| q |   )^{\frac{1}{2} }  } \cdot | \Lie_{Z^K} h  | \;.\\
  \eeaa
  Consequently,
  \beaa
&& | \derm ( \Lie_{Z^I}  H ) |^2  \\
\notag
&\les & \sum_{|J| \leq |I|  }   \Big( 1+ C ( \lfloor \frac{|I|}{2} \rfloor  ) \cdot E ( \lfloor \frac{|I|}{2} \rfloor  +  \lfloor  \frac{n}{2} \rfloor  +1) \cdot \eps \Big)^2 \cdot | \derm ( \Lie_{Z^J}  h ) |^2  \\
\notag
&& +    \sum_{  |K|\leq |I|}     C (  \lfloor \frac{|I|}{2} \rfloor   ) \cdot E ( \lfloor \frac{|I|}{2} \rfloor   +  \lfloor  \frac{n}{2} \rfloor  +1)   \cdot \frac{\eps  }{(1+t+|q|)^{3}\cdot (1+|q|) }  \cdot | \Lie_{Z^K} h  |^2 \;.\\
  \eeaa
 Based on the Hardy-type inequality that we showed in Lemma \ref{Hardytypeinequalitywithintegralonthewholespaceslice}, we get that if $ \Lie_{Z^K} h$ is such that for all time $t$\,, for $n \geq 4$\,,
\bea
 \int_{\SSS^{n-1}} \lim_{r \to \infty} \Big( \frac{r^{n-1}}{ (1+|q|) } w(q) \cdot  | \Lie_{Z^K} h  |^2  \Big)  d\si^{n-1} (t ) &=& 0 \; ,
\eea
then, for $\ga \neq 0$\,, 
  \beaa
\notag
  \int_{\Si_{t}} \frac{1}{(1+|q|)^2} \cdot |\Lie_{Z^K} h |^2 \cdot   w       &\leq& c(\ga) \cdot  \int_{\Si_{t}}   \cdot | \derm \Lie_{Z^K} h |^2  \cdot   w    \; .
 \eeaa
 
Hence, for $ \Lie_{Z^K} h $ decaying fast enough, we have
\beaa
 && \sum_{|J|\leq |I|}   \int_{\Si_{t}}   | \derm ( \Lie_{Z^{J}} H ) |^2  \cdot   w   \\
\notag
&\les & \sum_{|J| \leq |I|  }   \int_{\Si_{t}} \Big( 1+ C ( \lfloor \frac{|I|}{2} \rfloor  ) \cdot E ( \lfloor \frac{|I|}{2} \rfloor  +  \lfloor  \frac{n}{2} \rfloor  +1) \cdot \eps \Big)^2 \cdot | \derm ( \Lie_{Z^J}  h ) |^2 \cdot   w \\
\notag
&& +    \sum_{  |J|\leq |I|}    \int_{\Si_{t}}  C (  \lfloor \frac{|I|}{2} \rfloor   ) \cdot E ( \lfloor \frac{|I|}{2} \rfloor   +  \lfloor  \frac{n}{2} \rfloor  +1)   \cdot \frac{\eps  }{(1+|q|)^{ 2}} \cdot  | \Lie_{Z^J} h  |^2  \cdot   w \\
\notag
&\les & \sum_{|J| \leq |I|  }   \int_{\Si_{t}}  \Big( 1+ C ( \lfloor \frac{|I|}{2} \rfloor  ) \cdot E ( \lfloor \frac{|I|}{2} \rfloor  +  \lfloor  \frac{n}{2} \rfloor  +1) \cdot \eps \Big)^2 \cdot | \derm ( \Lie_{Z^J}  h ) |^2 \cdot   w \\
\notag
&& +     \sum_{  |J|\leq |I|}    \int_{\Si_{t}}  c(\ga) \cdot  C (  \lfloor \frac{|I|}{2} \rfloor   ) \cdot E ( \lfloor \frac{|I|}{2} \rfloor   +  \lfloor  \frac{n}{2} \rfloor  +1)   \cdot \eps  \cdot  | \derm (\Lie_{Z^J} h )  |^2  \cdot   w \; .\\
  \eeaa
Finally, we obtain for small $E ( \lfloor \frac{|I|}{2} \rfloor  +  \lfloor  \frac{n}{2} \rfloor  +1) \cdot \eps \leq 1 $, and for $\Phi$ and $\Lie_{Z^J} h$, $|J|\leq |I|$, decaying fast enough at spatial infinity, the following energy estimate
    \bea
\notag
 &&\int_{\Si_{t}}  (1+t )^{1+\la}  \cdot  | g^{\a\b} \derm_{\a}   \derm_{\b}   ( \Lie_{ Z^I}   \Phi ) |^2  \cdot w  \\
 &\les& C ( |I| ) \cdot E ( \lfloor \frac{|I|}{2} \rfloor+  \lfloor  \frac{n}{2} \rfloor  +1)  \cdot  c(\ga)  \cdot   \frac{\eps}{(1+t)^{2-\la}}  \cdot  \sum_{|J|\leq |I|}   \int_{\Si_{t}}  | \derm ( \Lie_{Z^{J}} H ) |^2  \cdot   w    \\
\notag
&& +C ( |I| ) \cdot E ( \lfloor \frac{|I|}{2} \rfloor+  \lfloor  \frac{n}{2} \rfloor  +1)    \cdot   \frac{\eps}{(1+t)^{2-\la}} \cdot  \sum_{|J|\leq |I|}    \int_{\Si_{t}}  |\derm ( \Lie_{Z^{K}} \Phi )  |^2    \cdot   w    \\
\notag
&& +\int_{\Si_{t}}  (1+t )^{1+\la}  \cdot \sum_{|K| \leq |I| }  | \Lie_{Z^K}  g^{\a\b} \derm_{\a}   \derm_{\b}     \Phi |^2   \cdot w \\
\notag
  &\les&    C ( |I| ) \cdot  c(\ga) \cdot E ( \lfloor \frac{|I|}{2} \rfloor+  \lfloor  \frac{n}{2} \rfloor  +1)    \cdot  \frac{\eps}{(1+t)^{2-\la}}     \cdot \Big(  \sum_{|J|\leq |I|}   \int_{\Si_{t}} \big(  | \derm ( \Lie_{Z^{J}} h ) |^2 + |\derm ( \Lie_{Z^{K}} \Phi )  |^2 \big) \cdot   w  \Big)   \\
      \notag
&& + \int_{\Si_{t}}  (1+t )^{1+\la}  \cdot \sum_{|K| \leq |I| }  | \Lie_{Z^K}  g^{\a\b} \derm_{\a}   \derm_{\b}     \Phi |^2   \cdot w       \; . \\
\eea
 \end{proof}

\begin{remark}
It is straightforward to see that if we restrict ourselves to the case $n\geq 5$\,, excluding $n=4$\,, this would relax slightly the decay assumption on spatial infinity for $H$\,, and we obtain the following lemma.
\begin{lemma}\label{estimatethespace-timeintegralofthecommutatortermneededfortheenergyestimatengeq5}
For $n\geq 5$\,, let $H$ such that for all time $t$\,, for $\ga \neq 0$\,, and $ 0 < \la \leq \frac{1}{2}$\,, 
\bea
 \int_{\SSS^{n-1}} \lim_{r \to \infty} \Big( \frac{r^{n-1}}{(1+t+r)^{3-\la} \cdot (1+|q|) } w(q) \cdot |H|^2  \Big)  d\si^{n-1} (t ) &=& 0 \; ,
\eea
and let $h$ such that for all time $t$\,, for all $|K| \leq |I|$\,,
\bea
 \int_{\SSS^{n-1}} \lim_{r \to \infty} \Big( \frac{r^{n-1}}{ (1+|q|) } w(q) \cdot  | \Lie_{Z^K} h  |^2  \Big)  d\si^{n-1} (t ) &=& 0 \; ,
\eea
then, for $\de = 0$\,, we have
   \bea
      \notag
  &&   \int_0^t \Big(  \int_{\Si_{t}}  (1+t )^{1+\la}  \cdot  | g^{\a\b} \derm_{\a}   \derm_{\b}   ( \Lie_{ Z^I}   \Phi ) |^2  \cdot w  \cdot dx^1 \ldots dx^n  \Big) \cdot dt  \\
 \notag
&\les&  \int_0^t   \frac{\eps}{(1+t)^{3-\la}}  \cdot  C ( |I| ) \cdot E ( \lfloor \frac{|I|}{2} \rfloor+  \lfloor  \frac{n}{2} \rfloor  +1)  \cdot  c(\ga)  \\
      \notag
&&  \times \Big(  \sum_{|J|\leq |I|}   \int_{\Si_{t}} \big(  | \derm ( \Lie_{Z^{J}} h ) |^2 + |\derm ( \Lie_{Z^{K}} \Phi )  |^2 \big) \cdot   w  \cdot dx^1 \ldots dx^n   \Big) \cdot dt   \\
      \notag
&& + \int_0^t  \Big( \int_{\Si_{t}}  (1+t )^{1+\la}  \cdot \sum_{|K| \leq |I| }  | \Lie_{Z^K}  g^{\a\b} \derm_{\a}   \derm_{\b}     \Phi |^2   \cdot w   \cdot dx^1 \ldots dx^n  \Big) \cdot dt  \; . \\
\eea
\end{lemma}

 \end{remark}

      \section{The energy estimate for $n\geq 4$}

      \begin{corollary}
 For $n\geq 4$\;,   $\de = 0$\;, we have for all  $|I|$\;,
              \beaa
 \notag
 |\derm ( \Lie_{Z^I} H ) (t,x)  |    &\leq&       C ( |I| ) \cdot E (|I| +  \lfloor  \frac{n}{2} \rfloor  +1)   \cdot \frac{\eps  }{(1+t+|q|)^{\frac{3}{2}}(1+|q|)^{\frac{1}{2} }} \; ,\\
 \notag
|  \Lie_{Z^I} H (t,x)  |      &\leq&    C ( |I| ) \cdot E (|I| +  \lfloor  \frac{n}{2} \rfloor  +1) \cdot  \frac{\eps}{ (1+ t + | q | )^{ \frac{3}{2}} } \cdot (1+| q |   )^{\frac{1}{2} }  \; .
    \eeaa

\end{corollary}

\begin{proof}
    
      We showed in Lemma \ref{EstimateonLiederivativeZofthemetricbigH}, that for all $|I|$\,, $\delta  \leq \frac{(n-2)}{2}$\,, \,$\eps \leq 1$\,,

              \beaa
 \notag
 |\derm ( \Lie_{Z^I} H ) (t,x)  |    &\leq& \begin{cases} C ( |I| ) \cdot E (|I| +  \lfloor  \frac{n}{2} \rfloor  +1)  \cdot \frac{\eps }{(1+t+|q|)^{\frac{(n-1)}{2}-\delta} (1+|q|)^{1+\ga}},\quad\text{when }\quad q>0,\\
       C ( |I| ) \cdot E (|I| +  \lfloor  \frac{n}{2} \rfloor  +1)  \cdot \frac{\eps  }{(1+t+|q|)^{\frac{(n-1)}{2}-\delta}(1+|q|)^{\frac{1}{2} }}  \,,\quad\text{when }\quad q<0 , \end{cases} \\
      \eeaa
      and
       \beaa
 \notag
|  \Lie_{Z^I} H (t,x)  | &\leq& \begin{cases} c (\delta) \cdot c (\gamma) \cdot C ( |I| ) \cdot E ( |I| +  \lfloor  \frac{n}{2} \rfloor  +1) \cdot  \frac{\eps}{ (1+ t + | q | )^{\frac{(n-1)}{2}-\delta }  (1+| q |   )^{\ga}}  ,\quad\text{when }\quad q>0,\\
\notag
    C ( |I| ) \cdot E (|I| +  \lfloor  \frac{n}{2} \rfloor  +1) \cdot  \frac{\eps}{ (1+ t + | q | )^{\frac{(n-1)}{2}-\delta }  } (1+| q |   )^{\frac{1}{2} }  , \,\quad\text{when }\quad q<0 . \end{cases} \\ 
    \eeaa
    
           Taking $\delta = 0$\,, we have $0  \leq \frac{(4-2)}{2} = 1  \leq  \frac{(n-2)}{2}  $\,, for $n \geq 4$\,. Assume also that the energy is small such that $\eps \leq 1$\,. We get that for $\ga \geq - \frac{1}{2} $\,,

              \beaa
 \notag
 |\derm ( \Lie_{Z^I} H ) (t,x)  |    &\leq& \begin{cases} C ( |I| ) \cdot E (|I| +  \lfloor  \frac{n}{2} \rfloor  +1)  \cdot \frac{\eps }{(1+t+|q|)^{\frac{3}{2}} (1+|q|)^{1+\ga}},\quad\text{when }\quad q>0,\\
       C ( |I| ) \cdot E (|I| +  \lfloor  \frac{n}{2} \rfloor  +1)  \cdot \frac{\eps  }{(1+t+|q|)^{\frac{3}{2}}(1+|q|)^{\frac{1}{2} }}  \,,\quad\text{when }\quad q<0 , \end{cases} \\
       &\les&       C ( |I| ) \cdot E (|I| +  \lfloor  \frac{n}{2} \rfloor  +1)  \cdot \frac{\eps  }{(1+t+|q|)^{\frac{3}{2}}(1+|q|)^{\frac{1}{2} }}  \; ,
      \eeaa
      and for $\ga \geq 0$, we have
       \beaa
 \notag
|  \Lie_{Z^I} H (t,x)  | &\leq& \begin{cases} c (\delta) \cdot c (\gamma) \cdot C ( |I| ) \cdot E ( |I| +  \lfloor  \frac{n}{2} \rfloor  +1) \cdot  \frac{\eps}{ (1+ t + | q | )^{\frac{3}{2} }  (1+| q |   )^{\ga}}  ,\quad\text{when }\quad q>0,\\
\notag
    C ( |I| ) \cdot E (|I| +  \lfloor  \frac{n}{2} \rfloor  +1) \cdot  \frac{\eps}{ (1+ t + | q | )^{\frac{3}{2} }  } (1+| q |   )^{\frac{1}{2} }  , \,\quad\text{when }\quad q<0 . \end{cases} \\ 
       &\les&     c (\delta) \cdot c (\gamma) \cdot C ( |I| ) \cdot E (|I| +  \lfloor  \frac{n}{2} \rfloor  +1) \cdot  \frac{\eps }{ (1+ t + | q | )^{ \frac{3}{2} } } \cdot (1+| q |   )^{\frac{1}{2} }   \; .
    \eeaa
    
    \end{proof}

\begin{lemma}\label{derivativeoftildwandrelationtotildew}
Let $w$ as in Definition \ref{defw}, where $\ga >0$.
Then, since $\ga > -\frac{1}{2}$\,, we have for all $q$,
\beaa
w^{\prime}(q) \geq 0 \; ,
\eeaa
and
\beaa
w^{\prime}(q) \leq \frac{w(q)}{(1+|q|)} \; .
\eeaa

\end{lemma}
\begin{proof}

We have from Definition \ref{defw}, that $\ga > 0$ and that
\beaa
w(q) &=&\begin{cases} (1+ q)^{1+2\gamma} \quad\text{when }\quad q>0 , \\
       1  \,\quad\text{when }\quad q<0 . \end{cases} 
\eeaa
We compute,
\beaa
 w^{\prime}(q)&=&\begin{cases} (1+2\gamma)(1+|q|)^{2\gamma} \quad\text{when }\quad q>0 , \\
         0 \,\quad\text{when }\quad q<0 . \end{cases} \\
          &=&\begin{cases}         (1+2\gamma)   \frac{w(q)}{(1+|q|)}   \quad\text{when }\quad q>0 , \\
         0    \,\quad\text{when }\quad q<0 . \end{cases} \\
\eeaa
Hence, since $\ga > -\frac{1}{2}$, we have $w^{\prime}(q) \geq 0$. Also, for $q > 0$, we have
\beaa
w^{\prime}(q) \sim \frac{w(q)}{(1+|q|)} \; .
\eeaa
For $q < 0$, we have
\beaa
w^{\prime}(q) \leq \frac{1}{(1+|q|)} = \frac{w(q)}{(1+|q|)}  \; .
\eeaa
Therefore, for all $q$, 
\beaa
w^{\prime}(q) \leq  \frac{w(q)}{(1+|q|)}  \; .
\eeaa
\end{proof}

\begin{lemma}\label{energyestimateforngeq4forphidecayingfastenough}

Let $n\geq 4$. Assume that $ H^{\mu\nu} = g^{\mu\nu}-m^{\mu\nu}$ satisfies 
\bea
| H| \leq C < \frac{1}{n} \; , \quad \text{where $n$ is the space dimension},
\eea
and assume that in wave coordinates $\{t, x^1, \ldots, x^n \}$, we have for $j$ running over spatial indices $\{ x^1, \ldots, x^n \}$, for all time $t$\,,
\bea
\lim_{r \to \infty }   \int_{\SSS^{n} }  g^{rj}  \cdot < \pa_t  \Phi ,\pa_j \Phi > \cdot w \cdot r^{n-1} d\si^{n-1} =0 \\
\lim_{r \to \infty }   \int_{\SSS^{n} }  g^{tr}  \cdot < \pa_t  \Phi ,\pa_t \Phi > \cdot w \cdot r^{n-1} d\si^{n-1} =0 \;.
\eea

Then, for  $ 0 < \la \leq \frac{1}{2}$, for $\ga > - \frac{1}{2} $ and $\mu \neq 0$, we have

\bea
\notag
 && \int_{\Si_{t}}  | \derm \Phi|^2 \cdot w \cdot dx^1 \ldots dx^n \\
 \notag
  &\les&  \int_{\Si_{t=0}}  | \derm \Phi|^2  \cdot w \cdot dx^1 \ldots dx^n \\
  \notag
  &&  + \int_0^t   \frac{E (  \lfloor  \frac{n}{2} \rfloor  +1) }{(1+t )^{1+\la}}  \cdot  \Big( \int_{\Si_{t}}   | \derm \Phi |^2  \cdot w \cdot dx^1 \ldots dx^n  \Big) \cdot dt \\
   \notag
  && +  \int_0^t \Big(  \int_{\Si_{t}}  (1+t )^{1+\la}  \cdot  | g^{\a\b} \derm_{\a}   \derm_{\b}   \Phi |^2  \cdot w  \cdot dx^1 \ldots dx^n  \Big) \cdot dt \; . \\
\eea

\end{lemma}

    \begin{proof}

Let $ \Phi_{\mu\nu}$ be a tensor solution of the following tensorial wave equation
\bea
 g^{\a\b} \derm_{\a}   \derm_{\b}   \Phi_{\mu\nu}= S_{\mu\nu} \, , 
\eea
where $S_{\mu\nu}$ is the source term, with a sufficiently smooth metric $g$\;, with $\Phi$ satisfying the assumptions of the lemma. Then, based on Lemma \ref{Theenerhyestimatefornequalfive}, we have the following

\bea
\notag
 && \int_{\Si_{t}}  | \derm \Phi|^2 \cdot w \\
 \notag
  &\leq&  \int_{\Si_{t=0}}  | \derm \Phi|^2  \cdot w + \int_0^t \int_{\Si_{t}} - 2 < \derm_t \Phi , S  >  \cdot w \\
\notag
&&+  \int_0^t \int_{\Si_{t}}   | \derm H | \cdot |\derm \Phi |^2  \cdot  w  + \int_0^t  \int_{\Si_{t}}  | H |  \cdot |\derm \Phi |^2 \cdot | w^{\prime} (q) |  \\
\notag
&&-  \int_0^t   \int_{\Si_{t}} \Big(   < \derm_t  \Phi + \derm_r \Phi, \derm_t  \Phi + \derm_r \Phi >    \\
\notag
&& +  \de^{ij}  < ( \derm_i - \frac{x_i}{r} \derm_{r}  )\Phi , (\derm_j - \frac{x_j}{r} \derm_{r} ) \Phi >  \Big) \cdot w^{\prime} (q)   \; ,
\eea

where the integration on $\Sigma_t$ is taken with respect to the measure $dx^1 \ldots dx^n$\;, and the integration in $t$ is taken with respect to the measure $dt$\;.\\

However, we showed in Lemma \ref{derivativeoftildwandrelationtotildew}, that for $\ga \neq - \frac{1}{2} $, we have $w^{\prime} (q) \geq 0$, thus we can ignore on the right hand side of the inequality the negative term
\beaa
&&-  \int_0^t   \int_{\Si_{t}} \Big(   < \derm_t  \Phi + \derm_r \Phi, \derm_t  \Phi + \derm_r \Phi >    \\
\notag
&& +  \de^{ij}  < ( \derm_i - \frac{x_i}{r} \derm_{r}  )\Phi , (\derm_j - \frac{x_j}{r} \derm_{r} ) \Phi >  \Big) \cdot w^{\prime} (q)   \; .
\eeaa

Hence, we get for $n \geq 4$\,,
\bea
\notag
 && \int_{\Si_{t}}  | \derm \Phi|^2 \cdot w \\
 \notag
  &\leq&  \int_{\Si_{t=0}}  | \derm \Phi|^2  \cdot w + \int_0^t \int_{\Si_{t}}  2 | \derm_t \Phi | \cdot | g^{\a\b} \derm_{\a}   \derm_{\b}   \Phi |  \cdot w \\
\notag
&&+  \int_0^t \int_{\Si_{t}}   E ( \lfloor  \frac{n}{2} \rfloor  +1)  \cdot \frac{\eps  }{(1+t+|q|)^{\frac{3}{2}}(1+|q|)^{\frac{1}{2} }} \cdot |\derm \Phi |^2  \cdot  w  \\
\notag
&& + \int_0^t  \int_{\Si_{t}}    E (  \lfloor  \frac{n}{2} \rfloor  +1) \cdot  \frac{\eps}{ (1+ t + | q | )^{\frac{3}{2}} } \cdot (1+| q |   )^{\frac{1}{2} }    \cdot |\derm \Phi |^2 \cdot | w^{\prime} (q) |  \\
\notag
&&-  \int_0^t   \int_{\Si_{t}} \Big(   | \derm_t  \Phi + \derm_r \Phi |^2     +  \sum_{i=1}^{n}  | \derm_i - \frac{x_i}{r} \derm_{r}  )\Phi |^2  \Big) \cdot w^{\prime} (q)   \\
\notag
  &\les&  \int_{\Si_{t=0}}  | \derm \Phi|^2  \cdot w + \int_0^t \int_{\Si_{t}}   \frac{1}{(1+t )^{1+\la}}  \cdot  | \derm_t \Phi |^2 \cdot  w\\
  \notag
  &&  +  \int_0^t \int_{\Si_{t}}  (1+t )^{1+\la}  \cdot  | g^{\a\b} \derm_{\a}   \derm_{\b}   \Phi |^2  \cdot w \\
\notag
&&+  \int_0^t \int_{\Si_{t}}     E (  \lfloor  \frac{n}{2} \rfloor  +1)  \cdot \frac{\eps  }{(1+t+|q|)^{\frac{3}{2}}(1+|q|)^{\frac{1}{2} }} \cdot |\derm \Phi |^2  \cdot  w  \\
\notag
&& + \int_0^t  \int_{\Si_{t}}   E (  \lfloor  \frac{n}{2} \rfloor  +1) \cdot  \frac{\eps}{ (1+ t + | q | )^{\frac{3}{2} } } \cdot (1+| q |   )^{\frac{1}{2} }    \cdot |\derm \Phi |^2 \cdot | w^{\prime} (q) |  \\
\notag
&& \text{(using $a\cdot b \les a^2 + b^2$)}  .
\eea
We also showed in Lemma \ref{derivativeoftildwandrelationtotildew}, that for $\ga > - \frac{1}{2} $\,, we have 
\beaa
w^{\prime}(q) \leq \frac{w(q)}{(1+|q|)} \; .
\eeaa

Thus,
\bea
\notag
 && \int_{\Si_{t}}  | \derm \Phi|^2 \cdot w \\
 \notag
  &\les&  \int_{\Si_{t=0}}  | \derm \Phi|^2  \cdot w + \int_0^t \int_{\Si_{t}}   \frac{1}{(1+t )^{1+\la}}  \cdot  | \derm \Phi |^2 \cdot  w +  \int_0^t \int_{\Si_{t}} (1+t )^{1+\la}  \cdot  | g^{\a\b} \derm_{\a}   \derm_{\b}   \Phi |^2  \cdot w \\
\notag
&&+  \int_0^t \int_{\Si_{t}}    C ( |I| ) \cdot E (  \lfloor  \frac{n}{2} \rfloor  +1)  \cdot \frac{\eps  }{(1+t+|q|)^{\frac{3}{2}}(1+|q|)^{\frac{1}{2} }} \cdot |\derm \Phi |^2  \cdot  w  \\
\notag
&& + \int_0^t  \int_{\Si_{t}}    C ( |I| ) \cdot E (  \lfloor  \frac{n}{2} \rfloor  +1) \cdot  \frac{\eps}{ (1+ t + | q | )^{\frac{3}{2} } (1+|q|)^{\frac{1}{2} } }   \cdot |\derm \Phi |^2 \cdot w  \; .
\eea
Choosing $ 0 < \la \leq \frac{1}{2}$, we obtain
\bea
\notag
 && \int_{\Si_{t}}  | \derm \Phi|^2 \cdot w \\
 \notag
  &\les&  \int_{\Si_{t=0}}  | \derm \Phi|^2  \cdot w + \int_0^t   \frac{ E (  \lfloor  \frac{n}{2} \rfloor  +1) }{(1+t )^{1+\la}}  \cdot  \int_{\Si_{t}}   | \derm \Phi |^2  \cdot w \\
   \notag
  && +  \int_0^t \int_{\Si_{t}}  (1+t )^{1+\la}  \cdot  | g^{\a\b} \derm_{\a}   \derm_{\b}   \Phi |^2  \cdot w \; .
\eea
\end{proof}

 \subsection{The main energy estimate for $A$ and $ h$ for $n\geq 4$}\

We now state the main energy estimate that we would like to apply for higher dimensions $n \geq 5$ with a bootstrap argument for $\de = 0$. However, the estimate is true for all $n\geq 4$.   
    \begin{lemma}\label{the main energy estimate forAandhforn4}

Assume that $ H^{\mu\nu} = g^{\mu\nu}-m^{\mu\nu}$ satisfies 
\bea
| H| \leq C < \frac{1}{n} \; , \quad \text{where $n$ is the space dimension},
\eea
and assume that in wave coordinates $\{t, x^1, \ldots, x^n\}$\,, we have for $j$ running over spatial indices $\{ x^1, \ldots, x^n \}$\,, for any $I$ as in Definition \ref{DefinitionofZI}, for all time $t$\,,
\bea 
\lim_{r \to \infty }   \int_{\SSS^{n} }  g^{rj}  \cdot < \derm_t   (\Lie_{Z^I} \Phi  ) ,\derm_j  (\Lie_{Z^I}  \Phi  ) > \cdot w \cdot r^{n-1} d\si^{n-1} =0 \;,\\
\lim_{r \to \infty }   \int_{\SSS^{n} }  g^{tr}  \cdot < \derm_t   (\Lie_{Z^I}  \Phi   ) ,\derm_t  (\Lie_{Z^I} \Phi  ) > \cdot w \cdot r^{n-1} d\si^{n-1} =0 \;.
\eea

Let $n\geq 4$\,, $ 0 < \la \leq \frac{1}{2}$\,, $\ga > 0$\,. Assume that $H$ is such that for all time $t$\,, 
\bea
 \int_{\SSS^{n-1}} \lim_{r \to \infty} \Big( \frac{r^{n-1}}{(1+t+r)^{2-\la} \cdot (1+|q|) } w(q) \cdot |H|^2  \Big)  d\si^{n-1} (t ) &=& 0 \; ,
\eea
and that $h$ is such that for all $|J| \leq |I|$\,, we have
\bea
 \int_{\SSS^{n-1}} \lim_{r \to \infty} \Big( \frac{r^{n-1}}{ (1+|q|) } w(q) \cdot  | \Lie_{Z^J} h  |^2  \Big)  d\si^{n-1} (t ) &=& 0 \; .
\eea
Then, for either $\Phi = H$ or $\Phi = A$\,, using the bootstrap assumption on $\Phi$\,, with $\de = 0$\,, and for $E ( \lfloor \frac{|I|}{2} \rfloor  +  \lfloor  \frac{n}{2} \rfloor  +1) \leq 1 $\,, the following energy estimate holds
\bea
\notag
 && \int_{\Si_{t}}  | \derm  (\Lie_{Z^I}  \Phi ) |^2 \cdot w \cdot dx^1 \ldots dx^n \\
 \notag
  &\les&  \int_{\Si_{t=0}}  | \derm  (\Lie_{Z^I} \Phi   ) |^2  \cdot w \cdot dx^1 \ldots dx^n \\
  \notag
&& +  C ( |I| ) \cdot  c(\ga) \cdot E ( \lfloor \frac{|I|}{2} \rfloor+  \lfloor  \frac{n}{2} \rfloor  +1)   \\
\notag
&&  \times \int_0^t   \frac{1}{(1+t)^{1+\la}}     \cdot \Big(  \sum_{|J|\leq |I|}   \int_{\Si_{t}} \big(  | \derm ( \Lie_{Z^{J}} h ) |^2 + |\derm ( \Lie_{Z^{K}} \Phi )  |^2 \big) \cdot   w  \cdot dx^1 \ldots dx^n   \Big) \cdot dt   \\
      \notag
&& + \int_0^t  \Big( \int_{\Si_{t}}  (1+t )^{1+\la}  \cdot \sum_{|K| \leq |I| }  | \Lie_{Z^K}  g^{\a\b} \derm_{\a}   \derm_{\b}     \Phi |^2   \cdot w   \cdot dx^1 \ldots dx^n  \Big) \cdot dt  \; . \\
\eea

\end{lemma}

\begin{proof}
We showed in Lemma \ref{energyestimateforngeq4forphidecayingfastenough} and in Lemma \ref{estimatethespace-timeintegralofthecommutatortermneededfortheenergyestimate}, that under these assumptions on the metric and on the spatial asymptotic behaviour of the field, we get
\bea
\notag
 && \int_{\Si_{t}}  | \derm  (\Lie_{Z^I}  \Phi ) |^2 \cdot w \cdot dx^1 \ldots dx^n \\
 \notag
  &\les&  \int_{\Si_{t=0}}  | \derm  (\Lie_{Z^I} \Phi   ) |^2  \cdot w \cdot dx^1 \ldots dx^n \\
  \notag
  &&  + \int_0^t   \frac{  E (  \lfloor  \frac{n}{2} \rfloor  +1)  }{(1+t )^{1+\la}}  \cdot  \Big( \int_{\Si_{t}}   | \derm  (\Lie_{Z^I}  \Phi  ) |^2  \cdot w \cdot dx^1 \ldots dx^n  \Big) \cdot dt \\
   \notag
&& +  C ( |I| ) \cdot E ( \lfloor \frac{|I|}{2} \rfloor+  \lfloor  \frac{n}{2} \rfloor  +1)  \cdot  c(\ga)  \\
\notag
&&  \times \int_0^t   \frac{\eps}{(1+t)^{2-\la}}     \cdot \Big(  \sum_{|J|\leq |I|}   \int_{\Si_{t}} \big(  | \derm ( \Lie_{Z^{J}} h ) |^2 + |\derm ( \Lie_{Z^{K}} \Phi )  |^2 \big) \cdot   w  \cdot dx^1 \ldots dx^n   \Big) \cdot dt   \\
      \notag
&& + \int_0^t  \Big( \int_{\Si_{t}}  (1+t )^{1+\la}  \cdot \sum_{|K| \leq |I| }  | \Lie_{Z^K}  g^{\a\b} \derm_{\a}   \derm_{\b}     \Phi |^2   \cdot w   \cdot dx^1 \ldots dx^n  \Big) \cdot dt    \; .\\
\eea

Finally, we obtain for small $E ( \lfloor \frac{|I|}{2} \rfloor  +  \lfloor  \frac{n}{2} \rfloor  +1) \leq 1 $, and for $\Phi$ and $\Lie_{Z^J} h$, $|J|\leq |I|$, decaying fast enough at spatial infinity, the following energy estimate

\bea
\notag
 && \int_{\Si_{t}}  | \derm  (\Lie_{Z^I}  \Phi ) |^2 \cdot w \cdot dx^1 \ldots dx^n \\
 \notag
  &\les&  \int_{\Si_{t=0}}  | \derm  (\Lie_{Z^I} \Phi   ) |^2  \cdot w \cdot dx^1 \ldots dx^n \\
  \notag
&& +  C ( |I| ) \cdot  c(\ga) \cdot E ( \lfloor \frac{|I|}{2} \rfloor+  \lfloor  \frac{n}{2} \rfloor  +1)   \\
\notag
&&  \times  \int_0^t   \frac{1}{(1+t)^{1+\la}}     \cdot \Big(  \sum_{|J|\leq |I|}   \int_{\Si_{t}} \big(  | \derm ( \Lie_{Z^{J}} h ) |^2 + |\derm ( \Lie_{Z^{K}} \Phi )  |^2 \big) \cdot   w  \cdot dx^1 \ldots dx^n   \Big) \cdot dt   \\
      \notag
&& + \int_0^t  \Big( \int_{\Si_{t}}  \frac{(1+t )^{1+\la}}{\eps}  \cdot \sum_{|K| \leq |I| }  | \Lie_{Z^K}  g^{\a\b} \derm_{\a}   \derm_{\b}     \Phi |^2   \cdot w   \cdot dx^1 \ldots dx^n  \Big) \cdot dt  \; . \\
\eea

\end{proof}

\begin{remark}
It is straightforward to see that if we restrict the lemma to $n\geq5$\,, then the decay assumption on $H$ could be relaxed to become that for $ 0 < \la \leq \frac{1}{2}$\,, $\ga > 0$\,,
\bea
 \int_{\SSS^{n-1}} \lim_{r \to \infty} \Big( \frac{r^{n-1}}{(1+t+r)^{3-\la} \cdot (1+|q|) } w(q) \cdot |H|^2  \Big)  d\si^{n-1} (t ) &=& 0 \; .
\eea

\end{remark}

\section{The proof of global stability for $n\geq 5$}

Now, we fix $n\geq 5$ and $\de = 0$\,.

\subsection{Using the Hardy type inequality for the space-time integrals of the source terms for $n \geq 5$}\
      
\begin{lemma}\label{squareofthesourcetermsforAforngeq5withdeltaequalzeroforenergyestimate}
For $n \geq 5$\,, $\de = 0$\,, we have

    \beaa
 \notag
&& (1+t )^{1+\la} \cdot |  \Lie_{Z^I}   ( g^{\la\mu} \derm_{\la}   \derm_{\mu}   A  )   |^2 \, \\
   \notag
   &\les&  \sum_{|K|\leq |I|}  \Big( |    \derm (\Lie_{Z^K} A ) |^2 +  |    \derm (\Lie_{Z^K} h ) |^2 \Big) \cdot E (   \lfloor \frac{|I|}{2} \rfloor + \lfloor  \frac{n}{2} \rfloor  + 1) \cdot  \frac{\eps  }{(1+t+|q|)^{2-\la}   }  \\
      \notag 
&& + \sum_{|K|\leq |I|}  \Big( |   \Lie_{Z^K} A  |^2 +|   \Lie_{Z^K} h  |^2 \Big) \cdot E (   \lfloor \frac{|I|}{2} \rfloor + \lfloor  \frac{n}{2} \rfloor  + 1) \cdot       \frac{\eps  }{(1+t+|q|)^{3-\la} \cdot ( 1+|q| )^{} }  \; .
      \notag 
   \eeaa

\end{lemma}

\begin{proof}

We showed in Lemma \ref{ttimesthesquareofthesourcetermsforAforngeq5}, that for $n \geq 5$, we have
 \beaa
 \notag
&& (1+t ) \cdot |  \Lie_{Z^I}   ( g^{\la\mu} \derm_{\la}   \derm_{\mu}   A  )   |^2 \, \\
   \notag
   &\les&  \sum_{|K|\leq |I|}  \Big( |    \derm (\Lie_{Z^K} A ) |^2 +  |    \derm (\Lie_{Z^K} h ) |^2 \Big) \cdot E (   \lfloor \frac{|I|}{2} \rfloor + \lfloor  \frac{n}{2} \rfloor  + 1)\\
 \notag
   && \cdot \Big(  \begin{cases}  \frac{\eps }{(1+t+|q|)^{3-2\delta} (1+|q|)^{2\gamma}},\quad\text{when }\quad q>0,\\
           \notag
      \frac{\eps  }{(1+t+|q|)^{3-2\delta} \cdot ( 1+|q| )^{-1} } \,\quad\text{when }\quad q<0 . \end{cases} \Big) \\
      \notag 
&& + \sum_{|K|\leq |I|}  \Big( |   \Lie_{Z^K} A  |^2 +|   \Lie_{Z^K} h  |^2 \Big) \cdot E (   \lfloor \frac{|I|}{2} \rfloor + \lfloor  \frac{n}{2} \rfloor  + 1) \\
 \notag
   && \cdot \Big(  \begin{cases}  \frac{\eps }{(1+t+|q|)^{3-2\delta} (1+|q|)^{2+2\gamma}},\quad\text{when }\quad q>0,\\
           \notag
      \frac{\eps  }{(1+t+|q|)^{3-2\delta} \cdot ( 1+|q| )^{} } \,\quad\text{when }\quad q<0 . \end{cases} \Big) \; .
      \notag 
   \eeaa
   Taking $\de = 0$, we obtain
       \beaa
 \notag
&& (1+t ) \cdot |  \Lie_{Z^I}   ( g^{\la\mu} \derm_{\la}   \derm_{\mu}   A  )   |^2 \, \\
   \notag
   &\les&  \sum_{|K|\leq |I|}  \Big( |    \derm (\Lie_{Z^K} A ) |^2 +  |    \derm (\Lie_{Z^K} h ) |^2 \Big) \cdot E (   \lfloor \frac{|I|}{2} \rfloor + \lfloor  \frac{n}{2} \rfloor  + 1) \cdot  \frac{\eps  }{(1+t+|q|)^{3} \cdot ( 1+|q| )^{-1} }  \\
      \notag 
&& + \sum_{|K|\leq |I|}  \Big( |   \Lie_{Z^K} A  |^2 +|   \Lie_{Z^K} h  |^2 \Big) \cdot E (   \lfloor \frac{|I|}{2} \rfloor + \lfloor  \frac{n}{2} \rfloor  + 1) \cdot       \frac{\eps  }{(1+t+|q|)^{3} \cdot ( 1+|q| )^{} }  \; .
      \notag 
   \eeaa
   \end{proof}

\begin{lemma}\label{squareofthesourcetermsforthemetrichforngeq5withdeltaequalzeroforenergyestimate}
For $n \geq 5$ and $\de = 0$, 
\beaa
   \notag
 &&  (1+t )^{1+\la}\cdot  | \Lie_{Z^I}   ( g^{\a\b} \derm_{\a}   \derm_{\b}    h ) |^2   \\
     \notag
   &\les&  \sum_{|K|\leq |I|}  \Big( |    \derm (\Lie_{Z^K} A ) |^2 + |    \derm (\Lie_{Z^K} h ) |^2 \Big) \cdot  E (   \lfloor \frac{|I|}{2} \rfloor + \lfloor  \frac{n}{2} \rfloor  + 1) \cdot       \frac{\eps  }{(1+t+|q|)^{3 -\la} \cdot ( 1+|q| )^{} }  \\
      \notag 
   &&+  \sum_{|K|\leq |I|}  \Big( |   \Lie_{Z^K} A  |^2 + |   \Lie_{Z^K} h  |^2 \Big) \cdot  E (   \lfloor \frac{|I|}{2} \rfloor + \lfloor  \frac{n}{2} \rfloor  + 1 ) \cdot    \frac{\eps  }{(1+t+|q|)^{7-\la}  }  \; .
      \notag 
   \eeaa

\end{lemma}

\begin{proof}

We showed in Lemma \ref{ttimesthesquareofthesourcetermsforthemetrichforngeq5}, that for $n \geq 5$, 
\beaa
   \notag
 &&  (1+t ) \cdot  | \Lie_{Z^I}   ( g^{\a\b} \derm_{\a}   \derm_{\b}    h ) |^2   \\
     \notag
   &\les&  \sum_{|K|\leq |I|}  \Big( |    \derm (\Lie_{Z^K} A ) |^2 + |    \derm (\Lie_{Z^K} h ) |^2 \Big) \cdot E (   \lfloor \frac{|I|}{2} \rfloor + \lfloor  \frac{n}{2} \rfloor  + 1)\\
 \notag
   && \cdot \Big(  \begin{cases}  \frac{\eps }{(1+t+|q|)^{3-2\delta} (1+|q|)^{2+2\gamma}},\quad\text{when }\quad q>0,\\
           \notag
      \frac{\eps  }{(1+t+|q|)^{3-2\delta} \cdot ( 1+|q| )^{} } \,\quad\text{when }\quad q<0 . \end{cases} \Big) \\
      \notag 
   &&+  \sum_{|K|\leq |I|}  \Big( |   \Lie_{Z^K} A  |^2 + |   \Lie_{Z^K} h  |^2 \Big) \cdot E (   \lfloor \frac{|I|}{2} \rfloor + \lfloor  \frac{n}{2} \rfloor  + 1)\\
 \notag
   && \cdot \Big(  \begin{cases}  \frac{\eps }{(1+t+|q|)^{7-4\delta} (1+|q|)^{2+4\gamma}},\quad\text{when }\quad q>0,\\
           \notag
      \frac{\eps  }{(1+t+|q|)^{7-4\delta}  } \,\quad\text{when }\quad q<0 . \end{cases} \Big) \; .
      \notag 
   \eeaa
Taking $\de = 0$, we obtain
\beaa
   \notag
 &&  (1+t ) \cdot  | \Lie_{Z^I}   ( g^{\a\b} \derm_{\a}   \derm_{\b}    h ) |^2   \\
     \notag
   &\les&  \sum_{|K|\leq |I|}  \Big( |    \derm (\Lie_{Z^K} A ) |^2 + |    \derm (\Lie_{Z^K} h ) |^2 \Big) \cdot  E (   \lfloor \frac{|I|}{2} \rfloor + \lfloor  \frac{n}{2} \rfloor  + 1) \cdot       \frac{\eps  }{(1+t+|q|)^{3} \cdot ( 1+|q| )^{} }  \\
      \notag 
   &&+  \sum_{|K|\leq |I|}  \Big( |   \Lie_{Z^K} A  |^2 + |   \Lie_{Z^K} h  |^2 \Big) \cdot  E (   \lfloor \frac{|I|}{2} \rfloor + \lfloor  \frac{n}{2} \rfloor  + 1 ) \cdot    \frac{\eps  }{(1+t+|q|)^{7}  }  \; .
      \notag 
   \eeaa
   
      \end{proof}

\begin{lemma}\label{estimateonthesoacetimeintegralforLiederivativesZofthesourcetermsAusefulforagronwallinequalityinenergyestimate}
For $n \geq 5$, $\de = 0$, and for fields decaying fast enough at spatial infinity, such that for all time $t$, for all $|K| \leq |I| $,
\bea
\notag
 \int_{\SSS^{n-1}} \lim_{r \to \infty} \Big( \frac{r^{n-1}}{(1+t+r)^{2-\la} \cdot (1+|q|) } w(q) \cdot   \Big( |   \Lie_{Z^K} A  |^2 + |   \Lie_{Z^K} h  |^2  \Big)  d\si^{n-1} (t ) &=& 0  \; , \\
\eea
then, for $\ga \neq 0$\,,
 \beaa
   \notag
 &&   \int_0^t \Big(  \int_{\Si_{t}}  (1+t )^{1+\la}  \cdot  | \Lie_{Z^I}   ( g^{\la\mu} \derm_{\la}   \derm_{\mu}    A ) |^2  \cdot   w \Big) \cdot dt   \\
     \notag
    &\les& c(\ga) \cdot E (   \lfloor \frac{|I|}{2} \rfloor + \lfloor  \frac{n}{2} \rfloor  + 1) \\
    &&  \times   \int_0^t          \frac{\eps  }{(1+t)^{2-\la} \ }  \cdot  \Big(  \int_{\Si_{t}}   \sum_{|K|\leq |I|}  \Big( |    \derm (\Lie_{Z^K} A ) |^2 + |    \derm (\Lie_{Z^K} h ) |^2 \Big)   \cdot   w \Big) \cdot dt  \;. 
   \eeaa
\end{lemma}

\begin{proof}
Based on what we have shown in Lemma \ref{squareofthesourcetermsforAforngeq5withdeltaequalzeroforenergyestimate}, for $n \geq 5$, $\de = 0$, we have
 \beaa
 \notag
&& (1+t ) \cdot |  \Lie_{Z^I}   ( g^{\la\mu} \derm_{\la}   \derm_{\mu}   A  )   |^2 \, \\
   \notag
   &\les&  \sum_{|K|\leq |I|}  \Big( |    \derm (\Lie_{Z^K} A ) |^2 +  |    \derm (\Lie_{Z^K} h ) |^2 \Big) \cdot E (   \lfloor \frac{|I|}{2} \rfloor + \lfloor  \frac{n}{2} \rfloor  + 1)\\
 \notag
   && \cdot \Big(  \begin{cases}  \frac{\eps }{(1+t+|q|)^{3} (1+|q|)^{2\gamma}},\quad\text{when }\quad q>0,\\
           \notag
      \frac{\eps  }{(1+t+|q|)^{3} \cdot ( 1+|q| )^{-1} } \,\quad\text{when }\quad q<0 . \end{cases} \Big) \\
      \notag 
   && + \sum_{|K|\leq |I|}  \Big( |   \Lie_{Z^K} A  |^2 +|   \Lie_{Z^K} h  |^2 \Big) \cdot E (   \lfloor \frac{|I|}{2} \rfloor + \lfloor  \frac{n}{2} \rfloor  + 1) \\
 \notag
   && \cdot \Big(  \begin{cases}  \frac{\eps }{(1+t+|q|)^{3} (1+|q|)^{2+2\gamma}},\quad\text{when }\quad q>0,\\
           \notag
      \frac{\eps  }{(1+t+|q|)^{3} \cdot ( 1+|q| )^{} } \,\quad\text{when }\quad q<0 . \end{cases} \Big) \\
      \notag 
   &\les&  \sum_{|K|\leq |I|}  \Big( |    \derm (\Lie_{Z^K} A ) |^2 +  |    \derm (\Lie_{Z^K} h ) |^2 \Big) \cdot E (   \lfloor \frac{|I|}{2} \rfloor + \lfloor  \frac{n}{2} \rfloor  + 1)  \cdot       \frac{\eps  }{(1+t+|q|)^{2}  }  \\
     \notag
   && + \sum_{|K|\leq |I|}  \Big( |   \Lie_{Z^K} A  |^2 +|   \Lie_{Z^K} h  |^2 \Big) \cdot E (   \lfloor \frac{|I|}{2} \rfloor + \lfloor  \frac{n}{2} \rfloor  + 1)  \cdot      \frac{\eps  }{(1+t+|q|)^{3} \cdot ( 1+|q| )^{} }   \; .
      \notag 
   \eeaa
   Hence,
       \beaa
 \notag
&& (1+t )^{1+\la}  \cdot |  \Lie_{Z^I}   ( g^{\la\mu} \derm_{\la}   \derm_{\mu}   A  )   |^2 \, \\
   \notag
   &\les&  \sum_{|K|\leq |I|}  \Big( |    \derm (\Lie_{Z^K} A ) |^2 +  |    \derm (\Lie_{Z^K} h ) |^2 \Big) \cdot E (   \lfloor \frac{|I|}{2} \rfloor + \lfloor  \frac{n}{2} \rfloor  + 1)  \cdot       \frac{\eps  }{(1+t+|q|)^{2-\la}  }  \\
     \notag
   && + \sum_{|K|\leq |I|}  \Big( |   \Lie_{Z^K} A  |^2 +|   \Lie_{Z^K} h  |^2 \Big) \cdot E (   \lfloor \frac{|I|}{2} \rfloor + \lfloor  \frac{n}{2} \rfloor  + 1)  \cdot      \frac{\eps  }{(1+t+|q|)^{2-\la} \cdot ( 1+|q| )^{2} }   \; .
      \notag 
   \eeaa
   
Under the assumption again that $\Lie_{Z^K} A$ and $\Lie_{Z^K} h$ decay fast enough at spatial infinity for all time $t$, for all for $|K| \leq |I| $, such that
\bea
\notag
 \int_{\SSS^{n-1}} \lim_{r \to \infty} \Big( \frac{r^{n-1}}{(1+t+r)^{2-\la} \cdot (1+|q|) } w(q) \cdot   \Big( |   \Lie_{Z^K} A  |^2 + |   \Lie_{Z^K} h  |^2  \Big)  d\si^{n-1} (t ) &=& 0  \; , \\
\eea
we get by then, that for $\ga \neq 0$ and $ 0 < \la \leq \frac{1}{2}$ (and therefore $2-\la \leq n-1$ for $n \geq 5$), that 
  \beaa
\notag
  && \int_{\Si_{t}} \frac{1}{(1+t+|q| )^{2-\la}(1+|q|)^2} \cdot \Big( |   \Lie_{Z^K} A  |^2 + |   \Lie_{Z^K} h  |^2  \Big) \cdot   w    \\
     &\leq& c(\ga) \cdot  \int_{\Si_{t}}  \frac{1}{(1+t+|q|)^{2-\la}}  \cdot  \Big( |  \derm( \Lie_{Z^K} A ) |^2 + | \derm(    \Lie_{Z^K} h )  |^2  \Big)  \cdot   w    \; .
 \eeaa
 As a result,
        \beaa
 \notag
&& \int_0^t \Big( \int_{\Si_{t}}  (1+t )^{1+\la}  \cdot |  \Lie_{Z^I}   ( g^{\la\mu} \derm_{\la}   \derm_{\mu}   A  )   |^2 \cdot w \Big) \cdot dt\, \\
   \notag
   &\les&  \int_0^t \Big( \int_{\Si_{t}} \sum_{|K|\leq |I|}  \Big( |    \derm (\Lie_{Z^K} A ) |^2 +  |    \derm (\Lie_{Z^K} h ) |^2 \Big) \cdot E (   \lfloor \frac{|I|}{2} \rfloor + \lfloor  \frac{n}{2} \rfloor  + 1)  \cdot       \frac{\eps  }{(1+t+|q|)^{2-\la}  } \cdot w \Big) \cdot dt  \\
     \notag
   && +  \int_0^t \Big( \int_{\Si_{t}} \sum_{|K|\leq |I|}  \Big( |   \Lie_{Z^K} A  |^2 +|   \Lie_{Z^K} h  |^2 \Big) \cdot E (   \lfloor \frac{|I|}{2} \rfloor + \lfloor  \frac{n}{2} \rfloor  + 1)  \cdot      \frac{\eps  }{(1+t+|q|)^{2-\la} \cdot ( 1+|q| )^{2} }  \cdot w \Big) \cdot dt \\
      \notag 
         &\les&  \int_0^t \Big( \int_{\Si_{t}} \sum_{|K|\leq |I|}  \Big( |    \derm (\Lie_{Z^K} A ) |^2 +  |    \derm (\Lie_{Z^K} h ) |^2 \Big) \cdot E (   \lfloor \frac{|I|}{2} \rfloor + \lfloor  \frac{n}{2} \rfloor  + 1)  \cdot       \frac{\eps  }{(1+t+|q|)^{2-\la}  } \cdot w \Big) \cdot dt  \\
     \notag
   && +  \int_0^t \Big( \int_{\Si_{t}} \sum_{|K|\leq |I|}  \Big( | \derm (   \Lie_{Z^K} A ) |^2 +| \derm (  \Lie_{Z^K} h )   |^2 \Big)  \cdot E (   \lfloor \frac{|I|}{2} \rfloor + \lfloor  \frac{n}{2} \rfloor  + 1)  \cdot      \frac{ c(\ga, \mu) \cdot\eps  }{(1+t+|q|)^{2-\la}  }  \cdot w \Big) \cdot dt \; .
      \notag 
   \eeaa
   
   \end{proof}
   
\begin{lemma}\label{estimateonthesoacetimeintegralforLiederivativesZofthesourcetermsformetrichusefulforagronwallinequalityinenergyestimate}
For $n \geq 5$\,, $\de = 0$\,, and for fields decaying fast enough at spatial infinity, such that for all time $t$\,, for $|K| \leq |I| $\,,
\bea
\notag
 \int_{\SSS^{n-1}} \lim_{r \to \infty} \Big( \frac{r^{n-1}}{(1+t+r)^{2-\la} \cdot (1+|q|) } w(q) \cdot   \Big( |   \Lie_{Z^K} A  |^2 + |   \Lie_{Z^K} h  |^2  \Big)  d\si^{n-1} (t ) &=& 0  \; , \\
\eea
then, for $\ga \neq 0$\,,
 \beaa
   \notag
 &&   \int_0^t \Big(  \int_{\Si_{t}}   (1+t )^{1+\la}   \cdot  | \Lie_{Z^I}   ( g^{\la\mu} \derm_{\la}   \derm_{\mu}    h ) |^2  \cdot   w \Big) \cdot dt   \\
     \notag
    &\les& c(\ga) \cdot E (   \lfloor \frac{|I|}{2} \rfloor + \lfloor  \frac{n}{2} \rfloor  + 1)  \\
    &&  \times   \int_0^t          \frac{\eps  }{(1+t)^{2-\la} \ }  \cdot  \Big(  \int_{\Si_{t}}   \sum_{|K|\leq |I|}  \Big( |    \derm (\Lie_{Z^K} A ) |^2 + |    \derm (\Lie_{Z^K} h ) |^2 \Big)   \cdot   w \Big) \cdot dt  \; .
   \eeaa

\end{lemma}

\begin{proof}

We showed in Lemma \ref{squareofthesourcetermsforthemetrichforngeq5withdeltaequalzeroforenergyestimate}, that for $n \geq 5$, $\de = 0$,
\beaa
   \notag
 &&  (1+t ) \cdot  | \Lie_{Z^I}   ( g^{\la\mu} \derm_{\la}   \derm_{\mu}    h ) |^2   \\
     \notag
   &\les&  \sum_{|K|\leq |I|}  \Big( |    \derm (\Lie_{Z^K} A ) |^2 + |    \derm (\Lie_{Z^K} h ) |^2 \Big) \cdot E (   \lfloor \frac{|I|}{2} \rfloor + \lfloor  \frac{n}{2} \rfloor  + 1)\\
 \notag
   && \cdot \Big(  \begin{cases}  \frac{\eps }{(1+t+|q|)^{3} (1+|q|)^{2+2\gamma}},\quad\text{when }\quad q>0,\\
           \notag
      \frac{\eps  }{(1+t+|q|)^{3} \cdot ( 1+|q| )^{} } \,\quad\text{when }\quad q<0 . \end{cases} \Big) \\
      \notag 
   &&+  \sum_{|K|\leq |I|}  \Big( |   \Lie_{Z^K} A  |^2 + |   \Lie_{Z^K} h  |^2 \Big) \cdot E (   \lfloor \frac{|I|}{2} \rfloor + \lfloor  \frac{n}{2} \rfloor  + 1)\\
 \notag
   && \cdot \Big(  \begin{cases}  \frac{\eps }{(1+t+|q|)^{7} (1+|q|)^{2+4\gamma}},\quad\text{when }\quad q>0,\\
           \notag
      \frac{\eps  }{(1+t+|q|)^{7}  } \,\quad\text{when }\quad q<0 . \end{cases} \Big) \; .
      \notag 
   \eeaa

Hence,
\beaa
   \notag
 &&  (1+t )^{1+\la} \cdot  | \Lie_{Z^I}   ( g^{\la\mu} \derm_{\la}   \derm_{\mu}    h ) |^2   \\
     \notag
   &\les&  \sum_{|K|\leq |I|}  \Big( |    \derm (\Lie_{Z^K} A ) |^2 + |    \derm (\Lie_{Z^K} h ) |^2 \Big) \cdot E (   \lfloor \frac{|I|}{2} \rfloor + \lfloor  \frac{n}{2} \rfloor  + 1)  \cdot      \frac{\eps  }{(1+t+|q|)^{3-\la} \cdot ( 1+|q| )^{} }  \\
      \notag 
   &&+  \sum_{|K|\leq |I|}  \Big( |   \Lie_{Z^K} A  |^2 + |   \Lie_{Z^K} h  |^2 \Big) \cdot E (   \lfloor \frac{|I|}{2} \rfloor + \lfloor  \frac{n}{2} \rfloor  + 1) \cdot        \frac{\eps  }{(1+t+|q|)^{7-\la}  } \\
         \notag
   &\les&  \sum_{|K|\leq |I|}  \Big( |    \derm (\Lie_{Z^K} A ) |^2 + |    \derm (\Lie_{Z^K} h ) |^2 \Big) \cdot E (   \lfloor \frac{|I|}{2} \rfloor + \lfloor  \frac{n}{2} \rfloor  + 1)  \cdot      \frac{\eps  }{(1+t+|q|)^{2-\la} \cdot ( 1+|q| )^{2} }  \\
      \notag 
   &&+  \sum_{|K|\leq |I|}  \Big( |   \Lie_{Z^K} A  |^2 + |   \Lie_{Z^K} h  |^2 \Big) \cdot E (   \lfloor \frac{|I|}{2} \rfloor + \lfloor  \frac{n}{2} \rfloor  + 1) \cdot        \frac{\eps  }{(1+t+|q|)^{5-\la} \cdot (1+|q|)^2  } \; .
      \notag 
   \eeaa

   Assuming that both $\Lie_{Z^K} A$ and $\Lie_{Z^K} h$ decay fast enough at spatial infinity for all time $t$, i.e. that
\bea
\notag
 \int_{\SSS^{n-1}} \lim_{r \to \infty} \Big( \frac{r^{n-1}}{(1+t+r)^{2-\la} \cdot (1+|q|) } w(q) \cdot   \Big( |   \Lie_{Z^K} A  |^2 + |   \Lie_{Z^K} h  |^2  \Big)  d\si^{n-1} (t ) &=& 0  \; . \\
\eea
Then, for $\ga \neq 0$ and $ 0 < \la \leq \frac{1}{2}$,  we have $0 \leq 2-\la \leq 4 \leq n-1$, and consequently, 
  \beaa
\notag
  && \int_{\Si_{t}} \frac{1}{(1+t+|q| )^{2-\la}(1+|q|)^2} \cdot \Big( |   \Lie_{Z^K} A  |^2 + |   \Lie_{Z^K} h  |^2  \Big) \cdot   w    \\
     &\leq& c(\ga) \cdot  \int_{\Si_{t}}  \frac{1}{(1+t+|q|)^{2-\la}}  \cdot  \Big( |  \derm( \Lie_{Z^K} A ) |^2 + | \derm(    \Lie_{Z^K} h )  |^2  \Big)  \cdot   w \; .   \\
 \eeaa
 As a result,
 \beaa
   \notag
 &&  \int_{\Si_{t}}   (1+t )^{1+\la}  \cdot  | \Lie_{Z^I}   ( g^{\la\mu} \derm_{\la}   \derm_{\mu}    h ) |^2  \cdot   w  \\
     \notag
   &\les&  \int_{\Si_{t}}   \sum_{|K|\leq |I|}  \Big( |    \derm (\Lie_{Z^K} A ) |^2 + |    \derm (\Lie_{Z^K} h ) |^2 \Big) \cdot E (   \lfloor \frac{|I|}{2} \rfloor + \lfloor  \frac{n}{2} \rfloor  + 1)  \cdot      \frac{\eps  }{(1+t+|q|)^{2-\la} \cdot ( 1+|q| )^{2} }  \cdot   w \\
      \notag 
   &&+ \int_{\Si_{t}}   \sum_{|K|\leq |I|}  \Big( |   \Lie_{Z^K} A  |^2 + |   \Lie_{Z^K} h  |^2 \Big) \cdot E (   \lfloor \frac{|I|}{2} \rfloor + \lfloor  \frac{n}{2} \rfloor  + 1) \cdot        \frac{\eps  }{(1+t+|q|)^{2-\la} \cdot (1+|q|)^2  }  \cdot   w \\
      \notag 
       &\les&   c(\ga) \cdot  E (   \lfloor \frac{|I|}{2} \rfloor + \lfloor  \frac{n}{2} \rfloor  + 1)  \cdot      \frac{\eps  }{(1+t)^{2-\la} \ }  \cdot  \int_{\Si_{t}}   \sum_{|K|\leq |I|}  \Big( |    \derm (\Lie_{Z^K} A ) |^2 + |    \derm (\Lie_{Z^K} h ) |^2 \Big)   \cdot   w \\
      \notag 
   &&+ \int_{\Si_{t}}   \sum_{|K|\leq |I|}  \Big( |   \derm ( \Lie_{Z^K} A ) |^2 + |  \derm ( \Lie_{Z^K} h )  |^2 \Big) \cdot E (   \lfloor \frac{|I|}{2} \rfloor + \lfloor  \frac{n}{2} \rfloor  + 1) \cdot        \frac{\eps  }{(1+t+|q|)^{4-\la}  }  \cdot   w \\
   \notag
    &\les&   c(\ga) \cdot  E (   \lfloor \frac{|I|}{2} \rfloor + \lfloor  \frac{n}{2} \rfloor  + 1)  \cdot      \frac{\eps  }{(1+t)^{2-\la} \ }  \cdot  \int_{\Si_{t}}   \sum_{|K|\leq |I|}  \Big( |    \derm (\Lie_{Z^K} A ) |^2 + |    \derm (\Lie_{Z^K} h ) |^2 \Big)   \cdot   w \; .
   \eeaa
   
   \end{proof}

\subsection{Grönwall type inequality on the energy for $n\geq 5$}\
   
       \begin{lemma}\label{TheGronwalltypeinequalityonenergyforngeq5}

Assume that $ H^{\mu\nu} = g^{\mu\nu}-m^{\mu\nu}$ satisfies 
\bea
| H| \leq C < \frac{1}{n} \; , \quad \text{where $n$ is the space dimension},
\eea
and assume that in wave coordinates $\{t, x^1, \ldots, x^n \}$\,, we have for $j$ running over spatial indices $\{ x^1, \ldots, x^n \}$\,, for all $|J| \leq |I|$, for all time $t$\,,
\bea
\lim_{r \to \infty }   \int_{\SSS^{n} }  g^{rj}  \cdot < \derm_t   (\Lie_{Z^J} A  ) ,\derm_j  (\Lie_{Z^J}  A  ) > \cdot w \cdot r^{n-1} d\si^{n-1} =0 \;, \\
\lim_{r \to \infty }   \int_{\SSS^{n} }  g^{tr}  \cdot < \derm_t   (\Lie_{Z^J} A   ) ,\derm_t  (\Lie_{Z^J} A  ) > \cdot w \cdot r^{n-1} d\si^{n-1} =0 \;,
\eea
\bea
\lim_{r \to \infty }   \int_{\SSS^{n} }  g^{rj}  \cdot < \derm_t   (\Lie_{Z^J} h  ) ,\derm_j  (\Lie_{Z^J}  h ) > \cdot w \cdot r^{n-1} d\si^{n-1} =0\;, \\
\lim_{r \to \infty }   \int_{\SSS^{n} }  g^{tr}  \cdot < \derm_t   (\Lie_{Z^J} h   ) ,\derm_t  (\Lie_{Z^J} h  ) > \cdot w \cdot r^{n-1} d\si^{n-1} =0 \;.
\eea

Let $n\geq 5$\,, $ 0 < \la \leq \frac{1}{2}$\,. Assume that $H$ is such that for all time $t$\,, 
\bea
 \int_{\SSS^{n-1}} \lim_{r \to \infty} \Big( \frac{r^{n-1}}{(1+t+r)^{3-\la} \cdot (1+|q|) } w(q) \cdot |H|^2  \Big)  d\si^{n-1} (t ) &=& 0 \; ,
\eea
and assume $h$ and $A$ are such that for $|J| \leq |I|$\,,
\bea
 \int_{\SSS^{n-1}} \lim_{r \to \infty} \Big( \frac{r^{n-1}}{ (1+|q|) } w(q) \cdot  | \Lie_{Z^J} h  |^2  \Big)  d\si^{n-1} (t ) &=& 0 \; ,
\eea
and
\bea
\notag
 \int_{\SSS^{n-1}} \lim_{r \to \infty} \Big( \frac{r^{n-1}}{(1+t+r)^{2-\la} \cdot (1+|q|) } w(q) \cdot   \Big( |   \Lie_{Z^J} A  |^2 + |   \Lie_{Z^J} h  |^2  \Big)  d\si^{n-1} (t ) &=& 0  \; . \\
\eea
Then, for $\ga \neq 0$ and for
\bea
\E_{  \lfloor \frac{|I|}{2} \rfloor + \lfloor  \frac{n}{2} \rfloor  + 1 } (\tau) \leq  E(  \lfloor \frac{|I|}{2} \rfloor + \lfloor  \frac{n}{2} \rfloor  + 1  ) \cdot \eps \, ,
\eea
for all $ 0 \leq \tau \leq t$\,, and for small  $E ( \lfloor \frac{|I|}{2} \rfloor  +  \lfloor  \frac{n}{2} \rfloor  +1) $\,, we have the following energy estimate
\bea
\notag
 &&  \E_{|I|}^2 (t)  \\
 \notag
  &\les&  \E_{|I|}^2 (0) +    c(\ga) \cdot E (   \lfloor \frac{|I|}{2} \rfloor + \lfloor  \frac{n}{2} \rfloor  + 1)  \cdot  C(|I|) \cdot  \int_0^t          \frac{\eps  }{(1+\tau)^{1+\la} \ }  \cdot  \E_{|I|}^2 (\tau)  \cdot d\tau \, , \\
\eea
where
\beaa
\E_{|I|} (\tau) :=  \sum_{|J|\leq |I|} \big( \|w^{1/2}   \derm ( \Lie_{Z^J} h^1   (t,\cdot) )  \|_{L^2} +  \|w^{1/2}   \derm ( \Lie_{Z^J}  A   (t,\cdot) )  \|_{L^2} \big) \, .
\eeaa
\end{lemma}

\begin{proof}
Based on what we have shown in Lemma \ref{the main energy estimate forAandhforn4}, under the assumption on the metric, we obtain
\bea\label{combiningenergyestimateandcommutatortermforngeq5}
\notag
 && \int_{\Si_{t}} \big( | \derm  (\Lie_{Z^I}  A) |^2 +   | \derm  (\Lie_{Z^I}  h) |^2 \big)  \cdot w \cdot dx^1 \ldots dx^n \\
 \notag
  &\les&  \int_{\Si_{t=0}} \big(  | \derm  (\Lie_{Z^I} A   ) |^2 +  | \derm  (\Lie_{Z^I} h   ) |^2 \big)  \cdot w \cdot dx^1 \ldots dx^n \\
  \notag
&& +  C ( |I| ) \cdot  c(\ga) \cdot E ( \lfloor \frac{|I|}{2} \rfloor+  \lfloor  \frac{n}{2} \rfloor  +1)   \\
\notag
&&  \times \int_0^t   \frac{1}{(1+\tau)^{1+\la}}     \cdot \Big(  \sum_{|J|\leq |I|}   \int_{\Si_{\tau}} \big(  | \derm ( \Lie_{Z^{J}} A ) |^2 + |\derm ( \Lie_{Z^{K}} h )  |^2 \big) \cdot   w  \cdot dx^1 \ldots dx^n   \Big) \cdot d\tau   \\
      \notag
&& + \int_0^t  \Big( \int_{\Si_{\tau}}   (1+t )^{1+\la}  \cdot \sum_{|K| \leq |I| }  \Big( | \Lie_{Z^K}  g^{\a\b} \derm_{\a}   \derm_{\b}    A |^2  \\
      \notag
&& + | \Lie_{Z^K}  g^{\a\b} \derm_{\a}   \derm_{\b}    h |^2 \Big)  \cdot w   \cdot dx^1 \ldots dx^n  \Big) \cdot d\tau  \; . \\
\eea
For $n \geq 5$, $\de = 0$, and for fields decaying fast enough at spatial infinity, such that for all time $t$,
\bea
\notag
 \int_{\SSS^{n-1}} \lim_{r \to \infty} \Big( \frac{r^{n-1}}{(1+t+r)^{2-\la} \cdot (1+|q|) } w(q) \cdot   \Big( |   \Lie_{Z^K} A  |^2 + |   \Lie_{Z^K} h  |^2  \Big)  d\si^{n-1} (t ) &=& 0  \; , \\
\eea
then, by Lemmas \ref{estimateonthesoacetimeintegralforLiederivativesZofthesourcetermsAusefulforagronwallinequalityinenergyestimate} and \ref{estimateonthesoacetimeintegralforLiederivativesZofthesourcetermsformetrichusefulforagronwallinequalityinenergyestimate}, for $\ga \neq 0$\,,
 \beaa
   \notag
 &&   \int_0^t \Big(  \int_{\Si_{\tau}}  (1+t )^{1+\la}   \cdot  \big( | \Lie_{Z^I}   ( g^{\la\mu} \derm_{\la}   \derm_{\mu}    A ) |^2 + | \Lie_{Z^I}   ( g^{\la\mu} \derm_{\la}   \derm_{\mu}    h ) |^2 \big) \cdot   w \Big) \cdot d\tau   \\
     \notag
    &\les& c(\ga) \cdot E (   \lfloor \frac{|I|}{2} \rfloor + \lfloor  \frac{n}{2} \rfloor  + 1)  \cdot    \int_0^t          \frac{\eps  }{(1+\tau)^{2-\la} \ }  \cdot  \Big(  \int_{\Si_{\tau}}   \sum_{|K|\leq |I|}  \Big( |    \derm (\Lie_{Z^K} A ) |^2 + |    \derm (\Lie_{Z^K} h ) |^2 \Big)   \cdot   w \Big) \cdot d\tau  \\
       &\les& c(\ga) \cdot E (   \lfloor \frac{|I|}{2} \rfloor + \lfloor  \frac{n}{2} \rfloor  + 1)  \cdot    \int_0^t          \frac{\eps  }{(1+\tau)^{2-\la} \ }  \cdot  \E_{|I|}^2 (\tau)  \cdot d\tau \;. \\
   \eeaa
   For $0<\la \leq \frac{1}{2}$, we have  $2-\la \geq 1+\la$ and therefore 
 \beaa
   \notag
 &&   \int_0^t \Big(  \int_{\Si_{\tau}}   (1+t )^{1+\la}  \cdot  \big( | \Lie_{Z^I}   ( g^{\la\mu} \derm_{\la}   \derm_{\mu}    A ) |^2 + | \Lie_{Z^I}   ( g^{\la\mu} \derm_{\la}   \derm_{\mu}    h ) |^2 \big) \cdot   w \Big) \cdot d\tau   \\
     \notag
       &\les& c(\ga) \cdot E (   \lfloor \frac{|I|}{2} \rfloor + \lfloor  \frac{n}{2} \rfloor  + 1)  \cdot    \int_0^t          \frac{\eps  }{(1+\tau)^{1+\la} \ }  \cdot  \E_{|I|}^2 (\tau)  \cdot d\tau  \;.\\
   \eeaa
Finally, injecting in \eqref{combiningenergyestimateandcommutatortermforngeq5}, we obtain,
\bea
\notag
 &&  \E_{|I|}^2 (t)  \\
 \notag
  &\les&  \E_{|I|}^2 (0) +    c(\ga) \cdot E (   \lfloor \frac{|I|}{2} \rfloor + \lfloor  \frac{n}{2} \rfloor  + 1)  \cdot  C(|I|) \cdot  \int_0^t          \frac{\eps  }{(1+\tau)^{1+\la} \ }  \cdot  \E_{|I|}^2 (\tau)  \cdot d\tau \; .
\eea
\end{proof}

\subsection{The proof of the theorem for $n\geq 5$}\

\begin{proposition}\label{Thetheoremofglobalstabilityanddecayforngeq 5}

Let $n \geq 5$ and let $N \geq 2 \lfloor  \frac{n}{2} \rfloor  + 2$\,. Assume that for all $I$, as in Definition \ref{DefinitionofZI}, with $|I| \leq N$, we have in wave coordinates $\{t, x^1, \ldots, x^n \}$\,, for $j$ running over spatial indices $\{ x^1, \ldots, x^n \}$\,, for time $t=0$\,,
\bea
\lim_{r \to \infty }   \int_{\SSS^{n} }  g^{rj}  \cdot < \derm_t   (\Lie_{Z^I} A  ) ,\derm_j  (\Lie_{Z^I}  A  ) > \cdot w \cdot r^{n-1} d\si^{n-1} =0 \;, \\
\lim_{r \to \infty }   \int_{\SSS^{n} }  g^{tr}  \cdot < \derm_t   (\Lie_{Z^I} A   ) ,\derm_t  (\Lie_{Z^I} A  ) > \cdot w \cdot r^{n-1} d\si^{n-1} =0 \;,
\eea
\bea
\lim_{r \to \infty }   \int_{\SSS^{n} }  g^{rj}  \cdot < \derm_t   (\Lie_{Z^I} h  ) ,\derm_j  (\Lie_{Z^I}  h ) > \cdot w \cdot r^{n-1} d\si^{n-1} =0\;, \\
\lim_{r \to \infty }   \int_{\SSS^{n} }  g^{tr}  \cdot < \derm_t   (\Lie_{Z^I} h   ) ,\derm_t  (\Lie_{Z^I} h  ) > \cdot w \cdot r^{n-1} d\si^{n-1} =0 \;.
\eea

Also, assume that for $\ga > 0$ and for $ 0 < \la \leq \frac{1}{2}$\,, we have for time $t=0$\,, 
\bea
 \int_{\SSS^{n-1}} \lim_{r \to \infty} \Big( \frac{r^{n-1}}{(1+r)^{3-\la} \cdot (1+|q|) } w(q) \cdot |H|^2  \Big)  d\si^{n-1}  &=& 0 \; , 
 \eea
 and for for all $|I| \leq N$,
 \bea
 \int_{\SSS^{n-1}} \lim_{r \to \infty} \Big( \frac{r^{n-1}}{ (1+|q|) } w(q) \cdot  | \Lie_{Z^I} h  |^2  \Big)  d\si^{n-1} &=& 0 \; ,
 \eea
\bea
\notag
 \int_{\SSS^{n-1}} \lim_{r \to \infty} \Big( \frac{r^{n-1}}{(1+r)^{2-\la} \cdot (1+|q|) } w(q) \cdot   \Big( |   \Lie_{Z^I} A  |^2 + |   \Lie_{Z^I} h  |^2  \Big)  d\si^{n-1}  &=& 0  \; . \\
\eea

Under these stated assumptions, for any constant $E(N)$ (that is there to bound $\E_N (t)$ in \eqref{Theboundontheglobaleenergybyaconstantthatwechoose}), there exist two constants, a constant $c_1$ that depends on $\ga > 0$ and on $n \geq 5$, and a constant $c_2$ (to bound $ \overline{\E_N} (0)$ defined in \eqref{definitionoftheenergynormforinitialdata}), that depends on $E(N)$\,, on $N \geq 2 \lfloor  \frac{n}{2} \rfloor  + 2$ and on $w$ (i.e. depends on $\gamma$), such that if
\bea\label{AssumptiononinitialdataforcertainfirstminimalLiederivativesglobalexistenceanddecay}
 \overline{\E}_{ ( \lfloor  \frac{n}{2} \rfloor  +1)} (0)  \leq c_1(\ga, n ) \; ,
\eea
and if
\bea\label{Assumptiononinitialdataforglobalexistenceanddecay}
 \overline{\E_N} (0) \leq c_2 (E(N), N, \ga) \; ,
\eea
then, we have for all time $t$\,, 
\bea\label{Theboundontheglobaleenergybyaconstantthatwechoose}
 \E_{N} (t) \leq E(N) \; ,
\eea
where
\beaa
\E_{N} (\tau) :=  \sum_{|J|\leq N} \big( \|w^{1/2}   \derm ( \Lie_{Z^J} h^1   (t,\cdot) )  \|_{L^2} +  \|w^{1/2}   \derm ( \Lie_{Z^J}  A   (t,\cdot) )  \|_{L^2} \big) \; .
\eeaa
Consequently, the initial value Cauchy problem for the Einstein Yang-Mills equations in the Lorenz gauge and in wave coordinates, that we defined in the set-up, will admit a global solution in time $t$ for initial data satisfying \eqref{AssumptiononinitialdataforcertainfirstminimalLiederivativesglobalexistenceanddecay} and \eqref{Assumptiononinitialdataforglobalexistenceanddecay}. As a result, in the Lorenz gauge, the Yang-Mills potential decays to zero and the metric decays to the Minkowski metric in wave coordinates. More precisely, for all $|I| \leq N -  \lfloor  \frac{n}{2} \rfloor  - 1$\;, we have,
 \beaa
 \notag
 |\derm  ( \Lie_{Z^I}  A ) (t,x)  |     + |\derm  ( \Lie_{Z^I}  h ) (t,x)  |     &\les& \begin{cases}   \frac{E ( N ) }{(1+t+|q|)^{\frac{(n-1)}{2}} (1+|q|)^{1+\gamma}},\quad\text{when }\quad q>0,\\
           \notag
        \frac{E ( N )  }{(1+t+|q|)^{\frac{(n-1)}{2}}(1+|q|)^{\frac{1}{2} }}  \,,\quad\text{when }\quad q<0 , \end{cases} \\
      \eeaa
and
 \beaa
 \notag
 |\Lie_{Z^I} A (t,x)  | + |\Lie_{Z^I}  h (t,x)  |  &\les& \begin{cases}   \frac{E ( N )  }{(1+t+|q|)^{\frac{(n-1)}{2}} (1+|q|)^{\gamma}},\quad\text{when }\quad q>0,\\
        \frac{E ( N )  \cdot (1+|q|)^{\frac{1}{2} }  }{(1+t+|q|)^{\frac{(n-1)}{2}}}  \,,\quad\text{when }\quad q<0 . \end{cases} \\
      \eeaa
\end{proposition}

\begin{proof}
We start with the bootstrap assumption explained in Section \ref{Theassumptionforthebootstrapandthenotation}. We have by then, thanks to Lemma \ref{aprioriestimateonthezeroLiederivativeZforgradiantbigHandbigH}, for $n\geq 5$ and for $\de = 0$\;, that
       \beaa
 \notag
|   H (t,x)  | &\les& \begin{cases}  c (\gamma)  \cdot  \frac{ \E_{ ( \lfloor  \frac{n}{2} \rfloor  +1)}}{ (1+ t + | q | )^{2 }  (1+| q |   )^{\ga}}  ,\quad\text{when }\quad q>0,\\
\notag
    \frac{ \E_{ ( \lfloor  \frac{n}{2} \rfloor  +1)}}{ (1+ t + | q | )^{2 }  } (1+| q |   )^{\frac{1}{2} }  , \,\quad\text{when }\quad q<0 . \end{cases} \\ 
     &\les&  c (\gamma)  \cdot   \E_{ ( \lfloor  \frac{n}{2} \rfloor  +1)}\\
               &\les&  c (\gamma)  \cdot   E (  \lfloor  \frac{n}{2} \rfloor  +1) \cdot \eps \\
                   && \text{(where we used that we chose $\de = 0$\;, see \eqref{delataqualtozero})}.\\
                &\les&  c (\gamma)  \cdot    E (  \lfloor  \frac{n}{2} \rfloor  +1)  \\
     && \text{(where we used that we chose $\eps \leq 1$\;, see \eqref{epsissmallerthan1},}\\
     &&\text{\; and in fact, in this paper, we chose $\eps =1$, see \eqref{epsequaltoone}).}  
    \eeaa
   
By choosing $E (  \lfloor  \frac{n}{2} \rfloor  +1)$ small enough, depending on $\ga$ and on $n$\, (which imposes the condition on the initial data by \eqref{theboostrapimposestheconditiononinitialdatasothatnonemptyset}), we have
\bea
 c (\gamma)  \cdot  E (  \lfloor  \frac{n}{2} \rfloor  +1) < \frac{1}{n} \; .
\eea
In addition, we claim that the decay assumptions on the initial data, stated in the proposition, will propagate in time, under the bootstrap assumption, since the fields satisfy a wave equation, and thus, they will be satisfied for all time $t$. Consequently, we could use Lemma \ref{TheGronwalltypeinequalityonenergyforngeq5}, where we fix  $ 0 < \la \leq \frac{1}{2}$ arbitrary, and we get that
\bea
\notag
 &&  \E_{|I|}^2 (t)  \\
 \notag
  &\leq& C \cdot \E_{|I|}^2 (0) +    c(\ga) \cdot E (   \lfloor \frac{|I|}{2} \rfloor + \lfloor  \frac{n}{2} \rfloor  + 1)  \cdot  C(|I|) \cdot  \int_0^t          \frac{\eps  }{(1+\tau)^{1+\la} \ }  \cdot  \E_{|I|}^2 (\tau)  \cdot d\tau \;.
\eea
Now, using the celebrated Grönwall lemma, we get
\bea
\notag
 &&  \E_{|I|}^2 (t)  \leq C\cdot \E_{|I|}^2 (0) \cdot \exp \Big(  \int_0^t    c(\ga) \cdot E (   \lfloor \frac{|I|}{2} \rfloor + \lfloor  \frac{n}{2} \rfloor  + 1)  \cdot  C(|I|) \cdot   \eps \cdot       \frac{1 }{(1+\tau)^{1+\la} \ }  \cdot d\tau     \Big) \\
\notag
 & \leq& C \cdot \E_{|I|}^2 (0) \cdot \exp \Big(     c(\ga) \cdot E (   \lfloor \frac{|I|}{2} \rfloor + \lfloor  \frac{n}{2} \rfloor  + 1)  \cdot  C(|I|) \cdot   \eps  \cdot \Big[  \frac{-1 }{\la (1+\tau)^{\la} }   \Big]^{\infty}_{0}      \Big) \\
  & \leq& C \cdot \E_{|I|}^2 (0) \cdot \exp \Big(    c(\ga) \cdot E (   \lfloor \frac{|I|}{2} \rfloor + \lfloor  \frac{n}{2} \rfloor  + 1)  \cdot  C(|I|) \cdot   \eps  \cdot \frac{1 }{\la }     \Big) \; ,
\eea
which also leads to, using that we chose $\eps \leq 1$ and that $E(k) \leq 1$ (see \eqref{epsissmallerthan1} and \eqref{assumptionontheconstantboundsforenergy}), that
\beaa
  \E_{|I|} (t)  & \leq& C \cdot \E_{|I|} (0) \cdot \exp \Big(    c(\ga) \cdot E (   \lfloor \frac{|I|}{2} \rfloor + \lfloor  \frac{n}{2} \rfloor  + 1)  \cdot  C(|I|) \cdot   \eps  \cdot \frac{1 }{\la }     \Big) \\
  & \leq& C \cdot \E_{|I|} (0) \cdot \exp \Big(    c(\ga)   \cdot  C(|I|)   \cdot \frac{1 }{\la }     \Big) \; .
\eeaa

Thus, choosing an initial data such that the energy norm defined in \eqref{definitionoftheenergynormforinitialdata} satsfies
\bea
 \overline{\E_{|I|}} (0) \leq  \frac{1}{2\cdot C  \cdot \exp \Big(    c(\ga)  \cdot  C(|I|)   \cdot \frac{1 }{\la }     \Big)} \cdot  E (   \lfloor \frac{|I|}{2} \rfloor + \lfloor  \frac{n}{2} \rfloor  + 1) \; ,
\eea
implies that 
\bea
 \E_{|I|} (0) \leq  \frac{1}{2\cdot C  \cdot \exp \Big(    c(\ga)  \cdot  C(|I|)   \cdot \frac{1 }{\la }     \Big)} \cdot  E (   \lfloor \frac{|I|}{2} \rfloor + \lfloor  \frac{n}{2} \rfloor  + 1) \; .
\eea
This leads to
\beaa
\notag
 &&  \E_{|I|} (t)  \leq \frac{1}{2} \cdot E (   \lfloor \frac{|I|}{2} \rfloor + \lfloor  \frac{n}{2} \rfloor  + 1)  \; .
\eeaa
However, for $|I| \geq  \lfloor \frac{|I|}{2} \rfloor + \lfloor  \frac{n}{2} \rfloor  + 1$, which means for $\frac{|I|}{2} \geq  \lfloor  \frac{n}{2} \rfloor  + 1$, we have
\beaa
  \E_{   \lfloor \frac{|I|}{2} \rfloor + \lfloor  \frac{n}{2} \rfloor  + 1}  (t)  \leq   \E_{|I|} (0)  \; .
 \eeaa
 Thus,
 \bea
  &&  \E_{   \lfloor \frac{|I|}{2} \rfloor + \lfloor  \frac{n}{2} \rfloor  + 1}  (t)  \leq \frac{1}{2} \cdot E (   \lfloor \frac{|I|}{2} \rfloor + \lfloor  \frac{n}{2} \rfloor  + 1)    \; .
  \eea
This shows that the estimate $  \E_{   \lfloor \frac{|I|}{2} \rfloor + \lfloor  \frac{n}{2} \rfloor  + 1}  (t)  \leq E (   \lfloor \frac{|I|}{2} \rfloor + \lfloor  \frac{n}{2} \rfloor  + 1)    $\, is in fact a true estimate and therefore, we can close the bootstrap argument explained in Section \ref{Thebootstrapargumentandnotationonboundingtheenergy}, for  $  \E_{   \lfloor \frac{|I|}{2} \rfloor + \lfloor  \frac{n}{2} \rfloor  + 1}  (t)$\,, with $\eps=1$ and $\de=0$. For this, we have used the condition that
\beaa
|I| \geq  \lfloor \frac{|I|}{2} \rfloor + \lfloor  \frac{n}{2} \rfloor  + 1\;,
\eeaa
 which imposes that $|I| \geq 2 \lfloor  \frac{n}{2} \rfloor  + 2$, and we also got that
\bea
 &&  \E_{|I|} (t)  \leq \frac{1}{2} \cdot E (   \lfloor \frac{|I|}{2} \rfloor + \lfloor  \frac{n}{2} \rfloor  + 1) \; .
\eea
This in turn gives, using Lemmas \ref{aprioriestimatefrombootstraponzerothderivativeofAandh1} and \ref{aprioriestimatesongradientoftheLiederivativesofthefields}, the stated decay estimates on the fields.
\end{proof}

\end{document}